\documentclass[reqno]{amsart}

\usepackage{a4wide}
\usepackage{amssymb}
\usepackage{graphicx}
\usepackage{hyperref}
\usepackage[mathscr]{euscript}
\usepackage{mathtools}
\usepackage{subfigure}
\usepackage{calc}
\usepackage{tikz}
\usetikzlibrary{cd,arrows,matrix}
\usepackage{newtxtext}
\usepackage{newtxmath}
\usepackage{anyfontsize}
\usepackage{bm}

\usepackage[english]{babel}

\usepackage[makeroom]{cancel}

\hypersetup{colorlinks=true,linkcolor=blue}

\numberwithin{equation}{section}

\theoremstyle{plain}
\newtheorem{theorem}{Theorem}[subsection]
\newtheorem{lemma}[theorem]{Lemma}
\newtheorem{proposition}[theorem]{Proposition}
\newtheorem{corollary}[theorem]{Corollary}

\newtheorem{proposition/definition}[theorem]{Proposition/Definition}
\newtheorem{theorem/definition}[theorem]{Theorem/Definition}
\newtheorem*{bigtheorem*}{Main Theorem}

\theoremstyle{definition}
\newtheorem{definition}[theorem]{Definition}
\newtheorem{example}[theorem]{Example}
\newtheorem{example/definition}[theorem]{Example/Definition}
\newtheorem{assumption}[theorem]{Assumption}
\newtheorem{notation}[theorem]{Notation}

\theoremstyle{remark}
\newtheorem{remark}[theorem]{Remark}

\DeclareMathOperator{\augnumber}{aug}

\DeclareMathOperator{\cone}{Cone}
\DeclareMathOperator{\diag}{Diag}

\DeclareMathOperator{\End}{End}

\DeclareMathOperator{\GNR}{GNR}
\let\hom\relax\DeclareMathOperator{\hom}{Hom}
\DeclareMathOperator{\im}{Image}
\DeclareMathOperator{\MC}{MC}

\DeclareMathOperator{\NR}{NR}

\DeclareMathOperator{\Stab}{Stab}

\DeclareMathOperator{\identity}{Id}

\DeclareMathOperator{\Res}{Res^\mathsf{Ng}}

\DeclareMathOperator{\sgn}{sgn}

\DeclareMathOperator{\val}{val}

\newcommand{\CC}{\mathbb{C}}
\newcommand{\FF}{\mathbb{F}}
\newcommand{\KK}{\mathbb{K}}
\newcommand{\NN}{\mathbb{N}}
\newcommand{\RR}{\mathbb{R}}
\newcommand{\ZZ}{\mathbb{Z}}

\newcommand{\bfR}{\mathbf{R}}

\newcommand{\bfU}{\mathbf{U}}

\newcommand{\bff}{\mathbf{f}}

\newcommand{\bfi}{\mathbf{i}}

\newcommand{\bfx}{\mathbf{x}}

\newcommand{\bfinfty}{\boldsymbol{\infty}}

\newcommand{\bfgraf}{\mathbf{\graf}}
\newcommand{\bfzero}{\mathbf{0}}
\newcommand{\bfmu}{{\bm{\mu}}}

\newcommand{\bftau}{{\bm{\tau}}}
\newcommand{\bfpi}{{\bm{\pi}}}

\newcommand{\cA}{\mathcal{A}}

\newcommand{\cC}{\mathcal{C}}

\newcommand{\cE}{\mathcal{E}}

\newcommand{\cP}{\mathcal{P}}

\newcommand{\cT}{\mathcal{T}}

\newcommand{\cV}{\mathcal{V}}

\newcommand{\scrT}{\mathscr{T}}

\newcommand{\sfA}{\mathsf{A}}

\newcommand{\sfG}{\mathsf{G}}

\newcommand{\sfI}{\mathsf{I}}

\newcommand{\sfT}{\mathsf{T}}

\DeclareSymbolFont{sfletters}{OT1}{cmss}{m}{n}
\DeclareMathSymbol{\sTheta}{\mathord}{sfletters}{"02}

\newcommand{\sfc}{\mathsf{c}}

\newcommand{\CE}{\mathsf{CE}}

\newcommand{\std}{\mathsf{std}}
\newcommand{\front}{\mathsf{fr}}
\newcommand{\lag}{\mathsf{Lag}}

\newcommand{\ttB}{\mathtt{B}}
\newcommand{\ttE}{\mathtt{E}}

\newcommand{\ttV}{\mathtt{V}}

\newcommand{\ttv}{\mathtt{v}}

\newcommand{\fd}{\mathfrak{d}}

\renewcommand{\bar}{\overline}
\renewcommand{\emptyset}{\varnothing}
\renewcommand{\hat}{\widehat}
\renewcommand{\tilde}{\widetilde}

\newcommand{\Zmod}[1]{\ZZ/#1\ZZ}

\newcommand{\graf}{\Gamma}
\newcommand{\differential}{\partial}
\newcommand{\boundary}{\partial}
\newcommand{\field}{\KK}
\newcommand{\ring}{\mathbb{Z}}

\newcommand{\Left}{\mathsf{L}}
\newcommand{\Right}{\mathsf{R}}

\newcommand{\grading}{\ZZ}

\newcommand{\vg}{\xi}

\newcommand{\modulispace}{\mathcal{M}}

\newcommand{\homotopic}{\simeq}
\newcommand{\isomorphic}{\cong}
\newcommand{\relation}{\sim}

\newcommand{\RM}[1]{\left(\mathrm{#1}\right)}

\newcommand{\dgafont}{\mathcal}
\newcommand{\algfont}{\mathsf}

\newcommand{\categoryfont}{\mathscr}

\newcommand{\alg}{\algfont{A}}
\newcommand{\dga}{\dgafont{A}}

\newcommand{\aug}{\mathsf{Aug}}

\newcommand{\ruling}{\rho}

\newcommand{\LT}{\categoryfont{LG}}
\newcommand{\BLT}{\categoryfont{BLG}}

\newcommand{\Sh}{\categoryfont{S}\mathsf{h}}

\newcommand{\valency}{n}

\newcommand*{\DecorationScale}{0.3}

\DeclareRobustCommand*{\crossing}{{
\vcenter{\hbox{\begin{tikzpicture}[scale=\DecorationScale]
\draw[thick,-](-1,1)--(1,-1);
\draw[thick,-](-1,-1)--(1,1);
\end{tikzpicture}}}
}}

\DeclareRobustCommand*{\crossingpositive}{{
\vcenter{\hbox{\begin{tikzpicture}[scale=\DecorationScale]
\draw[thick,-](-1,1)--(1,-1);
\draw[thick,-](-1,-1)--(-0.3,-0.3);
\draw[thick,-](0.3,0.3)--(1,1);
\end{tikzpicture}}}
}}

\DeclareRobustCommand*{\crossingnegative}{{
\vcenter{\hbox{\begin{tikzpicture}[scale=\DecorationScale]
\draw[thick,-](-1,-1)--(1,1);
\draw[thick,-](-1,1)--(-0.3,0.3);
\draw[thick,-](0.3,-0.3)--(1,-1);
\end{tikzpicture}}}
}}

\DeclareRobustCommand*{\crossinghorizontal}{{
\vcenter{\hbox{\begin{tikzpicture}[scale=\DecorationScale]
\draw[thick,-](-1,1) arc (225:315:1.414);
\draw[thick,-](-1,-1) arc (-225:-315:1.414);
\end{tikzpicture}}}
}}

\DeclareRobustCommand*{\rightkink}{{
\vcenter{\hbox{\begin{tikzpicture}[scale=\DecorationScale]
\draw[thick,-](-1,1) -- (0.3, -0.3) arc (-135: 135: 0.4242);
\draw[thick,-](-0.3,-0.3)--(-1,-1);
\end{tikzpicture}}}
}}

\DeclareRobustCommand*{\leftarc}{{
\vcenter{\hbox{\begin{tikzpicture}[scale=\DecorationScale]
\draw[thick,-](0,1) arc ( 90: 275: 1);
\end{tikzpicture}}}
}}

\DeclareRobustCommand*{\Lcusp}{{
\vcenter{\hbox{\begin{tikzpicture}[scale=\DecorationScale]
\draw[thick,-] (1,1) arc (-45:-90:2.414) arc (90:45:2.414);
\end{tikzpicture}}}
}}

\DeclareRobustCommand*{\Rcusp}{{
\vcenter{\hbox{\begin{tikzpicture}[scale=\DecorationScale]
\draw[thick,-] (-1,1) arc (-135:-90:2.414) arc (90:135:2.414);
\end{tikzpicture}}}
}}

\title{Augmentations and ruling polynomials for Legendrian graphs}

\author[B. H. An]{Byung Hee An}
\email{anbyhee@ibs.re.kr}
\address{Center for Geometry and Physics, Institute for Basic Science (IBS), Pohang 37673, Korea}

\author[Y. Bae]{Youngjin Bae}
\email{ybae@kurims.kyoto-u.ac.jp}
\address{Research Institute for Mathematical Sciences, Kyoto University, Kyoto 606-8317, Japan}

\author[T. Su]{Tao Su}
\email{taosu@dma.ens.fr}
\address{Department of Mathematics, \'{E}cole Normale Sup\'{e}rieure, Paris, France}


\begin{document}

\begin{abstract}
In this article, associated to a (bordered) Legendrian graph, we study and show the equivalence between two Legendrian isotopy invariants: augmentation number via point-counting over a finite field, for the augmentation variety of the associated Chekanov-Eliashberg differential graded algebra, and ruling polynomial via combinatorics of the decompositions of the associated front projection.   
\end{abstract}

\maketitle

\tableofcontents

\section{Introduction}

The theory of Legendrian knots has been extremely fruitful. 
Classical Legendrian isotopy invariants include rotation numbers and Thurston-Bennequin numbers \cite{Gei2008}.
However, the much more powerful invariants are Chekanov-Eliashberg differential graded algebras (CE DGAs), which distinguish a pair of Legendrian knots with the same classical invariants \cite{Chekanov2002}.
The CE DGAs admit higher dimensional generalizations, known as Legendrian contact homology differential graded algebras (LCH DGAs), associated to any pairs of a Legendrian submanifold contained in a contact manifold. The LCH DGAs fit into the more general framework of symplectic field theory \cite{EGH2000}, and have a Morse-Floer-Fukaya theoretic definition in terms of counting of holomorphic disks. There is an advantage in the case of Legendrian knots contained in the standard contact three-space $\RR_\std^3$. That is, the LCH DGAs in this case also admit a combinatorial description \cite{Chekanov2002,ENS2002}, allowing concrete computations. The LCH DGAs are Legendrian isotopy invariants, up to homotopy equivalence.

On the other hand, there are natural singular generalizations of Legendrian knots, i.e. singular Legendrian 1-dimensional submanifolds in contact 3-manifolds, called \emph{Legendrian graphs} \cite{OP2012}. The motivation for studying these objects is quite natural. For example, they have already appeared in the proof of the Giroux correspondence \cite{Gir2002,Etn2006}, a bijection between the set of oriented contact structures up to isotopy and the set of open book decompositions up to positive stabilization, on compact oriented 3-manifolds. Also, Legendrian graphs are used in  \cite{EF2009} to show that the classical invariants are complete invariants for topologically trivial Legendrian knots.
Recently, they also appear in the construction and study of Weinstein manifolds whose skeleton are singular, especially the arboreal type, see \cite{Nad2017}.
 
In this article, we study the Legendrian isotopy invariants of Legendrian graphs, or more generally, \emph{bordered Legendrian graphs}. A bordered Legendrian graph $\cT=(T_\Left\rightarrow T\leftarrow T_\Right)$ with a Maslov potential $\bfmu=(\mu_\Left, \mu, \mu_\Right)$ is a singular Legendrian 1-submanifold $T$ in $J^1\bfU\isomorphic \bfU\times\RR_{yz}^2$, which is `transverse' to the left and right boundaries $\boundary(J^1 \bfU)$ along the borders $T_\Left$ and $T_\Right$, respectively.
Here, $\bfU\subset\RR_x$ is a closed interval. 

Similar to Legendrian knots, there are Legendrian isotopy invariants, the LCH DGAs 
\[
A^\CE(\cT,\mu)=\left(A^\CE(T_\Left,\mu_\Left)\leftarrow A^\CE(T,\mu)\rightarrow A^\CE(T_\Right,\mu_\Right)\right)
\]
for bordered Legendrian graphs $(\cT,\bfmu)$ defined combinatorially as in \cite{AB2018}. 
A standard way to extract numerical invariants from the LCH DGAs $A^\CE(\cT,\bfmu)$ is via counting the ``functor-of-points'' for $A^\CE(\cT,\bfmu)$ over a finite field $\FF_q$.
More specifically, for any base field $\field$, we consider the set of augmentations, i.e. DGA maps, $\epsilon:A^\CE(T,\mu)\to \field$ onto $\field$ with zero differential, whose restrictions $\epsilon_*:A^\CE(T_*,{\mu_*})\to \field$ for $*=\Left$ and $\Right$ are specified by boundary conditions. This defines an algebraic variety, called \emph{augmentation variety} (with boundary conditions), whose normalized point-counting over $\FF_q$ defines a numerical Legendrian isotopy invariant, called the \emph{augmentation number} and denoted by $\augnumber(\cT,\bfmu;\ruling_\Left,\ruling_\Right;\FF_q)$.

Our main result in this article solves the question of counting the augmentation number. More precisely, we generalize the results in \cite{HR2015,Su2017} to show that the augmentation numbers are computed by ruling polynomials of $(\cT,\bfmu)$, defined via the combinatorics of decompositions of the front projection $\pi_{xz}(\cT)$ as follows:

\begin{bigtheorem*}[Theorem~\ref{theorem:augmentation number is a ruling polynomial}, Corollary~\ref{cor:gluing property of ruling polynomials}]\label{thm:count augmentations}
For a bordered Legendrian graph $\cT$ with a Maslov potential $\bfmu$, let $\ruling_\Left\in\NR(T_\Left,\mu_\Left)$ and $\ruling_\Right\in\NR(T_\Right,\mu_\Right)$ be two boundary conditions,i.e., normal rulings. Then the following two Legendrian isotopy invariants coincide:
\begin{equation*}
\augnumber(\cT,\bfmu;\ruling_\Left,\ruling_\Right;\FF_q)=q^{-\frac{d+\hat{B}}{2}}z^{\hat{B}}\langle \ruling_\Left|R(\cT,\bfmu;q,z)|\ruling_\Right\rangle.
\end{equation*}
Here, $\langle \ruling_\Left|R(\cT,\bfmu;q,z)|\ruling_\Right\rangle\in\ZZ[q^{\pm\frac{1}{2}},z^{\pm 1}]$ is the two-variable ruling polynomial for $(\cT,\bfmu)$ with boundary conditions $(\ruling_\Left,\ruling_\Right)$, $d\coloneqq\max\deg_z \langle \ruling_\Left|R(\cT,\bfmu;z^2,z)|\ruling_\Right\rangle$. In the formula, we take $z=q^{\frac{1}{2}}-q^{-\frac{1}{2}}$, and 
\[
\hat{B}\coloneqq B+\sum_{v\in V(\cT)}\frac{\val(v)}{2}
\]
counts the number of ``generalized'' base points in $T$.

Moreover, the ruling polynomial satisfies the \emph{gluing property}, that is, for the concatenation 
$(\cT,\bfmu)=(\cT^1,\bfmu^1)\cdot(\cT^2,\bfmu^2)$
of two bordered Legendrian graphs with Maslov potentials with $(T^1_\Right,\mu^1_\Right)=(T^2_\Left,\mu^2_\Left)$, we have
\[
\langle \ruling_\Left|R(\cT,\bfmu;q,z)|\ruling_\Right\rangle=\sum_{\ruling\in\NR(T^1_\Right, \mu^1_\Right)}\langle \ruling_\Left|R(\cT^1,\bfmu^1;q,z)|\ruling\rangle\cdot \langle \ruling|R(\cT^2,\bfmu^2;q,z)|\ruling_\Right\rangle.
\]
\end{bigtheorem*}

By the second statement above, we tend to view the augmentation number $\augnumber(\cT,\bfmu;\ruling_\Left,\ruling_\Right;\FF_q)$ or the ruling polynomial $\langle \ruling_\Left|R(\cT,\bfmu;q,z)|\ruling_\Right\rangle$ as the Legendrian analogue of Jones polynomials \cite{Jon1985} in topological knot theory, which is well-known to fit into a 3D TQFT. 
Hence, we may consider them interesting problems to build a ``contact version of 3D TQFT'' for Legendrian knots and graphs in contact 3-manifolds.

Another purpose of this article is to set up a foundation for our next article \cite{ABS2019aug_sheaf}, in which we show the slogan``augmentations are sheaves'' for Legendrian graphs that generalizes the case of Legendrian knots \cite{NRSSZ2015}. 
More specifically, for each closed interval $\bfU\subset\RR_x$ and $M=\bfU\times\RR_z$, we first identify $J^1\bfU$ with the contact submanifold $T^{\infty,-}M$ of the co-sphere bundle $T^\infty M$ which consists of covectors which are negative in the $z$-directions. 
Now for a bordered Legendrian graph $\cT$ in $J^1\bfU$ with a Maslov potential $\bfmu$, we define $\Sh_{\cT,\bfmu}(M;\field)$ to be the DG category of constructible sheaves on $M$ with micro-support at infinity contained in $\cT$, and define $\cC_1(\cT,\bfmu;\field)$ to be the full subcategory of $\Sh_{\cT,\bfmu}(M;\field)$ whose objects are microlocal rank $1$ sheaves with acyclic stalks for large enough $|z|$.

\begin{theorem}\cite{ABS2019aug_sheaf}\label{thm:augs are sheaves}
There is a well-defined diagram of unital $A_{\infty}$-categories:
\begin{equation*}
\aug_+(\cT,\bfmu;\field)\coloneqq(\aug_+(T_\Left,\mu_\Left;\field)\leftarrow\aug_+(T,\mu;\field)\rightarrow\aug_+(T_\Right,\mu_\Right;\field)),
\end{equation*}
where the objects of $\aug_+(\cT,\bfmu;\field)$ are the (acyclic) augmentations of the LCH DGA $A^\CE(T,\mu)$. The diagram is invariant under Legendrian isotopy up to $A_{\infty}$-equivalence. Moreover, there is an $A_{\infty}$-equivalence between the two diagrams of unital $A_{\infty}$-categories:
\begin{equation*}
\aug_+(\cT,\bfmu;\field)\stackrel{\homotopic}\longrightarrow\cC_1(\cT,\bfmu;\field).
\end{equation*}
\end{theorem}

Here, $\cC_1(\cT,\bfmu;\field)\coloneqq(\cC_1(T_\Left,\mu_\Left;\field)\leftarrow\cC_1(T,\mu;\field)\rightarrow\cC_1(T_\Right,\mu_\Right;\field))$ is a diagram of DG-categories of constructible sheaves, with each category defined as above. 
In particular, when $T=\Lambda\subset T^{\infty,-}M$ is a Legendrian graph, we get an $A_{\infty}$-equivalence $\aug_+(\Lambda,\mu;\field)\stackrel{\homotopic}\longrightarrow\cC_1(\Lambda,\mu;\field)$. Namely, the slogan ``augmentations are sheaves'' holds in the singular case. 

As a consequence of Theorem~\ref{thm:augs are sheaves}, up to a normalization, our main theorem 
also gives the point-counting of the moduli space $\modulispace_1(\cT,\bfmu;\field)$ of sheaves in $\cC_1(\cT,\bfmu;\field)$ (with boundary conditions) over a finite field $\field=\FF_q$. By a theorem of N.~Katz \cite[Appendix]{HRV2008}, this is also the same as giving the $E$-polynomial of $\modulispace_1(\cT,\bfmu;\CC)$. 
Notice that $\modulispace_1(\cT,\bfmu;\field)$ is an analogue of the wild character varieties over a punctured Riemann surface, the latter is well-known as the Betti moduli space in non-abelian Hodge theory, the study of which remains a central subject of current research (for example, see \cite{HRV2008}). Thus, it would be interesting to explore the Hodge theory of $\modulispace_1(\cT,\bfmu;\field)$ as well. See also Remark \ref{rem:Hodge str} for more information.

\addtocontents{toc}{\protect\setcounter{tocdepth}{1}}
\subsection*{Organization}
The article is organized as follows. In Section \ref{section:preliminaries}, we give the basic backgrounds on bordered Legendrian graphs $\cT=(T_\Left\rightarrow T\leftarrow T_\Right)$ with a Maslov potential $\bfmu=(\mu_\Left,\mu,\mu_\Right)$. 
In Section \ref{section:bordered graphs}, we introduce bordered LCH DGAs $A^\CE(\cT,\bfmu)=(A^\CE(T_\Left,\mu_\Left)\rightarrow A^\CE(T,\mu)\leftarrow A^\CE(T_\Right,\mu_\Right))$. 
The key result in this section is to show the invariance of the diagrams of LCH DGAs for bordered Legendrian graphs, up to (generalized) stabilizations.

In Section \ref{section:augmentations}, we introduce augmentations and augmentation varieties (with boundary conditions) $\aug(\cT,\bfmu;\epsilon_\Left,\ruling_\Right;\field)$ for the LCH DGAs $A^\CE(\cT,\bfmu)$. The normalized point-counting over a finite field $\FF_q$ of $\aug(\cT,\bfmu;\epsilon_L,\ruling_R;\FF_q)$ then defines the augmentation number $\augnumber(\cT,\bfmu;\ruling_\Left,\ruling_\Right;\FF_q)$. The key ingredient in this section is to show the ``invariance'' of the augmentation varieties $\aug(\cT,\bfmu;\epsilon_\Left,\ruling_\Right;\field)$, which implies the invariance of augmentation numbers $\augnumber(\cT,\bfmu;\ruling_\Left,\ruling_\Right;\FF_q)$. Unlike the Legendrian knot case, this is not a trivial task, due to the involvement of generalized stabilization of LCH DGAs for (bordered) Legendrian graphs. 

In Section \ref{section:rulings}, we introduce and show the invariance of normal rulings $\bfR(\cT,\bfmu;\ruling_\Left,\ruling_\Right)$ and ruling polynomials (with boundary conditions) $\langle \ruling_\Left|R(\cT,\bfmu;q,z)|\ruling_\Right\rangle$ for bordered Legendrian graphs $(\cT,\bfmu)$ with Maslov potentials. The gluing property of ruling polynomials will follow directly from the definition. 
We then show our main result that augmentation numbers are ruling polynomials up to a normalization. The proof is via a tangle approach as in \cite{Su2017}.

\addtocontents{toc}{\protect\setcounter{tocdepth}{1}}
\subsection*{Acknowledgements}
We would like to thank RIMS in Japan, IBS-CGP in South Korea, and ENS Paris - CNRS in France for supporting the visits, where much of this project was developed. 
The first author is supported by IBS-R003-D1.
The second author is supported by Japan Society for the Promotion of Science (JSPS) International Research Fellowship Program, and he thanks Research Institute for Mathematical Sciences, Kyoto University for its warm hospitality.
The third author is supported by ANR-15-CE40-0007. He would like to thank St\'{e}phane Guillermou for the invitation to visit CRM, Montreal, where part of this project was improved. In addition, he is grateful to Vivek Shende, David Nadler, and Lehnard Ng for the help in his early career.

\addtocontents{toc}{\protect\setcounter{tocdepth}{2}}

\section{Preliminaries}
\label{section:preliminaries}

\subsection{Bordered Legendrian graphs}
A graph is a finite regular one dimensional CW complex, whose 0-cells and closed 1-cells are called vertices and edges.
For each vertex $\ttv$, a \emph{half-edge} at $\ttv$ is a small enough restriction of an edge adjacent to $\ttv$.
Then as usual, the \emph{valency} of $\ttv$ is the number of half-edges at $v$ and denoted by $\val(\ttv)$.

A (\emph{based}) \emph{bordered graph} $\graf=(\ttV,\ttV_\Left,\ttV_\Right,\ttB,\ttE)$ of type $(n_\Left,n_\Right)$ consists of the following data:
\begin{itemize}
\item the pair $(\ttV\amalg\ttV_\Left\amalg\ttV_\Right\amalg\ttB,\ttE)$ defines a graph $|\graf|$,
\item each $b\in\ttB$ of $|\graf|$ is bivalent, and
\item two disjoint subsets $\ttV_\Left$ and $\ttV_\Right$ consist of $n_\Left$ and $n_\Right$ univalent vertices of $|\graf|$.
\end{itemize}

Elements in $\ttV,\ttV_\Left,\ttV_\Right,\ttB$ and $\ttE$ will be called \emph{vertices}, \emph{left borders}, \emph{right borders}, \emph{base points} and \emph{edges}, respectively.
The \emph{interior} $\mathring{\graf}$ of $\graf$ is define to be the complement of $\ttV_\Left\amalg\ttV_\Right$.

\begin{notation}
In order to emphasize the border structures, we will denote the bordered graph $\graf$ as
\[
\bfgraf=(\graf_\Left\to \graf \leftarrow \graf_\Right),
\]
where $\graf_\Left$ and $\graf_\Right$ are defined by $\ttV_\Left$ and $\ttV_\Right$ by, respectively, and both arrows are inclusions.

From now on, we mean by a \emph{graph} a bordered graph with empty borders $(\emptyset\to \graf\leftarrow\emptyset)$, which will be denoted simply by $\graf$.
\end{notation}

For a closed interval $U=[x_\Left,x_\Right]\subset \RR_x$, let the \emph{bordered one-jet space} $J^1\bfU=(J^1U_\Left\to J^1U\leftarrow J^1U_\Right)$ be the one-jet space $J^1 U\coloneqq(U\times\RR_{yz}, dz-ydx)\subset J^1\RR_x= (\RR^3_{xyz}, dz-ydx)$ together with two contact submanifolds
\[
J^1 U_\Left\coloneqq \left(\{x_\Left\}\times\RR_z,dz\right)\quad\text{ and }\quad
J^1 U_\Right\coloneqq \left(\{x_\Right\}\times\RR_z,dz\right).
\]

\begin{definition}[bordered Legendrian graphs]\label{definition:bordered Legendrian graph}
A \emph{bordered Legendrian graph} $\scrT=(\sfT_\Left\to\sfT\leftarrow \sfT_\Right)$ of a bordered graph $\bfgraf$ of type $(n_\Left,n_\Right)$ in $J^1\bfU$ is defined as an embedding $\sfT:\graf\to J^1U$ such that 
\begin{enumerate}
\item $\sfT$ is transverse to the boundary $\boundary J^1U=\boundary U\times\RR_{yz}$ and the restrictions on the interior $\mathring{\graf}$ and both borders $\graf_\Left$ and $\graf_\Right$ are contained in $J^1\mathring{U}$, $J^1U_\Left$ and $J^1U_\Right$, respectively.
\begin{align*}
\sfT&\pitchfork \boundary J^1U,&
\sfT_\Left&\coloneqq\sfT(\graf_\Left)\subset J^1U_\Left,&
\sfT_\Right&\coloneqq\sfT(\graf_\Right)\subset J^1U_\Right,&
\mathring{\sfT}&\coloneqq\sfT(\mathring{\graf})\subset J^1\mathring{U}.
\end{align*}
\item $\sfT$ on edges are smooth Legendrian with boundary and pairwise non-tangent to each other at all vertices and base points, in particular, two edges adjacent to each base point form a smooth arc.
\end{enumerate}

By labeling borders in $\sfT_\Left$ and $\sfT_\Right$ in top-to-bottom ways with respect to $z$-coordinates, we identify the left and right border $\sfT_\Left$ and $\sfT_\Right$ with the set $[n_\Left]=\{1,\dots,n_\Left\}$ and $[n_\Right]=\{1,\dots,n_\Right\}$.
\end{definition}

\subsection{Front and Lagrangian projections}\label{section:projections}
There are two projections for $J^1\RR_x\isomorphic \RR^3_{xyz}$, called the \emph{front} and \emph{Lagrangian} projections $\pi_{\front}:\RR^3_{xyz}\to\RR^2_{xz}$ and $\pi_{\lag}:\RR^3_{xyz}\to\RR^2_{xy}$, respectively.

\begin{definition}[Regular projections]\label{definition:regular projections}
For a bordered Legendrian graph $\scrT=(\sfT_\Left\to\sfT\leftarrow\sfT_\Right)$, the \emph{front} and \emph{Lagrangian projections} $\cT=(T_\Left\to T\leftarrow T_\Right)\coloneqq \pi_\front(\scrT)$ and $\cT_\lag=(T_{\lag,\Left}\to T_\lag\leftarrow T_{\lag,\Right})\coloneqq \pi_\lag(\scrT)$ are said to be \emph{regular} if in their interiors,
\begin{enumerate}
\item there are only finitely many transverse double points, called \emph{crossings}, and
\item no vertices, basepoints or \emph{$x$-extreme points} are crossings,
\item each edge containing a $x$-maximal point must involve at least one vertex or a base point,
\end{enumerate}
where a point in the interior $\mathring{\sfT}$ is said to be \emph{$x$-maximal} or \emph{$x$-minimal} if it is maximal or minimal with respect to the $x$-coordinate, and \emph{$x$-extreme} if it is either $x$-maximal or $x$-minimal.\footnotemark

A bordered Legendrian graph of type $(0,0)$ is called a \emph{Legendrian graph} and we denote the sets of regular front and Lagrangian projections of non-bordered and bordered Legendrian graphs by
$\LT_\front, \LT_\lag$ and $\BLT_\front, \BLT_\lag$, respectively.
\end{definition}
\footnotetext{In the front projection, an $x$-extreme point is either a cusp or a vertex of type $(0,n)$ or $(n,0)$.}

\begin{remark}
Due to the Legendrian condition, there are no vertical tangencies and no non-transverse double points in the front projection.
Instead, it contains \emph{cusps}, which is obviously, $x$-extreme.
\end{remark}

\begin{notation}
The front and Lagrangian projection of $\sfT=\sfT(\graf)$ with $\graf=(\ttV,\ttV_\Left, \ttV_\Right, \ttB,\ttE)$ will be denoted by $T=(V,V_\Left, V_\Right, B, E)$ and $T_\lag=(V_\lag,V_{\lag,\Left},V_{\lag,\Right},B_\lag,E_\lag)$, respectively.
\end{notation}

For examples of regular and non-regular projection, see Figure~\ref{figure:non-regular projections}.
To avoid any confusion, we denote vertices and basepoints by small dots and bars, respectively.

\begin{figure}[ht]
\subfigure[Non-regular local front projections]{\makebox[0.45\textwidth]{$
\begin{tikzpicture}[baseline=-.5ex,scale=0.5]
\begin{scope}
\draw[thick] (-1,0)--(1,0);
\draw[thick] (-.5,-1)--(.5,1);
\draw[white,fill=white] (0,0) circle (0.2);
\draw[thick] (-.5,1)--(.5,-1);
\end{scope}
\begin{scope}[xshift=2.5cm]
\draw[thick] (-.5,-1)--(.5,1);
\draw[white,fill=white] (0,0) circle (0.2);
\draw[thick] (-1,0.5) to[out=-45,in=180] (0,0) to[out=180,in=45] (-1,-0.5);
\draw[thick] (1,0.3) to[out=-150,in=0] (0,0) to[out=0,in=150] (1,-0.3);
\draw[thick] (-1,0.25) to[out=-22.5,in=180] (0,0);
\draw[fill] (0,0) circle (0.1);
\end{scope}
\begin{scope}[xshift=5cm]
\draw[thick] (-.5,-1)--(.5,1);
\draw[white,line width=6] (-1,0) -- (1,0);
\draw[thick] (-1,0)--(0,0) node{$|$} -- (1,0);
\end{scope}
\begin{scope}[xshift=7.5cm]
\draw[thick] (-.5,-1)--(.5,1);
\draw[white,fill=white] (0,0) circle (0.2);
\draw[thick] (-1,0.5) to[out=-45,in=180] (0,0) to[out=180,in=45] (-1,-0.5);
\end{scope}
\end{tikzpicture}
$}}
\subfigure[Non-regular local Lagrangian projections]{\makebox[0.5\textwidth]{$
\begin{tikzpicture}[baseline=-.5ex,scale=0.5]
\begin{scope}
\draw[thick] (-1,0)--(1,0);
\draw[thick] (-.5,-1)--(.5,1);
\draw[white,fill=white] (0,0) circle (0.2);
\draw[thick] (-.5,1)--(.5,-1);
\end{scope}
\begin{scope}[xshift=2.5cm]
\draw[thick] (-1,-0.5) to[out=45,in=180] (0,0) to[out=0,in=135] (1,-0.5);
\draw[white,line width=5] (-1,0.5) to[out=-45,in=180] (0,0) to[out=0,in=-135] (1,0.5);
\draw[thick] (-1,0.5) to[out=-45,in=180] (0,0) to[out=0,in=-135] (1,0.5);
\end{scope}
\begin{scope}[xshift=5cm]
\draw[thick] (-0.5,-1) -- (0.5,1);
\draw[white,fill=white] (0,0) circle (0.2);
\draw[thick] (-1,0.5) -- (0,0) (-1,0.25) -- (0,0) (-1,-0.5) -- (0,0) (1,0.3) -- (0,0) (1,-0.3) -- (0,0);
\draw[fill] (0,0) circle (0.1);
\end{scope}
\begin{scope}[xshift=7.5cm]
\draw[thick] (-.5,-1)--(.5,1);
\draw[white,line width=6] (-1,0) -- (1,0);
\draw[thick] (-1,0)--(0,0) node{$|$} -- (1,0);
\end{scope}
\begin{scope}[xshift=10cm]
\draw[thick] (-1,0) -- (1,0);
\draw[white,line width=5] (-1,-0.5) to[out=0,in=-90] (0,0) to[out=90,in=0] (-1,0.5);
\draw[thick] (-1,-0.5) to[out=0,in=-90] (0,0) to[out=90,in=0] (-1,0.5);
\end{scope}
\end{tikzpicture}
$}}
\subfigure[Regular front projections\label{figure:regular projections}]{\makebox[0.55\textwidth]{$
\begin{tikzpicture}[baseline=-.5ex]
\begin{scope}
\draw[thick] (-0.75,0.75) -- (0,0.75);
\draw[thick] (-0.75,0.25) -- (0,0.25);
\draw[thick](-0.25,-0.5) to[out=180,in=0] node[midway,sloped] {$|$} (-0.75,-0.25);
\draw[thick](-0.25,-0.5) to[out=180,in=0] (-0.75,-0.75);
\draw[red](-0.75,-1)--(-0.75,1);
\draw[red](0,-1)--(0,1);
\end{scope}
\end{tikzpicture}\quad
\cT_n\coloneqq
\begin{tikzpicture}[baseline=-.5ex]
\foreach \i in {-1.5,-0.5,0.5,1.5} {
\draw[thick] (-0.75, {\i*0.5}) -- +(0.75,0);
}
\draw[red](-0.75,-1)--(-0.75,1) (0,-1)--(0,1);
\draw (0,0.75) node[right] {$1$};
\draw (0,0) node[right] {$\vdots$};
\draw (0,-0.75) node[right] {$n$};
\end{tikzpicture}\quad
\bfzero_n\coloneqq
\begin{tikzpicture}[baseline=-.5ex]
\draw[fill](0,0) circle (0.05) node[above] {$0$};
\foreach \i in {0.75, 0.25, -0.25, -0.75} {
\draw[thick, rounded corners](0,0) to[out=0,in=180] (0.75,\i);
}
\draw[red](0.75,-1)--(0.75,1);
\draw (0.75,0.8) node[right] {$1$};
\draw (0.75,0) node[right] {$\vdots$};
\draw (0.75,-0.8) node[right] {$n$};
\end{tikzpicture}\quad
\bfinfty_n\coloneqq
\begin{tikzpicture}[baseline=-.5ex]
\draw[fill](0,0) circle (0.05) node[above] {$\infty$};
\foreach \i in {0.75, 0.25, -0.25, -0.75} {
\draw[thick, rounded corners](0,0) to[out=180,in=0] (-0.75,\i);
}
\draw[red](-0.75,-1)--(-0.75,1);
\draw (-0.75,0.8) node[left] {$1$};
\draw (-0.75,0) node[left] {$\vdots$};
\draw (-0.75,-0.8) node[left] {$n$};
\end{tikzpicture}
$}}
\subfigure[Non-regular front projections]{\makebox[0.4\textwidth]{$
\begin{tikzpicture}[baseline=-.5ex]
\begin{scope}
\draw[thick] (-0.75,0.6) -- (0,-0.6);
\draw[thick] (-0.75,0) -- (0,0);
\draw[thick] (-0.75,-0.6) -- (0,0.6);
\draw[red](-0.75,-1)--(-0.75,1);
\draw[red](0,-1)--(0,1);
\end{scope}
\begin{scope}[xshift=1.5cm]
\draw[thick] (-0.75,0.75) -- (0,0.75);
\draw[thick] (-0.75,0.25) -- (0,0.25);
\draw[thick, rounded corners](-0.25,-0.5) to[out=180,in=0] (-0.75,-0.25);
\draw[thick, rounded corners](-0.25,-0.5) to[out=180,in=0] (-0.75,-0.75);
\draw[red](-0.75,-1)--(-0.75,1);
\draw[red](0,-1)--(0,1);
\end{scope}
\end{tikzpicture}
$}}
\caption{Regular and non-regular projections of bordered Legendrian graphs}
\label{figure:non-regular projections}
\end{figure}
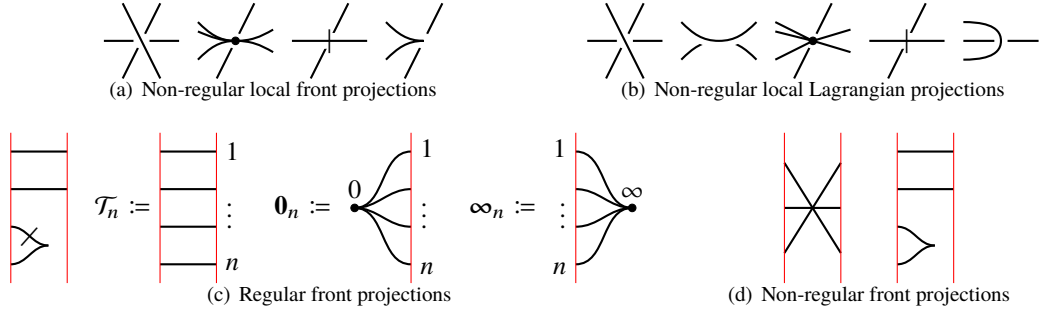

\begin{definition}[Types and orientations]\label{definition:types of vertices}
For a vertex $v$ or a base point $b$ of a bordered Legendrian graph, we say that it is of \emph{type $(\ell,r)$} if there are $\ell$ and $r$ half-edges adjacent to $v$ or $b$ on the left and right, respectively. 
We label the set $H_v\coloneqq{\{h_{v,1},\dots,h_{v,m}\}}$ of (small enough) half-edges in front and Lagrangian projections as follows:
\[
\begin{tikzpicture}[baseline=-0.5ex,scale=0.8]
\begin{scope}[xshift=-5cm]
\foreach \i in {4,...,8} {
	\draw[thick] (-1,{(6-\i)/3}) to[out=0,in=180] (0,0);
}
\draw (-1,0.67) node[left] {$h_{v,1}$};
\draw (-1,0.33) node[left] {$h_{v,2}$};
\draw (-1,-0.167) node[left] {$\vdots$};
\draw (-1,-0.67) node[left] {$h_{v,\ell}$};
\foreach \i in {1,2,3} {
	\draw[thick] (1,{(2-\i)/1.5}) to[out=180,in=0] (0,0);
}
\draw (1,0.67) node[right] {$h_{v,\ell+1}$};
\draw (1,0) node[right] {$\vdots$};
\draw (1,-0.67) node[right] {$h_{v,\ell+r}$};
\draw[fill] (0,0) circle (2pt) node[above] {$v$};
\end{scope}
\begin{scope}
\foreach \i in {4,...,8} {
	\draw[thick] (-1,{(\i-6)/3}) -- (0,0);
}
\draw (-1,-0.67) node[left] {$h_{v,1}$};
\draw (-1,-0.33) node[left] {$h_{v,2}$};
\draw (-1,0.167) node[left] {$\vdots$};
\draw (-1,0.67) node[left] {$h_{v,\ell}$};
\foreach \i in {1,2,3} {
	\draw[thick] (1,{(2-\i)/1.5}) -- (0,0);
}
\draw (1,0.67) node[right] {$h_{v,\ell+1}$};
\draw (1,0) node[right] {$\vdots$};
\draw (1,-0.67) node[right] {$h_{v,\ell+r}$};
\draw[fill] (0,0) circle (2pt) node[above] {$v$};
\end{scope}
\begin{scope}[xshift=5cm]
\draw[thick] (-1,0) node[left] {$h_{b,1}$} -- (0,0) node {$|$} node[above] {$b$} -- (1,0) node[right] {$h_{b,2}$};
\end{scope}
\end{tikzpicture}
\quad\cdots
\]

In particular, each base point $b\in B$ is assumed to be \emph{oriented} from the half-edge $h_{b,1}$ to $h_{b,2}$ in the above convention.
\end{definition}

\begin{example/definition}[The trivial and vertex bordered Legendrian graphs]\label{example/definition:trivial graph}
Let $n\ge 1$. The front projections of the trivial bordered Legendrian graph $\cT_n=(T_{n,\Left}\to T_n\leftarrow T_{n,\Right})$ of type $(n,n)$ and the vertex bordered graphs $\bfzero_n=(\emptyset\to 0_n\leftarrow 0_{n,\Right})$ and $\bfinfty_n=(\infty_{n,\Left}\to \infty_n\leftarrow\emptyset)$ of type $(0,n)$ and $(n,0)$ as depicted in Figure~\ref{figure:regular projections},
whose left and right borders are points lying on the red lines at the left and right, respectively.

For convenience's sake, we define $\cT_0=\bfzero_0=\bfinfty_0=(\emptyset\to\emptyset\leftarrow\emptyset)$.
\end{example/definition}

As usual, we say that two bordered Legendrian graphs $\scrT$ and $\scrT'$ are \emph{equivalent} if
they are isotopic, that is, there exists a family of bordered Legendrian graphs
\begin{align*}
\scrT_t&:\graf\times[0,1]\to J^1\bfU_t\subset J^1\RR_x,&
\scrT_0&=\scrT,&
\scrT_1&=\scrT'.
\end{align*}
\begin{remark}
It is important to note that during the isotopy between two bordered Legendrian graphs, the ambient manifold $J^1\bfU_t$ may changes. For example, any translation along the $x$-axis will give us an equivalence.
\end{remark}

In terms of projections, it is known that two front projections $\cT'$ and $\cT$ or two Lagrangian projections $\cT'_\lag$ and $\cT_\lag$ are projections of equivalent Legendrian graphs if and only if they can be connected by a sequence of front or Lagrangian \emph{Reidemeister moves} depicted in Figures~\ref{fig:front_RM} or \ref{fig:RM}, respectively. One may refer \cite[Proposition~4.2]{BI2009} or \cite[Proposition~2.1]{ABK2018}.

\begin{figure}[ht]
\subfigure[\label{fig:front_RM}Front Reidemeister moves]{$
\begin{tikzcd}[row sep=0pc,ampersand replacement=\&]
\vcenter{\hbox{\includegraphics[scale=0.7]{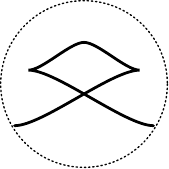}}}\arrow[r,"\RM{I}"]\&
\vcenter{\hbox{\includegraphics[scale=0.7]{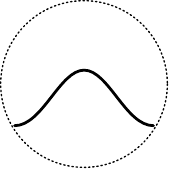}}}
\&
\vcenter{\hbox{\includegraphics[scale=0.7]{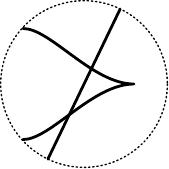}}}\arrow[r,"\RM{II}"]\&
\vcenter{\hbox{\includegraphics[scale=0.7]{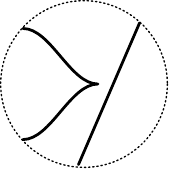}}}
\&
\vcenter{\hbox{\includegraphics[scale=0.7]{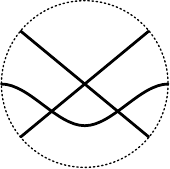}}}\arrow[r,leftrightarrow,"\RM{III}"]\&
\vcenter{\hbox{\includegraphics[scale=0.7]{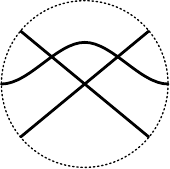}}}
\\
\vcenter{\hbox{\includegraphics[scale=0.7]{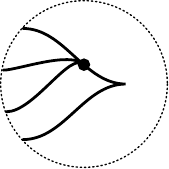}}}\arrow[r,"\RM{IV}"]\&
\vcenter{\hbox{\includegraphics[scale=0.7]{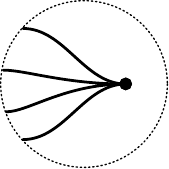}}}
\&
\begin{tikzpicture}[baseline=-.5ex,xscale=-1,scale=0.6]
\draw[densely dotted](0,0) circle (1);
\clip(0,0) circle (1);
\draw[thick] (1,.6) to[out=180,in=0] (-0.5,0);
\draw[thick] (1,.2) to[out=180,in=0] (-0.5,0);
\draw[thick] (1,-.2) to[out=180,in=0] (-0.5,0);
\draw[thick] (1,-.6) to[out=180,in=0] (-0.5,0);
\draw[fill] (-0.5,0) circle (2pt);
\draw[thick] (0,1) -- (0.5,-1);
\end{tikzpicture}\arrow[r,"\RM{V}"]
\&
\begin{tikzpicture}[baseline=-.5ex,xscale=-1,scale=0.6]
\draw[densely dotted](0,0) circle (1);
\clip(0,0) circle (1);
\draw[thick] (1,.6) to[out=180,in=0] (0,0);
\draw[thick] (1,.2) to[out=180,in=0] (0,0);
\draw[thick] (1,-.2) to[out=180,in=0] (0,0);
\draw[thick] (1,-.6) to[out=180,in=0] (0,0);
\draw[fill] (0,0) circle (2pt);
\draw[thick] (-0.5,1) -- (0,-1);
\end{tikzpicture}
\&
\vcenter{\hbox{\includegraphics[scale=0.7]{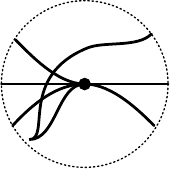}}}\arrow[r,"\RM{VI}"]\&
\vcenter{\hbox{\includegraphics[scale=0.7]{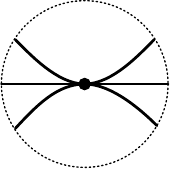}}}
\end{tikzcd}
$}

\subfigure[\label{fig:RM}Lagrangian Reidemeister moves]{$
\begin{tikzcd}[row sep=0pc,ampersand replacement=\&]
\vcenter{\hbox{\includegraphics[scale=0.7]{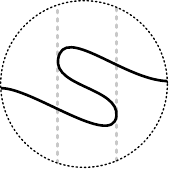}}}\arrow[r,"\RM{0_a}"]\&
\vcenter{\hbox{\includegraphics[scale=0.7]{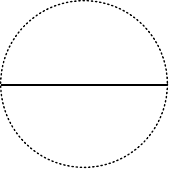}}}
\&
\vcenter{\hbox{\includegraphics[scale=0.7]{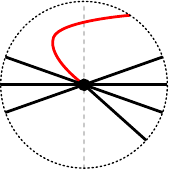}}}\arrow[r,"\RM{0_b}"]\&
\vcenter{\hbox{\includegraphics[scale=0.7]{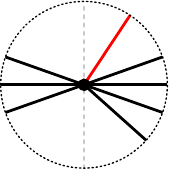}}}
\&
\vcenter{\hbox{\includegraphics[scale=0.7]{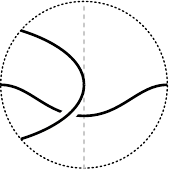}}}\arrow[r,"\RM{0_c}"]\&
\vcenter{\hbox{\includegraphics[scale=0.7]{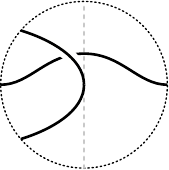}}}
\\
\vcenter{\hbox{\includegraphics[scale=0.7]{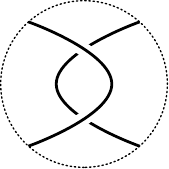}}}\arrow[r,"\RM{ii}"]\&
\vcenter{\hbox{\includegraphics[scale=0.7]{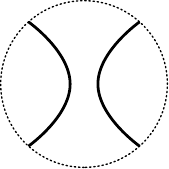}}}
\&
\vcenter{\hbox{\rotatebox[origin=c]{90}{\includegraphics[scale=0.7]{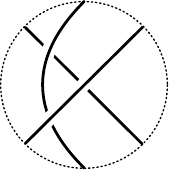}}}}\arrow[r,leftrightarrow,"\RM{iii_a}"]\&
\vcenter{\hbox{\rotatebox[origin=c]{90}{\includegraphics[scale=0.7]{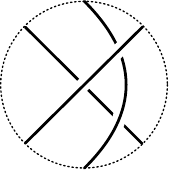}}}}
\&
\vcenter{\hbox{\rotatebox[origin=c]{90}{\includegraphics[scale=0.7]{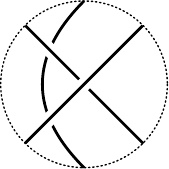}}}}\arrow[r,leftrightarrow,"\RM{iii_b}"]\&
\vcenter{\hbox{\rotatebox[origin=c]{90}{\includegraphics[scale=0.7]{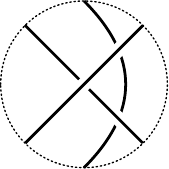}}}}
\\
\begin{tikzpicture}[baseline=-.5ex,scale=0.6]
\draw[densely dotted](0,0) circle (1);
\clip(0,0) circle (1);
\draw[thick] (-0.5,-1) -- (0,1);
\draw[white,line width=5] (-1,.6) -- (0.5,0) (-1,.2) -- (0.5,0) (-1,-.2) -- (0.5,0) (-1,-.6) -- (0.5,0);
\draw[thick] (-1,.6) -- (0.5,0) (-1,.2) -- (0.5,0) (-1,-.2) -- (0.5,0) (-1,-.6) -- (0.5,0);
\draw[fill] (0.5,0) circle (2pt);
\end{tikzpicture}
\arrow[r,"\RM{iv}"]\&
\begin{tikzpicture}[baseline=-.5ex,scale=0.6]
\draw[densely dotted](0,0) circle (1);
\clip(0,0) circle (1);
\draw[thick] (-1,.6) -- (0,0) (-1,.2) -- (0,0) (-1,-.2) -- (0,0) (-1,-.6) -- (0,0);
\draw[fill] (0,0) circle (2pt);
\draw[thick] (0,-1) -- (0.5,1);
\end{tikzpicture}
\end{tikzcd}
$}
\caption{Front and Lagrangian Reidemeister moves: Refelections are possible, the valency of vertex is arbitrary and the vertex can be replaced with a basepoint if it is bivalent.}
\end{figure}
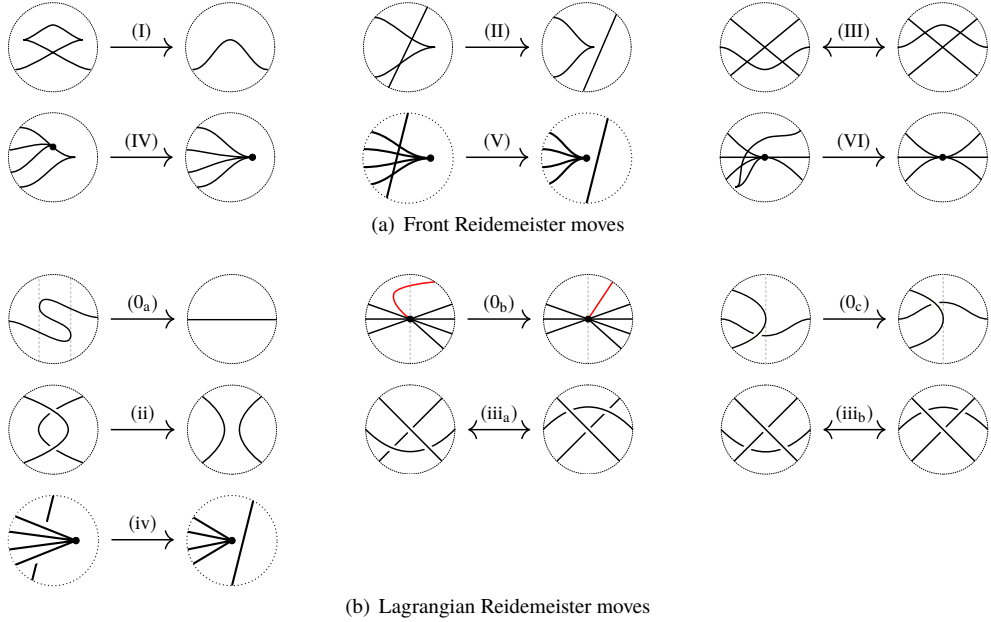

\subsubsection{Maslov potentials for bordered Legendrian graphs}\label{section:Maslov potentials}
We consider $\grading$-valued Maslov potentials on bordered Legendrian graphs as follows:

Let $\cT\in\BLT_\front$ and $S=S(T)\subset\mathring{T}$ be the set of $x$-extreme points in its interior.
An \emph{$\grading$-valued Maslov potential} of $T$ is a function $\mu:R\to\grading$ from the set $R\coloneqq\pi_0\left(T\setminus(V\cup S)\right)$ of connected components of the complement of vertices and cusps such that for all $s\in S\setminus V$,
\begin{align}\label{equation:defining equation of Maslov potentials}
\mu(s^+)-\mu(s^-)&=1\in \grading,
\end{align}
where $s^+$ (resp. $s^-$) is the upper (resp. lower) strand near $s$.

For $T_\lag\in\BLT_\lag$, let $S_\lag=S(T_\lag)\subset\mathring{T}_\lag$ be the set of $x$-extreme points.
As before we define the set $R_\lag\coloneqq \pi_0(T_\lag\setminus(V_\lag\cup S_\lag))$ of connected components of the complement of vertices and $x$-extreme points. Then an $\grading$-valued Maslov potential of $T_\lag$ is a function $\mu:R_\lag\to\grading$ satisfying the following condition: for each $s\in S_\lag\setminus V_\lag$,
\begin{align}\label{equation:defining equation of Maslov potentials of Lagrangian projection}
\mu(s^+)-\mu(s^-) = \begin{cases}
1 & s\text{ is $x$-minimal};\\
-1 & s\text{ is $x$-maximal}.
\end{cases}
\end{align}

Diagrammatically, the above definition is depicted in Figures~\ref{figure:defining equation of Maslov potentials for front} and \ref{figure:defining equation of Maslov potentials for Lagrangian}.

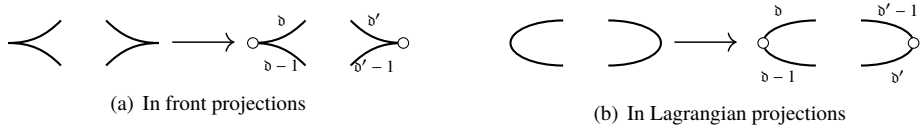
\begin{figure}[ht]
\subfigure[In front projections\label{figure:defining equation of Maslov potentials for front}]{\makebox[0.45\textwidth]{
\begin{tikzcd}[ampersand replacement=\&]
\begin{tikzpicture}[baseline=-.5ex]
\begin{scope}[xshift=-1cm]
\draw[thick] (.7,.3) to[out=225,in=0] (0,0) to[out=0,in=135] (.7,-.3);
\end{scope}
\begin{scope}[xshift=1cm]
\draw[thick] (-.7,.3) to[out=-45,in=180] (0,0) to[out=180,in=45] (-.7,-.3);
\end{scope}
\end{tikzpicture}\arrow[r] \& 
\begin{tikzpicture}[baseline=-.5ex]
\begin{scope}[xshift=-1cm]
\draw[thick] (.7,.3) to[out=225,in=0] node[midway,above] {\tiny$\fd$} (0,0) to[out=0,in=135] node[midway,below] {\tiny$\fd-1$} (.7,-.3);
\draw[fill=white] (0,0) circle (2pt);
\end{scope}
\begin{scope}[xshift=1cm]
\draw[thick] (-.7,.3) to[out=-45,in=180] node[midway,above] {\tiny$\fd'$}  (0,0) to[out=180,in=45] node[midway,below] {\tiny$\fd'-1$} (-.7,-.3);
\draw[fill=white] (0,0) circle (2pt);
\end{scope}
\end{tikzpicture}
\end{tikzcd}
}}
\subfigure[In Lagrangian projections\label{figure:defining equation of Maslov potentials for Lagrangian}]{\makebox[0.45\textwidth]{
\begin{tikzcd}[ampersand replacement=\&]
\begin{tikzpicture}[baseline=-.5ex]
\begin{scope}[xshift=-1cm]
\draw[thick] (.7,.3) arc (90:270:.7 and 0.3);
\end{scope}
\begin{scope}[xshift=1cm]
\draw[thick] (-.7,.3) arc (90:-90:.7 and 0.3);
\end{scope}
\end{tikzpicture}\arrow[r] \& 
\begin{tikzpicture}[baseline=-.5ex]
\begin{scope}[xshift=-1cm]
\draw[thick] (.7,.3) arc (90:180:.7 and 0.3) node[midway,above] {\tiny$\fd$} arc (180:270:.7 and 0.3) node[midway,below] {\tiny$\fd-1$};
\draw[fill=white] (0,0) circle (2pt);
\end{scope}
\begin{scope}[xshift=1cm]
\draw[thick] (-.7,.3) arc (90:0:.7 and 0.3) node[midway,above] {\tiny$\fd'-1$} arc (0:-90:.7 and 0.3) node[midway,below] {\tiny$\fd'$};
\draw[fill=white] (0,0) circle (2pt);
\end{scope}
\end{tikzpicture}
\end{tikzcd}
}}
\caption{Defining diagrams for Maslov potentials}
\end{figure}

Moreover, one can define $\mu_\Left\coloneqq \mu|_{T_\Left}:[n_\Left]\to\grading$ and $\mu_\Right\coloneqq \mu|_{T_\Right}:[n_\Right]\to\grading$ as the restrictions of $\mu$ to the connected components containing $T_\Left$ and $T_\Right$, respectively, under the canonical identifications $T_\Left\isomorphic [n_\Left]$ and $\cT_\Right\isomorphic [n_\Right]$.

\begin{definition}[Maslov potentials for bordered Legendrian graphs]
A Maslov potential for a bordered Legendrian graph $\cT$ is a triple $\bfmu=(\mu_\Left,\mu,\mu_\Right)$ and we denote the category of front and Lagrangian projections of Legendrian graphs with Maslov potentials by $\BLT_\front^\mu$ and $\BLT_\lag^\mu$, respectively.
\end{definition}

\begin{example}\label{example:Maslov potentials for trivial graphs}
Recall the bordered Legendrian graphs $\cT_n$, $\bfzero_n$ and $\bfinfty_n$ defined in Example/Definition~\ref{example/definition:trivial graph}.
Since they have no $x$-extreme points except for a vertex, there are no conditions for Maslov potentials. 
That is, any function $\mu:[n]=\{1,\dots,n\}\to\grading$ can be realized as a Maslov potential for $\cT_n$, $\bfzero_n$ or $\bfinfty_n$.
\end{example}

\begin{lemma}\cite[Theorem~2.21]{AB2018}\label{lemma:Reidemeister move preserves potential}
Let $\RM{m}:\cT'_\lag\to \cT_\lag$ be a Lagrangian Reidemeister move.
Then it lifts to $\RM{m}:(\cT'_\lag,{\mu'})\to (\cT_\lag,\mu)$. Namely, for given $\mu'$, there is a unique Maslov potential $\mu$ on $\cT_\lag$ such that 
\[
\RM{m}_{*}\mu'=\mu.
\]
\end{lemma}

\begin{definition}[Restrictions of Maslov potentials on vertices]
For each vertex $v$ of type $(\ell,r)$ with $\ell+r=n$, the set $H_v$ of half-edges can be identified with $\Zmod{n}$ by Definition~\ref{definition:types of vertices} and we denote the restriction $\mu|_{H_v}:\Zmod{n}\to\grading$ of a Maslov potential to $H_v$ by $\mu_v$.
\end{definition}

\subsubsection{Concatenations of bordered Legendrian graphs}
For $i=1,2$, let $\cT^i\in\BLT_\front$ be a front projection of type $\left(n^i_\Left, n^i_\Right\right)$.
Suppose that $n^1_\Right=n^2_\Left$.
Then we can naturally define the \emph{concatenation} $\cT^1_\front\cdot \cT^2_\front$ simply by concatenating and identifying $T^1_\Right$ and $T^2_\Left$.
The result $T$ becomes a front projection of the bordered Legendrian graph whose left and right borders $T_\Left\coloneqq T^1_\Left$ and $T_\Right\coloneqq T^2_\Right$ come naturally from $\cT^1$ and $\cT^2$, respectively.

\begin{remark}
It is important to note that we will not regard the points of concatenations as vertices of $T$.
Therefore $T$ has $n$-less edges than the disjoint union of $T^1$ and $T^2$.
\end{remark}

\begin{definition}[Closure]
For $\cT\in\BLT_\front$ of type $(n_\Left,n_\Right)$, the \emph{closure} $\hat{\cT}$ is defined by the Legendrian graph obtained by the concatenation
\[
\hat {\cT}\coloneqq \bfzero_{n_\Left}\cdot \cT \cdot \bfinfty_{n_\Right}\in\LT_\front
\]
as depicted in Figure~\ref{figure:closure}.
\end{definition}

\begin{figure}[ht]
\[
\begin{tikzcd}
\cT=\begin{tikzpicture}[baseline=-.5ex]
\foreach \i in {0.75, 0.25, -0.25, -0.75} {
\draw[thick, rounded corners] (-1,\i) -- (-0.75,\i) (0.25,\i) -- (0.5,\i);
}
\draw(-0.75,-1) rectangle (0.25,1);
\draw(-0.25,0) node {$T$};
\draw[red](-1,-1)--(-1,1);
\draw[red](0.5,-1)--(0.5,1);
\end{tikzpicture}\arrow[r,"\hat\cdot",mapsto]&
\begin{tikzpicture}[baseline=-.5ex]
\draw[fill](0.25,0) circle (0.05) node[above] {$0$};
\foreach \i in {0.75, 0.25, -0.25, -0.75} {
\draw[thick, rounded corners](0.25,0) to[out=0,in=180] (1,\i);
}
\draw[red](1,-1)--(1,1);
\end{tikzpicture}
\cdot
\begin{tikzpicture}[baseline=-.5ex]
\foreach \i in {0.75, 0.25, -0.25, -0.75} {
\draw[thick, rounded corners] (-1,\i) -- (-0.75,\i) (0.25,\i) -- (0.5,\i);
}
\draw(-0.75,-1) rectangle (0.25,1);
\draw(-0.25,0) node {$T$};
\draw[red](-1,-1)--(-1,1);
\draw[red](0.5,-1)--(0.5,1);
\end{tikzpicture}
\cdot
\begin{tikzpicture}[baseline=-.5ex]
\draw[fill](-0.25,0) circle (0.05) node[above] {$\infty$};
\foreach \i in {0.75, 0.25, -0.25, -0.75} {
\draw[thick, rounded corners](-0.25,0) to[out=180,in=0] (-1,\i);
}
\draw[red](-1,-1)--(-1,1);
\end{tikzpicture}
\arrow[r,equal]&
\begin{tikzpicture}[baseline=-.5ex]
\begin{scope}[xshift=-2cm]
\draw[fill](0.5,0) circle (0.05) node[above] {$0$};
\foreach \i in {0.75, 0.25, -0.25, -0.75} {
\draw[thick, rounded corners](0.5,0) to[out=0,in=180] (1.25,\i);
}
\end{scope}
\draw(-0.75,-1) rectangle (0.25,1);
\draw(-0.25,0) node {$T$};
\begin{scope}[xshift=1.5cm]
\draw[fill](-0.5,0) circle (0.05) node[above] {$\infty$};
\foreach \i in {0.75, 0.25, -0.25, -0.75} {
\draw[thick, rounded corners](-0.5,0) to[out=180,in=0] (-1.25,\i);
}
\end{scope}
\end{tikzpicture}=\hat \cT
\end{tikzcd}
\]
\caption{The closure of the front projection $\cT$}
\label{figure:closure}
\end{figure}
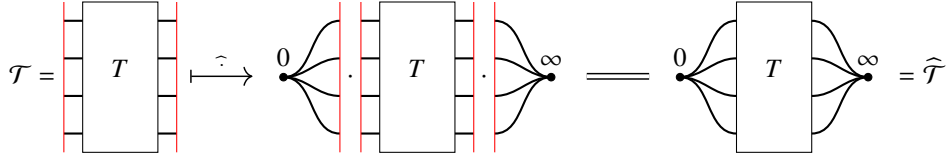

\begin{lemma}\label{lemma:potential extends to the closure}
The closure $\hat\cdot:\BLT_\front^\mu\to \LT_\front^\mu$ is well-defined.
\end{lemma}
\begin{proof}
It is obvious that for each $(\cT,\mu)\in\BLT_\front^\mu$, there is a unique way to extend $\mu$ on $T$ to $\hat \mu$ on $\hat T$ by definition of the closure, which is well-defined since any function on $[n]$ can be realized as a Maslov potential of $\bfzero_n$ or $\bfinfty_n$ as seen in Example~\ref{example:Maslov potentials for trivial graphs}.
\end{proof}

\subsubsection{Ng's resolution}
We introduce a combinatorial way, called the \emph{Ng's resolution} to obtain a regular Lagrangian projection $T_\lag\in\LT_\lag$ for given front projection $T\in\LT_\front$ defining the equivalent Legendrian graphs.
\begin{remark}
This is an extension of the original Ng's resolution for Legendrian knots to Legendrian graphs.
\end{remark}

\begin{definition}[Ng's resolution]\cite[Definition~2.1]{Ng2003}\label{definition:Ng's resolution}
For $T\in\LT_\front$, the \emph{Ng's resolution} $\Res(T)$ is a Lagrangian projection obtained by (combinatorially) replacing the local pieces as follows:
\begin{align*}
\Lcusp&\mapsto \leftarc,&
\Rcusp&\mapsto\rightkink,&
\crossing&\mapsto \crossingpositive,
\end{align*}
and for a vertex $v$ of type $(\ell,r)$, we take a replacement as follows:
\[
\begin{tikzcd}
\begin{tikzpicture}[baseline=-0.5ex,scale=0.8]
\foreach \i in {4,...,8} {
	\draw[thick] (-1,{(6-\i)/3}) to[out=0,in=180] (0,0);
}
\draw (-1,0.67) node[left] {$h_{v,1}$};
\draw (-1,0.33) node[left] {$h_{v,2}$};
\draw (-1,-0.167) node[left] {$\vdots$};
\draw (-1,-0.67) node[left] {$h_{v,\ell}$};
\foreach \i in {1,2,3} {
	\draw[thick] (1,{(2-\i)/1.5}) to[out=180,in=0] (0,0);
}
\draw (1,0.67) node[right] {$h_{v,\ell+1}$};
\draw (1,0) node[right] {$\vdots$};
\draw (1,-0.67) node[right] {$h_{v,\ell+r}$};
\draw[fill] (0,0) circle (2pt) node[above] {$v$};
\end{tikzpicture}\arrow[r,mapsto] &
\begin{tikzpicture}[baseline=-0.5ex,scale=0.8]
\draw[thick,rounded corners] (-2, -2/3 ) -- (-1.8, -2/3 ) -- (-0.6,2/3) -- (0,0);
\foreach \i in {7,6,5} {
	\draw[white,line width=3,rounded corners] (-2, {(6-\i)/3} ) -- (-1.8, {(6-\i)/3} ) -- ({-1.8+0.3*(8-\i)}, {-2/3} ) -- (-0.7,{(\i-6)/3});
	\draw[thick,rounded corners] (-2, {(6-\i)/3} ) -- (-1.8, {(6-\i)/3} ) -- ({-1.8+0.3*(8-\i)}, {-2/3} ) -- (-0.6,{(\i-6)/3}) --  (0,0);
}
\draw[white,line width=3,rounded corners] (-2, 2/3 ) -- (-1.8, 2/3 ) -- (-0.6,-2/3);
\draw[thick,rounded corners] (-2, 2/3 ) -- (-1.8, 2/3 ) -- (-0.6,-2/3) -- (0,0);
\draw (-2,0.67) node[left] {$h_{v,1}$};
\draw (-2,0.33) node[left] {$h_{v,2}$};
\draw (-2,-0.167) node[left] {$\vdots$};
\draw (-2,-0.67) node[left] {$h_{v,\ell}$};
\foreach \i in {1,2,3} {
	\draw[thick] (1,{(2-\i)/1.5}) -- (0,0);
}
\draw (1,0.67) node[right] {$h_{v,\ell+1}$};
\draw (1,0) node[right] {$\vdots$};
\draw (1,-0.67) node[right] {$h_{v,\ell+r}$};
\draw[fill] (0,0) circle (2pt) node[above] {$v$};
\end{tikzpicture}
\end{tikzcd}
\]

In particular, a right cusp with a base point will be mapped as follows:
\[
\begin{tikzcd}
\begin{tikzpicture}[baseline=-.5ex,scale=0.7]
\draw[thick] (-1,0.5) to[out=-45,in=180] (0,0) to[out=180,in=45] (-1,-0.5);
\draw (0,-0.2) -- (0,0.2);
\end{tikzpicture}\arrow[r,mapsto]&
\begin{tikzpicture}[baseline=-.5ex,scale=0.7]
\draw[thick] (0,-0.4) arc (-90:90:0.4) to[out=180,in=45] (-1,-0.5);
\draw[white,line width=5] (-1,0.5) to[out=-45,in=180] (0,-0.4);
\draw[thick] (-1,0.5) to[out=-45,in=180] (0,-0.4);
\draw (0.2,0) -- (0.6,0);
\end{tikzpicture}
\end{tikzcd}
\]
\end{definition}

For example, let $\cT\in\BLT_\front$ be a front projection of a bordered Legendrian graph. Then the Ng's resolution of its closure look like a picture in Figure~\ref{figure:resolution of the closure}.
Notice that there are $\binom{n_\Right}2$ many additional crossings in $\Res(\hat \cT)$ than $\cT$, which came from the right closure $\bfinfty_{n_\Right}$.

\begin{figure}[ht]
\[
\Res(\hat \cT)=
\begin{tikzpicture}[baseline=-0.5ex]
\begin{scope}
\foreach \t in {-2,...,2} {
\draw[thick,rounded corners] (-1,0) -- (-0.5,\t*0.4) -- (0,\t*0.4);
}
\draw (-0.5,0.8) node[left] {$1$};
\draw (-0.5,-0.8) node[left] {$n_\Left$};
\draw[fill](-1,0) circle (2pt) node[left] {$0$};
\end{scope}
\begin{scope}[xshift=3cm]
\draw[thick,rounded corners] (-1,-0.8)--(-0.8,-0.8) -- (0.8, 0.8) --(1.3,0);
\draw[white,line width=5,rounded corners] (-1,-0.4)--(-0.8,-0.4) -- (-.4,-0.8) -- (0.8, 0.4);
\draw[thick,rounded corners] (-1,-0.4)--(-0.8,-0.4) -- (-.4,-0.8) -- (0.8, 0.4)--(1.3,0);
\draw[white,line width=5,rounded corners] (-1,0)--(-0.8,0) -- (0, -0.8) -- (0.8, 0);
\draw[thick,rounded corners] (-1,0)--(-0.8,0) -- (0, -0.8) -- (0.8, 0)--(1.3,0);
\draw[white,line width=5,rounded corners] (-1,0.4)--(-0.8,0.4) -- (.4, -0.8) -- (0.8, -0.4);
\draw[thick,rounded corners] (-1,0.4)--(-0.8,0.4) -- (.4, -0.8) -- (0.8, -0.4) --(1.3,0);
\draw[white,line width=5,rounded corners] (-1,0.8)--(-0.8,0.8) -- (0.8, -0.8);
\draw[thick,rounded corners] (-1,0.8)--(-0.8,0.8) -- (0.8, -0.8)--(1.3,0);
\draw[fill] (1.3,0) circle (2pt) node[right] {$\infty$};
\draw (0.8,0.8) node[right] {$n_\Right$};
\draw (0.8,-0.8) node[right] {$1$};
\end{scope}
\draw (0,-1) rectangle (2,1);
\draw (1,0) node {$\Res(T)$};
\end{tikzpicture}
\]
\caption{The Ng's resolution of the closure of $\cT$}
\label{figure:resolution of the closure}
\end{figure}
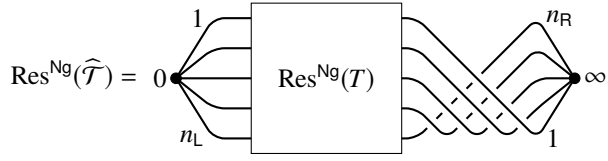

\begin{lemma}\label{lemma:functoriality of Ng's resolution}
The Ng's resolution $\Res:\LT_\front^\mu\to\LT_\lag^\mu$ is well-defined and preserves the equivalence.

In other words, each front Reidemeister move will be mapped to a sequence of Lagrangian Reidemeister moves with Maslov potentials.
\end{lemma}
\begin{proof}
The well-definedness is obvious since the Ng's resolution is indeed defined as follows:
\begin{center}
\scriptsize\def\svgscale{0.7}
\begingroup%
  \makeatletter%
  \providecommand\color[2][]{%
    \errmessage{(Inkscape) Color is used for the text in Inkscape, but the package 'color.sty' is not loaded}%
    \renewcommand\color[2][]{}%
  }%
  \providecommand\transparent[1]{%
    \errmessage{(Inkscape) Transparency is used (non-zero) for the text in Inkscape, but the package 'transparent.sty' is not loaded}%
    \renewcommand\transparent[1]{}%
  }%
  \providecommand\rotatebox[2]{#2}%
  \ifx\svgwidth\undefined%
    \setlength{\unitlength}{356.07366575bp}%
    \ifx\svgscale\undefined%
      \relax%
    \else%
      \setlength{\unitlength}{\unitlength * \real{\svgscale}}%
    \fi%
  \else%
    \setlength{\unitlength}{\svgwidth}%
  \fi%
  \global\let\svgwidth\undefined%
  \global\let\svgscale\undefined%
  \makeatother%
  \begin{picture}(1,0.17525977)%
    \put(0,0){\includegraphics[width=\unitlength,page=1]{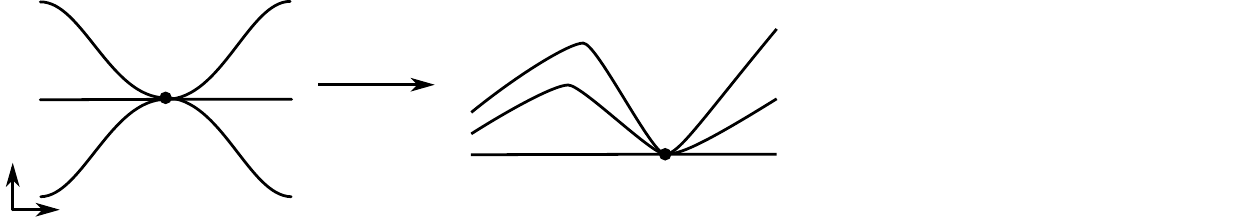}}%
    \put(0.05523567,0.00562905){\color[rgb]{0,0,0}\makebox(0,0)[lb]{\smash{$x$}}}%
    \put(-0.00093248,0.05056357){\color[rgb]{0,0,0}\makebox(0,0)[lb]{\smash{$z$}}}%
    \put(0,0){\includegraphics[width=\unitlength,page=2]{vertex_isotopy.pdf}}%
    \put(0.40415968,0.00576535){\color[rgb]{0,0,0}\makebox(0,0)[lb]{\smash{$x$}}}%
    \put(0.34799153,0.05069988){\color[rgb]{0,0,0}\makebox(0,0)[lb]{\smash{$z$}}}%
    \put(0.26162477,0.11796426){\color[rgb]{0,0,0}\makebox(0,0)[lb]{\smash{{\it isotopy}}}}%
    \put(0,0){\includegraphics[width=\unitlength,page=3]{vertex_isotopy.pdf}}%
    \put(0.65019555,0.12669009){\color[rgb]{0,0,0}\makebox(0,0)[lb]{\smash{$y=\frac{dz}{dx}$}}}%
    \put(0,0){\includegraphics[width=\unitlength,page=4]{vertex_isotopy.pdf}}%
    \put(0.79707586,0.01686268){\color[rgb]{0,0,0}\makebox(0,0)[lb]{\smash{$x$}}}%
    \put(0.74090771,0.05069988){\color[rgb]{0,0,0}\makebox(0,0)[lb]{\smash{$y$}}}%
    \put(0,0){\includegraphics[width=\unitlength,page=5]{vertex_isotopy.pdf}}%
  \end{picture}%
\endgroup%

\end{center}
Moreover, it is not hard to check that the Ng's resolution for each of front Reidemeister moves and their inverses gives us an equivalence. See Figure~\ref{figure:Reidemeister moves for Ng's resolution}.

Finally, by Lemma~\ref{lemma:Reidemeister move preserves potential}, we are done.
\end{proof}

\begin{figure}[ht]
\subfigure[$\Res\RM{I}=\RM{ii}$]{$
\begin{tikzcd}[ampersand replacement=\&,column sep=1.5pc]
\begin{tikzpicture}[baseline=-.5ex,scale=0.6]
\draw[dashed](0,0) circle (1);
\clip(0,0) circle (1);
\draw[thick] (-1,-0.5) to[out=0,in=180] (0.5,0) to[out=180,in=0] (0,0.5) to[out=180,in=0] (-0.5,0) to[out=0,in=180] (1,-0.5);
\end{tikzpicture}
\arrow[r,"\RM{I}"]\arrow[d,Rightarrow] \&
\begin{tikzpicture}[baseline=-.5ex,scale=0.6]
\draw[dashed](0,0) circle (1);
\clip(0,0) circle (1);
\draw[thick] (-1,-0.5) to[out=0,in=180] (0,0) to[out=0,in=180] (1,-0.5);
\end{tikzpicture}
\arrow[d,Rightarrow]\\
\begin{tikzpicture}[baseline=-.5ex,scale=0.6]
\draw[dashed](0,0) circle (1);
\clip(0,0) circle (1);
\draw[thick] (-1,-0.5) to[out=0,in=180] (0.5,0.5) arc (90:-90:0.3);
\draw[white,line width=5] (0.5,-0.1) to[out=180,in=0] (-0.3,0.5) arc (90:180:0.3) to[out=270,in=0] (1,-0.5);
\draw[thick] (0.5,-0.1) to[out=180,in=0] (-0.3,0.5) arc (90:180:0.3) to[out=270,in=0] (1,-0.5);
\end{tikzpicture}
\arrow[r,"\RM{ii}"] \&
\begin{tikzpicture}[baseline=-.5ex,scale=0.6]
\draw[dashed](0,0) circle (1);
\clip(0,0) circle (1);
\draw[thick] (-1,-0.5) to[out=0,in=180] (0,0) to[out=0,in=180] (1,-0.5);
\end{tikzpicture}
\end{tikzcd}
$}
\qquad
\subfigure[$\Res\RM{II}=\RM{ii}$ or $\RM{iii_a}\circ\RM{ii}$]{$
\begin{tikzcd}[ampersand replacement=\&,column sep=1.5pc]
\begin{tikzpicture}[baseline=-.5ex,scale=0.6]
\draw[dashed](0,0) circle (1);
\clip(0,0) circle (1);
\draw[thick] (1,-.5) to[out=180,in=0] (-0.5,0) to[out=0,in=180] (1,.5);
\draw[thick] (0,-1) -- (0.5,1);
\end{tikzpicture}\arrow[r,"\RM{II}"]\arrow[d,Rightarrow] \&
\begin{tikzpicture}[baseline=-.5ex,scale=0.6]
\draw[dashed](0,0) circle (1);
\clip(0,0) circle (1);
\draw[thick] (1,-.5) to[out=180,in=0] (-0.5,0) to[out=0,in=180] (1,.5);
\draw[thick] (-1,-1) -- (-0.5,1);
\end{tikzpicture}\arrow[d,Rightarrow]\\
\begin{tikzpicture}[baseline=-.5ex,scale=0.6]
\draw[dashed](0,0) circle (1);
\clip(0,0) circle (1);
\draw[thick] (0,-1) -- (0.5,1);
\draw[white,line width=5] (1,0.5) arc (90:270:1.5 and 0.5);
\draw[thick] (1,0.5) arc (90:270:1.5 and 0.5);
\end{tikzpicture}\arrow[r,"\RM{ii}"] \&
\begin{tikzpicture}[baseline=-.5ex,scale=0.6]
\draw[dashed](0,0) circle (1);
\clip(0,0) circle (1);
\draw[thick] (-1,-1) -- (-0.5,1);
\draw[white,line width=5] (1,0.5) arc (90:270:1.5 and 0.5);
\draw[thick] (1,0.5) arc (90:270:1.5 and 0.5);
\end{tikzpicture}
\end{tikzcd}\quad
\begin{tikzcd}[ampersand replacement=\&,column sep=1.5pc]
\begin{tikzpicture}[baseline=-.5ex,scale=0.6]
\draw[dashed](0,0) circle (1);
\clip(0,0) circle (1);
\draw[thick] (-1,-.5) to[out=0,in=180] (0.5,0) to[out=180,in=0] (-1,.5);
\draw[thick] (-0.5,-1) -- (0,1);
\end{tikzpicture}\arrow[rr,"\RM{II}"]\arrow[d,Rightarrow] \& \&
\begin{tikzpicture}[baseline=-.5ex,scale=0.6]
\draw[dashed](0,0) circle (1);
\clip(0,0) circle (1);
\draw[thick] (-1,-.5) to[out=0,in=180] (0.5,0) to[out=180,in=0] (-1,.5);
\draw[thick] (0.5,-1) -- (1,1);
\end{tikzpicture}\arrow[d,Rightarrow]\\
\begin{tikzpicture}[baseline=-.5ex,scale=0.6]
\draw[dashed](0,0) circle (1);
\clip(0,0) circle (1);
\draw[thick] (-0.5,-1) -- (0,1);
\draw[white,line width=5] (-1,-0.5) -- (-0.5,-0.5) to[out=0,in=180] (0.3,0.2) arc (90:-90:0.2);
\draw[thick] (-1,-0.5) -- (-0.5,-0.5)  to[out=0,in=180] (0.3,0.2) arc (90:-90:0.2);
\draw[white,line width=5] (0.3,-0.2) to[out=180,in=0] (-0.5,0.5) -- (-1,0.5);
\draw[thick] (0.3,-0.2) to[out=180,in=0] (-0.5,0.5) -- (-1,0.5);
\end{tikzpicture}\arrow[r,"\RM{iii_a}"] \&
\begin{tikzpicture}[baseline=-.5ex,scale=0.6]
\draw[dashed](0,0) circle (1);
\clip(0,0) circle (1);
\draw[thick] (0,-1) -- (0.5,1);
\draw[white,line width=5] (-1,-0.5) -- (-0.5,-0.5) to[out=0,in=180] (0.3,0.2) arc (90:-90:0.2);
\draw[thick] (-1,-0.5) -- (-0.5,-0.5)  to[out=0,in=180] (0.3,0.2) arc (90:-90:0.2);
\draw[white,line width=5] (0.3,-0.2) to[out=180,in=0] (-0.5,0.5) -- (-1,0.5);
\draw[thick] (0.3,-0.2) to[out=180,in=0] (-0.5,0.5) -- (-1,0.5);
\end{tikzpicture}\arrow[r,"\RM{ii}"] \&
\begin{tikzpicture}[baseline=-.5ex,scale=0.6]
\draw[dashed](0,0) circle (1);
\clip(0,0) circle (1);
\draw[thick] (0.5,-1) -- (1,1);
\draw[white,line width=5] (-1,-0.5) -- (-0.5,-0.5) to[out=0,in=180] (0.3,0.2) arc (90:-90:0.2);
\draw[thick] (-1,-0.5) -- (-0.5,-0.5)  to[out=0,in=180] (0.3,0.2) arc (90:-90:0.2);
\draw[white,line width=5] (0.3,-0.2) to[out=180,in=0] (-0.5,0.5) -- (-1,0.5);
\draw[thick] (0.3,-0.2) to[out=180,in=0] (-0.5,0.5) -- (-1,0.5);
\end{tikzpicture}
\end{tikzcd}
$}

\subfigure[$\Res\RM{III}=\RM{iii_a}$]{$
\begin{tikzcd}[ampersand replacement=\&,column sep=1.5pc]
\begin{tikzpicture}[baseline=-.5ex,scale=0.6]
\draw[dashed](0,0) circle (1);
\clip(0,0) circle (1);
\draw[thick] (-1,0) to[out=0,in=180] (0,-0.5) to[out=0,in=180] (1,0);
\draw[thick] (-1,-1) -- (1,1) (-1,1) -- (1,-1);
\end{tikzpicture}
\arrow[r,"\RM{III}"]\arrow[d,Rightarrow] \&
\begin{tikzpicture}[baseline=-.5ex,scale=0.6]
\draw[dashed](0,0) circle (1);
\clip(0,0) circle (1);
\draw[thick] (-1,0) to[out=0,in=180] (0,0.5) to[out=0,in=180] (1,0);
\draw[thick] (-1,-1) -- (1,1) (-1,1) -- (1,-1);
\end{tikzpicture}
\arrow[d,Rightarrow]\\
\begin{tikzpicture}[baseline=-.5ex,scale=0.6]
\draw[dashed](0,0) circle (1);
\clip(0,0) circle (1);
\draw[thick] (-1,-1) -- (1,1);
\draw[white,line width=5] (-1,0) to[out=0,in=180] (0,-0.5) to[out=0,in=180] (1,0);
\draw[thick] (-1,0) to[out=0,in=180] (0,-0.5) to[out=0,in=180] (1,0);
\draw[white,line width=5] (-1,1) -- (1,-1);
\draw[thick] (-1,1) -- (1,-1);
\end{tikzpicture}
\arrow[r,"\RM{iii_b}"] \&
\begin{tikzpicture}[baseline=-.5ex,scale=0.6]
\draw[dashed](0,0) circle (1);
\clip(0,0) circle (1);
\draw[thick] (-1,-1) -- (1,1);
\draw[white,line width=5] (-1,0) to[out=0,in=180] (0,0.5) to[out=0,in=180] (1,0);
\draw[thick] (-1,0) to[out=0,in=180] (0,0.5) to[out=0,in=180] (1,0);
\draw[white,line width=5] (-1,1) -- (1,-1);
\draw[thick] (-1,1) -- (1,-1);
\end{tikzpicture}
\end{tikzcd}
$}
\qquad
\subfigure[$\Res\RM{V}=\RM{iv}$ or $\RM{iv}\circ\RM{iii_a}\circ\cdots\circ\RM{iii_a}$]{$
\begin{tikzcd}[ampersand replacement=\&,column sep=1.5pc]
\begin{tikzpicture}[baseline=-.5ex,scale=0.6]
\draw[dashed](0,0) circle (1);
\clip(0,0) circle (1);
\draw[thick] (1,.6) to[out=180,in=0] (-0.5,0);
\draw[thick] (1,.2) to[out=180,in=0] (-0.5,0);
\draw[thick] (1,-.2) to[out=180,in=0] (-0.5,0);
\draw[thick] (1,-.6) to[out=180,in=0] (-0.5,0);
\draw[fill] (-0.5,0) circle (2pt);
\draw[thick] (0,-1) -- (0.5,1);
\end{tikzpicture}\arrow[r,"\RM{V}"]\arrow[d,Rightarrow] \&
\begin{tikzpicture}[baseline=-.5ex,scale=0.6]
\draw[dashed](0,0) circle (1);
\clip(0,0) circle (1);
\draw[thick] (1,.6) to[out=180,in=0] (-0.5,0);
\draw[thick] (1,.2) to[out=180,in=0] (-0.5,0);
\draw[thick] (1,-.2) to[out=180,in=0] (-0.5,0);
\draw[thick] (1,-.6) to[out=180,in=0] (-0.5,0);
\draw[fill] (-0.5,0) circle (2pt);
\draw[thick] (-1,-1) -- (-0.5,1);
\end{tikzpicture}\arrow[d,Rightarrow]\\
\begin{tikzpicture}[baseline=-.5ex,scale=0.6]
\draw[dashed](0,0) circle (1);
\clip(0,0) circle (1);
\draw[thick] (0,-1) -- (0.5,1);
\draw[white,line width=5] (1,.6) -- (-0.5,0) (1,.2) -- (-0.5,0) (1,-.2) -- (-0.5,0) (1,-.6) -- (-0.5,0);
\draw[thick] (1,.6) -- (-0.5,0) (1,.2) -- (-0.5,0) (1,-.2) -- (-0.5,0) (1,-.6) -- (-0.5,0);
\draw[fill] (-0.5,0) circle (2pt);
\end{tikzpicture}\arrow[r,"\RM{iv}"] \&
\begin{tikzpicture}[baseline=-.5ex,scale=0.6]
\draw[dashed](0,0) circle (1);
\clip(0,0) circle (1);
\draw[thick] (-1,-1) -- (-0.5,1);
\draw[thick] (1,.6) -- (-0.5,0) (1,.2) -- (-0.5,0) (1,-.2) -- (-0.5,0) (1,-.6) -- (-0.5,0);
\draw[fill] (-0.5,0) circle (2pt);
\end{tikzpicture}
\end{tikzcd}
\quad
\begin{tikzcd}[ampersand replacement=\&,column sep=1.5pc]
\begin{tikzpicture}[baseline=-.5ex,scale=0.6]
\draw[dashed](0,0) circle (1);
\clip(0,0) circle (1);
\draw[thick] (-1,.6) to[out=0,in=180] (0.5,0);
\draw[thick] (-1,.2) to[out=0,in=180] (0.5,0);
\draw[thick] (-1,-.2) to[out=0,in=180] (0.5,0);
\draw[thick] (-1,-.6) to[out=0,in=180] (0.5,0);
\draw[fill] (0.5,0) circle (2pt);
\draw[thick] (-0.5,-1) -- (0,1);
\end{tikzpicture}\arrow[rr,"\RM{V}"]\arrow[d,Rightarrow] \& \&
\begin{tikzpicture}[baseline=-.5ex,scale=0.6]
\draw[dashed](0,0) circle (1);
\clip(0,0) circle (1);
\draw[thick] (-1,.6) to[out=0,in=180] (0.5,0);
\draw[thick] (-1,.2) to[out=0,in=180] (0.5,0);
\draw[thick] (-1,-.2) to[out=0,in=180] (0.5,0);
\draw[thick] (-1,-.6) to[out=0,in=180] (0.5,0);
\draw[fill] (0.5,0) circle (2pt);
\draw[thick] (0.5,-1) -- (1,1);
\end{tikzpicture}\arrow[d,Rightarrow]\\
\begin{tikzpicture}[baseline=-.5ex,scale=0.6]
\draw[dashed](0,0) circle (1);
\clip(0,0) circle (1);
\draw[thick] (-0.75,-1) -- (-0.25,1);
\draw[white,line width=3,rounded corners] (-1,-0.45) -- (-0.6,-0.6) -- (0.2,0.6);
\draw[thick,rounded corners] (-1,-0.45) -- (-0.6,-0.6) -- (0.2,0.6) -- (0.5,0);
\draw[white,line width=3,rounded corners] (-1,-0.15) -- (-0.6,-0.2) -- (-0.33,-0.6) -- (0.2,0.2);
\draw[thick,rounded corners] (-1,-0.15) -- (-0.6,-0.2) -- (-0.33,-0.6) -- (0.2,0.2) -- (0.5,0);
\draw[white,line width=3,rounded corners] (-1,0.15) -- (-0.6,0.2) -- (-0.066,-0.6) -- (0.2,-0.2);
\draw[thick,rounded corners] (-1,0.15) -- (-0.6,0.2) -- (-0.066,-0.6) -- (0.2,-0.2)-- (0.5,0);
\draw[white,line width=3,rounded corners] (-1,0.45) -- (-0.6,0.6) -- (0.2,-0.6);
\draw[thick,rounded corners] (-1,0.45) -- (-0.6,0.6) -- (0.2,-0.6) -- (0.5,0);
\draw[fill] (0.5,0) circle (2pt);
\end{tikzpicture}\arrow[r,"\binom{m}2\RM{iii_a}"] \&
\begin{tikzpicture}[baseline=-.5ex,scale=0.6]
\draw[dashed](0,0) circle (1);
\clip(0,0) circle (1);
\draw[thick] (0,-1) -- (0.5,1);
\draw[white,line width=3,rounded corners] (-1,-0.45) -- (-0.6,-0.6) -- (0.2,0.6) -- (0.425, 0.15);
\draw[thick,rounded corners] (-1,-0.45) -- (-0.6,-0.6) -- (0.2,0.6) -- (0.5,0);
\draw[white,line width=3,rounded corners] (-1,-0.15) -- (-0.6,-0.2) -- (-0.33,-0.6) -- (0.2,0.2) -- (0.35,0.1);
\draw[thick,rounded corners] (-1,-0.15) -- (-0.6,-0.2) -- (-0.33,-0.6) -- (0.2,0.2) -- (0.5,0);
\draw[white,line width=3,rounded corners] (-1,0.15) -- (-0.6,0.2) -- (-0.066,-0.6) -- (0.2,-0.2) -- (0.35,-0.1);
\draw[thick,rounded corners] (-1,0.15) -- (-0.6,0.2) -- (-0.066,-0.6) -- (0.2,-0.2)-- (0.5,0);
\draw[white,line width=3,rounded corners] (-1,0.45) -- (-0.6,0.6) -- (0.2,-0.6) -- (0.35,-0.3);
\draw[thick,rounded corners] (-1,0.45) -- (-0.6,0.6) -- (0.2,-0.6) -- (0.5,0);
\draw[thick] (0.5,0) -- (0.35,0.3);
\draw[fill] (0.5,0) circle (2pt);
\end{tikzpicture}\arrow[r,"\RM{iv}"] \&
\begin{tikzpicture}[baseline=-.5ex,scale=0.6]
\draw[dashed](0,0) circle (1);
\clip(0,0) circle (1);
\draw[thick] (0.5,-1) -- (1,1);
\draw[white,line width=3,rounded corners] (-1,-0.45) -- (-0.6,-0.6) -- (0.2,0.6);
\draw[thick,rounded corners] (-1,-0.45) -- (-0.6,-0.6) -- (0.2,0.6) -- (0.5,0);
\draw[white,line width=3,rounded corners] (-1,-0.15) -- (-0.6,-0.2) -- (-0.33,-0.6) -- (0.2,0.2);
\draw[thick,rounded corners] (-1,-0.15) -- (-0.6,-0.2) -- (-0.33,-0.6) -- (0.2,0.2) -- (0.5,0);
\draw[white,line width=3,rounded corners] (-1,0.15) -- (-0.6,0.2) -- (-0.066,-0.6) -- (0.2,-0.2);
\draw[thick,rounded corners] (-1,0.15) -- (-0.6,0.2) -- (-0.066,-0.6) -- (0.2,-0.2)-- (0.5,0);
\draw[white,line width=3,rounded corners] (-1,0.45) -- (-0.6,0.6) -- (0.2,-0.6);
\draw[thick,rounded corners] (-1,0.45) -- (-0.6,0.6) -- (0.2,-0.6) -- (0.5,0);
\draw[fill] (0.5,0) circle (2pt);
\end{tikzpicture}
\end{tikzcd}
$}

\subfigure[$\Res\RM{VI}=\RM{0_b}\circ\RM{ii}\circ\cdots\circ\RM{ii}$]{$
\begin{tikzcd}[ampersand replacement=\&,column sep=1.5pc]
\begin{tikzpicture}[baseline=-.5ex,scale=0.6]
\draw[dashed](0,0) circle (1);
\clip(0,0) circle (1);
\draw[thick] (-1,.6) to[out=0,in=180] (0,0) to[out=0,in=180] (1,0.2);
\draw[thick] (-1,.2) to[out=0,in=180] (0,0) to[out=0,in=180] (1,-0.2);
\draw[thick] (-1,-.2) to[out=0,in=180] (0,0);
\draw[thick] (1,0.6) -- (0,0.6) to[out=180,in=0] (-0.75,-.4) to[out=0,in=180] (0,0);
\draw[fill] (0,0) circle (2pt);
\end{tikzpicture}\arrow[rr,"\RM{V}"]\arrow[d,Rightarrow] \& \&
\begin{tikzpicture}[baseline=-.5ex,scale=0.6]
\draw[dashed](0,0) circle (1);
\clip(0,0) circle (1);
\draw[thick] (-1,.4) to[out=0,in=180] (0,0) to[out=0,in=180] (1,0);
\draw[thick] (-1,0) to[out=0,in=180] (0,0) to[out=0,in=180] (1,-0.4);
\draw[thick] (-1,-.4) to[out=0,in=180] (0,0) to[out=0,in=180] (1,0.4);
\draw[fill] (0,0) circle (2pt);
\end{tikzpicture}\arrow[d,Rightarrow]\\
\begin{tikzpicture}[baseline=-.5ex,scale=0.6]
\draw[dashed](0,0) circle (1);
\clip(0,0) circle (1);
\draw[thick,rounded corners] (1,0.6) -- (0,0.6) -- (-0.9,-0.4) -- (-0.6,-0.7) -- (0.2,0.5) -- (0.5,0);
\draw[white,line width=3,rounded corners] (-1,-0.15) -- (-0.6,-0.2) -- (-0.33,-0.6) -- (0.2,0.2);
\draw[thick,rounded corners] (-1,-0.15) -- (-0.6,-0.2) -- (-0.33,-0.6) -- (0.2,0.2) -- (0.5,0);
\draw[white,line width=3,rounded corners] (-1,0.15) -- (-0.6,0.2) -- (-0.066,-0.6) -- (0.2,-0.2);
\draw[thick,rounded corners] (-1,0.15) -- (-0.6,0.2) -- (-0.066,-0.6) -- (0.2,-0.2)-- (0.5,0);
\draw[white,line width=3,rounded corners] (-1,0.45) -- (-0.6,0.6) -- (0.2,-0.6);
\draw[thick,rounded corners] (-1,0.45) -- (-0.6,0.6) -- (0.2,-0.6) -- (0.5,0);
\draw[thick] (0.5,0) -- (1,0) (0.5,0) -- (1,-0.4);
\draw[fill] (0.5,0) circle (2pt);
\end{tikzpicture}\arrow[r,"\ell\RM{ii}"] \&
\begin{tikzpicture}[baseline=-.5ex,scale=0.6]
\draw[dashed](0,0) circle (1);
\clip(0,0) circle (1);
\draw[thick,rounded corners] (-1,-0.4) -- (-0.8,-0.4) -- (-0.25,0.2) -- (0,0);
\draw[white,line width=3,rounded corners] (-1,0) -- (-0.8,0) -- (-0.5,-0.4) -- (-0.25,0);
\draw[thick,rounded corners] (-1,0) -- (-0.8,0) -- (-0.5,-0.4) -- (-0.25,0) -- (0,0);
\draw[white,line width=3,rounded corners] (-1,0.4) -- (-0.8,0.4) -- (-0.25,-0.4);
\draw[thick,rounded corners] (-1,0.4) -- (-0.8,0.4) -- (-0.25,-0.4) -- (0,0);
\draw[thick,rounded corners] (0,0) -- (-0.3,0.4) -- (1,0.4);
\draw[thick] (0,0) -- (1,0) (0,0) -- (1,-0.4);
\draw[fill] (0,0) circle (2pt);
\end{tikzpicture}\arrow[r,"\RM{0_b}"] \&
\begin{tikzpicture}[baseline=-.5ex,scale=0.6]
\draw[dashed](0,0) circle (1);
\clip(0,0) circle (1);
\draw[thick,rounded corners] (-1,-0.4) -- (-0.8,-0.4) -- (-0.25,0.4) -- (0,0);
\draw[white,line width=3,rounded corners] (-1,0) -- (-0.8,0) -- (-0.5,-0.4) -- (-0.25,0);
\draw[thick,rounded corners] (-1,0) -- (-0.8,0) -- (-0.5,-0.4) -- (-0.25,0) -- (0,0);
\draw[white,line width=3,rounded corners] (-1,0.4) -- (-0.8,0.4) -- (-0.25,-0.4);
\draw[thick,rounded corners] (-1,0.4) -- (-0.8,0.4) -- (-0.25,-0.4) -- (0,0);
\draw[thick] (0,0) -- (1,0.4) (0,0) -- (1,0) (0,0) -- (1,-0.4);
\draw[fill] (0,0) circle (2pt);
\end{tikzpicture}
\end{tikzcd}
$}

\caption{Ng's resolutions for front Reidemeister moves}
\label{figure:Reidemeister moves for Ng's resolution}
\end{figure}
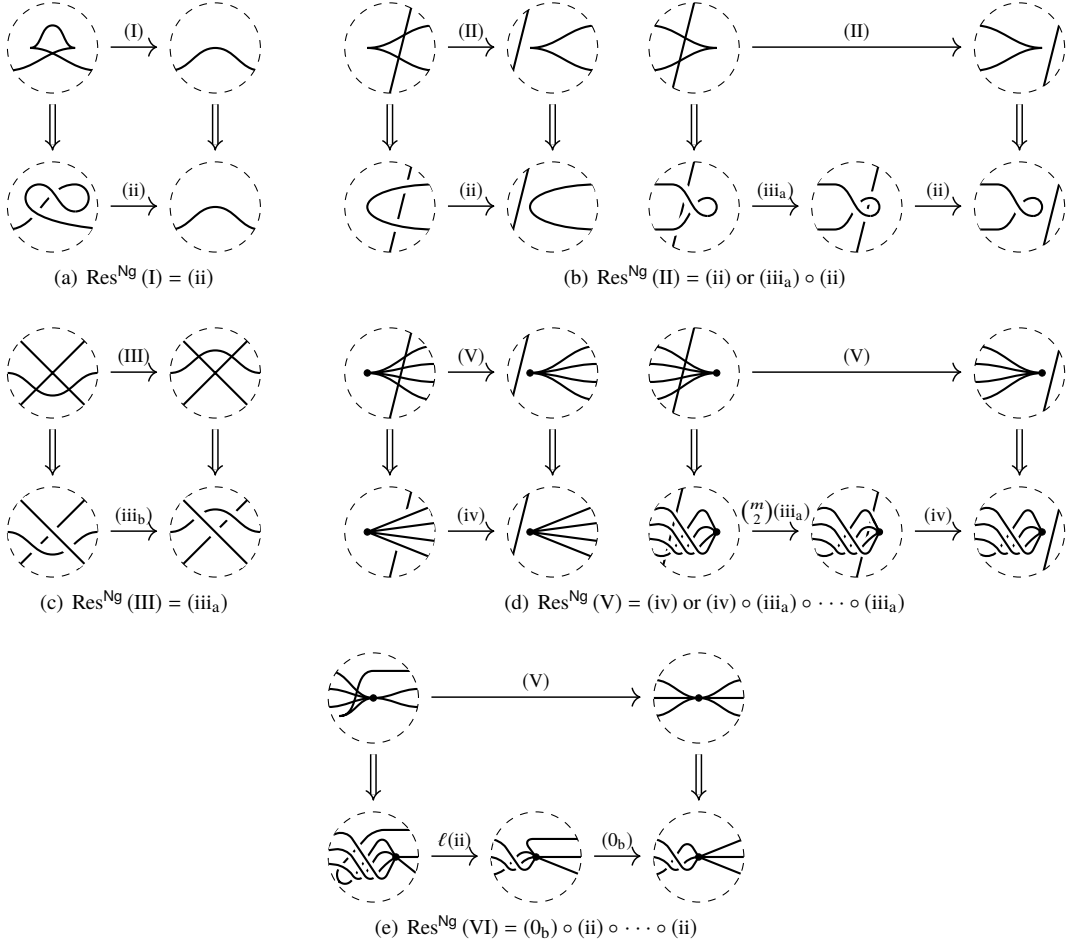

\section{Bordered DGA invariants for bordered Legendrian graphs}\label{section:bordered graphs}
In this section, we will consider a \emph{bordered Chekanov-Eliashberg's DGA} associated to each bordered Legendrian graph.

\subsection{Preliminaries on DGAs}
Throughout this paper, we mean by a \emph{differential graded algebra} (DGA) a pair $A=(\alg,\differential)$ of a unital associative graded free algebra $\alg$ over the Laurent polynomial ring $\ring[t_1^{\pm1},\dots,t^{\pm1}_k]$ for some $k\ge 0$.

\begin{example/definition}[Border DGAs]\label{example:border DGA}
Let $n\ge 1$ and $\mu:[n]\to\grading$ be a function.
The \emph{border DGA} $A_n(\mu)=(\alg_n,\differential_n)$ is defined as follows:
\begin{itemize}
\item The graded algebra $\alg_n\coloneqq\ring\left\langle K_n\right\rangle$ is freely generated by the graded set $K_n$, where 
\begin{align*}
K_n&\coloneqq \{k_{ab}\mid 1\le a<b\le n\},&
|k_{ab}|&\coloneqq \mu(a)-\mu(b)-1.
\end{align*}
\item The differential is given as 
\begin{equation}\label{equation:differential for border dga}
\differential_n(k_{ab})\coloneqq\sum_{a<c<b}(-1)^{|k_{ac}|-1}k_{ac}k_{cb}.
\end{equation}
\end{itemize}
\end{example/definition}

\begin{example/definition}[Internal DGAs]\label{example:internal DGA}
Let $\ell,r\ge0$ with $\ell+r=n\ge1$ and $\mu:\Zmod{n}\to\grading$ be a function.
The \emph{internal DGA} $I_{(\ell,r)}(\mu)=(\sfI_{(\ell,r)},\differential_n)$ of type $(\ell,r)$ is defined as follows:
\begin{itemize}
\item The graded algebra $\sfI_{(\ell,r)}\coloneqq\ring\left\langle \tilde V_{(\ell,r)} \right\rangle$ is freely generated over $\ring$ by the graded set $\tilde V_{(\ell,r)}$, where 
\begin{align*}
\tilde V_{(\ell,r)}&\coloneqq \{ \vg_{a,i}\mid a\in\Zmod{n},i\ge1 \},&
|\vg_{a,i}|&\coloneqq \mu(a)-\mu(a+i)+(n(\ell,r,a,i)-1),
\end{align*}
where $n(\ell,r,a,i)$ is the number defined as follows:\footnotemark
\begin{align*}
n(\ell,r,a,i)&\coloneqq\sum_{j=a}^{a+i-1} n(\ell,r,j,1),&
n(\ell,r,a,1)&\coloneqq\begin{cases}
2 & \ell r=0, a=n;\\
1 & \ell r\neq 0, a=r\text{ or }n;\\
0 &\text{otherwise}.
\end{cases}
\end{align*}
Therefore for $\ell_1+\ell_2=\ell$, we have
\[
|\vg_{a,\ell}| = |\vg_{a,\ell_1}|+|\vg_{a+\ell_1,\ell_2}|+1
\]
and moreover, $|\vg_{a,n}|=1$ for any $a\in \Zmod{n}$.

\item The differential is given as
\[
\differential_n (\vg_{a,i}) \coloneqq \delta_{i,n}+\sum_{i_1+i_2=i}(-1)^{|\vg_{a,i_1}|-1}\vg_{a,i_1}\vg_{a+i_1,i_2}.
\]
\end{itemize}

We will simply denote $I_{(0,n)}(\mu)$ by $I_n(\mu)$.
\end{example/definition}
\footnotetext{The geometric description for the number $n(\ell,r,a,i)$ will be given in Definition~\ref{definition:spiral curves}.}

\begin{remark}
For any $l+r=n$ and $\mu:\Zmod{n}\to\grading$, there exists $\mu'$ defined as $\mu'(a)=\mu(a)+1$ for $a\le \ell$ or $\mu(a)$ for $a>\ell$, where $I_{(\ell,r)}(\mu)\isomorphic I_n(\mu')$.
Therefore all internal DGAs can be assumed to be of type $(0,n)$ for some $n>0$.
\end{remark}

\begin{proposition}\label{proposition:differentials for border DGAs}
The assignment $k_{ab}\mapsto \vg_{a,b-a}$ defines a inclusion $A_n(\mu)\to I_n(\mu)$ between DGAs.
\end{proposition}
\begin{proof}
This follows easily from the direct computation.
\end{proof}

\begin{definition}[DGA homotopies]
Let $f,g:A'\to A$ be two DGA morphisms. 
We say that $f$ and $g$ are \emph{homotopic} via $H$, denoted by $f\stackrel{H}\homotopic g$
if there exists a \emph{DGA homotopy} $H:A\to A'$ which is an algebra homomorphism of degree 1 satisfying the following:
\begin{itemize}
\item $H$ is a chain homotopy:
\begin{equation}\label{equation:chain homotopy}
f-g=\differential\circ H + H\circ\differential'.
\end{equation}
\item $H$ is a $(f,g)$-derivation: for all $x,y\in\alg'$,
\begin{equation}\label{equation:derivation}
H(xy) = H(x) g(y)+(-1)^{|x|}f(x) H(y).
\end{equation}
\end{itemize}
\end{definition}

\begin{remark}\label{remark:homotopy_equiv}
It is known that the homotopy relation becomes an equivalence relation, but this is not straightforward at all. See \cite[Ch.~26]{FHT01}.
\end{remark}

Let us recall the notion of tame isomorphisms on DGAs.
\begin{definition}[Tame isomorphisms]
Let $A'=(\alg'=\ring\langle G\rangle,\differential')$ and $A=(\alg=\ring\langle G\rangle,\differential)$ be DGAs.
A DGA map $f: A'\to A$ is called an \emph{elementary isomorphism} if for some $g\in G$,
\begin{align*}
f(h)=
\begin{cases}
g+u&\text{ if }h=g\\
h&\text{ if }h\neq g,
\end{cases}
\end{align*}
where $u$ is a word in $A$ not containing $g$.

A \emph{tame isomorphism} is a composition of countably many (possibly finite) elementary isomorphisms.
\end{definition}

\begin{definition}[Stabilizations {\cite[Definition~2.16]{EN2015}}]\label{definition:stabilizations}
For a DGA $A=(\alg=\ring\langle G\rangle,\differential)$, a \emph{stabilization} of $A$ is a DGA which is tame isomorphic to a DGA $SA=(S\alg,\bar\differential)$ obtained from $A$ by adding a countably many (possibly finite) number of canceling pairs of generators $\{\hat e^i, e^i\mid i\in I\}$ for some index set $I$ so that
\begin{align*}
S\alg&=\ring\left\langle G\amalg\{\hat e^i, e^i\mid i\in I\}\right\rangle,&
|\hat e^i| &=|e^i| +1,&
\bar\differential(\hat e^i)&= e^i,&
\bar\differential(e^i)&= 0,
\end{align*}
and for each $\fd\in\grading$, there are only finitely many $\hat e^i$'s and $e^i$'s of degree $\fd$.
\end{definition}

For a stabilization $A'$ of $A$, there are two natural DGA morphisms $\pi:A'\to A$ and $\iota:A\to A'$ defined as the canonical projection and inclusion
\begin{align*}
\pi(s)&\coloneqq \begin{cases}
s & s\in A;\\
0 & \text{otherwise},
\end{cases}&
\iota(s)&\coloneqq s.
\end{align*}

We introduce a special kind of a DGA operation called a \emph{generalized stabilization} originally defined in \cite{AB2018} as follows:
\begin{definition}[Generalized stabilization]\label{def:Generalized stabilization}
Let $A=(\alg=\ring\langle G\rangle, \differential)$ be a DGA generated by the set $G$ and $\varphi: I_n(\mu)\to A$ be a DGA morphism for some $\mu:\Zmod{n}\to\grading$.

The \emph{$\fd$-th positive} or \emph{negative stabilization} of $A$ with respect to $\varphi$ is the DGA $S_\varphi^{\fd \pm}A=\left(S^{\fd}_\varphi\alg,\bar\differential^\pm\right)$ defined as follows: 
let $v_{a,i}\coloneqq\varphi(\vg_{a,i})\in A$.
\begin{itemize}
\item The graded algebra $S^{\fd}_\varphi \sfA$ is given as
\begin{align*}
S^{\fd}_\varphi \sfA&\coloneqq\ring\left\langle G\amalg\{e^1,\dots,e^n\}\right\rangle,&
|e^b| &\coloneqq\fd+\sum_{a<b}(|v_{a,1}|+1).
\end{align*}
\item The differentials $\bar\differential^\pm$ for $e^b$ are given as
\begin{align*}
\bar\differential^+ (e^b)&\coloneqq \sum_{a<b}(-1)^{|e^a|-1}e^a v_{a,b-a},&
\bar\differential^- (e^b)&\coloneqq \sum_{a<b}v_{n+1-b,b-a} e^a.
\end{align*}
\end{itemize}
\end{definition}

As observed in \cite[Remark~3.8]{AB2018}, the generalized stabilization is a composition of a stabilization and a destabilization.
For the completeness, we will give a proof in Appendix~\ref{appendix:proof of EN's stabilization}.
\begin{proposition}\cite[Appendix~B.2]{EN2015}\label{proposition:generalized stabilizations}
There exists a DGA $\tilde S^{\fd\pm}_\varphi A$ which is a common stabilization of both $A$ and $S^{\fd\pm}_\varphi A$.
\end{proposition}

\subsection{Bordered DGAs}
We consider DGAs together with additional structures, called \emph{bodered DGAs}.
\begin{definition}[Bordered DGAs]
A \emph{bordered DGA $\dga$ of type $\left(n_\Left, n_\Right\right)$} is a diagram
\[
\dga=\left(
A_\Left\stackrel{\phi_\Left}\longrightarrow A\stackrel{\phi_\Right}\longleftarrow A_\Right
\right)
\]
consisting of DGAs $A, A_\Left$ and $A_\Right$, and two DGA morphisms $\phi_\Left:A_\Left\to A$ and $\phi_\Right:A_\Right\to A$ such that
\begin{enumerate}
\item $\phi_\Left$ is \emph{injective}, and
\item for some $\mu_\Left:[n_\Left]\to\grading$ and $\mu_\Right:[n_\Right]\to\grading$,
\[
A_\Left\isomorphic A_{n_\Left}(\mu_\Left)\quad\text{ and }\quad
A_\Right\isomorphic A_{n_\Right}(\mu_\Right).
\]
\end{enumerate}

A \emph{bordered quasi-morphism} $\bff:\dga'\to\dga$ between two bordered DGAs $\dga$ and $\dga'$ is a triple $(f_\Left,f,f_\Right)$ of DGA morphisms such that in the following diagram, the left square is commutative and the right square is commutative \emph{up to homotopy}
\begin{align*}
\begin{tikzcd}[column sep=4pc, row sep=2pc, ampersand replacement=\&]
\dga'\arrow[d,"\bff"']\\
\dga
\end{tikzcd}&=\ \left(
\begin{tikzcd}[column sep=4pc, row sep=2pc, ampersand replacement=\&]
A'_\Left\arrow[r,"\phi'_\Left"]\arrow[d,"f_\Left"']
\& A'\arrow[d,"f"'] \& A'_\Right\arrow[l,"\phi'_\Right"']\arrow[d,,"f_\Right"]\arrow[Rightarrow,ld,"\exists H_\Right",sloped]\\
A_\Left\arrow[r,"\phi_\Left"] \& A \& A_\Right\arrow[l,"\phi_\Right"']
\end{tikzcd}\right)
\end{align*}
and a bordered quasi-morphism is a \emph{bordered morphism} if the homotopy $H_\Right$ is trivial, i.e., the squares above are strictly commutative.
\end{definition}

Now we consider stabilizations of bordered DGAs as follows:
\begin{definition}[Stabilizations of bordered DGAs]
for each bordered DGA $\dga$, we say that a bordered DGA $\dga'$ of $\dga'$ is said to be a \emph{(weak) stabilization} if there exists a bordered morphism $\bfpi=(\identity_\Left,\pi,\identity_\Right):\dga'\to\dga$ such that $\pi:A'\to A$ is the canonical projection, and a \emph{strong stabilization} if it is a weak stabilization and there exists a bordered morphism $\bfi=(\identity, \iota, \identity)$ where $\iota:A\to A'$ is the canonical inclusion.
\end{definition}

As an example of (weak) stabilizations, we consider the mapping cylinder in the bordered setting as follows: let $\dga=\left(A_\Left\stackrel{\phi_\Left}\longrightarrow A\stackrel{\phi_\Right}\longleftarrow A_\Right\right)$ be a bordered DGA of type $\left(n_\Left,n_\Right\right)$. That is,
\begin{align*}
A&=(\alg,\differential),& \alg&=\ring\langle G\rangle,\\
A_\Left&=(\alg_\Left,\differential_\Left),&
\alg_\Left&=\ring\langle K_\Left\rangle,&
K_\Left&\coloneqq\{k_{a'b'}\mid 1\le a'<b'\le n_\Left\},\\
A_\Right&=(\alg_\Right,\differential_\Right),&
\alg_\Right&=\ring\langle K_\Right\rangle,&
K_\Right&\coloneqq\{k_{ab}\mid 1\le a<b\le n_\Right\}.
\end{align*}

\begin{definition}[Mapping cylinders of bordered DGAs]\label{definition:mapping cylinder}
The \emph{mapping cylinder} $\hat\dga\coloneqq\left(
A_\Left\stackrel{\hat\phi_\Left}\longrightarrow \hat A \stackrel{\hat\phi_\Right}\longleftarrow A_\Right
\right)$ of $\dga$ will be defined as follows:
\begin{itemize}
\item The DGA $\hat A=(\hat\alg,\hat\differential)$ and its graded algebra $\hat \alg$ is defined as
\begin{align*}
\hat\alg&\coloneqq \ring\left\langle G\amalg K_\Right \amalg \hat K_\Right\right\rangle,&
\hat K_\Right&\coloneqq\left\{\hat k_{ab}~\middle|~k_{ab}\in K_\Right\right\},&
\left|\hat k_{ab}\right|&\coloneqq |k_{ab}|+1.
\end{align*}
\item The differential $\hat\differential$ for each $k_{ab}$ is the same as $\differential_\Right(k_{ab})$ and for $\hat k_{ab}$ it is defined as 
\begin{align}\label{equation:differential for mapping cylinder}
\hat\differential(\hat k_{ab})&\coloneqq 
k_{ab}-\phi_\Right(k_{ab})+\sum_{a<c<b}(-1)^{|\hat k_{ac}|-1}\hat k_{ac}\phi_\Right( k_{cb})+ k_{ac}\hat k_{cb}.
\end{align}
\item The morphism $\hat\phi_\Left$ is the composition of $\phi_\Left$ and the canonical inclusion $A\to\hat A$, and $\hat\phi_\Right$ is defined by $\hat\phi_\Right(k_{ab})=k_{ab}\in \hat\alg$.
\end{itemize}
\end{definition}

\begin{lemma}\label{lemma:mapping cylinders are stabilizations}
The DGA $\hat A$ is a weak stabilization of $A$ but not a strong stabilization.
\end{lemma}
\begin{proof}
Notice that in each differential $\hat\differential(\hat k_{ab})$ there is one and only one generator $k_{ab}$.
Therefore one may find a tame automorphism $\Phi$ on $\hat A$ that sends $\hat\differential(\hat k_{ab})$ to $k_{ab}$ so that the DGA $(\hat \alg,\hat\differential_\Phi)$ with the twisted differential becomes a stabilization of $A$.

Indeed, we have a canonical projection $\hat\bfpi:\hat\dga\to\dga$ which maps $k_{ab}\mapsto \phi_\Right(k_{ab})$ and $\hat  k_{ab}\mapsto0$.
Then it gives us a strictly commutative diagram
\begin{align}\label{equation:canonical projection of mapping cylinder}
\begin{tikzcd}[column sep=4pc, row sep=2pc, ampersand replacement=\&]
\hat\dga\arrow[d,"\hat\bfpi"']\\
\dga
\end{tikzcd}&=\ \left(
\begin{tikzcd}[column sep=4pc, row sep=2pc, ampersand replacement=\&]
A_\Left\arrow[r,"\hat\phi_\Left"]\arrow[d,equal] \& \hat A\arrow[d,"\hat\pi"] \& A_\Right\arrow[l,"\hat\phi_\Right"']\arrow[d,equal]\\
A_\Left\arrow[r,"\phi_\Left"] \& A \& A_\Right\arrow[l,"\phi_\Right"']
\end{tikzcd}\right)~,
\end{align}
and therefore it is a bordered morphism.
However, the canonical inclusion $\hat \bfi:\dga\to\hat\dga$ makes the following diagram commutative up to homotopy
\begin{align}\label{equation:canonical inclusion of mapping cylinder}
\begin{tikzcd}[column sep=4pc, row sep=2pc, ampersand replacement=\&]
\dga\arrow[d,"\hat \bfi"']\\
\hat\dga
\end{tikzcd}&=\ \left(
\begin{tikzcd}[column sep=4pc, row sep=2pc, ampersand replacement=\&]
A_\Left\arrow[r,"\phi_\Left"]\arrow[d,equal] \& A\arrow[d,"\hat \iota"] \& A_\Right\arrow[l,"\phi_\Right"']\arrow[d,equal]\arrow[Rightarrow,ld,"\exists \hat H_\Right",sloped]\\
A_\Left\arrow[r,"\hat\phi_\Left"] \& \hat A \& A_\Right\arrow[l,"\hat\phi_\Right"']
\end{tikzcd}\right)~,
\end{align}
where the homotopy $\hat H_\Right$ is obviously defined as $\hat H_\Right( k_{ab})\coloneqq \hat  k_{ab}$.
Therefore it only gives us a quasi-morphism.
\end{proof}

On the other hand, we have the following general lemma.
\begin{lemma}\label{lemma:mapping cylinder differential}
Let $f:K_\Right\to A$ be a function and $\bar f:A_\Right\to A$ be an algebra homomorphism extending $f$.
Then the endomorphism
\[
\hat\differential(\hat k_{ab})\coloneqq k_{ab}-f(k_{ab})+\sum_{a<c<b} (-1)^{|\hat k_{ac}|-1}\hat k_{ac} f(k_{cb})+k_{ac}\hat k_{cb}
\]
defines a differential on $\hat A$, i.e., $\hat\differential^2=0$, if and only if $\bar f$ is a DGA morphism.
\end{lemma}
\begin{proof}
It follows easily from the induction on $b-a$ and the direct computation. We omit the proof.
\end{proof}

\subsection{Chekanov-Eliashberg DGAs for bordered Legendrian graphs}\label{section:CE DGA}
Now we briefly review the construction of Chekanov-Eliashberg DGAs for Legendrian graphs with Maslov potentials.

Let $T_\lag=(V_\lag, E_\lag,B_\lag)\in\LT_\lag$ be a Lagrangian projection of a Legendrian graph with a Maslov potential $\mu$.
Then we consider two graded sets as follows:
\begin{itemize}
\item The set $C_\lag$ of crossings in $T_\lag$, where for each $c\in C_\lag$, its grading is defined as
\[
|c|\coloneqq\mu(c^+)-\mu(c^-)\quad\text{ if }c=\crossingpositive,\quad\text{ or }\quad
|c|\coloneqq\mu(c^+)-\mu(c^-)-1\quad\text{ if }c=\crossingnegative,
\]
where $c^+$ and $c^-$ in $R_\lag=T_\lag\setminus(V_\lag\amalg S_\lag)$ are components containing the upper and lower preimage of the crossing $c$ with respect to the $z$-coordinate.
\item The set $\tilde V_\lag$ of infinitely many generators from each vertex,
\[
\tilde V_\lag\coloneqq\{v_{a,i} \mid a\in \Zmod{\val(v)},\ i\geq1,\ v\in V_\lag\},
\]
whose grading is defined as
\begin{equation}\label{equation:degree of vertex generators}
|v_{a,i}| \coloneqq\mu_v(h_{v,a})-\mu_v(h_{v,a+i})+(n(v,a,i)-1),
\end{equation}
where $n(v,a,i)\coloneqq n(\ell,r,a,i)$ is the number defined in Example~\ref{example:internal DGA}.
\end{itemize}

Indeed, one can interpret each $v_{a,i}$ geometrically as follows:
\begin{definition}[Vertex generators and spiral curves]\label{definition:spiral curves}
Let $v$ be a vertex of type $(\ell,r)$.
Each $v_{a,i}$ are called a \emph{vertex generator} which may be regarded as a spiral curve $\gamma(v,a,i)$ which starts from the $a$-th half-edge at the vertex $v$ and rotates $i$ sectors in a clockwise direction as depicted in Figure~\ref{figure:vertex generators}.
\end{definition}

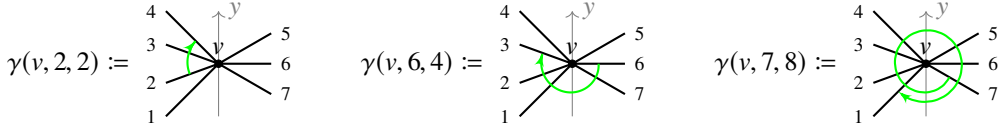
\begin{figure}[ht]
\begin{align*}
\gamma(v,2,2)&\coloneqq
\begin{tikzpicture}[baseline=-.5ex,scale=0.7]
\draw[gray,->] (0,-1)--(0,1) node[right] {$y$};
\draw[fill] (0,0) circle (2pt) node[above] {$v$};
\draw[thick] (0,0) -- (1,{tan(30)}) node[right] {\scriptsize$5$} (0,0) -- (1,0) node[right] {\scriptsize$6$} (0,0) -- (1,{-tan(30)}) node[right] {\scriptsize$7$};
\draw[thick] (0,0) -- (-1,-1) node[left] {\scriptsize$1$} (0,0) -- (-1,{tan(-200)}) node[left] {\scriptsize$2$} (0,0) -- (-1,{tan(200)}) node[left] {\scriptsize$3$} (0,0) -- (-1,1) node[left] {\scriptsize$4$};
\draw[thick,green,-latex'] plot[domain=-160:-225,smooth] ({0.5*(1-\x/1000)*cos(\x)},{0.5*(1-\x/1000)*sin(\x)});
\end{tikzpicture}&
\gamma(v,6,4)&\coloneqq
\begin{tikzpicture}[baseline=-.5ex,scale=0.7]
\draw[gray,->] (0,-1)--(0,1) node[right] {$y$};
\draw[fill] (0,0) circle (2pt) node[above] {$v$};
\draw[thick] (0,0) -- (1,{tan(30)}) node[right] {\scriptsize$5$} (0,0) -- (1,0) node[right] {\scriptsize$6$} (0,0) -- (1,{-tan(30)}) node[right] {\scriptsize$7$};
\draw[thick] (0,0) -- (-1,-1) node[left] {\scriptsize$1$} (0,0) -- (-1,{tan(-200)}) node[left] {\scriptsize$2$} (0,0) -- (-1,{tan(200)}) node[left] {\scriptsize$3$} (0,0) -- (-1,1) node[left] {\scriptsize$4$};
\draw[thick,green,-latex'] plot[domain=0:-200,smooth] ({0.5*(1-\x/1000)*cos(\x)},{0.5*(1-\x/1000)*sin(\x)});
\end{tikzpicture}&
\gamma(v,7,8)&\coloneqq
\begin{tikzpicture}[baseline=-.5ex,scale=0.7]
\draw[gray,->] (0,-1)--(0,1) node[right] {$y$};
\draw[fill] (0,0) circle (2pt) node[above] {$v$};
\draw[thick] (0,0) -- (1,{tan(30)}) node[right] {\scriptsize$5$} (0,0) -- (1,0) node[right] {\scriptsize$6$} (0,0) -- (1,{-tan(30)}) node[right] {\scriptsize$7$};
\draw[thick] (0,0) -- (-1,-1) node[left] {\scriptsize$1$} (0,0) -- (-1,{tan(-200)}) node[left] {\scriptsize$2$} (0,0) -- (-1,{tan(200)}) node[left] {\scriptsize$3$} (0,0) -- (-1,1) node[left] {\scriptsize$4$};
\draw[thick,green,-latex'] plot[domain=-30:-495,smooth] ({0.5*(1-\x/1000)*cos(\x)},{0.5*(1-\x/1000)*sin(\x)});
\end{tikzpicture}
\end{align*}
\caption{Examples of spiral curves}
\label{figure:vertex generators}
\end{figure}

\begin{remark}
The number $n(v,a,i)$ is the number of intersection between the $y$-axis and the spiral curve $\gamma(v,a,i)$.
For examples, the spiral curves in Figure~\ref{figure:vertex generators} give us the numbers $n$ as
\begin{align*}
n(v,2,2)&=0,&
n(v,6,4)&=1,&
n(v,7,8)&=3.
\end{align*}

Moreover, the following additivity holds for any $a\in\Zmod{\val(v)}$ and $i_1+i_2=i$
\begin{equation}\label{equation:additivity of n}
n(v,a,i) = n(v,a,i_1)+n(v,a+i_1,i_2)
\end{equation}
\end{remark}

\begin{definition}
The graded algebra $\alg^\CE(T_\lag,\mu)$ is defined to be the unital associative algebra over the Laurent polynomial ring
$\ring\left[t_b,t_b^{-1}~\middle|~b\in B_\lag\right]$, which is freely generated by the graded set $C_\lag \amalg \tilde V_\lag$.
\begin{align*}
\alg^\CE(T_\lag,\mu)&\coloneqq\ring\left[t_b,t_b^{-1}~\middle|~b\in B_\lag\right]\left\langle C_\lag \amalg \tilde V_\lag\right\rangle.
\end{align*}
\end{definition}

The differential on $\tilde V_\lag$ will be defined as the differential for the internal DGA. That is,
\begin{equation}\label{equation:differential for vertex generators}
\differential(v_{a,i})\coloneqq \delta_{\val(v),i}+\sum_{i_1+i_2=i}(-1)^{|v_{a,i_1}|+1}v_{a,i_1}v_{a+i_1,i_2}.
\end{equation}
Then for each $v$, the DG-subalgebra $I_v\subset A^\CE(T_\lag,\mu)$ generated by $v_{a,i}$'s is isomorphic to an internal DGA 
\begin{equation}\label{equation:internal DG-subalgebra}
I_v\isomorphic I_{(\ell,r)}(\mu_v),
\end{equation}
under the identification $\vg_{a,i}\mapsto v_{a,i}$, where $\mu_v$ is the restriction of $\mu$ to the set of half-edges of $v$.

Let $\Pi_t$ be a $(t+1)$-gon, and denote its boundary and vertices by $\boundary \Pi_t$ and $V\Pi_t=\{\bfx_0,\dots, \bfx_t \}$, respectively.
The differential for each crossing $c\in C$ are given by counting orientation preserving immersed polygons 
\begin{align}\label{eqn:polygon_map}
f:(\Pi_t,\boundary \Pi_t, V\Pi_t)\to(\RR^2,T_\lag,C_\lag\amalg V_\lag\amalg B_\lag)
\end{align}
which satisfies the following conditions:
\begin{itemize}
\item near $\bfx_0$, $f$ passes a convex corner of $c$ with positive Reeb sign in Figure~\ref{figure:Reeb_sign};
\item near $\bfx_k$ for each $k\ge 1$, $f$ passes either a convex corner of a crossing with negative Reeb sign, a half space of a base point, or a vertex.
\end{itemize}

\begin{figure}[ht]
\subfigure[Reeb sign\label{figure:Reeb_sign}]{\makebox[0.4\textwidth]{
$\begin{tikzpicture}[scale=0.5,baseline=-.5ex]
\draw[thick] (-1,-1) -- (1,1);
\draw[line width=5pt,white] (-1,1) -- (1,-1);
\draw[thick] (-1,1) -- (1,-1);
\draw(0,0) node[above] {$-$} node[below] {$-$} node[left] {$+$} node[right] {$+$};
\begin{scope}[xshift=3cm]
\draw[thick] (-1,1) -- (1,-1);
\draw[line width=5pt,white] (-1,-1) -- (1,1);
\draw[thick] (-1,-1) -- (1,1);
\draw(0,0) node[above] {$+$} node[below] {$+$} node[right] {$-$} node[left] {$-$};
\end{scope}
\end{tikzpicture}$}}
\subfigure[Orientation sign\label{figure:ori_sign}
]{\makebox[0.4\textwidth]{
$\begin{tikzpicture}[scale=0.5,baseline=-.5ex]
\begin{scope}
\draw[thick] (-1,-1) -- (1,1);
\draw[line width=5pt,white] (-1,1) -- (1,-1);
\fill[gray,opacity=0.4](-0.5,-0.5)--(0.5,0.5) arc (45:-135:0.707);
\draw[thick] (-1,1) -- (1,-1);
\draw(0,0) node[above] {$+$} node[below] {$\epsilon$} node[left] {$+$} node[right] {$\epsilon$};
\end{scope}
\begin{scope}[xshift=3cm]
\draw[thick] (-1,1) -- (1,-1);
\draw[line width=5pt,white] (-1,-1) -- (1,1);
\fill[gray,opacity=0.4](0.5,-0.5)--(-0.5,0.5) arc (135:-45:0.707);
\draw[thick] (-1,-1) -- (1,1);
\draw(0,0) node[above] {$\epsilon$} node[below] {$+$} node[right] {$\epsilon$} node[left] {$+$};
\draw(1.5,0) node[right] {$\epsilon\coloneqq \epsilon(\sfc)=(-1)^{|\sfc|-1}$};
\end{scope}
\end{tikzpicture}$}}
\caption{Reeb sign and orientation sign}
\label{figure:sign}
\end{figure}
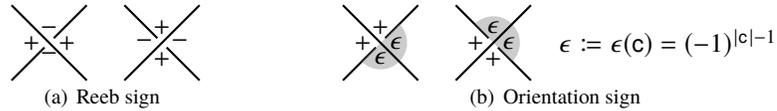

For such a map $f$ defined on $\Pi_t$ with a vertex $\bfx\in V\Pi_t$, the \emph{sign} $\sgn(f,\bfx)$ of  $f$ at $\bfx$ is the orientation sign defined as follows:
\begin{align*}
\sgn(f,\bfx)=
\begin{cases}
\text{the sign in Figure}~\ref{figure:ori_sign}&f(\bfx)\in C_\lag;\\
1 &f(\bfx)\in V_\lag.
\end{cases}
\end{align*}
The \emph{evaluation} $\tilde f(\bfx)$ of $f$ at $\bfx$ is given by
\begin{align*}
\widetilde f(\bfx)=
\begin{cases}
\sgn(f,\bfx)c & \text{ if }f(\bfx)=c\in C_\lag;\\
v_{a,j} & \text{ if } f(\bfx)=v\in V_\lag \text{, $f$ occupies $j$-sectors from $h_{v,a}$ to $h_{v,a+j}$ near $\bfx$};\\
t_b & f(\bfx)=b\in B_\lag, \text{ the orientation of $f$ matches with that of $b$};\\
t_b^{-1} & f(\bfx)=b\in B_\lag, \text{ the orientation of $f$ does not matches with that of $b$}.
\end{cases}
\end{align*}
Here, we are using the orientation convention for each base point $b$ as defined in Definition~\ref{definition:types of vertices}.

The \emph{grading} $|f|$ of $f$ defined on $\Pi_t$ is assigned by
\[
|f|=|\widetilde f(\bfx_0) |-\sum_{i=1}^t|\widetilde f(\bfx_i)|.
\]
For a crossing $c\in C_\lag$, consider the following moduli space
\[
\modulispace_t(c)\coloneqq\left\{f
~\middle|~ f \text{ is a map of }\eqref{eqn:polygon_map},\ f(\bfx_0)=c,\ |f|=1\right\}\big/\operatorname{Diff}^+(\Pi_t,\boundary\Pi_t,V\Pi_t).
\]

The \emph{differential} $\differential(c)$ of the DGA $A^\CE(T_\lag^\mu)$ is defined by
\[
\differential(c)\coloneqq \sum_{t\geq 0} \sum_{f\in \modulispace_t(c)}\sgn(f,\bfx_0)\tilde f(\bfx_1)\cdots \tilde f(\bfx_t).
\]

\begin{definition}[LCH DGAs for Legendrian graphs]
The \emph{Legendrian contact homology DGA}(LCH DGA) or \emph{Chekanov-Eliashberg} DGA $A^\CE(T_\lag,\mu)$ for a Legendrian graph $T_\lag^\mu\in\LT_\lag^\mu$ with a Maslov potential is defined as
\[
A^\CE(T_\lag,\mu)\coloneqq\left(\alg^\CE(T_\lag,\mu),\differential\right).
\]
\end{definition}

\begin{theorem}\cite[Theorem~A, B and E]{AB2018}\label{theorem:invariance of DGAs for Legendrian graphs}
For a Legendrian graph $(T_\lag,\mu)\in\LT_\lag^\mu$ with a Maslov potential, the pair of the LCH DGA $A^\CE(T_\lag,\mu)$ and the set of its internal DG-subalgebras $\left\{I_v\subset A^\CE(T_\lag,\mu)\mid v\in V\right\}$ corresponding to vertices is invariant under the Reidemeister moves up to generalized stable-tame isomorphisms.
\end{theorem}

\begin{remark}\label{remark:peripheral structures}
In \cite{AB2018}, internal DG-subalgebras corresponding to vertices are called \emph{peripheral structures}.
\end{remark}

\begin{corollary}\label{corollary:zig-zags of stabilizations}
Let $(T'_\lag,\mu'), (T_\lag,\mu)\in\LT_\lag^\mu$ be Lagrangian projections of two equivalent Legendrian graphs.
Then there exists a zig-zag of stabilizations between two LCH DGAs $A^\CE(T'_\lag,\mu')$ and $A^\CE(T_\lag,\mu)$
\[
\begin{tikzcd}
A^\CE(T'_\lag,\mu')\arrow[r,equal]& A_0
& A_1
\arrow[from=l,"i_1",yshift=.7ex,hook]\arrow[l,"\pi_1",yshift=-.7ex,->>] 
\arrow[from=r,"i_1'"',yshift=.7ex,hook']\arrow[r,"\pi_1'"',yshift=-.7ex,->>]
&\cdots
& A_{n-1}
\arrow[from=l,"i_{n-1}",yshift=.7ex,hook]\arrow[l,"\pi_{n-1}",yshift=-.7ex,->>] 
\arrow[from=r,"i_{n-1}'"',yshift=.7ex,hook']\arrow[r,"\pi_{n-1}'"',yshift=-.7ex,->>]
&A_n\arrow[r,equal]& A^\CE(T_\lag,\mu).
\end{tikzcd}
\]
\end{corollary}
\begin{proof}
It is known that up to tame isomorphisms, each Lagrangian Reidemeister move induces a DGA morphism which is either a canonical inclusion into or a projection from a (generalized) stabilization.
Since a generalized stabilization is a zig-zag of stabilizations by Proposition~\ref{proposition:generalized stabilizations}, we are done.
\end{proof}

Now let $(\cT,\bfmu)\in\BLT_\front^\mu$ be a front projection of a bordered Legendrian graph with a Maslov potential.
Then by Lemma~\ref{lemma:potential extends to the closure}, we have a Legendrian graph $(\hat \cT,\hat\bfmu)$ with a Maslov potential.
Let $A^\CE(\hat \cT,\hat\bfmu)$ be the Chekanov-Eliashberg's DGA for the Ng's resolution
\[
A^\CE(\hat \cT,\hat\bfmu)=\left(\alg^\CE(\hat \cT,\hat\bfmu),\hat\differential\right)\coloneqq A^\CE(\Res(\hat \cT,\hat\bfmu))
\]
of the closure $(\hat \cT,\hat\mu)$ as depicted in Figure~\ref{figure:resolution of the closure}.
Clearly, it consists of three parts, where the left and the right parts corresponding to the resolutions of $\bfzero_{n_\Left}$ and $\bfinfty_{n_\Right}$, and we denote the middle part by $\Res(T)$.

As mentioned earlier, there are $\binom{n_\Right}2$ more crossings in $\Res(\hat T)$, which will be denoted by 
\begin{align*}
\hat K_\Right&\coloneqq\left\{\hat k_{ab}~\middle|~1\leq a<b \leq n_\Right\right\},&
\left|\hat k_{ab}\right|&=|k_{ab}|+1.
\end{align*}
whose upper and lower strand corresponds to $a$ and $b$ in $T_\Right$, respectively.

For two vertices $0$ and $\infty$ coming from the closure, we define two subsets of vertex generators
\[
K_\Left\coloneqq\{k_{a'b'}=0_{a',b'-a'}\mid1\le a'<b' \le n_\Left\}\quad\text{ and }\quad
K_\Right\coloneqq\{k_{ab}=\infty_{a,b-a}\mid1\le a<b\le n_\Right\}.
\]
Then by Proposition~\ref{proposition:differentials for border DGAs}, the DG-subalgebras generated by $k_{a'b'}$'s and $k_{ab}$'s are isomorphic to border DGAs $A_{n_\Left}(\mu_\Left)$ and $A_{n_\Right}(\mu_\Right)$, where $\mu_\Left$ and $\mu_\Right$ are the restriction of $\mu$ to the left and right borders, respectively, as seen in Section~\ref{section:Maslov potentials}.

\begin{lemma}\label{lem:subDGA}
The bordered DGA below is well-defined:
\begin{equation}\label{equation:mapping cylinder}
\hat A^\CE(\cT,\bfmu)\coloneqq\left(A^\CE(T_\Left,\mu_\Left)\stackrel{\hat\phi_\Left}\longrightarrow \hat A^\CE(T,\mu)
\stackrel{\hat\phi_\Right}\longleftarrow A^\CE(T_\Right,\mu_\Right)\right),
\end{equation}
where 
\[
A^\CE(T_\Left,\mu_\Left)\coloneqq (\ring\langle K_\Left\rangle,\differential_\Left)\isomorphic A_{n_\Left}
\quad\text{ and }\quad
A^\CE(T_\Right,\mu_\Right)\coloneqq (\ring\langle K_\Right\rangle,\differential_\Right)\isomorphic A_{n_\Right},
\]
and $\hat A^\CE(T,\mu)$ is the subalgebra of $A^\CE(\hat \cT,\hat\bfmu)$ generated by
\begin{itemize}
\item all crossings and vertex generators in $\Res(T)$,
\item all elements in the set $\hat K_\Right$, and
\item all $k_{a'b'}$'s and $k_{ab}$'s for $1\le a'<b'<n_\Left$ and $1\le a<b<n_\Right$,
\end{itemize}
and the maps $\hat\phi_\Left$ and $\hat\phi_\Right$ are the canonical inclusions.
\end{lemma}
\begin{proof}
Let $f$ be an immersed polygon defined in \eqref{eqn:polygon_map} with $f(\bfx_0)\in C(\cT)$.
Since two vertices $0$ and $\infty$ in $\Res(\hat\cT)$ faces the unbounded region which correspond to the sectors between $h_{0,n_\Left}$ and $h_{0,1}$ and between $h_{\infty,n_\Right}$ and $h_{\infty,1}$, respectively, these unbounded regions are never covered by $f$.
\end{proof}

Furthermore, we have a smaller DG-subalgebra $A^\CE(T,\mu)$ as follows:
\begin{theorem}\label{theorem:subDGA}
The DG-subalgebra $A^\CE(T,\mu)$ of $\hat A^\CE(T,\mu)$ generated by $K_\Left$ and both crossings and vertex generators in $\Res(T)$ is well-defined.
\end{theorem}
\begin{proof}
Since the action of Reeb chords increases as we go to the right and we regard crossings $\hat k_{ab}$ have very large actions relative to the action of $c$, the map $f$ never escape $\Res(T_\front)$ to the right. Therefore no generators in $K_\Right$ appear and we are done.
\end{proof}

\begin{corollary}\label{corollary:subDGA}
The map $\hat\phi_\Left$ is the composition of the canonical inclusions $\phi_\Left:A^\CE(T_\Left,\mu_\Left)\to A^\CE(T,\mu)$ and $A^\CE(T,\mu)\to \hat A^\CE(T,\mu)$.
\end{corollary}
\begin{proof}
This is a direct consequence of Lemma~\ref{lem:subDGA} and Theorem~\ref{theorem:subDGA}.
\end{proof}

Now let us focus on the differential $\hat\differential(\hat k_{ab})$. Indeed, there is a trichotomy on the moduli space $\modulispace_t(\hat k_{ab})$ as follows:
\begin{enumerate}
\item disks hitting the vertex $\infty$ have the contribution
\[
(-1)^{|\hat k_{ab}|-1}\left(k_{ab}+\sum_{a<c<b}\hat k_{ac}k_{cb}\right),\footnotemark
\]
\item disks not hitting $\infty$ but hitting $\hat k_{cb}$ have the contribution
\[
\sum_{a<c<b}g_{ac}\hat k_{cb}
\]
for some $g_{ac}\in A^\CE(T,\mu)$,
\item disks hitting neither $\infty$ nor $\hat k_{cb}$ have the contribution denoted by $g_{ab}\in A^\CE(T,\mu)$.
\end{enumerate}
\footnotetext{The sign convention given in Figure~\ref{figure:resolution of the closure} is used here.}

In summary, we have the following differential formula for $\hat k_{ab}$
\begin{equation}\label{equation:sign corrected differential}
\hat\differential(\hat k_{ab})=
(-1)^{|\hat k_{ab}|-1}k_{ab}+g_{ab}+\sum_{a<c<b}(-1)^{|\hat k_{ab}|-1}\hat k_{ac}k_{cb}+g_{ac}\hat k_{cb}
\end{equation}
for some $g_{ab}\in A^\CE(T,\mu)$.
Notice that this formula is the same as the differential formula for the mapping cylinder given in \ref{equation:differential for mapping cylinder} by replacing $\hat k_{ab}$ and $g_{ab}$ with $(-1)^{|\hat k_{ab}|-1}\hat k_{ab}$ and $(-1)^{|\hat k_{ab}|}g_{ab}$.
Therefore the assignment $k_{ab}\mapsto \bar g_{ab}$ will define a DGA morphism by Lemma~\ref{lemma:mapping cylinder differential} since $\hat\differential^2=0$.

Summarizing the above argument, we have the following theorem.
\begin{theorem/definition}[Well-definedness]\label{theorem:well-definedness of a bordered DGA}
Let $(\cT,\bfmu)\in\BLT_\front^\mu$ be a bordered Legendrian graph with a Maslov potential.
The bordered DGA 
\begin{equation}\label{equation:bordered DGAs for bordered Legendrian graphs}
A^\CE(\cT,\bfmu)\coloneqq\left(A^\CE(T_{\Left},{\mu_\Left})\stackrel{\phi_\Left}\longrightarrow A^\CE(T,\mu)\stackrel{\phi_\Right}\longleftarrow A^\CE(T_{\Right},{\mu_\Right})\right)
\end{equation}
is well-defined.
\end{theorem/definition}
\begin{proof}
The DGA $A^\CE(T,{\mu})$ is well-defined by Theorem~\ref{theorem:subDGA}, and two morphisms are DGA maps by Corollary~\ref{corollary:subDGA} and Lemma~\ref{lemma:mapping cylinder differential}, and we are done.
\end{proof}

\begin{example}[DGAs for the trivial and vertex bordered Legendrian graphs]\label{example:DGAs for trivial graphs}
Let $(\cT_n,\bfmu)\in\BLT_\front^\mu$ be the front projection of the trivial bordered graph with $n$ strands and with a Maslov potential $\bfmu=(\mu,\mu,\mu)$ for a function $\mu:[n]\to\grading$.
Then it is obvious that $T_n$ has no generators except for those for the left border $T_{n,\Left}$ and so is isomorphic to the border DGA $A_n(\mu)$ defined in Example/Definition~\ref{example:border DGA}.
Therefore the bordered LCH DGA for $T_n$ is
\[
A^\CE(\cT_n,\bfmu)=\left(A_n(\mu)\isomorphic A_n(\mu) \isomorphic A_n(\mu)\right).
\]

On the other hand, for the bordered Legendrian graph $(\bfzero_n,\bfmu)$ with a Maslov potential $\bfmu=(0,\mu,\mu)$ as shown in Example/Definitoin~\ref{example/definition:trivial graph}, the DGA $A^\CE(0_n,\mu)$ is the same as the internal DGA $I_n(\mu)$ since it has no left borders.
Therefore the bordered LCH DGA for $(\bfzero_n,\bfmu)$ is
\[
A^\CE(\bfzero_n,\bfmu)=\left(0\to I_n(\mu) \leftarrow A_n(\mu_\Right)\right),
\]
where the DGA morphism on the right is the inclusion described in Proposition~\ref{proposition:differentials for border DGAs}.
\end{example}

As a corollary of Theorem/Definition~\ref{theorem:well-definedness of a bordered DGA}, we have the following important observation.
\begin{proposition}\label{proposition:closure DGA is a mapping cylinder}
Let $(\cT,\bfmu)\in\BLT_\front^\mu$.
Then the bordered DGA $\hat A^\CE(\cT,{\bfmu})$ defined in \eqref{equation:mapping cylinder} is the mapping cylinder of the bordered DGA $A^\CE(\cT,{\bfmu})$ defined in \eqref{equation:bordered DGAs for bordered Legendrian graphs} in the sense of Definition~\ref{definition:mapping cylinder}.
\end{proposition}
\begin{proof}
This follows obviously from the equation \eqref{equation:sign corrected differential}.
\end{proof}

\begin{theorem}\label{theorem:concatenation}
For $i=1,2$, let $(\cT^i,\bfmu^i)\in\BLT_\front^\mu$ be two front projections of type $(n^i_\Left,n^i_\Right)$.
Suppose that $n\coloneqq n^1_\Right=n^2_\Left$ and $\mu_n\coloneqq\mu^1_\Right=\mu^2_\Left$.
Then the bordered DGA for the concatenation $(\cT,\bfmu)\coloneqq (\cT^1,\bfmu^1)\cdot(\cT^2,\bfmu^2)$ is given as follows:
\begin{itemize}
\item the DGA $A^\CE(T,\mu)$ is defined to be the push-out of two DGA morphisms $\phi^1_\Right$ and $\phi^2_\Left$,
\[
\begin{tikzcd}[column sep=3pc,row sep=2pc]
A_n(\mu_n)\arrow[r,"\phi^2_\Left"]\arrow[d,"\phi^1_\Right"'] & A^\CE(T^2,\mu^2)\arrow[d,dashed,"i_{\bfU^2\bfU}"]\\
A^\CE(T^1,\mu^1)\arrow[r,dashed,"i_{\bfU^1\bfU}"] & A^\CE(T,\mu),
\end{tikzcd}
\]
\item two border DGAs $A^\CE(T_{\Left},{\mu_\Left})$ and $A^\CE(T_{\Right},{\mu_\Right})$ and morphisms $\phi_\Left$ and $\phi_\Right$ are defined as
\begin{align*}
\phi_\Left&:A^\CE(T_{\Left},{\mu_\Left})\coloneqq A^\CE(T^1_\Left,\mu^1_\Left)\stackrel{\phi^1_\Left}\longrightarrow A^\CE(T^1,\mu^1)\stackrel{i_{\bfU^1\bfU}}\longrightarrow A^\CE(T,\mu),\\
\phi_\Right&:A^\CE(T_{\Right},{\mu_\Right})\coloneqq A^\CE(T^2_\Right,\mu^1_\Right)\stackrel{\phi^2_\Right}\longrightarrow A^\CE(T^2,\mu^2)\stackrel{i_{\bfU^2\bfU}}\longrightarrow A^\CE(T,\mu).
\end{align*}
\end{itemize}
\end{theorem}
\begin{proof}
This is essentially nothing but Theorem~C in \cite{AB2018}.
The generators for $A^\CE(T,\mu)$ come from those for $A^\CE(T^1,\mu^1)$ and $A^\CE(T^2,\mu^2)$.
For the differential of $T$, every immersed polygon $f$ can be cut into two pieces along the original border $T^1_\Right=T^2_\Left$ so that the right part will give a generator in $T^2_\Left$ while the left part will give the image under $\phi^1_\Right$ for the corresponding generator in $T^1_\Right$. We omit the detail.
\end{proof}

For the rest of this section, we will prove the invariance of the bordered LCH DGAs.
\begin{theorem}[Invariance of bordered LCH DGAs]\label{theorem:invariance of bordered LCH DGAs}
The bordered LCH DGAs $A^\CE(\cT,{\bfmu})$ is invariant under Legendrian isotopy up to stablizations.

Moreover, for each $\RM{M}:(\cT'{\bfmu'})\to (\cT,{\bfmu})$, there is an induced bordered quasi-morphism 
\[
\RM{M}_*:A^\CE(\cT',{\bfmu'})\to A^\CE(\cT,\bfmu)
\]
such that two compositions $\RM{M}_*\circ \RM{M^{-1}}_*$ and $\RM{M^{-1}}_*\circ\RM{M}_*$ are homotopic to the identities.
\end{theorem}
\begin{proof}
For convenience's sake, we denote the DGAs for these graphs as follows: for $*=\Left, \Right$ or empty,
\begin{align*}
A'_*&\coloneqq A^\CE(T'_*,\mu'_*),&
\dga'&\coloneqq \hat A^\CE(\cT',{\bfmu'}),&
\hat A'&\coloneqq \hat A^\CE(T',{\mu'}),\\
A_*&\coloneqq A^\CE(T_*,{\mu_*}),&
\dga&\coloneqq \hat A^\CE(\cT,\bfmu),&
\hat A&\coloneqq \hat A^\CE(T,\mu).
\end{align*}

Due to Proposition~\ref{proposition:closure DGA is a mapping cylinder}, we have the following situation:
\begin{equation}\label{equation:diagram for invariance}
\begin{tikzcd}[column sep=4pc, row sep=2pc]
\dga'\arrow[from=d,"\hat\pi'"']&
A'_\Left \arrow[r,"\phi'_\Left"] \arrow[d,equal]& A'\arrow[from=d,"\hat \pi'"'] & A'_\Right\arrow[l,"\phi'_\Right"']\arrow[d,equal]\\
\hat\dga'\arrow[dashed,d,"{\RM{M}}_*"]&
A'_\Left \arrow[r,"\hat \phi'_\Left"] \arrow[d,equal] & \hat A' \arrow[dashed,d,"{\RM{M}}_*"] & A'_\Right \arrow[l,"\hat\phi'_\Right"'] \arrow[d,equal]\\
\hat\dga\arrow[d,"\hat\pi"]&
A_\Left \arrow[r,"\hat \phi_\Left"] \arrow[d,equal] & \hat A \arrow[d,"\hat\pi"] & A_\Right \arrow[l,"\hat\phi_\Right"'] \arrow[d,equal]\\
\dga&
A_\Left \arrow[r,"\phi_\Left"] & A & A_\Right\arrow[l,"\phi_\Right"'],
\end{tikzcd}
\end{equation}
where all squares in both upper and lower parts are commutative.

Let us regard $\RM{M}:\cT'\to \cT$ as the move $\RM{\hat M}:\hat \cT'\to\hat \cT$. By taking Ng's resolution, we have a sequence $\Res{\RM{\hat M}}$ of Lagrangian Reidemeister moves between $\Res(\hat \cT')$ and $\Res(\hat \cT)$ as seen in Lemma~\ref{lemma:functoriality of Ng's resolution}.
Then Corollary~\ref{corollary:zig-zags of stabilizations} implies that we have a zig-zag of stabilizations as follows:
\[
\begin{tikzcd}
A'\arrow[from=r,"\hat\pi'"']& \hat A'=A_0
& A_1
\arrow[from=l,"\iota_1",yshift=.7ex,hook]\arrow[l,"\pi_1",yshift=-.7ex,->>] 
\arrow[from=r,"\iota_1'"',yshift=.7ex,hook']\arrow[r,"\pi_1'"',yshift=-.7ex,->>]
&\cdots
& A_{n-1}
\arrow[from=l,"\iota_{n-1}",yshift=.7ex,hook]\arrow[l,"\pi_{n-1}",yshift=-.7ex,->>] 
\arrow[from=r,"\iota_{n-1}'"',yshift=.7ex,hook']\arrow[r,"\pi_{n-1}'"',yshift=-.7ex,->>]
&A_n=\hat A\arrow[r,"\hat\pi"]& A.
\end{tikzcd}
\]

It is important to note that all DGA morphisms in this zig-zag \emph{preserve} two border DG-subalgebras since they are DG-subalgebras of the internal DG-subalgebras which are invariant as seen in Theorem~\ref{theorem:invariance of DGAs for Legendrian graphs}.
Therefore the above zig-zag induce the zig-zags of stabilizations between bordered DGAs $\hat\dga'$ and $\hat\dga$ 
\[
\begin{tikzcd}
\dga'\arrow[from=r,"\hat\bfpi'"']& \hat \dga'=\dga_0
& \dga_1
\arrow[from=l,"\bfi_1",yshift=.7ex,hook]\arrow[l,"\bfpi_1",yshift=-.7ex,->>] 
\arrow[from=r,"\bfi_1'"',yshift=.7ex,hook']\arrow[r,"\bfpi_1'"',yshift=-.7ex,->>]
&\cdots
& \dga_{n-1}
\arrow[from=l,"\bfi_{n-1}",yshift=.7ex,hook]\arrow[l,"\bfpi_{n-1}",yshift=-.7ex,->>] 
\arrow[from=r,"\bfi_{n-1}'"',yshift=.7ex,hook']\arrow[r,"\bfpi_{n-1}'"',yshift=-.7ex,->>]
&\dga_n=\hat \dga\arrow[r,"\hat\bfpi"]& \dga,
\end{tikzcd}
\]
where all DGAs in the middle are strong stabilizations while $\hat\bfpi'$ and $\hat\bfpi$ are weak stabilizations.
Therefore the invariance up to stabilizations is proved.

In order to define the bordered DGA morphism $\RM{M}_*=(\identity,\RM{M}_*,\identity)$, we use the bordered quasi-morphism $\hat\bfi:\dga'\to\hat \dga'$ given as \eqref{equation:canonical inclusion of mapping cylinder}.
Then we define the DGA morphism $\RM{M}_*$ as the composition
\begin{align*}
\RM{M}_*&\coloneqq \hat\bfpi\circ {\RM{\hat M}}_*\circ\hat\bfi':\dga'\to \dga,&
{\RM{\hat M}}_*&=\Res\RM{\hat M}_*\coloneqq \bfpi_n'\circ \bfi_{n-1} \circ\cdots\circ \bfi_1 :\hat \dga'\to\hat \dga,
\end{align*}
which makes the diagram below commutative up to homotopy 
\[
\begin{tikzcd}[column sep=4pc, row sep=2pc]
\dga'\arrow[d,"\RM{M}_*"']\\
\dga
\end{tikzcd}=\ \left(
\begin{tikzcd}[column sep=4pc, row sep=2pc]
A'_\Left \arrow[r,"\phi'_\Left"] \arrow[d,equal]& A'\arrow[d,"\RM{M}_*"] & A'_\Right\arrow[l,"\phi'_\Right"']\arrow[d,equal]\arrow[ld,"H'_{\Right,\RM{M}}",Rightarrow,sloped]\\
A_\Left \arrow[r,"\phi_\Left"] & A & A_\Right\arrow[l,"\phi_\Right"']
\end{tikzcd}\right)
\]
where the homotopy $H'_{\Right,\RM{M}}$ can be defined explicitly as
\[
H'_{\Right,\RM{M}}\coloneqq \hat\pi\circ {\RM{M}}_*\circ H'_\Right:A'_\Right\to A.
\]

The last statement follows easily by choosing the sequence $\Res\RM{M^{-1}}$ of Lagrangian Reidemeister moves corresponding to each front Reidemeister move $\RM{M^{-1}}$ as the inverse of the sequence $\Res\RM{M}$.
Then two compositions $\RM{M}_*\circ\RM{M^{-1}}_*$ and $\RM{M^{-1}}_*\circ\RM{M}_*$ of induced maps will be compositions of morphisms which are either the identities or homotopic to the identity morphisms.
\end{proof}

\section{Augmentations}\label{section:augmentations}
In this section, we introduce augmentation varieties (with boundary conditions) and augmentation numbers for bordered Legendrian graphs, and show their invariance. In addition, we consider two key examples: the augmentation varieties for a trivial bordered Legendrian graph and for a vertex.

\begin{definition}[Augmentation variety]
Let $A=(\alg,\differential)$ be a DGA. 
An \emph{augmentation} of $A$ over $\field$ is a DGA morphism $\epsilon:A\to\field$, where $\field$ is a field with the trivial differential, i.e., 
\begin{equation}\label{equation:defining equation for augmentations}
\epsilon\circ\differential\equiv0,
\end{equation}
and the (\emph{full}) \emph{augmentation variety} of $A$ is then defined as
\[
\tilde\aug(A;\field)\coloneqq\{\epsilon\mid \epsilon\text{ is an augmentation of $A$ over $\field$}\}.
\]
\end{definition}

We firstly introduce an equivalence relation $\relation$ on the set of affine algebraic varieties over $\field$.
\begin{definition}[Equivalence of augmentation varieties]\label{def:equiv rel on aug var}
Let $X$ any $Y$ be two affine algebraic varieties over $\field$.
We say that $X$ and $Y$ are \emph{equivalent}, denoted by $X\relation Y$ if they are isomorphic as affine algebraic varieties up to stabilizations, i.e., there exists $m\ge0$ and $n\ge0$ such that
\[
X\times\field^m\isomorphic Y\times\field^n.
\]
\end{definition}

\begin{lemma}\label{lemma:invariance of augmentation variety under stabilization}
Let $A'$ be a stabilization of $A$.
Then two full augmentation varieties $\tilde\aug(A;\field)$ and $\tilde\aug(A';\field)$ are equivalent.
\end{lemma}
\begin{proof}
This is obvious. Indeed, if $A'$ has additional cancelling pairs $\{\hat e^i, e^i\mid i\in I\}$ as in Definition~\ref{definition:stabilizations}, then the augmentation variety $\tilde\aug(A';\field)$ is the product space
\[
\tilde\aug(A';\field)\isomorphic \tilde\aug(A;\field)\times \field^N,
\]
where $N$ is the number of generators $\hat e^i$ of degree $0$.
\end{proof}

\subsection{Augmentation varieties for border DGAs}\label{sec:aug_variety_lines}
Let $(\cT_n,\bfmu)\in\BLT_\front^\mu$ be a front projection of the trivial bordered Legendrian graph consisting of $n$ parallel strands with a Maslov potential $\bfmu$ defined in Example/Definition~\ref{example/definition:trivial graph}.
Recall that $A^\CE(\cT_n,\bfmu)=(A_n(\mu)\isomorphic A_n(\mu)\isomorphic A_n(\mu))$ as seen in Example~\ref{example:DGAs for trivial graphs}, where $A_n(\mu)=(\alg_n(\mu),\differential_n)$ is generated by elements $k_{ab}$'s 
\begin{align*}
\alg_n(\mu)&=\ring\langle K_n\rangle,&
K_n&\coloneqq\{k_{ab}\mid 1\leq a<b\leq n\},&
|k_{ab}|&=\mu(a)-\mu(b)-1,
\end{align*}
whose differential is given by
\[
\differential_n(k_{ab})=\sum_{a<c<b}(-1)^{|k_{ac}|+1}k_{ac}k_{cb}
\]
as defined in Example/Definition~\ref{example:border DGA}.

Let us denote the augmentation variety for $A_n(\mu)$ by $\tilde\aug(T_n,\mu;\field)$. Then by regarding $\epsilon(k_{ab})$ as a variable $x_{ab}$, the augmentation variety of $(T_n,\mu)$ is an affine algebraic variety
\begin{align*}
\tilde\aug(T_n,\mu;\field)\isomorphic\left\{(x_{ab})\in \field^{n(n-1)/2}~\middle|~
1\le a<b\le n,
x_{ab}=0\text{ if }|k_{ab}|\neq0, 
\sum_{a<c<b} x_{ac}x_{ab}=0
\right\}.
\end{align*}

\begin{definition}[Morse complex]\label{def:Morse complex}
Let $\mu:[n]\to\grading$ be a function. 
We define a decreasing filtration $F^\bullet C_n$
\[
C_n=F^0 C_n\supset F^1 C_n\supset\cdots\supset F^n C_n=\{0\}
\]
of the free graded $\field$-module $C_n=C_n(\mu)$ as follows:
\begin{align*}
F^a C_n&\coloneqq\bigoplus_{a<b\le n}\field\langle e_b\rangle,&
|e_b|\coloneqq -\mu(b).\footnotemark
\end{align*}
\footnotetext{The convention used here differs from \cite[Def.4.4]{Su2017} by a negative sign, which is more convenient for our purpose in this article.}

The group of $\field$-linear automorphisms on $F^\bullet C_n$, i.e., automorphisms on the graded $\field$-module $C_n$ preserving the filtration, will be denoted by $B(T_n,\mu;\field)$.

A \emph{Morse complex} $d\in\End_\field(F^\bullet C_n)$ on $F^{\bullet}C_n$ makes $(C_n,d)$ a cochain complex and preserves the filtration, and a Morse complex $d$ is said to be \emph{acyclic} if the induced cohomology vanishes.
We denote the sets of all Morse complexes and all acyclic Morse complexes on $F^{\bullet}C_n$ by $\tilde{\MC}(T_n,\mu;\field)$ and $\MC(T_n,\mu;\field)$, respectively.
\end{definition}

\begin{remark}\label{remark:upper-triangular}
As a matrix, each element in $B(T_n,\mu;\field)$ is upper-triangular with respect to the ordered bases $\{e_1,\dots,e_n\}$ since it preserves the filtration.
\end{remark}

\begin{definition}\label{definition:GNR and NR}
We define $\GNR(T_n,\mu)$ to be the set of all involutions $\ruling$ on $[n]$ such that for each $a< b=\ruling(a)$,
\[
|k_{ab}|=0\quad\text{ or equivalently,}\quad|e_a|=|e_b|-1.
\]
The subset consisting of involutions {\em without} fixed points will be denoted by $\NR(T_n,\mu)$.
\end{definition}
We'll see in the next section that $\GNR(T_n,\mu)$ is known as the set of {\em generalized normal rulings}, and $\NR(T_n,\mu)$ is the set of {\em normal rulings} of $T_n,\mu$.

\begin{lemma}[{\cite[Definition~4.4]{Su2017}}]\label{lem:aug and Morse complex for trivial tangle}
There is a canonical identification 
\[
\setlength\arraycolsep{2pt}
\begin{array}{ccrcl}
\Xi&:&\tilde\aug(T_n,\mu;\field) &\stackrel{\isomorphic}{\longrightarrow}& \tilde{\MC}(T_n,\mu;\field)\\
&&\epsilon &\longmapsto& d=d(\epsilon),
\end{array}
\]
where
\[
d(e_a)\coloneqq\sum_{a<b}(-1)^{\mu(a)}\epsilon(k_{ab})e_b.
\]
\end{lemma}

Under this identification, 
$B(T_n,\mu;\field)$ acts on $\tilde\MC(T_n,\mu;\field)$ via conjugation:  for each $g\in B(T_n,\mu;\field)$,
\[
\Xi(g\cdot\epsilon)\coloneqq g\cdot \Xi(\epsilon)=g\circ \Xi(\epsilon)\circ g^{-1}.
\]
Moreover, we can say that an augmentation $\epsilon$ is \emph{acyclic} if so is $\Xi(\epsilon)$ and we denote the subvariety of $\tilde\aug(T_n,\mu;\field)$ of acyclic augmentations by $\aug(T_n,\mu;\field)$, which can be identified with $\MC(T_n,\mu;\field)$ by Lemma~\ref{lem:aug and Morse complex for trivial tangle}.

\begin{definition}[Canonical augmentations and differentials]\label{def:canonical differential}
For each involution $\ruling\in \GNR(T_n,\mu)$, the \emph{canonical augmentation} $\epsilon_\ruling\in\tilde\aug(T_n,\mu;\field)$ and \emph{differential} $d_{\ruling}\in\tilde{\MC}(T_n,\mu;\field)$ are defined as
\begin{align*}
\epsilon_\ruling(k_{ab})&\coloneqq\begin{cases}
1 & a<b=\ruling(a);\\
0 &\text{otherwise},
\end{cases}&
d_{\ruling}(e_a)&\coloneqq
\begin{cases}
(-1)^{\mu(a)}e_b &a<b=\ruling(a);\\
0&\text{otherwise}
\end{cases}
\end{align*}
so that $\Xi(\epsilon_\ruling)=d_\ruling$.

The orbit of $\epsilon_{\ruling}$ under the action of $B(T_n,\mu;\field)$ will be denoted by
\[
\aug^\ruling(T_n,\mu;\field)\coloneqq B(T_n,\mu;\field)\cdot \epsilon_{\ruling},
\]
whose stabilizer subgroup will be denoted by $\Stab^\ruling(T_n,\mu;\field)$.
\end{definition}

Notice that each $\ruling\in \GNR(T_n,\mu)$ defines a partition $[n]=U_\ruling\amalg H_\ruling\amalg L_\ruling$ together with a bijection $\ruling:U_\ruling\stackrel{\isomorphic}{\longrightarrow}L_\ruling$ satisfying that 
$|k_{ab}|=0$ for all $a\in U_\ruling<b=\ruling(a)\in L_\ruling$ and $c=\ruling(c)$ for all $c\in H_\ruling$.

\begin{definition}\label{def:indices for an involution at a trivial tangle}
For each $i\in [n]$, we define
\[
I(i):=\{j\in [n] \mid j>i, \mu(j)=\mu(i) \}.
\]
For any fixed $\ruling\in \GNR(T_n,\mu)$ and any $i\in H_\ruling$, \emph{define} $\ruling(i)=\infty$. For any $i\in U_\ruling\amalg H_\ruling$, \emph{define}
\[
A_{\ruling}(i):=\{j\in U_\ruling \mid j\in I(i) \text{ and $\ruling(j)<\ruling(i)$}\}.
\]

Now, define $A_b(\ruling)\in\NN$ by
\[
A_b(\ruling):=\sum_{i\in U_\ruling\amalg H_\ruling}|A_\ruling(i)|+\sum_{i\in L_\ruling}|I(i)|.
\]
\end{definition}

\begin{lemma}[{\cite[Lem.4.5, 5.7 and Cor.5.8]{Su2017}}]\label{lem:ruling stratification for trivial tangle}
Let $\mu:[n]\to\ZZ$ be a function as before.
\begin{enumerate}
\item
The group action of $B(T_n,\mu;\field)$ induces decompositions on $\tilde{\MC}(T_n,\mu;\field)$ and $\MC(T_n,\mu;\field)$ into finitely many orbits 
\begin{align*}
\tilde{\MC}(T_n,\mu;\field)&=\coprod_{\ruling\in \GNR(T_n,\mu)}B(T_n,\mu;\field)\cdot d_{\ruling},&
\MC(T_n,\mu;\field)&=\coprod_{\ruling\in \NR(T_n,\mu)}B(T_n,\mu;\field)\cdot d_{\ruling}.
\end{align*}
In particular, we have a decomposition of $\aug(T_n,\mu;\field)$ over the finite set $\NR(T_n,\mu)$
\[
\aug(T_n,\mu;\field)=\coprod_{\ruling\in \NR(T_n,\mu)}\aug^{\ruling}(T_n,\mu;\field),
\]
\item
For all $\ruling\in\GNR(T_n,\mu)$, the principal bundle 
\[
\pi_{\ruling}:B(T_n,\mu;\field)\rightarrow B(T_n,\mu;\field)\cdot d_{\ruling}
\]
admits a canonical section $\varphi_{\ruling}$. 
That is, $\varphi_{\ruling}(d)\cdot d_{\ruling}=d$ for all $d\in B(T_n,\mu;\field)\cdot d_{\ruling}$.

\item
For all $\ruling\in \GNR(T_n,\mu)$, we have:
\begin{equation}
B(T_n,\mu;\field)\cdot d_{\ruling}\isomorphic (\field^\times)^{|L_{\ruling}|}\times\field^{A_b(\ruling)}.
\end{equation}
\end{enumerate}
\end{lemma}
\begin{remark}\label{remark:indices for trivial}
Due to the result of Barannikov \cite{Bar1994}, for each Morse complex $d$, we have the Barannikov's normal form of $d$ which is a canonical differential $d_\ruling$ for some $\ruling\in\GNR(T_n,\mu)$ up to the action of $B(T_n,\mu;\field)$.
Indeed, $d_\ruling$ can be obtained as follows:
we pick a first nontrivial element from $(i,i+j)$ entry with the lexicographic order on $(i,j)$ such as
\[
(1,2)\to(2,3)\to\cdots\to(1,3)\to(2,4)\to\cdots\to(1,n)
\]
and use it as a povot to apply upward row operations and right column operations, where the entries to be cancelled out look like ``L''-shape.
After that, we pick the next pivot and do row and column operations to the corresponding ``L''-shape as before until we have only pivot entries.
\[
d_1=\left(\begin{tikzpicture}[baseline=-.5ex]
\matrix (m) [matrix of math nodes]
{
0 & 0 & a & b\\
0 & 0 & \mathbf{c}\vphantom{0} & d\\
0 & 0 & 0 & 0\\
0 & 0 & 0 & 0\\
};
\draw[->] (m-1-2.north west) -- (m-2-3.north west);
\fill[gray,opacity=0.3] (m-2-3.south west) -- (m-1-3.north west) -- (m-1-3.north east) -- (m-2-3.north east) -- (m-2-4.north east) -- (m-2-4.south east) -- cycle;
\end{tikzpicture}\right)\longrightarrow
d_{\ruling_1}=\left(\begin{tikzpicture}[baseline=-.5ex]
\matrix (m) [matrix of math nodes]
{
0 & 0 & 0 & \mathbf{b'}\\
0 & 0 & \mathbf{c}\vphantom{0} & 0\\
0 & 0 & 0 & 0\\
0 & 0 & 0 & 0\\
};
\draw[->] (m-2-3.south east) -- (m-3-4.south east);
\draw[->] (m-1-3.north west) -- (m-2-4.south east);
\end{tikzpicture}\right)\qquad
d_2=\left(\begin{tikzpicture}[baseline=-.5ex]
\matrix (m) [matrix of math nodes]
{
0 & 0 & \mathbf{a}\vphantom{0} & b\vphantom{0}\\
0 & 0 & 0 & d\\
0 & 0 & 0 & 0\\
0 & 0 & 0 & 0\\
};
\draw[->] (m-1-2.north west) -- (m-3-4.south east);
\fill[gray,opacity=0.3] (m-1-3.south west) -- (m-1-3.north west) -- (m-1-4.north east) -- (m-1-4.north east) -- (m-1-4.south east) -- cycle;
\end{tikzpicture}\right)\longrightarrow
d_{\ruling_2}=\left(\begin{tikzpicture}[baseline=-.5ex]
\matrix (m) [matrix of math nodes]
{
0 & 0 & \mathbf{a} & 0\\
0 & 0 & 0 & \mathbf{d}\\
0 & 0 & 0 & 0\\
0 & 0 & 0 & 0\\
};
\draw[->] (m-1-4.north west) -- (m-1-4.south east);
\end{tikzpicture}\right),
\]
where 
\[
\ruling_1=\{\{1,4\},\{2,3\}\}\quad\text{ and }\quad
\ruling_2=\{\{1,3\},\{2,4\}\}.
\]

Then there are exactly $A_b(\ruling)$-many non-pivot positions lying in the ``L''-shapes whose degrees are the same as of the pivot, which are cancelled during this process.
For example above, we have two such entries at $\{(1,3),(2,4)\}$ and only one entry $\{(1,4)\}$, or equivalently, $A_b(\ruling_1)=2$ and $A_b(\ruling_2)=1$, respectively.

All the detailed computations have been done in \cite{Su2017}.
\end{remark}

\begin{lemma}\label{lem:DGA homotopy for trivial tangle}
Let $\epsilon_1,\epsilon_2\in\tilde\aug(T_n,\mu;\field)$ be two augmentations. Then a function $h:K_n\to \field$ extends to a DGA homotopy between $\epsilon_1$ and $\epsilon_2$ if and only if 
\[
\Xi(\epsilon_1)=u\cdot \Xi(\epsilon_2),
\]
where $u$ is defined by 
\[
u(e_a)\coloneqq e_a+\sum_{a<b}h(k_{ab})e_b.
\] 
\end{lemma}
\begin{proof}
It is not hard to check that $h$ extends to a DGA homotopy from $\epsilon_1$ to $\epsilon_2$ if and only if 
\begin{itemize}
\item $h(k_{ab})=0$ for all $|k_{ab}|\neq -1$ and
\item for all $1\le a<b\le n$,
\begin{align*}
(\epsilon_2-\epsilon_1)(k_{ab})&
=h\circ\differential (k_{ab})\\
&=\sum_{a<c<b}(-1)^{|k_{ac}|-1}h(k_{acb})\\
&=\sum_{a<c<b}(-1)^{|k_{ac}|-1}h(k_{ac})\epsilon_1(k_{cb})+(-1)^{|k_{ac}|}\epsilon_2(k_{ac})h(k_{cb})\\
&=\sum_{a<c<b}(-1)^{\mu(a)-\mu(c)}\left(h(k_{ac})\epsilon_1(k_{cb})-\epsilon_2(k_{ac})h(k_{cb})\right).
\end{align*}
\end{itemize}

Notice that as mentioned in Remark~\ref{remark:upper-triangular}, the presentation matrix $D_i=((-1)^{\mu(a)}\epsilon_i(k_{ab}))$ for $\Xi(\epsilon_i)$ is a strictly upper-triangular matrix. 

Let $H=(h(k_{ab}))$. Then $H$ is also strictly upper-triangular. Then we can rewrite the above condition as 
\[
D_2-D_1=HD_1-D_2H \quad\text{ or equivalently, }\quad D_2(I+H)=(I+H)D_1,
\]
and again, it is equivalent to $u\circ \Xi(\epsilon_2)=\Xi(\epsilon_1)\circ u$.

Finally, by the definition of the action of $u\in B(T_n,\mu;\field)$, this is the same as the condition $\Xi(\epsilon_1)= u\cdot\Xi(\epsilon_2)$ as desired.
\end{proof}

As a consequence, we can say that two homotopic augmentations are in the same orbit.
\begin{corollary}\label{corollary:homotopy means the same orbit}
Let $\epsilon$ and $\epsilon'$ be two homotopic augmentations.
If $\epsilon\in\aug^\ruling(T_n,\mu;\field)$, then so is $\epsilon'$.
\end{corollary}

\subsection{Augmentation varieties with boundary conditions}
Let $(\cT,\bfmu)\in\BLT_\front^\mu$ be a front projection of a bordered Legendrian graph of type $(n_\Left,n_\Right)$.
Then we have a diagram of augmentation varieties
\begin{equation}\label{equation:bordered augmentation varieties}
\tilde\aug(\cT,\bfmu;\field)=\left(
\tilde\aug(T_\Left,{\mu_\Left};\field)\stackrel{\phi^*_\Left}\longleftarrow 
\tilde\aug(T,{\mu};\field)\stackrel{\phi^*_\Right}\longrightarrow 
\tilde\aug(T_\Right,{\mu_\Right};\field)
\right),
\end{equation}
where $\tilde\aug(T_*,{\mu_*};\field)\coloneqq\tilde\aug(A^\CE(T_*,{\mu_*});\field)$ for $*=\Left,\Right$ or empty.

\begin{proposition}[Gluing property of augmentation varieties]\label{prop:restriction_property_aug_variety}
Let $(\cT^i,\bfmu^i)$ be two bordered Legendrian graphs of type $(n_\Left^i,n_\Right^i)$, for $i=1,2$.
Suppose that $n \coloneqq n_\Right^1 = n_\Left^2$ and $\mu_n \coloneqq \mu_\Right^1 = \mu_\Left ^2$.
Then the augmentation variety for the concatenation $(\cT,\mu)\coloneqq (\cT^1,\mu^1)\cdot (\cT^2,\mu^2)$ is given as follows:
\begin{itemize}
\item the variety $\tilde\aug(T,\mu;\field)$ is defined to be the fiber product of two induced maps $(\phi^1_\Right)^*$ and $(\phi^2_\Left)^*$,
\[
\begin{tikzcd}[column sep=3pc,row sep=2pc]
\tilde\aug(T,\mu;\field)\arrow[r,"i^*_{U^2 U}",dashed]\arrow[d,"i^*_{U^1 U}"',dashed] & \tilde\aug(T^2,\mu^2;\field)\arrow[d,"(\phi^2_\Left)^*"]\\
\tilde\aug(T^1,\mu^1;\field)\arrow[r,"(\phi^1_\Right)^*"] & \tilde\aug(T_n,\mu_n).
\end{tikzcd}
\]
\item two border varieties $\tilde\aug(T_\Left,\mu_\Left;\field)$ and $\tilde\aug(T_\Right,\mu_\Right;\field)$ and morphisms $\phi^*_\Left$ and $\phi^*_\Right$ are defined as
\begin{align*}
\phi_\Left^*&:\tilde\aug(T,\mu;\field)\stackrel{i^*_{U^1 U}}\longrightarrow\tilde\aug(T^1,\mu^1;\field)\stackrel{(\phi^1_\Left)^*}\longrightarrow \tilde\aug(T^1_\Left,\mu^1_\Left;\field)=\tilde\aug(T_\Left,\mu_\Left;\field)\\
\phi_\Right^*&:\tilde\aug(T,\mu;\field)\stackrel{i^*_{U^2 U}}\longrightarrow\tilde\aug(T^2,\mu^2;\field)\stackrel{(\phi^2_\Right)^*}\longrightarrow \tilde\aug(T^2_\Right,\mu^2_\Right;\field)=\tilde\aug(T_\Right,\mu_\Right;\field).
\end{align*}
\end{itemize}
\end{proposition}

\begin{proof}
It follows directly from the definition of the augmentation variety and Theorem~\ref{theorem:concatenation}.
\end{proof}

Note that two augmentation varieties $\tilde\aug(T_\Left,{\mu_\Left};\field)$ and $\tilde\aug(T_\Right,{\mu_\Right};\field)$ are isomorphic to the augmentation varieties $\tilde\aug(T_{n_\Left},{\mu_\Left};\field)$ and $\tilde\aug(T_{n_\Right},{\mu_\Right};\field)$ of two trivial bordered Legendrian graphs $(T_{n_\Left},{\mu_\Left})$ and $(T_{n_\Right},{\mu_\Right})$, 
which can be decomposed over the finite sets $\GNR(T_{n_\Left},{\mu_\Left})$ and $\GNR(T_{n_\Right},{\mu_\Right})$ by Lemma~\ref{lem:ruling stratification for trivial tangle}.

We will consider the augmentation varieties for bordered Legendrian graphs with constraints, called \emph{boundary conditions}, and prove the invariance of augmentation varieties with boundary conditions.

\begin{definition}[Augmentation varieties with boundary conditions]\label{def:aug var with bdy conditions}
Let $(\cT,\mu)\in\BLT_\front^\mu$.
For $\ruling_\Left\in\NR(T_\Left,{\mu_\Left})$, $\ruling_\Right\in\NR(T_\Right,{\mu_\Right})$ and $\epsilon_\Left\in\aug(T_\Left,{\mu_\Left};\ruling_\Left;\field)$, we define subvarieties of $\tilde\aug(T,\mu;\field)$ as
\begin{align*}
\aug(\cT,\bfmu;\epsilon_\Left,\ruling_\Right;\field)&\coloneqq
\left(\phi^*_\Left\right)^{-1}(\epsilon_\Left) \cap 
\left(\phi^*_\Right\right)^{-1}(\aug^{\ruling_\Right}(T_\Right,{\mu_\Right};\field))\\
&=\left\{\left(\epsilon_\Left\stackrel{\phi^*_\Left}\longleftarrow \epsilon \stackrel{\phi^*_\Right}\longrightarrow \epsilon_\Right\right)\in\tilde\aug(T,\mu;\field)~\middle|~
\epsilon_\Right\in \aug^{\ruling_\Right}(T_\Right,{\mu_\Right};\field)\right\};\\
\aug(\cT,\bfmu;\ruling_\Left,\ruling_\Right;\field)&\coloneqq
\left(\phi^*_\Left\right)^{-1}(\aug^{\ruling_\Left}(T_\Right,{\mu_\Left};\field)) \cap 
\left(\phi^*_\Right\right)^{-1}(\aug^{\ruling_\Right}(T_\Right,{\mu_\Right};\field))\\
&=\left\{\left(\epsilon_\Left\stackrel{\phi^*_\Left}\longleftarrow \epsilon \stackrel{\phi^*_\Right}\longrightarrow \epsilon_\Right\right)\in\tilde\aug(T,\mu;\field)~\middle|~
\epsilon_*\in \aug^{\ruling_*}(T_*,{\mu_*};\field), *=\Left\text{ or }\Right
\right\}
\end{align*}
\end{definition}

\begin{remark}\label{rem:aug-orbit and orbit-orbit}
The augmentation-orbit and orbit-orbit boundary conditions are closed related, as follows:
\begin{align*}\label{eqn:aug-orbit and orbit-orbit}
\aug(\cT,\bfmu;\ruling_\Left,\ruling_\Right;\field)
&=\aug^{\ruling_\Left}(T_\Left,{\mu_\Left};\field)\times \aug(\cT,\bfmu;\epsilon_\Left,\ruling_\Right;\field)\\
&\isomorphic\left( (\field^\times)^{n_\Left/2} \times \field^{A_b(\ruling_\Left)} \right) \times \aug(\cT,\bfmu;\epsilon_\Left,\ruling_\Right;\field)
\end{align*}
See Proposition \ref{prop:aug-orbit vs orbit-orbit} for more details.
Thus we will see that the two definitions of augmentation varieties with boundary conditions lead to equivalent results \emph{up to a normalization}.
\end{remark}

\begin{theorem}[Invariance of augmentation varieties with boundary conditions]\label{thm:inv of aug var}
Let $(\cT,\mu)\in\BLT_\front^\mu$ and $\epsilon_\Left$, $\ruling_\Left$ and $\ruling_\Right$ be as above.
Then $\aug(\cT,\bfmu;\epsilon_\Left,\ruling_\Right;\field)$ and $\aug(\cT,\bfmu;\ruling_\Left,\ruling_\Right;\field)$ are Legendrian isotopy invariants up to equivalence.
\end{theorem}
\begin{proof}
It is enough to prove the invariance of the variety $\aug(\cT,\bfmu;\epsilon_\Left,\ruling_\Right;\field)$, and we will follow the argument described in the proof of Theorem~\ref{theorem:invariance of bordered LCH DGAs}.

Let us first consider the mapping cylinder $\hat A^\CE(\cT,\mu)$ of $A^\CE(\cT,\mu)$ defined in Lemma~\ref{lem:subDGA} and Proposition~\ref{proposition:closure DGA is a mapping cylinder}.
Then the canonical inclusion $\hat i$ commutes with the structure morphisms $\phi_\Right$ and $\hat\phi_\Right$ up to homotopy $\hat H_\Right$ which implies that there exists $u\in B(T_{n_\Right},{\mu_\Right};\field)$ that makes the following diagram commutative:
\[
\begin{tikzcd}[column sep=3pc, row sep=2pc]
\tilde\aug(T_\Left,{\mu_\Left};\field)\arrow[d,equal] 
&\tilde\aug(T,{\mu};\field) \arrow[l,"\phi^*_\Left"'] \arrow[r,"\phi^*_\Right"] 
&\tilde\aug(T_\Right,{\mu_\Right};\field)\\
\tilde\aug(T_\Left,{\mu_\Left};\field) 
&\hat\aug(T,\mu;\field)\arrow[u,"\hat i^*"',->>] \arrow[l,"\hat\phi^*_\Left"'] \arrow[r,"\hat\phi^*_\Right"]
&\tilde\aug(T_\Right,{\mu_\Right};\field)\arrow[u,"\isomorphic","u\cdot(-)"']
\end{tikzcd}
\]
where $\hat\aug(T,\mu;\field)\coloneqq\tilde\aug(\hat A^\CE(T,\mu);\field)$.
We denote the corresponding subvariety with boundary conditions in $\hat\aug(T,\mu;\field)$ by $\hat\aug(\cT,\bfmu;\epsilon_\Left,\ruling_\Right;\field)$.
\[
\hat\aug(\cT,\bfmu;\epsilon_\Left,\ruling_\Right;\field)\coloneqq
\left(\hat\phi_\Left^*\right)^{-1}(\epsilon_\Left)\cap
\left(\hat\phi_\Right^*\right)^{-1}(\aug^{\ruling_\Right}(T_{\Right},{\mu_\Right};\field)).
\]

By Lemma~\ref{lemma:mapping cylinder differential}, $\hat A^\CE(T,\mu)$ is isomorphic to a stabilization of $A^\CE(T,\mu)$ and the map $\hat i^*$ is the projection of the trivial vector bundle
\[
\hat i^*:\hat\aug(T,\mu;\field)\isomorphic \tilde\aug(T,{\mu};\field)\times\field^N\to\tilde\aug(T,{\mu};\field).
\]
Here the number $N$ is the number of generators $\hat k_{ab}$'s of degree $0$ in $\hat A^\CE(T,\mu)$.
Hence for each $\epsilon_\Right\in \tilde\aug(T_\Right,{\mu_\Right};\field)$,
\begin{align*}
\left(\hat \phi^*_\Right\right)^{-1}\left(u^{-1}\cdot \epsilon_\Right\right)&=
\left(\phi^*_\Right\circ\hat i^*\right)^{-1}(\epsilon_\Right)
\isomorphic\left(\phi^*_\Right\right)^{-1}(\epsilon_\Right)\times \field^N,
\end{align*}
which implies that
\begin{align*}
\hat\aug(\cT,\bfmu;\epsilon_\Left,\ruling_\Right;\field)&=
\left(\hat\phi_\Left^*\right)^{-1}(\epsilon_\Left)\cap
\left(\hat\phi_\Right^*\right)^{-1}(\aug^{\ruling_\Right}(T_{\Right},{\mu_\Right};\field))\\
&\isomorphic\left(\left(\phi^*_\Left\right)^{-1}(\epsilon_\Left)\times \field^N\right)\cap
\left(\left(\phi^*_\Right\right)^{-1}(\aug^{\ruling_\Right}(T_{\Right},{\mu_\Right};\field))\times \field^N\right)\\
&=\aug(T,\mu;\epsilon_\Left,\ruling_\Right;\field)\times\field^N.
\end{align*}
Obviously, this is a trivial vector bundle over $\aug(\cT,\bfmu;\epsilon_\Left,\ruling_\Right;\field)$ and therefore
\[
\hat\aug(\cT,\bfmu;\epsilon_\Left,\ruling_\Right;\field)\relation\aug(\cT,\bfmu;\epsilon_\Left,\ruling_\Right;\field).
\]

The rest of the proof follows obviously from the fact that each front Reidemeister move induces a zig-zag sequence of stabilizations of bordered DGAs as discussed already in the proof of Theorem~\ref{theorem:invariance of bordered LCH DGAs}, where both inclusion and projection of each stabilization fixes both left and right border DGAs.
\end{proof}

\begin{definition}[Augmentation numbers]
Let $(\cT,\bfmu)\in\BLT_\front^\mu$.
For $\ruling_\Left\in\NR(T_{\Left},{\mu_\Left})$ and $\ruling_\Right\in\NR(T_\Right,{\mu_\Right})$ and a finite field $\FF_q$, the \emph{augmentation number} with boundary condition $(\ruling_\Left,\ruling_\Right)$ for $(\cT,\bfmu)$ is the normalization of the number of $\FF_q$-points in the augmentation variety $\aug(\cT,\bfmu;\epsilon_\Left,\ruling_\Right;\field)$
\[
\augnumber(\cT,\bfmu;\ruling_\Left,\ruling_\Right;\FF_q)\coloneqq q^{-\dim_\field\aug(\cT,\bfmu;\epsilon_\Left,\ruling_\Right;\field)}\#\aug(\cT,\bfmu;\epsilon_\Left,\ruling_\Right;\FF_q),
\]
where $\epsilon_\Left=\epsilon_{\ruling_\Left}$ is the canonical augmentation for $\ruling_\Left$.
\end{definition}

\begin{remark}
If we use $\aug(\cT,\bfmu;\ruling_\Left,\ruling_\Right;\field)$ instead of $\aug(\cT,\bfmu;\epsilon_\Left,\ruling_\Right;\field)$, then as seen in Remark~\ref{rem:aug-orbit and orbit-orbit}, 
\[
\aug(\cT,\bfmu;\ruling_\Left,\ruling_\Right;\field)\isomorphic (\field^\times)^{n_\Left/2}\times\field^{A_b(\ruling_\Left)}\times\aug(\cT,\bfmu;\epsilon_\Left,\ruling_\Right;\field)
\]
and therefore the following relation holds.
\begin{align*}
&\mathrel{\hphantom{=}}q^{-\dim_\field\aug(\cT,\bfmu;\ruling_\Left,\ruling_\Right;\field)}\#\aug(\cT,\bfmu;\ruling_\Left,\ruling_\Right;\FF_q)\\
&=q^{-\left(n_\Left/2 + A_b(\ruling_\Left) \right)-\dim_\field\aug(\cT,\bfmu;\epsilon_\Left,\ruling_\Right;\field)}
\#\left((\field^\times)^{n_\Left/2}\times\field^{A_b(\ruling_\Left)}\right)
\#\aug(\cT,\bfmu;\epsilon_\Left,\ruling_\Right;\FF_q)\\
&=q^{-\left(n_\Left/2 + A_b(\ruling_\Left) \right)}(q-1)^{n_\Left/2} q^{A_b(\ruling_\Left)}\augnumber(\cT,\bfmu;\ruling_\Left,\ruling_\Right;\FF_q)\\
&=\left( \frac{q-1}{q} \right)^{n_\Left/2} \augnumber(\cT,\bfmu;\ruling_\Left,\ruling_\Right;\FF_q).
\end{align*}
\end{remark}

As a consequence of Theorem \ref{thm:inv of aug var}, we obtain the following proposition.
\begin{proposition}\label{prop:augnumber invarince}
The augmentation number $\augnumber(\cT,\bfmu;\ruling_\Left,\ruling_\Right;\FF_q)$ is a Legendrian isotopy invariant for $(\cT,\bfmu)$.
\end{proposition}
\begin{proof}
Recall the equivalence of algebraic varieties in Definition~\ref{def:equiv rel on aug var}, that is,
\[
X\relation Y\Longleftrightarrow X\times \field^m\isomorphic Y\times \field^n
\]
for some $m, n$.
Then the numbers of $\FF_q$ points and dimensions are
\begin{align*}
\#(X\times \field^m)(\FF_q) &= \# X(\FF_q)\cdot q^m&
\dim_\field(X\times \field^m) &= \dim_\field X+ m,\\
\#(Y\times \field^n)(\FF_q) &= \# Y(\FF_q)\cdot q^n,&
\dim_\field(Y\times \field^n) &= \dim_\field Y+ n.
\end{align*}
Therefore the normalizations of the numbers of $\FF_q$ points coincide
\begin{align*}
q^{-\dim_\field X}\#X(\FF_q) &= q^{-\dim_\field (X\times \field^m)} \#(X\times\field^m)(\FF_q)\\
&=q^{-\dim_\field (Y\times \field^n)} \#(Y\times\field^n)(\FF_q) = q^{-\dim_\field Y}\#Y(\FF_q).\qedhere
\end{align*}
\end{proof}

\subsection{Augmentation varieties for internal DGAs}\label{sec:aug_variety_vertex}
The key example concerns the information near a vertex of a bordered Legendrian graph.
Recall the construction of the DGA $A^\CE(\cT,\mu)$.
As seen in the equation \eqref{equation:internal DG-subalgebra}, for each vertex $v$, there is a DG-subalgebra of $I_v\subset A^\CE(\cT,\mu)$.

Let $v$ be a vertex of type $(\ell,r)$ with $n=\ell+r$, then $I_v=(\sfI_v,\differential_v)\isomorphic I_{(\ell,r)}(\mu_v)$ is as follows:
\begin{align*}
\sfI_v&=\ring\langle v_{a,i}\mid 1\le a\le n, i\ge 1\rangle,&
|v_{a,i}|&=\mu_v(a) - \mu_v(a+i)+ n(v,a,i)-1,
\end{align*}
where $n(v,a,i)$ is defined in Section~\ref{section:CE DGA}.
Notice that among these infinitely many geneartors, there are only finitely many generators in each degree.
Indeed, there exists $N=N(\mu_v)\ge m$ such that $|v_{a,i}|>1$ for any $a\in\Zmod{n}$ and $i>N$.

The differential $\differential_v$ is defined by
\[
\differential_v v_{a,i} \coloneqq \delta_{i,n}+\sum_{i_1+i_2=i}(-1)^{|v_{a,i_1}|-1} v_{a,i_1}v_{a+i_1,i_2}.
\]

As before, by regarding $\epsilon(v_{a,i})$ as $x_{a,i}$, the augmentation variety at $v$ is an affine algebraic variety
\[
\tilde\aug(v;\field)\isomorphic\left\{
(x_{a,i})\in \field^{n N}~\middle|~ 
a\in\Zmod{n}, i\in[N], 
x_{a,i}=0\text{ if }|x_{a,i}|\neq0, \sum_{i_1+i_2=i} x_{a,i_1}x_{a+i_1,i_2}=\delta_{i,n}
\right\},
\]
which will be denoted by $\aug(v;\field)$.

We introduce a Morse complex for a vertex as follows:
\begin{definition}[Morse complex at a vertex]\label{def:Morse complex at a vertex}
Let $\field[Z]$ be the graded polynomial ring in one variable $Z$ with $|Z|=1$. 
We define a free graded left $\field[Z]$-module 
\begin{align*}
C_v&\coloneqq \field[Z]\langle e_1,\dots,e_n\rangle,& |e_a|&\coloneqq -\mu_v(a)
\end{align*}
and a decreasing filtration 
\[
C_v\supset F^1 C_v\supset\cdots\supset F^\ell C_v\supset Z\cdot F^{\ell+1}C_v\supset\cdots\supset Z\cdot F^{\ell+r}C_v\supset Z^2\cdot F^{n+1} C_v = Z^2\cdot F^1 C_v
\]
of $C_v$ by free graded left $\field$-submodules such that for each $a\in\Zmod{n}$,
\begin{align*}
F^aC_v&\coloneqq \field\langle e_a\rangle\oplus\bigoplus_{i>0} \field\langle Z^{n(v,a,i)}e_{a+i} \rangle.
\end{align*}

The group of $\field[Z]$-superlinear automorphisms of $C_v$ preserving $F^\bullet$ will be denoted by $B(v;\field)$ and its unipotent subgroup will be denoted by $U(v;\field)\subset B(v;\field)$ consisting of automorphisms $u\in B(v;\field)$ such that $\langle u(e_a),e_a\rangle=1$ for all $a\in\Zmod{n}$.

We define $\MC(v;\field)$ to be the set of all $\field[Z]$-superlinear endormorphisms $d$ of degree $1$ which preserves $F^{\bullet}$ and satisfies $d^2+Z^2=0$.
\end{definition}

In the above definition, the pairing $\langle~,~\rangle:C_v\otimes C_v\rightarrow\field$ is $\field$-bilinear and satisfies $\langle Z^ie_a,Z^je_b\rangle\coloneqq \delta_{i,j}\cdot\delta_{a,b}$. Moreover, we assume that $d\in\MC(v;\field)$ is super-commutative with $Z$. That is, for all $x\in C_v$
\begin{equation}\label{equation:super commutativity}
d(Z\cdot x) = -Z\cdot d(x).
\end{equation}
and therefore $(d+Z)^2=0$ if and only if $d^2+Z^2=0$.

\begin{remark}\label{rem:morphisms of Morse complexes at a vertex}
A homogeneous $\field[Z]$-superlinear endormorphism $d\in\End(C_v)$ preserves $F^{\bullet}$ if and only if $d(e_a)\in F^a C_v$ for all $a\in\Zmod{m}$, that is,
\[
d(e_a)=c_{a,0}e_a+\sum_{i>0}c_{a,i}Z^{n(v,a,i)}e_{a+i}
\] 
for some $c_{a,i}\in\field$, with $c_{a,0}=0$ if $|d|\neq 0$.
\end{remark}

\begin{lemma}\label{lem:aug and Morse complex for a vertex}
There is a canonical identification
\begin{align*}
\Xi_v:\aug(v;\field)&\stackrel{\isomorphic}{\to} \MC(v;\field);\\
\epsilon &\mapsto d=\Xi_v(\epsilon),
\end{align*}
where
\[
d(e_a)=(-1)^{\mu_v(a)}\sum_{i>0}\epsilon(v_{a,i})Z^{n(v,a,i)}e_{a+i}.
\]
\end{lemma}
\begin{proof}
Since $d=\Xi_v(\epsilon)$ is of the form in Remark~\ref{rem:morphisms of Morse complexes at a vertex}, it preserves $F^{\bullet}C_v$ and is of degree 1 since $(-1)^{\mu_v(a)}\epsilon(v_{a,i})e_{a+i}\neq0$ only if $|v_{a,i}|=0$ and so
\[
|d|=|Z^{n(v,a,i)}e_{a+i}|-|e_a|=\mu_v(a)-\mu_v(a+i)+n(v,a,i)=|v_{a,i}|+1=1.
\]

Moreover, we have 
\begin{align*}\allowdisplaybreaks
d^2(e_a)&=d\left(\sum_{i_1>0}(-1)^{\mu_v(a)}\epsilon(v_{a,i_1})Z^{n(v,a,i_1)}e_{a+i_1}\right)\\
&=\sum_{i_1>0}(-1)^{\mu_v(a)+n(v,a,i_1)}\epsilon(v_{a,i_1})Z^{n(v,a,i_1)}d(e_{a+i_1})
\tag{by \eqref{equation:super commutativity}}\\
&=\sum_{i_1>0}(-1)^{\mu_v(a)+n(v,a,i_1)}\epsilon(v_{a,i_1})Z^{n(v,a,i_1)}\left(\sum_{i_2>0}(-1)^{\mu_v(a+i_1)}\epsilon(v_{a+i_1,i_2})Z^{n(v,a+i_1,i_2)}e_{a+i_1+i_2}\right)\\
&=\sum_{i_1,i_2>0}(-1)^{\mu_v(a)-\mu_v(a+i_1)+n(v,a,i_1)}\epsilon(v_{a,i_1}v_{a+i_1,i_2})Z^{n(v,a,i_1+i_2)}e_{a+i_1+i_2}
\tag{by \eqref{equation:additivity of n}}\\
&=\sum_{i>0}Z^{n(v,a,i)}\sum_{i_1+i_2=i}(-1)^{|v_{a,i_1}|+1}\epsilon(v_{a,i_1}v_{a+i_1,i_2})e_{a+i}
\tag{by \eqref{equation:degree of vertex generators}}\\
&=\sum_{i>0}Z^{n(v,a,i)}\left((\epsilon\circ\differential)(v_{a,i})-\delta_{i,n}\right)e_{a+i}
\tag{by \eqref{equation:differential for vertex generators}}\\
&=-Z^2e_a.
\end{align*}
The last equality holds by the defining equation of the augmentation \eqref{equation:defining equation for augmentations} and we are done.
\end{proof}

From now on, we will always use the identification above. In particular, the algebraic group $B(v;\field)$ acts naturally on $\aug(v;\field)$ via conjugation.

\begin{definition}[Canonical augmentations and differentials at a vertex]\label{def:canonical augmentation at a vertex}
Let us define as before the set $\NR(v)$
of all fixed-point-free involutions $\ruling$ on $[n]$ satisfying that $|v_{a,b-a}|=0$ for all $1\le a<b=\ruling(a)\le n$.

For each involution $\ruling\in \NR(v)$, define the \emph{canonical augmentation} $\epsilon_\ruling\in\aug(v;\field)$ and \emph{differential} $d_{\ruling}\in\MC(v;\field)$ as follows:
\begin{align*}
\epsilon_\ruling(v_{c,i})&\coloneqq
\begin{cases}
1 & 1\leq a=c<b=a+i=\ruling(a)\leq n;\\
1 & 1\leq a=b+i-n<c=b=\ruling(a)\leq n;\\
0 &\text{otherwise},
\end{cases}\\
d_\ruling(e_c)&\coloneqq \begin{cases}
(-1)^{\mu_v(c)}Z^{n(v,a,b-a)}e_b & a=c<b=\ruling(a);\\
(-1)^{\mu_v(c)}Z^{n(v,b,n+a-b)}e_a & a<c=b=\ruling(a).
\end{cases}
\end{align*}

We denote the orbit of $\epsilon_\ruling$ under the action of $B(v;\field)$ by
\[
\aug^\ruling(v;\field)\coloneqq B(v;\field)\cdot \epsilon_{\ruling}
\]
and denote its stabilizer subgroup by $\Stab^\ruling(v;\field)\subset B(v;\field)$.
\end{definition}

As before, each $\ruling\in \NR(v)$ defines a partition $[n]=U_\ruling\amalg L_\ruling$, together with a bijection $\ruling:U_\ruling\stackrel{\isomorphic}\longrightarrow L_\ruling$ such that $a<b=\ruling(a)\in L_\ruling$ for any $a\in U_\ruling$.

\begin{definition}\label{def:indices for an involution at a vertex}
For any $a\in[n]$, let 
\[
I_v(a)\coloneqq \{i>0 \mid \mu(a)-\mu(a+i)-n(v,a,i)=0,\text{i.e., }|v_{a,i}|=1\}
\]
which is finite due to the degree reason.

For any $a\in U_\ruling$, define
\begin{align*}
A_{\ruling}(a)\coloneqq \{b\in U \mid \ruling(b)<\ruling(a), \text{ and } b-a\in I_v(a)\}.
\end{align*}
By definition, we have $A_{\ruling}(b)\subsetneq A_{\ruling}(a)$ for all $b\in A_{\ruling}(a)$.
Now, define $A_v(\ruling)\in\NN$ by
\begin{align*}
A_v(\ruling)\coloneqq \sum_{a\in U}|A_{\ruling}(a)|+\sum_{a\in L}|I_v(a)|.
\end{align*}
\end{definition}

Then we have the similar result to Lemma \ref{lem:ruling stratification for trivial tangle}, which will be proved in Appendix~\ref{sec:Orbits of augmentation varieties for internal DGAs}.
\begin{lemma}\label{lem:stratification of aug var at a vertex}
Let $v$ be a vertex as above.

\begin{enumerate}
\item There are decompositions of $\aug(v;\field)$ and $\MC(v;\field)$ over the finite set $\NR(v)$
\begin{align*}
\aug(v;\field)&=\coprod_{\ruling\in \NR(v)}\aug^{\ruling}(v;\field)&
\MC(v;\field)&=\coprod_{\ruling\in \NR(v)}B(v;\field)\cdot d_{\ruling}
\end{align*}
In particular, $\aug(v;\field)=\emptyset$ if $\val(v)$ is odd.

\item For all $\ruling\in\NR(v)$, the $\Stab^\ruling(v;\field)$-principal bundle 
\[
\pi_{\ruling}:B(v;\field)\to B(v;\field)\cdot d_{\ruling}
\]
admits a natural algebraic section $\varphi_{\ruling}$, i.e. $\varphi_{\ruling}(d)\cdot d_{\ruling}=d$ for all $d\in B(v;\field)\cdot d_{\ruling}$. In other words, we have a trivialization of $\pi_{\ruling}$: 
\[
B(v;\field)\isomorphic  \Stab^\ruling(v;\field)\times B(v;\field)\cdot d_{\ruling}
\]
\item For all $\ruling\in\NR(v)$, we have:
\begin{equation}
B(v;\field)\cdot d_{\ruling}\isomorphic (\field^\times)^{\frac{\val(v)}{2}}\times\field^{A_v(\ruling)}.
\end{equation}
\end{enumerate}
\end{lemma}

\section{Rulings}\label{section:rulings}
In this section, we prove that augmentation numbers for (bordered) Legendrian graphs are computed by certain associated ruling polynomials.

\subsection{Resolution of vertices}\label{sec:resolution of vertex}

Let us recall from \cite{ABK2019} the construction of ruling invariants for (bordered) Legendrian graphs. 
The key is to resolve a vertex $v$ with respect to an involution $\ruling\in\NR(v)$. At first glance, this operation may look weird. However, there is a geometric intuition from the augmentation side. By studying the structure of the augmentation variety for an elementary bordered Legendrian graph with a single vertex, the operation of resolving the vertex arises naturally. See Remark \ref{rem:ruling decomposition}. 

Let $(\cT,\bfmu)\in\BLT_\front^\mu$ be a bordered Legendrian graph with Maslov potential $\bfmu$ and $v$ be a vertex of type $(\ell,r)$ with $\ell+r=n$. 
We identify the set $H_v$ of half-edges at $v$ with $[n]$ as defined in Definition~\ref{definition:types of vertices}.

Let $\ruling\in\NR(v)$ be a fixed-point-free involution (or a \emph{perfect matching}) on $[n]$.
We first recall the construction of the \emph{marked} bordered Legendrian graph $(\cT_{v,\ruling},M_{v,\ruling})$, which is the pair of the following:
\begin{itemize}
\item the bordered Legendrian link $\cT_{v,\ruling}$ of type $(\ell,r)$, and 
\item the set $M_{v,\ruling}$ of \emph{marking}, which is a subset of crossings in $\cT_{v,\ruling}$.
\end{itemize}

We first split $[n]$ into the following three subsets
\begin{align*}
L_v(\ruling)&=\{a\in [n] \mid a<\ruling(a)\le \ell \}=\{a^L_1<\dots<a^L_s\};\\
B_v(\ruling)&=\{a\in [n] \mid a\le\ell<\ruling(a) \}=\{a^B_1<\dots<a^B_t\};\\
R_v(\ruling)&=\{a\in [n] \mid \ell<a<\ruling(a)\}=\{a^R_1<\dots<a^R_u\}.
\end{align*}
Then $\ell=2s+t$ and $r=2u+t$.

\begin{example}\label{example:resolution of a vertex}
For example, let us assume that $v$ of type $(7,5)$ and $\ruling$ is given as
\begin{align*}
\ruling=\{(1,10), (2,5), (3,7), (4,8), (6,11), (9,12)\},
\end{align*}
where $(a,b)\in\ruling$ means that $a<b=\ruling(a)$.
Then three sets above are defined as
\begin{align*}
L_v(\ruling)&=\{2,3\},&
B_v(\ruling)&=\{1,4,6\},&
R_v(\ruling)&=\{9\}.
\end{align*}
\end{example}

For each $a^L_k\in L_v(\ruling)$, we assign a bordered Legendrian graph $\cT_{v,\ruling}(a^L_k)$ of type $(\ell-2k,\ell-2k-2)$ for some $k$ with one right cusp and several markings as depicted in Figure~\ref{figure:generalized cusps}.
By concatenating bordered Legendrian graphs $\cT_{v,\ruling}(a^L_1),\dots,\cT_{v,\ruling}(a^L_s)$, we obtain a bordered Legendrian graph $(\cT_{v,\ruling}^L,M_{v,\ruling}^L)$ of type $(\ell,t)$ as depicted in Figure~\ref{figure:concatenation of generalized left cusps}.
By the same procedure, we assign a bordered Legendrian graph $(\cT_{v,\ruling}^R,M_{v,\ruling}^R)$ of type $(t,r)$.

\begin{figure}[ht]
\subfigure[A marked cusp\label{figure:generalized cusps}]{\makebox[0.3\textwidth]{$
\cT_{v,\ruling}(a^L_k)=
\begin{tikzpicture}[baseline=-.5ex,scale=1]
\draw[red] (-1,-1) -- (-1,1);
\draw[thick] (-1,0.9) to[out=0,in=180] (0,0.6);
\draw[thick] (-1,0.3) to[out=0,in=180] (0,0.3);
\draw[thick] (-1,0) to[out=0,in=180] (0,0);
\draw[thick] (-1,-0.6) to[out=0,in=180] (0,-0.3);
\draw[thick] (-1,-0.9) to[out=0,in=180] (0,-0.6);
\draw[thick] (-1,0.6) node[left] {$a^L_k$} to[out=0,in=180] (-0.5,-0.3) -- (-1,-0.3) node[left] {$\ruling(a^L_k)$};
\draw[densely dotted,fill=black,opacity=0.5] (-0.77,0.3) circle (0.1) (-0.73,0) circle (0.1);
\draw[red] (0,-1) -- (0,1);
\end{tikzpicture}
$}}
\subfigure[Concatenated marked cusps\label{figure:concatenation of generalized left cusps}]{\makebox[0.6\textwidth]{$
\cT_{v,\ruling}^L=
\begin{tikzpicture}[baseline=-.5ex,scale=1]
\begin{scope}
\draw[red] (-1,-1) -- (-1,1);
\draw[thick] (-1,0.9) to[out=0,in=180] (0,0.6);
\draw[thick] (-1,0.3) to[out=0,in=180] (0,0.3);
\draw[thick] (-1,0) to[out=0,in=180] (0,0);
\draw[thick] (-1,-0.6) to[out=0,in=180] (0,-0.3);
\draw[thick] (-1,-0.9) to[out=0,in=180] (0,-0.6);
\draw[thick] (-1,0.6) to[out=0,in=180] (-0.5,-0.3) -- (-1,-0.3);
\draw[densely dotted,fill=black,opacity=0.5] (-0.77,0.3) circle (0.1) (-0.73,0) circle (0.1);
\end{scope}
\begin{scope}[xshift=1cm]
\draw[thick] (-1,0.6) to[out=0,in=180] (0,0.3);
\draw[thick] (-1,0) to[out=0,in=180] (0,0);
\draw[thick] (-1,-0.3) to[out=0,in=180] (0,-0.3);
\draw[thick] (-1,0.3) to[out=0,in=180] (-0.5,-0.6) -- (-1,-0.6);
\draw[densely dotted,fill=black,opacity=0.5] (-0.77,0) circle (0.1) (-0.73,-0.3) circle (0.1);
\draw[red] (0,-1) -- (0,1);
\end{scope}
\end{tikzpicture}\quad
\quad
\cT_{v,\ruling}^R=
\begin{tikzpicture}[baseline=-.5ex,scale=1]
\draw[thick] (-1,0.3) to[out=0,in=180] (0,0.6);
\draw[thick] (-1,0) to[out=0,in=180] (0,0);
\draw[thick] (-1,-0.3) to[out=0,in=180] (0,-0.3);
\draw[thick] (0,0.3) to[out=180,in=0] (-0.5,-0.6) -- (0,-0.6);
\draw[densely dotted,fill=black,opacity=0.5] (-0.23,0) circle (0.1) (-0.27,-0.3) circle (0.1);
\draw[red] (0,-1) -- (0,1);
\draw[red] (-1,-1) -- (-1,1);
\end{tikzpicture}
$}}
\caption{Two marked bordered Legendrian links $(\cT_{v,\ruling}^L,M_{v,\ruling}^L)$ and $(\cT_{v,\ruling}^R,M_{v,\ruling}^R)$}
\end{figure}

From the pairing $B_v(\ruling)$, we assign a braid $\beta=\beta(\ruling)$ with $t$-stands having minimal crossings. 
We regard $\beta$ as a bordered Legendrian graph of type $(t,t)$ 
consisting of $t$ Legendrian arcs without self-intersection, cusps, and markings. 
Then the minimality means that any pair of strands make at most 1 crossing in its front projection.

Let $\beta^c$ be a \emph{right complement} of $\beta$ satisfying $\beta\cdot \beta^c=\Delta_b$, where $\Delta_b$ is the $b$-braid of {half-twist}, and $\overline{\beta^c}$ is the mirror of $\beta^c$. 
We assume that all crossings in $\beta^c$ and $\bar\beta^c$ are marked.
Then the middle part of the resolution $(\cT_\ruling^B, M_\ruling^B)$ will be defined to be the concatenation of three marked bordered Legendrian graphs $\beta, \beta^c$ and $\bar\beta^c$.
\[
(\cT_{v,\ruling}^B, M_{v,\ruling}^B) \coloneqq \beta\cdot\beta^c\cdot\bar\beta^c
\]
See Figure~\ref{figure:complement and mirror}.

\begin{figure}[ht]
\subfigure[A right complement $\beta^c$ and its mirror $\bar\beta^c$\label{figure:complement and mirror}]{
$
\begin{aligned}
\beta&=
\begin{tikzpicture}[baseline=-.5ex]
\draw[thick] (-1,0.3) to[out=0,in=180] (0,0);
\draw[thick] (-1,0) to[out=0,in=180] (0,0.3);
\draw[thick] (-1,-0.3) to[out=0,in=180] (0,-0.3);
\draw[red] (-1,-0.5)--(-1,0.5);
\draw[red] (0,-0.5)--(0,0.5);
\end{tikzpicture},&
\Delta&=
\begin{tikzpicture}[baseline=-.5ex]
\draw[thick] (-1,0.3) to[out=0,in=180] (0,-0.3);
\draw[thick] (-1,0) to[out=0,in=180] (-0.5,0.3) to[out=0,in=180] (0,0);
\draw[thick] (-1,-0.3) to[out=0,in=180] (0,0.3);
\draw[red] (-1,-0.5)--(-1,0.5);
\draw[red] (0,-0.5)--(0,0.5);
\end{tikzpicture}=
\begin{tikzpicture}[baseline=-.5ex]
\draw[thick] (-1,0.3) to[out=0,in=180] (0,0);
\draw[thick] (-1,0) to[out=0,in=180] (0,0.3);
\draw[thick] (-1,-0.3) to[out=0,in=180] (0,-0.3);
\draw[red] (-1,-0.5)--(-1,0.5);
\draw[red] (0,-0.5)--(0,0.5);
\end{tikzpicture}
\cdot
\begin{tikzpicture}[baseline=-.5ex]
\draw[thick] (-1,0.3) to[out=0,in=180] (0,0);
\draw[thick] (-1,0) to[out=0,in=180] (0,-0.3);
\draw[thick] (-1,-0.3) to[out=0,in=180] (0,0.3);
\draw[red] (-1,-0.5)--(-1,0.5);
\draw[red] (0,-0.5)--(0,0.5);
\end{tikzpicture},&
\beta^c&=
\begin{tikzpicture}[baseline=-.5ex]
\draw[thick] (-1,0.3) to[out=0,in=180] (0,0);
\draw[thick] (-1,0) to[out=0,in=180] (0,-0.3);
\draw[thick] (-1,-0.3) to[out=0,in=180] (0,0.3);
\draw[densely dotted, fill, opacity=0.5] (-0.6, -0.1) circle (0.1) (-0.4, 0.1) circle (0.1);
\draw[red] (-1,-0.5)--(-1,0.5);
\draw[red] (0,-0.5)--(0,0.5);
\end{tikzpicture},&
\bar\beta^c&=
\begin{tikzpicture}[baseline=-.5ex]
\draw[thick] (-1,0.3) to[out=0,in=180] (0,-0.3);
\draw[thick] (-1,0) to[out=0,in=180] (0,0.3);
\draw[thick] (-1,-0.3) to[out=0,in=180] (0,0);
\draw[densely dotted, fill, opacity=0.5] (-0.6, 0.1) circle (0.1) (-0.4, -0.1) circle (0.1);
\draw[red] (-1,-0.5)--(-1,0.5);
\draw[red] (0,-0.5)--(0,0.5);
\end{tikzpicture},&
\end{aligned}
$}

\subfigure[The middle part of the $\ruling$-resolution]{$
(\cT_{v,\ruling}^B,M_{v,\ruling}^B)=\beta\cdot\beta^c\cdot\bar\beta^c=
\begin{tikzpicture}[baseline=-.5ex]
\draw[thick] (-1,0.3) to[out=0,in=180] (0,0);
\draw[thick] (-1,0) to[out=0,in=180] (0,0.3);
\draw[thick] (-1,-0.3) to[out=0,in=180] (0,-0.3);
\draw[red] (-1,-0.5)--(-1,0.5);
\draw[red] (0,-0.5)--(0,0.5);
\end{tikzpicture}\cdot
\begin{tikzpicture}[baseline=-.5ex]
\draw[thick] (-1,0.3) to[out=0,in=180] (0,0);
\draw[thick] (-1,0) to[out=0,in=180] (0,-0.3);
\draw[thick] (-1,-0.3) to[out=0,in=180] (0,0.3);
\draw[densely dotted, fill, opacity=0.5] (-0.6, -0.1) circle (0.1) (-0.4, 0.1) circle (0.1);
\draw[red] (-1,-0.5)--(-1,0.5);
\draw[red] (0,-0.5)--(0,0.5);
\end{tikzpicture}\cdot
\begin{tikzpicture}[baseline=-.5ex]
\draw[thick] (-1,0.3) to[out=0,in=180] (0,-0.3);
\draw[thick] (-1,0) to[out=0,in=180] (0,0.3);
\draw[thick] (-1,-0.3) to[out=0,in=180] (0,0);
\draw[densely dotted, fill, opacity=0.5] (-0.6, 0.1) circle (0.1) (-0.4, -0.1) circle (0.1);
\draw[red] (-1,-0.5)--(-1,0.5);
\draw[red] (0,-0.5)--(0,0.5);
\end{tikzpicture}
=
\begin{tikzpicture}[baseline=-.5ex]
\draw[thick] (-1,0.3) to[out=0,in=180] (-0.5,0);
\draw[thick] (-1,0) to[out=0,in=180] (-0.5,0.3);
\draw[thick] (-1,-0.3) to[out=0,in=180] (-0.5,-0.3);
\draw[thick] (-0.5,0.3) to[out=0,in=180] (0,0);
\draw[thick] (-0.5,0) to[out=0,in=180] (0,-0.3);
\draw[thick] (-0.5,-0.3) to[out=0,in=180] (0,0.3);
\draw[thick] (0,0.3) to[out=0,in=180] (0.5,-0.3);
\draw[thick] (0,0) to[out=0,in=180] (0.5,0.3);
\draw[thick] (0,-0.3) to[out=0,in=180] (0.5,0);
\draw[densely dotted, fill, opacity=0.5] (-0.2, 0.1) circle (0.1) (-0.3, -0.1) circle (0.1);
\draw[densely dotted, fill, opacity=0.5] (0.3, -0.1) circle (0.1) (0.2, 0.1) circle (0.1);
\draw[red] (-1,-0.5)--(-1,0.5);
\draw[red] (0.5,-0.5)--(0.5,0.5);
\end{tikzpicture}
$}
\end{figure}

Finally, the resolution of the vertex $v$ with respect to an involution $\ruling\in \NR(v)$ is a bordered Legendrian graph with markings $(\cT_{v,\ruling},M_{v,\ruling})$ is defined by the concatenation
\[
(\cT_{v,\ruling},M_{v,\ruling}) \coloneqq (\cT_{v,\ruling}^L,M_{v,\ruling}^L)\cdot  (\cT_{v,\ruling}^B, M_{v,\ruling}^B)\cdot (\cT_{v,\ruling}^R,M_{v,\ruling}^R).
\]

\begin{example}[cont.]
The marked bordered Legendrian graph $(\ruling,M)$ is given by
\[
(\cT_{v,\ruling},M_{v,\ruling})=
(\cT_{v,\ruling}^L,M_{v,\ruling}^L)\cdot  (\cT_{v,\ruling}^B, M_{v,\ruling}^B)\cdot (\cT_{v,\ruling}^R,M_{v,\ruling}^R)=
\begin{tikzpicture}[baseline=-.5ex]
\begin{scope}
\draw[red] (-1,-1) -- (-1,1);
\draw[thick] (-1,0.9) node[left] {$1$} to[out=0,in=180] (0,0.6);
\draw[thick] (-1,0.3) to[out=0,in=180] (0,0.3);
\draw[thick] (-1,0) node[left] {$\vdots$} to[out=0,in=180] (0,0);
\draw[thick] (-1,-0.6) to[out=0,in=180] (0,-0.3);
\draw[thick] (-1,-0.9) node[left] {$7$} to[out=0,in=180] (0,-0.6);
\draw[thick] (-1,0.6) to[out=0,in=180] (-0.5,-0.3) -- (-1,-0.3);
\draw[densely dotted,fill=black,opacity=0.5] (-0.77,0.3) circle (0.1) (-0.73,0) circle (0.1);
\end{scope}
\begin{scope}[xshift=1cm]
\draw[thick] (-1,0.6) to[out=0,in=180] (0,0.3);
\draw[thick] (-1,0) to[out=0,in=180] (0,0);
\draw[thick] (-1,-0.3) to[out=0,in=180] (0,-0.3);
\draw[thick] (-1,0.3) to[out=0,in=180] (-0.5,-0.6) -- (-1,-0.6);
\draw[densely dotted,fill=black,opacity=0.5] (-0.77,0) circle (0.1) (-0.73,-0.3) circle (0.1);
\end{scope}
\begin{scope}[xshift=2cm]
\draw[thick] (-1,0.3) to[out=0,in=180] (-0.5,0);
\draw[thick] (-1,0) to[out=0,in=180] (-0.5,0.3);
\draw[thick] (-1,-0.3) to[out=0,in=180] (-0.5,-0.3);
\draw[thick] (-0.5,0.3) to[out=0,in=180] (0,0);
\draw[thick] (-0.5,0) to[out=0,in=180] (0,-0.3);
\draw[thick] (-0.5,-0.3) to[out=0,in=180] (0,0.3);
\draw[thick] (0,0.3) to[out=0,in=180] (0.5,-0.3);
\draw[thick] (0,0) to[out=0,in=180] (0.5,0.3);
\draw[thick] (0,-0.3) to[out=0,in=180] (0.5,0);
\draw[densely dotted, fill, opacity=0.5] (-0.2, 0.1) circle (0.1) (-0.3, -0.1) circle (0.1);
\draw[densely dotted, fill, opacity=0.5] (0.3, -0.1) circle (0.1) (0.2, 0.1) circle (0.1);
\end{scope}
\begin{scope}[xshift=3.5cm]
\draw[thick] (-1,0.3) to[out=0,in=180] (0,0.6) node[right] {$8$};
\draw[thick] (-1,0) to[out=0,in=180] (0,0) node[right] {$\vdots$};
\draw[thick] (-1,-0.3) to[out=0,in=180] (0,-0.3);
\draw[thick] (0,0.3) to[out=180,in=0] (-0.5,-0.6) -- (0,-0.6) node[right] {$12$};
\draw[densely dotted,fill=black,opacity=0.5] (-0.23,0) circle (0.1) (-0.27,-0.3) circle (0.1);
\draw[red] (0,-1) -- (0,1);
\end{scope}
\end{tikzpicture}
\]
\end{example}

\begin{example}\label{ex:resolution of (3,3)}
Let $\ell+r=6$, then there are 15 possible marked bordered Legendrian links for the vertex of type $(\ell, r)$. Especially when $(\ell,r)=(3,3)$ we have the following list of resolutions with markings without considering a Maslov potential:
\[
\vcenter{\hbox{
\begingroup%
  \makeatletter%
  \providecommand\color[2][]{%
    \errmessage{(Inkscape) Color is used for the text in Inkscape, but the package 'color.sty' is not loaded}%
    \renewcommand\color[2][]{}%
  }%
  \providecommand\transparent[1]{%
    \errmessage{(Inkscape) Transparency is used (non-zero) for the text in Inkscape, but the package 'transparent.sty' is not loaded}%
    \renewcommand\transparent[1]{}%
  }%
  \providecommand\rotatebox[2]{#2}%
  \ifx\svgwidth\undefined%
    \setlength{\unitlength}{288.428624bp}%
    \ifx\svgscale\undefined%
      \relax%
    \else%
      \setlength{\unitlength}{\unitlength * \real{\svgscale}}%
    \fi%
  \else%
    \setlength{\unitlength}{\svgwidth}%
  \fi%
  \global\let\svgwidth\undefined%
  \global\let\svgscale\undefined%
  \makeatother%
  \begin{picture}(1,0.40081215)%
    \put(0,0){\includegraphics[width=\unitlength,page=1]{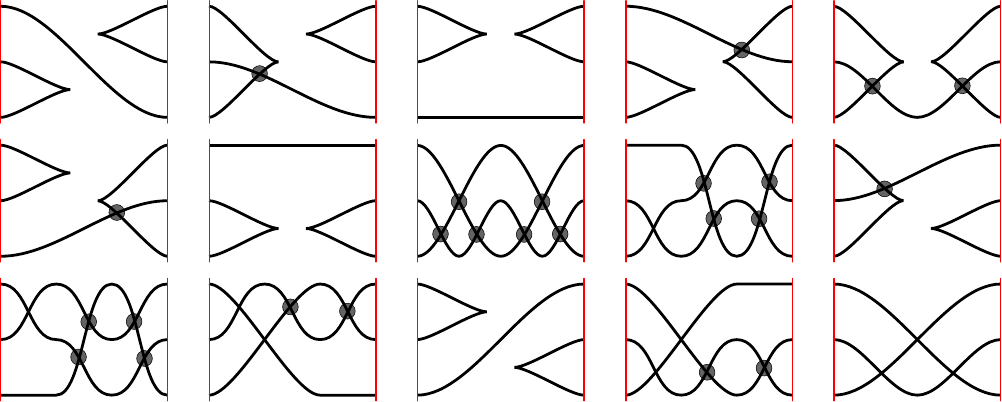}}%
  \end{picture}%
\endgroup%
}}
\]
\end{example}

\begin{definition}[$\ruling$-resolutions]\label{def:resolution at a vertex}
Let $(\cT,\bfmu)\in\BLT_\front^\mu$ and $v$ be a vertex of $T$ possibly with a set $M$ of markings.
For each $\ruling\in\NR(v)$, a $\ruling$-resolution $((\cT_\ruling,\mu_\ruling),M_\ruling)$ of $((\cT,\bfmu),M)$ at $v$ is defined to be the tangle replacement of a small neighborhood of $v$ with $(\cT_{v,\ruling},M_{v,\ruling})$.
That is, $M_\ruling=M\amalg M_{v,\ruling}$ and $\cT_\ruling=(T_\Left\to T_\ruling \leftarrow T_\Right)$, where
\[
T_\ruling\coloneqq (T\setminus U_v)\amalg \cT_{v,\ruling}.
\]
\end{definition}

Note that the Maslov potential $\bfmu$ on $\cT$ inherits to the result $\cT_\ruling$ of the tangle replacement because the construction of $\cT_{v,\ruling}$ obeys the condition of a Maslov potential.
Therefore the Maslov potential $\bfmu_\ruling$ on $\cT_\ruling$ is well-defined.

\begin{example}[A Legendrian graph with 6-valent vertex]\label{example:six valent}
Let us consider the Legendrian graph $\Lambda$ with a unique vertex $v$ of valency 6 which consists of three Legendrian unknots intersecting at one vertex $v$ as depicted in Figure~\ref{fig:six_valent_example}.

\begin{figure}[ht]
{\scriptsize\input{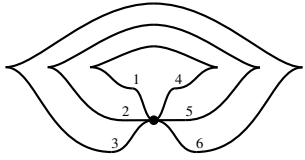}}
\caption{A Legendrian graph with six-valent vertex}
\label{fig:six_valent_example}
\end{figure}

The Maslov potential $\mu$ assigns 0 and 1 to the lower and the upper arcs, respectively.
Note that all the half edges $h_{v,i}$ near the vertex $v$ has the Maslov potential $\mu(h_{v,i})=0$.
As above, there are fifteen different involutions on $[6]$, see Example~\ref{ex:resolution of (3,3)}, but only six of them are possible due to Definition~\ref{definition:GNR and NR}, as listed below:
\begin{align*}
\ruling_v^1&=\{(1,6),(2,5),(3,4)\},&
\ruling_v^2&=\{(1,6),(2,4),(3,5)\},&
\ruling_v^3&=\{(1,5),(2,6),(3,4)\},\\
\ruling_v^4&=\{(1,4),(2,5),(3,6)\},&
\ruling_v^5&=\{(1,5),(2,4),(3,6)\},&
\ruling_v^6&=\{(1,4),(2,6),(3,5)\}.
\end{align*}

Then the corresponding resolutions are as depicted in Figure~\ref{figure:resolution_six_valent}.
\end{example}

\begin{figure}[ht]
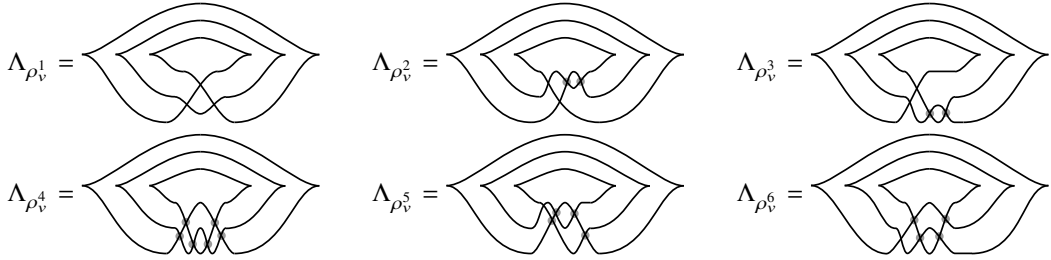

\begin{align*}
\Lambda_{\ruling_v^1}=\vcenter{\hbox{\scriptsize\def\svgscale{0.8}\input{six_res_front_1_input.tex}}}&\qquad
\Lambda_{\ruling_v^2}=\vcenter{\hbox{\scriptsize\def\svgscale{0.8}\input{six_res_front_2_input.tex}}}\qquad
\Lambda_{\ruling_v^3}=\vcenter{\hbox{\scriptsize\def\svgscale{0.8}\input{six_res_front_3_input.tex}}}\\
\Lambda_{\ruling_v^4}=\vcenter{\hbox{\scriptsize\def\svgscale{0.8}\input{six_res_front_4_input.tex}}}&\qquad
\Lambda_{\ruling_v^5}=\vcenter{\hbox{\scriptsize\def\svgscale{0.8}\input{six_res_front_5_input.tex}}}\qquad
\Lambda_{\ruling_v^6}=\vcenter{\hbox{\scriptsize\def\svgscale{0.8}\input{six_res_front_6_input.tex}}}
\end{align*}
\caption{Resolutions for $\Lambda$ in the front projection}
\label{figure:resolution_six_valent}
\end{figure}

\begin{definition}\label{def:full_resolution_marking}
Let $(\cT,\bfmu)\in\BLT_\front^\mu$ be a Legendrian graph with Maslov potential and possibly with marking $M$.
We define the set $\cP(\cT,\bfmu)$ of all resolutions of $\cT$ as
\begin{align*}
\cP(\cT,\bfmu) \coloneqq \prod_{v\in V} \NR(v)
= \{ \Phi=(\ruling_v) \mid \ruling_v\in \NR(v)\}.
\end{align*}

For each $\Phi=(\ruling_{v_i})\in\cP(\cT,\bfmu)$ and $M\subset C(T)$, we define the \emph{full resolution} of $((\cT,\bfmu),M)$ with respect to $\Phi$ by the pair $((\cT_\Phi,\bfmu_\Phi),M_\Phi)$ of
\[
((\cT_{\Phi},\mu_\Phi),M_\Phi) \coloneqq (\cdots ((((\cT,\mu),M)_{\ruling_{v_1}})_{\ruling_{v_2}})\cdots)_{\ruling_{v_k}}
\]
where $V=\{v_1,\dots, v_k\}$. 
\end{definition}

\subsection{Rulings for bordered Legendrian graphs}

Let $(\cT,\bfmu)\in\BLT_\front^\mu$ be a bordered Legendrian graph.
For any subset $N\subset C(T)$ of crossings, the \emph{$0$-resolution} of $\cT$ along $N$ is the bordered Legendrian graph $\cT_N=(T_\Left \to T_N\leftarrow T_\Right)$, where $T_N$ is obtained by replacing each crossing in $N$ with the $0$-resolution as follows:
\[
\crossing \mapsto \crossinghorizontal.
\]
Notice that when the crossings in $N$ are of degree $0$, then the Maslov potential $\bfmu$ induces the Maslov potential $\bfmu_N$ on $\cT_N$.

\begin{definition}[Graded normal rulings of bordered Legendrian links with markings]
Let 
\[
((\cT,\bfmu),M)=\left(((T_\Left,\mu_\Left),\emptyset)\stackrel{i_\Left}\longrightarrow ((T,\mu),M)\stackrel{i_\Right}\longleftarrow ((T_\Right,\mu_\Right),\emptyset)\right)
\] 
be a bordered Legendrian \emph{link} with the set $M$ of markings.
For $\ruling_\Left\in \NR(T_\Left,{\mu_\Left})$ and $\ruling_\Right\in\NR(T_\Right,{\mu_\Right})$, a \emph{($\ZZ$-graded) normal ruling} $\ruling$ of $((\cT,\bfmu),M)$ with respect to a boundary condition $(\ruling_\Left,\ruling_\Right)$ is a pair $(N,S_N)$ consisting of 
\begin{itemize}
\item a subset $N\subset C(T)\setminus M$ of crossings except \emph{markings}, and 
\item a \emph{decomposition (or ruling surface)} $S_N$ of the $0$-resolution $\cT_N$
\end{itemize}
satisfying the following:
\begin{enumerate}
\item Each $c\in N$ is of degree $0$.
\item The decomposition $S_N$ of the $0$-resolution $\cT_N$ consists of \emph{eyes, left half-eyes, right half-eyes,} and \emph{parallels} as follows:
\begin{align*}
\vcenter{\hbox{\includegraphics{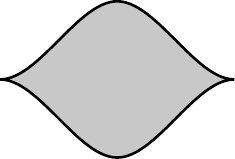}}}\qquad
\vcenter{\hbox{\includegraphics{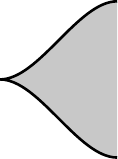}}}\qquad
\vcenter{\hbox{\includegraphics{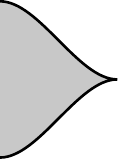}}}\qquad
\vcenter{\hbox{\includegraphics{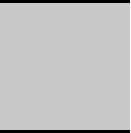}}}
\end{align*}
\item At each $c\in N$, the decomposition $S_N$ satisfies a \emph{non-interlacing} condition, i.e.,
the only following decompositions are allowed near $c\in N$:
\begin{align*}
\begin{tikzpicture}[baseline=-.5ex]
\begin{scope}
\fill[red,opacity=0.3] (-1,1)--(1,1) -- (1,0.5) -- (0,0) -- (-1,0.5)--cycle;
\fill[blue,opacity=0.3] (-1,-1)--(1,-1) -- (1,-0.5) -- (0,0) -- (-1,-0.5)--cycle;
\draw[thick,red] (-1,1)--(1,1);
\draw[thick,red] (-1,0.5)--(0,0)--(1,0.5);
\draw[thick,blue] (-1,-0.5)--(0,0) -- (1,-0.5);
\draw[thick,blue] (-1,-1)--(1,-1);
\end{scope}
\begin{scope}[xshift=3cm]
\fill[red,opacity=0.3] (-1,-1)--(1,-1) -- (1,1) -- (0,0.5) -- (-1,1)--cycle;
\fill[blue,opacity=0.3] (-1,-0.5)--(1,-0.5) -- (1,0) -- (0,0.5) -- (-1,0)--cycle;
\draw[thick,red] (-1,-1)--(1,-1);
\draw[thick,red] (-1,1)--(0,0.5)--(1,1);
\draw[thick,blue] (-1,0)--(0,0.5) -- (1,0);
\draw[thick,blue] (-1,-0.5)--(1,-0.5);
\end{scope}
\begin{scope}[xshift=6cm]
\fill[red,opacity=0.3] (-1,1)--(1,1) -- (1,-1) -- (0,-0.5) -- (-1,-1)--cycle;
\fill[blue,opacity=0.3] (-1,0.5)--(1,0.5) -- (1,0) -- (0,-0.5) -- (-1,0)--cycle;
\draw[thick,red] (-1,1)--(1,1);
\draw[thick,red] (-1,-1)--(0,-0.5)--(1,-1);
\draw[thick,blue] (-1,0)--(0,-0.5) -- (1,0);
\draw[thick,blue] (-1,0.5)--(1,0.5);
\end{scope}
\end{tikzpicture}
\end{align*}

\item The induced involutions $i^*_\Left(S_N)$ and $i^*_\Right(S_N)$ at left and right borders from $S_N$ will coincide with $\ruling_\Left$ and $\ruling_\Right$, respectively.
\begin{align*}
i^*_\Left(S_N)&=\ruling_\Left, &
i^*_\Right(S_N)&=\ruling_\Right.
\end{align*}
\end{enumerate}

We denote the set of normal ruling of $((\cT,\bfmu),M)$ with respect to $(\ruling_\Left,\ruling_\Right)$ by $\bfR((\cT,\bfmu),M;\ruling_\Left,\ruling_\Right)$.
\end{definition}

\begin{definition}[Graded normal rulings of bordered Legendrian graphs]\label{def:normal ruling LG}
Let $(\cT,\bfmu)\in\BLT_\front^\mu$ be a bordered Legendrian graph and $\cP(\cT,\bfmu)$ be the set of full resolutions.
For $\ruling_\Left\in\NR(T_\Left,{\mu_\Left})$ and $\ruling_\Right\in\NR(T_\Right,{\mu_\Right})$, we define the set of $\grading$-graded normal ruling of $((\cT,\bfmu),\emptyset)$ with a boundary condition $(\ruling_\Left,\ruling_\Right)$ by
\begin{align*}
\bfR(\cT,\bfmu;\ruling_\Left,\ruling_\Right) \coloneqq \coprod_{\Phi\in \cP(\cT,\bfmu)}\bfR((\cT_\Phi,\bfmu_\Phi),M_\Phi;\ruling_\Left,\ruling_\Right),
\end{align*}
where $((\cT_\Phi,\bfmu_\Phi),M_\Phi)=((\cT,\bfmu),\emptyset)_\Phi$ is the result of the full resolution defined in Definition~\ref{def:full_resolution_marking}.
\end{definition}

\begin{proposition}[Gluing property of normal rulings]\label{prop:sheaf of rulings}
Let $(\cT^1,{\bfmu^1})$ and $(\cT^2,{\bfmu^2})$ be bordered Legendrian graphs of type$\left(n^i_\Left,n^i_\Right\right)$ and $i=1,2$.
Suppose that $n\coloneqq n^1_\Right=n^2_\Left$ and $\mu \coloneqq \mu^1_\Right=\mu^2_\Left$ on $[n]$, then we have the following pull-back diagram: for $(\cT,\bfmu)=(\cT^1,\bfmu^1)\cdot (\cT^2,\bfmu^2)$,
\[
\begin{tikzcd}[column sep=3pc, row sep=2pc]
\bfR \left(\cT,\bfmu;\ruling^1_{\Left},\ruling^2_{\Right}\right) \arrow[r,dashed] \arrow[d,dashed] & \bfR \left(\cT^2,\bfmu^2;\ruling^2_{\Left},\ruling^2_{\Right}\right) \ar[d,"i_\Left^*"] \\
\bfR \left(\cT^1,\bfmu^1;\ruling^1_{\Left},\ruling^1_{\Right}\right) \ar[r,"i_\Right^*"] & \NR(T_n,{\mu})
\end{tikzcd}
\]
\end{proposition}
\begin{proof}
For each $S_{N^1} \in \bfR(\cT^1,{\bfmu^1};\ruling^1_{\Left},\ruling^1_{\Right})$ and $S_{N^2} \in \bfR(\cT^2,{\bfmu^2};\ruling^2_{\Left},\ruling^2_{\Right})$, we can glue them in a canonical way if and only if $\ruling^1_\Right=\ruling^2_\Left$ in $\NR(T_n,\mu)$.
The maps from $\bfR(\cT,\bfmu;\ruling^1_{\Left},\ruling^2_{\Right})$ are defined by restrictions and the square satisfies the universal property.
\end{proof}

\begin{definition}[Ruling polynomial]\label{definition:ruling polynomial}
Let $(\cT,\bfmu)\in\BLT_\front^\mu$ and $\ruling_\Left$ and $\ruling_\Right$ be as before.
For each $\Phi=(\ruling_v)_{v\in V}\in\cP(\cT,\bfmu)$ and normal ruling $\ruling \in\bfR((\cT_\Phi,\bfmu_\Phi),M_\Phi;\ruling_\Left,\ruling_\Right)$, let us define 
\[
A(\ruling) \coloneqq \sum_{v\in V}A_v(\ruling_v)\quad\text{ and }\quad
\chi(\ruling) \coloneqq \chi(S_N)-\chi(\iota_\Right^* S_N)
\]
and $A_v(\ruling_v)$ is defined in Lemma~\ref{lem:stratification of aug var at a vertex}.

The \emph{weight} $w(\ruling)$ of $\ruling=(N,S_N)$ is defined as
\begin{align}\label{eqn:weight for ruling}
w(\ruling)\coloneqq q^{\frac{A(\ruling)}2}z^{-\chi(\ruling)}
\end{align}
and the \emph{($\ZZ$-graded) ruling polynomial} $\langle \ruling_\Left|R(\cT,\bfmu;q,z)|\ruling_\Right\rangle$ and $R(\cT,\bfmu;q,z)$ are defined by the weight sum as follows:
\begin{align*}
\langle \ruling_\Left|R(\cT,\bfmu;q,z)|\ruling_\Right\rangle &\coloneqq \sum_{\ruling} w(\ruling),& 
R(\cT,\bfmu;q,z)&\coloneqq \sum_{(\ruling_\Left,\ruling_\Right)}\langle \ruling_\Left|R(\cT,\bfmu;q,z)|\ruling_\Right\rangle
\end{align*}
where the sums run over all $\ruling\in\bfR(\cT,\bfmu;\ruling_\Left,\ruling_\Right)$ and all $(\ruling_\Left,\ruling_\Right)\in \NR(T_\Left,{\mu_\Left})\times \NR(T_\Right,{\mu_\Right})$, respectively.
\end{definition}

\begin{example}\label{ex:ruling polynomial}
Let us compute ruling polynomial for the Legendrian graph in Figure~\ref{fig:six_valent_example}.
For each $\ruling_v^i\in \NR(v)$, by direct computation from the definition in Lemma~\ref{lem:stratification of aug var at a vertex}, we have $A_v({\ruling_v^i})$ as follows:
\begin{align*}
A_v({\ruling_v^1})&=6,& A_v({\ruling_v^2})&=5,& A_v({\ruling_v^3})&=5,&
A_v({\ruling_v^4})&=3,& A_v({\ruling_v^5})&=4,& A_v({\ruling_v^6})&=4.
\end{align*}
The corresponding ruling polynomial for each resolutions are
\begin{align*}
R(\Lambda_{\ruling_v^1};q,z)&=q^{3}(1+z^{-2}),&  R(\Lambda_{\ruling_v^2};q,z)&=q^{5/2} z^{-1},& R(\Lambda_{\ruling_v^3};q,z)&=q^{5/2} z^{-1};\\
R(\Lambda_{\ruling_v^4};q,z)&=q^{3/2}z^{-3},& R(\Lambda_{\ruling_v^5};q,z)&=q^{2}z^{-2},& R(\Lambda_{\ruling_v^6};q,z)&=q^{2}z^{-2}.
\end{align*}
Hence the ruling polynomial for the Legendrian graph $\Lambda$ becomes
\[
R(\Lambda;q,z)=q^{3}(1+z^{-2})+2q^{5/2} z^{-1} + 2q^{2} z^{-2} + q^{3/2}z^{-3}.
\]
\end{example}

\begin{corollary}\label{cor:gluing property of ruling polynomials}
Let $(\cT,\bfmu)=(\cT^1,\bfmu^1) \cdot (\cT^2,\bfmu^2)$ for $(\cT^i,\bfmu^i)$ of type $\left(n^i_\Left,n^i_\Right\right)$ and $i=1,2$.
Then
\begin{align}\label{eqn:ruling formula for concatenation}
\langle \ruling_\Left|R(\cT,\bfmu;q,z)|\ruling_\Right\rangle=\sum_{\ruling \in \NR(T_n,\mu_n)} \langle \ruling_\Left|R(\cT^1,\bfmu^1;q,z)|\ruling\rangle \cdot \langle \ruling|R(\cT^2,\bfmu^2;q,z)|\ruling_\Right\rangle,
\end{align}
where $n=n^1_\Right=n^2_\Left$ and $\mu=\mu^1_\Right=\mu^2_\Left$.
\end{corollary}
\begin{proof}
This is a direct consequence from the gluing property of the ruling in Proposition~\ref{prop:sheaf of rulings} and the definition of the ruling polynomial in Definition~\ref{definition:ruling polynomial}.
Note that the weight data $A(S_N)$ and $\bar\chi(S_N)$ is additive under the gluing of ruling surfaces along the borders.
\end{proof}

\begin{example}\label{ex:ruling_vertex}
Recall the Legendrian graph $\bfzero_n$ with $n$ even and let $\bfmu$ be a Maslov potential of $\bfzero_n$.
Then for any $\ruling\in \NR(T_n,\bfmu)$, it defines a unique graded normal ruling $S_N=S_N(\ruling)$ consisting of $\frac n2$ half-eyes.
Therefore its Euler characteristic $\chi(S_N)$ is exactly $\frac n2$ and the weight $w(S_N)=z^{\frac n2-\chi(S_N)}=1$ for any $\ruling$.
Hence the total ruling polynomial $R(\bfzero_n,\bfmu)$ is the same as $\#\NR(T_n,\bfmu)$.
\end{example}

\begin{theorem}[Invariance]\cite{ABK2019}\label{thm:invariance_ruling}
Let $(\cT,\bfmu)\in\BLT_\front^\mu$ and $(\ruling_\Left, \ruling_\Right)\in \NR(T_\Left,\mu_\Left)\times \NR(T_\Right,\mu_\Right)$ be as before.
Then the set of normal rulings $\bfR(\cT,\bfmu;\ruling_\Left, \ruling_\Right)$ is a Legendrian isotopy invariant up to bijections preserving $\chi(\ruling)$ and $A_v(\ruling)$.

In particular, the polynomial $\langle \ruling_\Left|R(\cT,\bfmu;q,z)|\ruling_\Right\rangle$ is an invariant.
\end{theorem}
\begin{proof}
Note that for bordered Legendrian links, the invariance has been already given in \cite[Lemma~2.9]{Su2017}.

Let $(\cT,\bfmu)$ and $(\cT',\bfmu')$ be two bordered Legendrian graphs which differ by a Reidemeister move.
Take boundary conditions $(\ruling_\Left,\ruling_\Right) \in \NR(T_\Left,\mu_\Left)\times \NR(T_\Right,\mu_\Right)$. 
In \cite{ABK2019}, there is a bijection between sets of rulings $\bfR(\cT,\bfmu; \ruling_\Left,\ruling_\Right)$ and $\bfR(\cT,\bfmu;\ruling_\Left,\ruling_\Right)$ preserving the Euler characteristic $\chi(\ruling)$. So it remains to show that $A(\ruling)$ is preserved. It is direct from the bijection in \cite{ABK2019} and $A(\ruling)$ only depends on data $\ruling$ at vertices not switches. The weight $z^{\chi(\ruling)} q^{A(\ruling)/2}$ for each $\ruling$ is preserved under that bijection. This proves the invariance theorem.
\end{proof}

\subsection{Augmentation numbers are ruling polynomials}
Let $(\cT,\bfmu)\in \BLT^\mu_\front$ be a bordered Legendrian graph of type $(n_\Left,n_\Right)$ contained in $[x_\Left,x_\Right]\times\RR_z$ and let $\ruling_\Left\in\NR(T_\Left,\mu_\Left)$ and $\ruling_\Right\in\NR(T_\Right,\mu_\Right)$ be fixed.

\begin{definition}[Elementary bordered Legendrian graphs]\label{definition:elementary bordered Legendrian graphs}
An \emph{elementary} bordered Legendrian graph $(\cE=(E_\Left\to E\leftarrow E_\Right),\bftau)\in \BLT^\mu_\front$ contains several horizontal strands and either (i) only one left cusp, (ii) only one right cusp with a base point, (iii) only one crossing or (iv) only one vertex $v$.
\end{definition}

\begin{assumption}
For the simplicity of the argument, throughout this section, let us assume that
\begin{enumerate}
\item $n_\Left$, $n_\Right$, $\val(v)$ for $v\in V(T)$ are all even;
\item each right cusp has a base point;
\item each vertex $v\in V(T)$ is of type $(0,\val(v))$;
\item each cusp, crossing and vertex has a different $x$-coordinate.
\end{enumerate}
\end{assumption}

\begin{lemma}
Let $(\cE,\bftau)$ be an elementary bordered Legendrian graph satisfying the above assumption and $(\sigma_\Left,\sigma_\Right)\in\NR(E_\Left,\tau_\Left)\times \GNR(E_\Right,\tau_\Right)$.
Then the set $\bfR(\cE,\bftau;\sigma_\Left,\sigma_\Right)$ of normal rulings with boundary conditions is either empty or consisting of a unique normal ruling $\sigma$.

In particular, if $\sigma_\Right\in \GNR(E_\Right,\tau_\Right)\setminus \NR(E_\Right,\tau_\Right)$, then $\bfR(\cE,\bftau;\sigma_\Left,\sigma_\Right)$ is empty.
\end{lemma}
\begin{proof}
Suppose that $\bfR(\cE,\bftau;\sigma_\Left,\sigma_\Right)$ is nonempty and contains $\sigma$.
Then it can be easily checked that $\sigma$ is uniquely determined by $\sigma_\Left$ and $\sigma_\Right$ for the cases (i), (ii) and (iii) in Definition~\ref{definition:elementary bordered Legendrian graphs}. Finally, when $\cE$ contains a vertex of type $(0,\val(v))$, $\sigma$ is again completely determined by boundary conditions since so is the resolution of the vertex $v$.
\end{proof}

\begin{remark}\label{remark:elementary with vertex of general types}
One can consider the bordered Legendrian graph $(\cE,\bftau)$ containing only one vertex of type $(\ell,r)$.
In this case, $\bfR(\cE,\bftau;\sigma_\Left,\sigma_\Right)$ may have several rulings.
\end{remark}

Due to the fourth assumption above, we may cut $(\cT,\bfmu)$ into elementary pieces
\begin{align*}
(\cE_i=(E_{i,\Left}\to E_i\leftarrow E_{i,\Right}),\bftau_i)&\coloneqq (\cT,\bfmu)|_{[x_{i-1},x_i]\times\RR_z},&
(\cT,\bfmu)&=(\cE_1,\bftau_1)\cdot\cdots\cdot(\cE_m,\bftau_m)
\end{align*}
along the vertical lines $\{x_1,\dots, x_{m-1}\}\times\RR_z$ for some $x_\Left=x_0<x_1<\dots<x_{m-1}<x_m=x_\Right$, where each vertical line contains no cusps, crossings, base points and vertices.

\begin{definition}\cite[Definition 2.4]{Su2017}\label{def:returns and departures}
Let $(\cT,\bfmu)\in \BLT^\mu_\front$ be a bordered Legendrian graph and $\ruling=(N,S_N)\in \bfR(\cT,\bfmu;\ruling_\Left, \ruling_\Right)$ be a normal ruling. A crossing $c$ is called a {\em return} (resp. {\em departure}) of $\ruling$ if the behavior of the ruling surface $S_N$ is one of (resp. the vertical reflection of) the following:
\begin{align*}
\begin{tikzpicture}[baseline=-.5ex]
\begin{scope}
\fill[red,opacity=0.3] (-1,1)--(1,1) -- (1,0.5) -- (0,0) -- (-1,-0.5)--cycle;
\fill[blue,opacity=0.3] (-1,-1)--(1,-1) -- (1,-0.5) -- (0,0) -- (-1,0.5)--cycle;
\draw[thick,red] (-1,1)--(1,1);
\draw[thick,red] (-1,-0.5)--(0,0)--(1,0.5);
\draw[thick,blue] (-1,0.5)--(0,0) -- (1,-0.5);
\draw[thick,blue] (-1,-1)--(1,-1);
\draw(0,0) node[below] {$c$};
\end{scope}
\begin{scope}[xshift=3cm]
\fill[red,opacity=0.3] (-1,-1)--(-1,0.5) -- (1,0.5) -- (1,0) --cycle;
\fill[blue,opacity=0.3] (-1,1)--(1,1) -- (1,-1) -- (-1,0)--cycle;
\draw[thick,red] (-1,0.5)--(1,0.5);
\draw[thick,red] (-1,-1)--(1,0);
\draw[thick,blue] (-1,1)--(1,1);
\draw[thick,blue] (-1,0)--(1,-1);
\draw(0,-0.5) node[below] {$c$};
\end{scope}
\begin{scope}[xshift=6cm]
\fill[red,opacity=0.3] (-1,1)--(1,0) -- (1,-0.5) -- (-1,-0.5)--cycle;
\fill[blue,opacity=0.3] (-1,0)--(1,1) -- (1,-1) -- (-1,-1)--cycle;
\draw[thick,blue] (-1,0) -- (1,1);
\draw[thick,blue] (-1,-1)--(1,-1);
\draw[thick,red] (-1,1)--(1,0);
\draw[thick,red] (-1,-0.5)--(1,-0.5);
\draw(0,0.5) node[below] {$c$};
\end{scope}
\end{tikzpicture}
\end{align*}
Moreover, returns (resp. departures) of degree $0$ are called {\em graded returns} (resp. {\em graded departures}) of the ruling $\ruling$. Let us denote the number of graded returns (resp. graded departures) by $r(\ruling)$ (resp. $d(\ruling)$).
\end{definition}

\begin{lemma}\label{lem:aug var elementary}
Let $(\cE,\bftau)\in \BLT^\mu_\front$ be an elementary bordered Legendrian graph and $(\sigma_\Left,\sigma_\Right)\in\NR(E_\Left,\tau_\Left)\times\GNR(E_\Right,\tau_\Right)$.
Then $\bfR(\cE,\bftau;\sigma_\Left,\sigma_\Right)$ is non-empty if and only if so is $\aug(\cE,\bftau;\epsilon_\Left, \sigma_\Right;\field)$ for any $\epsilon_\Left \in \aug^{\sigma_\Left}(E_\Left,\tau_\Left;\field)$. In this case, we have a unique ruling $\sigma\in\bfR(\cE,\bftau;\sigma_\Left,\sigma_\Right)$ and
\[
\aug(\cE,\bftau;\epsilon_\Left, \sigma_\Right;\field) \isomorphic (\field^\times)^{-\chi(\sigma)+\hat B}\times \field^{r(\sigma)+A(\sigma)},
\]
where 
\[
\hat B=\#B(E) + \sum_{v\in V(E)}\frac{\val(v)}{2}
\]
and $B(E)$ and $V(E)$ are the sets of base points and vertices of $E$.

In particular, it is independent of the choice of $\epsilon_\Left$.
\end{lemma}
\begin{proof}
For elementary bordered Legendrian graphs containing left or right cusps, or a crossing, we refer to \cite[Lemma 4.15]{Su2017} and we will focus on the case when $\cE$ is of type $(n, n+\val(v))$ and contains a vertex of type $(0,\val(v))$ which looks as follows:
\[
(\cE,\bftau)=
\begin{tikzpicture}[baseline=-.5ex]
\draw[fill](-0.5,0) circle (0.05) node[below] {$v$};
\foreach \i in {0.75, 0.25, -0.25, -0.75} {
\draw[thick, rounded corners](-0.5,0) to[out=0,in=180] (1,\i);
}
\foreach \i in {1.25, 1.75, -1.25, -1.75} {
\draw[thick](-1,\i) -- (1,\i);
}
\draw[red](1,-2)--(1,2);
\draw[red](-1,-2)--(-1,2);
\draw (1,1.8) node[right] {$1$};
\draw (-1,1.8) node[left] {$1$};
\draw (1,1.3) node[right] {$\vdots$};
\draw (1,0.8) node[right] {$j$};
\draw (1,0.1) node[right] {$\vdots$};
\draw (1,-0.8) node[right] {$j+\val(v)$};
\draw (1,-1.3) node[right] {$\vdots$};
\draw (1,-1.8) node[right] {$n+\val(v)$};
\draw (-1,-1.8) node[left] {$n$};
\end{tikzpicture}
\]

Let $\sigma_v\in\GNR(v)\coloneqq\GNR(v_\Right)$ be the restriction of $\sigma_\Right$ on $[\val(v)]\isomorphic\{j,\dots, j+\val(v)\}$.
Since the LCH DGA $A^\CE(E,\tau)\isomorphic A_\Left \coprod A_v$ is decomposed into two parts
\[
A_\Left \coloneqq \ring \langle k_{ab} \mid 1\leq a <b \leq n \rangle\quad\text{ and }\quad
A_v \coloneqq \ring \langle v_{a,i} \mid a\in \Zmod{\val(v)}, i>0 \rangle
\]
and the augmentations for generators $k_{ab}$ are given by $\epsilon_\Left$, we have the induced isomorphism 
\[
\aug(E;\epsilon_\Left,\sigma_\Right;\field)\stackrel{\isomorphic}\longrightarrow \aug^{\sigma_v}(v;\field).
\]
Therefore $\aug(E;\epsilon_\Left,\sigma_\Right;\field)$ is non-empty if and only if so is $\aug^{\sigma_v}(v;\field)$, i.e., $\sigma_v\in\NR(v)$ and 
\begin{align*}
\aug^{\sigma_v}(v;\field)\isomorphic (\field^\times)^{\val(v)/2}\times \field^{A_v(\sigma_v)}. \tag{Lemma~\ref{lem:stratification of aug var at a vertex}}
\end{align*}
Moreover, the ruling $\sigma_\Left$ and involution $\sigma_v$ determine a ruling $\sigma$ for $(\cE,\bftau)$ and \textit{vice versa}, we are done.
\end{proof}

\begin{remark}\label{rem:ruling decomposition}
Though it is not trivial and need some more work, one can show a natural generalization of Lemma \ref{lem:aug var elementary} for the case when a vertex $v$ is of any type, as shown in Proposition \ref{prop:ruling decomposition}.

Let $(\cE,\bftau)$ be an elementary bordered Legendrian graph with a single vertex $v$ of type $(\ell,r)$, then for any $\epsilon_\Left\in\aug^{\ruling_\Left}(E_\Left,\tau_\Left;\field)$, there is a natural decomposition via locally closed sub-varieties
\begin{equation*}
\aug(\cE,\bftau;\epsilon_\Left,\ruling_\Right;\field)=\coprod_{\ruling\in\bfR(\cE,\bftau;\ruling_\Left,\ruling_\Right)}\aug^{\ruling}(\cE,\bftau;\epsilon_\Left,\ruling_\Right;\field).
\end{equation*} 
In addition, 
\begin{equation*}
\aug^{\ruling}(\cE,\bftau;\epsilon_\Left,\ruling_\Right;\field)\isomorphic (\field^\times)^{-\chi(\ruling)+\hat{B}}\times\field^{\hat{r}(\ruling)+A(\ruling)}
\end{equation*}
where $\hat{r}(\ruling)$ is defined in Proposition \ref{prop:ruling decomposition}.

The above statement comes up naturally by purely working with augmentations, hence provides the geometric intuition for defining normal rulings and ruling polynomials for (bordered) Legendrian graphs as seen before.
\end{remark}

\begin{remark}\label{rem:Hodge str}
By Theorem \ref{thm:inv of aug var}, the mixed Hodge structure of the compactly supported cohomology $H_c^*(\aug(\cT,\bfmu;\epsilon_\Left,\ruling_\Right;\mathbb{C}),\mathbb{Q})$ is also a Legendrian isotopy invariant up to a normalization and generalizes augmentation numbers.
By a spectral sequence argument as in \cite{Su2018}, one can also show that $\aug(\cT,\bfmu;\epsilon_\Left,\ruling_\Right;\mathbb{C})$ is of Hodge type. This is similar to the Betti moduli space in non-abelian Hodge theory, which is well-known to have this property.
\end{remark}

\begin{definition}\label{def:ruling strata}
We denote for each $1\le i\le m$ 
\begin{align*}
(\cT_i=(T_{i,\Left}\to T_i\leftarrow T_{i,\Right}),\bfmu_i)&\coloneqq (\cE_1,\bftau_1)\cdot(\cE_2,\bftau_2)\cdot\quad \cdots\quad \cdot(\cE_i,\bftau_i).
\end{align*}

For $\ruling\in\bfR(\cT,\bfmu;\ruling_\Left,\ruling_\Right)$, we define the varieties
\begin{align*}
\aug^\ruling(\cT,\bfmu;\epsilon_L;\field) &\coloneqq \aug_1(\epsilon_L,\ruling_1)\times_{\aug^{\ruling_1}_1}\cdots \times_{\aug^{\ruling_{m-1}}_{m-1}} \aug_m(\ruling_{m-1},\ruling_\Right);\\
\aug^\ruling(\cT,\bfmu;\ruling_L;\field) &\coloneqq \aug_1(\ruling_L,\ruling_1)\times_{\aug^{\ruling_1}_1}\cdots \times_{\aug^{\ruling_{m-1}}_{m-1}} \aug_m(\ruling_{m-1},\ruling_\Right),
\end{align*}
where $\ruling_i\coloneqq \ruling|_{T_{i,\Right}}$ and
\[
\aug_i(-,-)\coloneqq\aug(\cE_i,\bfmu_i;-,-;\field)\quad\text{ and }\quad
\aug^{(-)}_i\coloneqq\aug^{(-)}(T_{i,\Right},\mu_{i,\Right};\field).
\]
\end{definition}

As a consequence of Lemma~\ref{lem:aug var elementary}, we obtain a partition of the augmentation variety as follows:
\begin{equation}\label{equation:partition of augmentation variety}
\aug(\cT,\bfmu;\epsilon_\Left,\ruling_R;\field)=\coprod_{\ruling\in\bfR(\cT,\bfmu;\ruling_\Left,\ruling_\Right)} \aug^\ruling(\cT,\bfmu;\epsilon_\Left;\field).
\end{equation}

\begin{remark}\label{rem:acyclic}
For a vertex of type $(\ell,r)$, we will use the following fact shown in Corollary \ref{cor:acyclic}:
For any augmentation $\epsilon\in\aug(\cT,\bfmu;\field)$, if $\epsilon_\Left\coloneqq \phi_\Left^*(\epsilon)\in\aug(T_\Left,\mu_\Left;\field)$ is acyclic, then $\epsilon_\Right\coloneqq \phi_\Right^*(\epsilon)$ is acyclic as well.
\end{remark}

\begin{lemma}\label{lem:augnumber}
For any $\epsilon_\Left \in \aug^{\ruling_\Left}(T_\Left,\mu_\Left;\field)$, we have
\[
\dim_\field \aug(\cT,\bfmu;\epsilon_\Left,\ruling_\Right;\field)= \max_{\ruling} \{ -\chi(\ruling) + \hat B + r(\ruling) + A(\ruling)\}
\]
and therefore
\[
\augnumber(T;\ruling_\Left, \ruling_\Right;\FF_q)=q^{-\max_\ruling\{-\chi(\ruling) + \hat B + r(\ruling) + A(\ruling)\}}\sum_{\ruling}(q-1)^{-\chi(\ruling)+\hat B}q^{r(\ruling)+A(\ruling)}.
\]
Here $\ruling$ runs over the set $\bfR(\cT,\bfmu;\ruling_\Left,\ruling_\Right)$.
\end{lemma}
\begin{proof}
First consider the natural projection
\begin{align*}
P_m: \aug^\ruling(\cT_m,\bfmu_m;\epsilon_L;\field) \to \aug^{\ruling|_{T_{m-1}}}(\cT_{m-1},\bfmu_{m-1};\epsilon_\Left;\field)
\end{align*}
whose fiber is $\aug(\cE_m,\bftau_m;\epsilon_{m-1},\ruling_m)$ for each $\epsilon_{m-1}\in \aug^{\ruling_{m-1}}(T_{m-1,\Right},\mu_{m-1,\Right};\field)$.

Lemma~\ref{lem:aug var elementary} implies that the fibers are independent of the choice of $\epsilon_{m-1}$ and the following formulas:
\begin{align*}
\dim_\field \aug^\ruling(\cT_m,\bfmu_m;\epsilon_\Left;\field) &= \dim_\field \aug^{\ruling|_{T_{m-1}}}(\cT_{m-1},\bfmu_{m-1};\epsilon_\Left;\field)\\
&\mathrel{\hphantom{=}}- \chi(\ruling|_{E_n}) + \hat B(E_n) + r(\ruling|_{E_n})+ A(\ruling|_{E_n});\\
\#\aug^\ruling(\cT_m,\bfmu_m;\epsilon_\Left;\FF_q) &= (q-1)^{- \chi(\ruling|_{E_m})+ \hat B(E_m)} q^{ r(\ruling|_{E_m})+ A(\ruling|_{E_m})}\#\aug^{\ruling|_{T_{m-1}}}(\cT_{m-1},\bfmu_{m-1};\epsilon_\Left;\FF_q).
\end{align*}

By the inductive process, we obtain
\begin{align*}
\dim_\field \aug^\ruling(\cT,\bfmu;\epsilon_\Left;\field)&=-\chi(\ruling) +\hat B + r(\ruling) + A(\ruling)\quad\text{ and }\quad\\
\#\aug^\ruling(\cT,\bfmu;\epsilon_\Left;\field)&= (q-1)^{-\chi(\ruling)+\hat B} q^{r(\ruling)+A(\ruling)}
\end{align*}
since all indices $\chi, \hat B, r$ and $A$ are additive under the concatenation.
Then the conclusion directly follows from the partition of the augmentation variety given in \eqref{equation:partition of augmentation variety}.
\end{proof}

\begin{corollary}\label{corollary:dimension of augmentation variety}
The following holds:
\[
\dim_\field\aug(\cT,\bfmu;\epsilon_\Left,\ruling_\Right)=\max_{\ruling\in\bfR(\cT,\bfmu;\ruling_\Left,\ruling_\Right)}
\{-\chi(\ruling) +\hat B + r(\ruling) + A(\ruling)\}.
\]
\end{corollary}
\begin{proof}
This is a direct consequence of the decomposition given in \eqref{equation:partition of augmentation variety} and Lemma~\ref{lem:augnumber}.
\end{proof}

Generalizing the result in \cite{HR2015, Su2017}, we need the following lemma to prove the main theorem.

\begin{lemma}\label{lem:inv quantity}
Suppose that each vertex $v$ is at the bottom of the vertical line that contains $v$.
For any $\ruling\in\bfR(\cT,\bfmu;\ruling_\Left,\ruling_\Right)$, the value
\begin{align*}
-\chi(\ruling)+2 r(\ruling)+A(\ruling)
\end{align*} 
is a constant and will be denoted by $D=D(\cT,\bfmu;\ruling_\Left,\ruling_\Right)$.
\end{lemma}
\begin{proof}
Let us denote 
\[
c_\Right \coloneqq \#(\text{right cusps of }T)\quad\text{ and }\quad
s(\ruling) \coloneqq \#(\text{switches of }\ruling).
\]
Recall that $d(\ruling)$ is the number of graded departures of $\ruling$ from Definition~\ref{def:returns and departures}.
Then $\chi(\ruling)=c_\Right - s(\ruling)$ and hence
\begin{align*}
-\chi(\ruling)+2 r(\ruling)+A(\ruling)=\big(s(\ruling)+ r(\ruling) + d(\ruling)- c_\Right \big) + r(\ruling) + A(\ruling) - d(\ruling).
\end{align*}
Note that $s(\ruling)+ r(\ruling) + d(\ruling)$ is the number of crossings of the front diagram $T$ and $c_\Right$ also depends only on the front diagram. So it remains to show that $r(\ruling) + A(\ruling) - d(\ruling)$ is independent of the choice of ruling.

Assume that $T\subset [x_0,x_1]\times \RR_z$, and take a normal ruling $\ruling\in \bfR(\cT,\bfmu;\ruling_\Left,\ruling_\Right)$.
For each $x\in [x_0,x_1]$ missing the crossings, cusps, and vertices, consider the restriction 
\[
\ruling_x \coloneqq \ruling|_{\{x\}\times \RR_z}\in \NR(T|_{\{x\}\times \RR_z},\mu|_{\{x\}\times \RR_z}),
\]
and $ A(x)\coloneqq A_b(\ruling_x)$ as in Definition \ref{def:indices for an involution at a trivial tangle}.
Note that $ A(x_0)=A_b(\ruling_\Left)$ and $A(x_1)=A_b(\ruling_\Right)$.

Let us look at the behavior of $A(x)$ by increasing $x\in[x_0,x_1]$.
It is direct from the definition of $A(x)$ that it increases (resp. decreases) by 1 when $x$ passes through a graded return (resp. a graded depature).
There is no change when $x$ passes switches and other crossings.

Now focus the change of $A(x)$ when it passes through a vertex. Denote $q\in[x_0, x_1]$ be a position of a vertex $v$ of type $(0,\val(v))$. 
Recall the corresponding definition $A_v(\ruling)$ from Definition~\ref{def:indices for an involution at a vertex} as follows: 

Let $\mu_v$ be the restriction of $\mu$ to a neighborhood of $v$, then for any $i\in I(v)=[\val(v)]$, we introduce
\begin{equation*}
I_v(i)\coloneqq \{j>0 \mid |e_i|=|Z^{n(v_{i,j})}e_{i+j}|, \text{ i.e.,} \mu_v(i)-\mu_v(i+j)+n(v_{i,j})=0\}.
\end{equation*}
For the induced restriction $\ruling_v \in \NR(v)$, we obtain a partition $I(v)=U_v\amalg L_v$ with a bijection $\ruling_v:U_v\stackrel{\isomorphic}\longrightarrow L_v$ as in Definition \ref{def:canonical augmentation at a vertex}. For any $i\in U_v$, let us define
\begin{align*}
A_{v,\ruling}(i)\coloneqq \{j\in U_v \mid \ruling_v(j)<\ruling_v(i), \textrm{ and } j-i\in I_v(i)\}.
\end{align*}
Now, define $A_v(\ruling)\in\NN$ by 
\begin{align*}
A_v(\ruling)\coloneqq \sum_{i\in U_v}|A_{v,\ruling}(i)|+\sum_{i\in L_v}|I_v(i)|.
\end{align*}

Now compare $(I_v,\ruling_v,A_v)$, $(I_{q\pm\delta},\ruling_{q\pm\delta}, A_{q\pm\delta})$ for sufficiently small $\delta>0$.
Let $T\cap \{{x=q-\delta}\}\isomorphic [n_{q-\delta}]$, then there is a canonical identification 
$\iota_v:I(v)\isomorphic [\val (v)]\hookrightarrow [n_{q-\delta}+\val(v)] \isomorphic [n_{q+\delta}].$
Moreover, $\iota_v$ induces the following commutative diagram:
\[
\begin{tikzcd}[column sep=3pc, row sep=2pc]
U_v \arrow[r,"\iota_v|_{U_v}",hook] \arrow[d,"\ruling_v","\relation"'] & U_{q+\delta} \ar[d,"\ruling_{q+\delta}","\relation"'] \\
L_v \ar[r,"\iota_v|_{L_v}",hook] & L_{q+\delta}
\end{tikzcd}
\]
Note that the image $\iota_v(I(v))$ is consecutive.
For the notational convenience, let us denote
\begin{align*}
I_{q-\delta}&=\{1,\dots,n_{q-\delta}\};\\
I_{q+\delta}&=\{1,\dots,n_{q-\delta}\}\amalg \iota_v(I(v))\isomorphic [n_{q+\delta}].
\end{align*}
Let $I_*(i)$ be defined as in Definition \ref{def:indices for an involution at a trivial tangle}, for $*\in\{q\pm\delta\}$. Then for $i\in[n_*]$ and $*\in\{q\pm\delta ,v\}$, we have
\begin{align*}
I_{q+\delta}(i)=
\begin{cases}
I_{q-\delta}(i)\amalg \iota_v(\{j\in I(v) \mid \mu(j)=\mu(i)\}) &\text{ if }i\leq n_{q-\delta};\\
\left(i+I_v(i-n_{q-\delta})\right)\cap [\val(v)] &\text{ if }i\in\iota_v(I(v)),
\end{cases}
\end{align*}
where $i+S$ means $\{i+s \mid s\in S\}$.
Note that the construction of  $I_{q+\delta}$ from $I_{q-\delta}$ and $I_v$ does not depend on the ruling $\ruling$, depend only on $\mu$ and $\iota_v$. This implies that
\begin{align}\label{eqn:I sum}
\sum_{i\in L_{q+\delta}}|I_{q+\delta}(i)|-\sum_{i\in L_{q-\delta}}|I_{q-\delta}(i)|=\sum_{i\in L_v}|I_v(i)|+C_v
\end{align}
and $C_v=C_v(\mu)$ does not depend on the ruling $\ruling$. 

Note that $\iota_v(I(v))\subset [n]$ has a perfect matching with respect to a given ruling $\ruling_v$.
Then the definition of $A_*(i)$ for $i\in U_*$, $*\in\{q\pm\delta ,v\}$ conclude the following:
For $i\in U_{q+\delta}$,
\begin{align*}
|A_{q+\delta}(i)|=
\begin{dcases}
|A_{q-\delta}(i)| &\text{ if }i\leq n_{q-\delta};\\
|A_v(i-n_{q-\delta})| &\text{ if }i\geq n_{q-\delta},
\end{dcases}
\end{align*}
and we obtain
\begin{align}\label{eqn:A sum}
\sum_{i\in U_{q_\delta}}|A_{q+\delta}(i)|-\sum_{i\in U_{q-\delta}}|A_{q-\delta}(i)|=\sum_{i\in U_v}|A_v(i)|.
\end{align}
By combining (\ref{eqn:I sum}) and (\ref{eqn:A sum}), we have
\[
A_{q+\delta}(\ruling)-A_{q-\delta}(\ruling)=A_v(\ruling)+C_v
\]
Basically we apply same strategy for left and right cusps, let $c$ be a cusp with the $x$-coordinate $p$, then
\[
A_{p+\delta}(\ruling)-A_{p-\delta}(\ruling)=C_c,
\]
where the constant $C_c$ depends only on $(\cT,\bfmu)$, see \cite[Proposition 4.22]{Su2017}.

Now summing up all the equations which come from the crossings, cusps, and vertices. Then we have
\[
A_b(\ruling_\Right)-A_b(\ruling_\Left)=A(x_1)-A(x_0)=r(\ruling)-d(\ruling)+\sum_{v\in V(T)}A_v(\ruling)+\sum_{v\in V(T)}C_v +\sum_{\text{cusps}}C_c.
\]
Since $A_b(\ruling_\Right)-A_b(\ruling_\Left)$ is fixed and $\sum_{v\in V(T)}C_v +\sum_{\text{cusps}}C_c$ independent of $\ruling$, we conclude that $r(\ruling)-d(\ruling)+A(\ruling)$ is independent of the choice of the ruling $\ruling$.
This proves the lemma.
\end{proof}

\begin{remark}\label{rem:inv quantity}
With some more work, one can also show that, for any bordered Legendrian graph $(\cT,\bfmu)$, the index $-\chi(\ruling)+2\hat{r}(\ruling)+A(\ruling)$ is a constant in $\ruling\in\bfR(\cT,\mu;\ruling_\Left,\ruling_\Right)$.
\end{remark}

Now we consider the operations manipulating base points.
There are essentially two operations on base points as depicted in Figure~\ref{figure:base points}.
\begin{remark}
Note that the operation that moves a base point through a crossing below or above can be realized as a sequence of front Reidemeister moves and induces an isomorphism between LCH DGAs.
Therefore their augmentation varieties and augmentation numbers coincide.
See \cite[Lemma 4.21]{Su2017} for detail.
\end{remark}

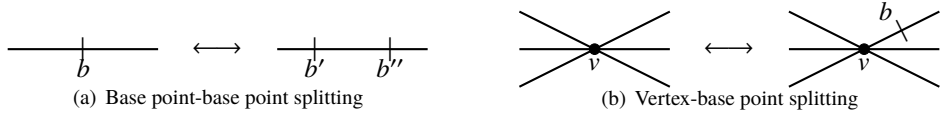
\begin{figure}[ht]
\subfigure[Base point-base point splitting]{\makebox[0.45\textwidth]{$
\begin{tikzpicture}[baseline=-.5ex]
\useasboundingbox(-1,-0.5)--(1,0.5);
\draw[thick] (-1,0)-- (0,0) node{$|$} node[below] {$b$} -- (1,0);
\end{tikzpicture}\quad\longleftrightarrow\quad
\begin{tikzpicture}[baseline=-.5ex]
\draw[thick] (-1,0)-- node[near start]{$|$}node[near start, below]{$b'$} node[near end]{$|$} node[near end,below]{$b''$} (1,0);
\end{tikzpicture}
$}}
\subfigure[Vertex-base point splitting]{\makebox[0.45\textwidth]{$
\begin{tikzpicture}[baseline=-.5ex]
\useasboundingbox(-1,-0.5)--(1,0.5);
\draw[thick] (-1,-0.5) -- (1,0.5);
\draw[thick] (-1,0)-- (1,0);
\draw[thick] (-1,0.5) -- (1,-0.5);
\draw[fill] (0,0) circle (2pt) node[below] {$v$};
\end{tikzpicture}\quad\longleftrightarrow\quad
\begin{tikzpicture}[baseline=-.5ex]
\draw[thick] (-1,-0.5) -- (0,0) -- node[midway,sloped] {$|$} node[midway,above left] {$b$} (1,0.5);
\draw[thick] (-1,0)-- (1,0);
\draw[thick] (-1,0.5) -- (1,-0.5);
\draw[fill] (0,0) circle (2pt) node[below] {$v$};
\end{tikzpicture}
$}}
\caption{Operations on base points}
\label{figure:base points}
\end{figure}

\begin{lemma}\label{lem:base point}
Let $d\coloneqq \max \deg_z \langle\ruling_\Left|R(\cT,\bfmu; z^2,z)| \ruling_\Right\rangle$ and
\begin{align*}
\hat B
&\coloneqq \#(\text{base points in $T$}) +\sum_{v\in V(T)}\frac{val(v)}{2}.
\end{align*}
Then the normalized augmentation number
\[
q^{\frac{d+\hat B}{2}}z^{-\hat B}\augnumber(\cT,\bfmu;\ruling_\Left, \ruling_\Right;\FF_q)\big{|}_{z=(q^{1/2}-q^{-1/2})}
\]
is independent of the number and position of the base points on $T$.

In other words, it is invariant under the operations on base points described in Figure~\ref{figure:base points}.
\end{lemma}
\begin{proof}
Notice that the numbers $d$ and $\sum_v \val(v)/2$ are preserved under the operations on base points.
Therefore we only need to concern $q^{\frac{\#B}2}z^{-\#B}$ for the normalization, where $B$ is the set of base points.
Since both operations increase the number of base points by $1$, the contribution of the additional base point in the normalization is precisely 
\[
q^{\frac 12}z^{-1}=\frac{q}{q-1}.
\]
Therefore it suffices to show that for each operation $(\cT',\bfmu')\to(\cT,\bfmu)$ on base points, 
\[
\aug(\cT,\bfmu;\ruling_\Left,\ruling_\Right;\field)\isomorphic
\aug(\cT',\bfmu';\ruling_\Left,\ruling_\Right;\field)\times\field^\times
\]
which induces the desired equality
\[
\augnumber(\cT,\bfmu;\ruling_\Left,\ruling_\Right;\field)=\frac{q-1}q \aug(\cT',\bfmu';\ruling_\Left,\ruling_\Right;\field).
\]

The base point-base point splitting operation $(\cT',\bfmu')\to(\cT,\bfmu)$ induces a bordered DGA morphism $\bff:\dga^\CE(\cT',\bfmu')\to\dga^\CE(\cT,\bfmu)$ since it is a special case of the tangle replacement defined in \cite{AB2018}.
Notice that two DGAs $A^\CE(T',\mu')$ and $A^\CE(T,\mu)$ have the same generating sets but over the different rings 
\[
\ring[t_b^{\pm1}\mid b\in B']\quad\text{ and }\quad\ring[t_b^{\pm1}\mid b\in B],
\]
where $B=(B'\setminus\{b\})\amalg\{b',b''\}$. 
Indeed, the DGA morphism $f:A^\CE(T',\mu')\to A^\CE(T,\mu)$ maps $t_b$ to $t_{b'}t_{b''}$ and each occurrence of $t_{b'}$ in any differential of $A^\CE(T,\mu)$ comes together with $t_{b''}$ and \textit{vice versa}.
Therefore, their augmentation categories are as follows:
\begin{align*}
\aug(\cT,\bfmu;\ruling_\Left,\ruling_\Right;\field)&\isomorphic
\aug(\cT',\bfmu';\ruling_\Left,\ruling_\Right;\field)\times\{(x_{b'}, x_{b''})\in (\field^\times)^2\mid x_{b'}x_{b''}=x_b\}\\
&\isomorphic \aug(\cT',\bfmu';\ruling_\Left,\ruling_\Right;\field)\times \field^\times,
\end{align*}
where $x_b, x_{b'}$ and $x_{b''}$ are regarded as variables corresponding to the values of augmentations of $t_b, t_{b'}$ and $t_{b''}$, respectively.

Now let $(\cT',\bfmu')\to(\cT,\bfmu)$ be the vertex-base point splitting operation.
Then as before, we have a bordered morphism $\bff:\dga^\CE(\cT',\bfmu')\to\dga^\CE(\cT,\bfmu)$ since it is another tangle replacement again, and moreover, two DGAs $A^\CE(T',\mu')$ and $A^\CE(T,\mu)$ only differ by the base rings 
\[
\ring[t_b^{\pm1}\mid b\in B']\quad\text{ and }\quad\ring[t_b^{\pm1}\mid b\in B],
\]
where $B=B'\amalg\{b\}$.
The DGA morphism $f:A^\CE(T',\mu')\to A^\CE(T,\mu)$ sends all generators to the very corresponding generators except for the following vertex generators $v_{a,i}$:
\begin{itemize}
\item the vertex $v$ is the vertex where the splitting happens;
\item the half-edge $h_{v,a}$ or $h_{v,a+i}$ is the half-edge on where the additional base point $b$ lies.
\end{itemize}
In this case, $f$ maps $v_{a,i}$ to either $v_{a,i}t_b$ or $t_b^{-1}v_{a,i}$, and as before each occurrence of $t_b$ in any differential of $A^\CE(T,\mu)$ comes together with such a vertex generator and \textit{vice versa}.

In other words, whatever we have an augmentation $\epsilon'$ for $A^\CE(T',\mu')$ and we assign a nonzero value $x_b$ to $t_b$, by multiplying $x_b^{-1}$ to each augmentation value for those vertex generators, we obtain an augmentation $\epsilon$ for $A^\CE(T,\mu)$.
Conversely, for each augmentation $\epsilon$ for $A^\CE(T,\mu)$, by assigning the value $\epsilon(v_{a,i}t_b)$ to $v_{a,i}$, we have an augmentation for $A^\CE(T',\mu')$.
It is obvious that these two constructions are inverses to each other up to the value of $t_b$.
Therefore we have
\[
\aug(\cT,\bfmu;\ruling_\Left,\ruling_\Right;\field)\isomorphic
\aug(\cT',\bfmu';\ruling_\Left,\ruling_\Right;\field)\times\{x_b\in\field^\times\}
\isomorphic\aug(\cT',\bfmu';\ruling_\Left,\ruling_\Right;\field)\times \field^\times.
\]
This completes the proof.
\end{proof}

\begin{theorem}\label{theorem:augmentation number is a ruling polynomial}
The augmentation number and the ruling polynomial are related as follows:
\begin{align*}
\augnumber(\cT,\bfmu;\ruling_\Left, \ruling_\Right;\FF_q)= q^{-\frac{d+\widehat B}{2}}z^{\widehat B}\langle\ruling_\Left| R(\cT,\bfmu;q,z)|\ruling_\Right\rangle,
\end{align*}
where $z=q^{\frac{1}{2}}-q^{-\frac{1}{2}}$, $d\coloneqq \max \deg_z \langle\ruling_\Left| R(\cT,\bfmu;z^2,z)|\ruling_\Right\rangle$ and
\begin{align*}
\hat B
&\coloneqq \#(\text{base points in $T$}) +\sum_{v\in V(T)}\frac{val(v)}{2}.
\end{align*}
\end{theorem}

\begin{proof}
Let us first show the statement for the special front diagram as described in Lemma~\ref{lem:inv quantity}. That is, we assume that each vertex $v$ is at the bottom of the vertical line containing $v$.

Let us choose a normal ruling $\ruling_0\in\bfR(\cT,\bfmu;\ruling_\Left,\ruling_\Right)$ such that
\[
\dim_\field \aug(\cT,\bfmu;\epsilon_\Left,\ruling_\Right;\field) = -\chi(\ruling_0) + \hat B + r(\ruling_0) + A(\ruling_0).
\]
Then we claim that $d=-\chi(\ruling_0)+A(\ruling_0)$.
By Corollary~\ref{corollary:dimension of augmentation variety} and Lemma~\ref{lem:inv quantity} imply that
\begin{align*}
D+\hat B-r(\ruling_0)&= -\chi(\ruling_0) +\hat B+ r(\ruling_0) + A(\ruling_0)\\
&=\max_\ruling\{ -\chi(\ruling)+\hat B+ r(\ruling) + A(\ruling)\}\\
&= \max_\ruling\{D+\hat B-r(\ruling)\}.
\end{align*}
In other words, $-r(\ruling_0)=\min_\ruling\{-r(\ruling)\}$ since $D$ and $\hat B$ are constant for any ruling.
Hence by Definition~\ref{definition:ruling polynomial},
\begin{equation}\label{eqn:d=-chi+A}
d\coloneqq\max_\ruling\{-\chi(\ruling)+A(\ruling)\}=\max_\ruling\{D-2r(\ruling)\}=D-2r(\ruling_0)=-\chi(\ruling_0)+A(\ruling_0).
\end{equation}

On the other hand, by Lemma~\ref{lem:inv quantity} again, we have 
\begin{align}\label{eqn:r-r_0}
r(\ruling)-r(\ruling_0)=\frac{1}{2}(-\chi(\ruling_0)+A(\ruling_0)+\chi(\ruling)-A(\ruling))=\frac{1}{2}(d+\chi(\ruling)-A(\ruling))
\end{align}
and therefore we obtain
\begin{align*}
\augnumber(\cT,\bfmu;\ruling_\Left, \ruling_\Right;\FF_q)
&=q^{\chi(\ruling_0) - \hat B - r(\ruling_0) - A(\ruling_0)}\sum_{\ruling}(q-1)^{-\chi(\ruling)+\hat B}q^{r(\ruling)+A(\ruling)} \tag{by Lemma~\ref{lem:augnumber}}\\
&=q^{-d - \hat B}(q-1)^{\hat B}\sum_{\ruling}(q-1)^{-\chi(\ruling)}q^{A(\ruling)}q^{r(\ruling)-r(\ruling_0)} \tag{by (\ref{eqn:d=-chi+A})}\\
&=q^{-d - \hat B}(q-1)^{\hat B}\sum_{\ruling}(q-1)^{-\chi(\ruling)}q^{A(\ruling)}q^{\frac{1}{2}(d+\chi(\ruling)-A(\ruling))} \tag{by (\ref{eqn:r-r_0})}\\
&=q^{-\frac{d+\hat B}{2}}z^{\hat B}\sum_{\ruling}z^{-\chi(\ruling)}q^{\frac{A(\ruling)}{2}} \tag{$z=q^{\frac{1}{2}}-q^{-\frac{1}{2}}$}\\
&=q^{-\frac{d+\hat B}{2}}z^{\hat B} \langle\ruling_\Left| R(\cT,\bfmu;q,z)|\ruling_\Right\rangle.
\end{align*}

So we have the statement for the special front diagram of a bordered Legendrian graph.
Note that the invariance of augmentation number and the ruling polynomial under Legendrian isotopy are already discussed in Proposition~\ref{prop:augnumber invarince} and Theorem~\ref{thm:invariance_ruling}.
This shows that the statement holds for general front diagrams.

The only remaining issue is about base points. The ruling polynomial $\langle\ruling_\Left| R(\cT,\bfmu;q,z)|\ruling_\Right\rangle$ has nothing to do with base points, and $\augnumber(\cT,\bfmu;\ruling_\Left,\ruling_\Right;\FF_q)$ depends on (the number of) base points. 
However, the normalized augmentation number in Lemma~\ref{lem:base point} is independent of the number and position of the base points on $T$. This prove the theorem.
\end{proof}

\begin{example}[Higher valency vertices]
Recall Example~\ref{example:six valent}.
For the Legendrian graph $\Lambda$ consisting of three Legendrian unknots intersecting at one vertex $v$ of valency 6, which has six possible resolutions $\ruling_v^1,\dots,\ruling_v^6$.

\begin{figure}[ht]
\[
\Res(\Lambda)=
\begin{tikzpicture}[baseline=-.5ex,scale=0.7]
\begin{scope}[xshift=2cm]
\draw[thick,rounded corners](7.5,2)-- (0.75,2) -- (-0.5, -0.5)--(0,-1.5) -- (0.5,-1.5) -- (1.5,.5) -- (2.5,0.5) -- (3,-0.5);
\draw[line width=8pt,white,rounded corners](0.25, 0) -- (1,-1.5) -- (1.5,-1.5) -- (2, -0.5) -- (2.5,-0.5);
\draw[thick,rounded corners](5,1.5)--(1,1.5) --(0.25, 0) -- (1,-1.5) -- (1.5,-1.5) -- (2, -0.5) -- (2.5,-0.5) -- (3,-0.5);
\draw[line width=8pt,white,rounded corners](1,0.5) -- (2,-1.5) -- (2.5,-1.5);
\draw[thick,rounded corners](2.5,1)--(1.25,1) -- (1,0.5) -- (2,-1.5) -- (2.5,-1.5) -- (3,-0.5);
\draw[thick,rounded corners] (3,-0.5) -- (3.75,1) -- (5,1) -- (5.25, 0.5) -- (5, 0) -- (4.5,0);
\draw[line width=8pt,white,rounded corners](3.75,0) -- (3.25,1);
\draw[thick,rounded corners] (4.5,0) -- (3.75,0) -- (3.25,1) -- (2.5,1);
\draw[thick,rounded corners] (3,-0.5)-- (5.25,-0.5)--(5.5,0);
\draw[thick,rounded corners] (3,-0.5)-- (3.5,-1.5)--(7,-1.5)--(8,0.5);
\fill[gray,opacity=0.4](1.65,-1.2)--(1.85,-.8) arc (63.43:-116.57:0.224);
\fill[gray,opacity=0.4](1.15,-.2)--(1.35,.2) arc (63.43:-116.57:0.224);
\fill[gray,opacity=0.4](0.65,-1.2)--(0.85,-0.8) arc (63.43:-116.57:0.224);
\draw[fill] (3,-0.5) circle (2pt) node[below] {\scriptsize$v$};
\draw (1.75,-1) node[right] {\scriptsize$a_{1,2}$};
\draw (1.25,0) node[right] {\scriptsize$a_{1,3}$};
\draw (0.75,-1) node[right] {\scriptsize$a_{2,3}$};
\end{scope}
\begin{scope}[xshift=4.5cm,yshift=0.5cm]
\draw[thick,rounded corners] (3,-0.5) -- (3.75,1) -- (5,1) -- (5.25, 0.5) -- (5, 0) -- (4.5,0);
\draw[line width=8pt,white,rounded corners](3.75,0) -- (3.25,1);
\draw[thick,rounded corners] (4.5,0) -- (3.75,0) -- (3.25,1) -- (2.5,1);
\end{scope}
\begin{scope}[xshift=7cm,yshift=1cm]
\draw[thick,rounded corners] (3,-0.5) -- (3.75,1) -- (5,1) -- (5.25, 0.5) -- (5, 0) -- (4.5,0);
\draw[line width=8pt,white,rounded corners](3.75,0) -- (3.25,1);
\draw[thick,rounded corners] (4.5,0) -- (3.75,0) -- (3.25,1) -- (2.5,1);
\end{scope}
\draw (5.5,0.5) node[right] {\scriptsize$a_{1}$};
\draw (8,1) node[right] {\scriptsize$a_{2}$};
\draw (10.5,1.5) node[right] {\scriptsize$a_{3}$};
\draw (7.2,0.5) node {\scriptsize$-$} node[left] {\scriptsize$b_1$};
\draw (9.7,1) node {\scriptsize$-$} node[left] {\scriptsize$b_2$};
\draw (12.2,1.5) node {\scriptsize$-$} node[left] {\scriptsize$b_3$};
\end{tikzpicture}
\]
\caption{Ng's resolution of $\Lambda$ in Example~\ref{example:six valent}}
\label{figure:Ng resolution of example}
\end{figure}

Let us compute the DGA $A(\Lambda,\mu)=(\alg,\differential)\coloneqq A^\CE(\Lambda,\mu)$ for $\Res(\Lambda)$ equipped with Maslov potential $\mu$. The algebra is generated by the crossings and the vertex generators:
\begin{align*}
G &=\{a_1,a_2,a_3,a_{1,2},a_{2,3},a_{1,3} \}\amalg \tilde{V},&
\alg&=\ring[t_1^{\pm1}, t_2^{\pm1}, t_3^{\pm1}]\langle G \rangle
\end{align*}
Note that $|a_i|=1$ for $i=1,2,3$ and let us list all generators of degree $0$:
\begin{align*}
a_{1,2}, a_{2,3}, a_{1,3};\\
v_{1,3}, v_{1,4}, v_{1,5}, v_{2,2}, v_{2,3}, v_{2,4}, v_{3,1}, v_{3,2}, v_{3,3};\\
v_{4,3}, v_{4,4}, v_{4,5}, v_{5,2}, v_{5,3}, v_{5,4}, v_{6,1}, v_{6,2}, v_{6,3}.
\end{align*}
The differential for the vertex generators are the same as before and for the crossings are as follows:
\begin{align}\label{eqn:exmaple differential}
\differential a_1 &= t_1 + a_{1,3} v_{3,1} + a_{1,2} v_{2,2} + v_{1,3};\nonumber\\
\differential a_2 &= t_2 + a_{2,3} v_{3,2} + v_{2,3} + (a_{2,3} v_{3,1} + v_{2,2})t_1^{-1}(a_1v_{4,1}+a_{1,3} v_{3,2} + a_{1,2} v_{2,3} + v_{1,4});\nonumber\\
\differential a_3 &= t_3 + v_{3,3} + v_{3,1}t_1^{-1}(a_1v_{4,2}+a_{1,3} v_{3,3}+a_{1,2} v_{2,4} + v_{1,5})
+v_{3,2} t_2^{-1}(a_2v_{5,1}+a_{2,3} v_{3,3} + v_{2,4}) \nonumber\\
&\mathrel{\hphantom{=}}+ v_{3,2} t_2^{-1} (a_{2,3} v_{3,1} + v_{2,2}) t_1^{-1} (a_1 v_{4,2} + a_{1,3}v_{3,3} + a_{1,2} v_{2,4} + v_{1,5})\nonumber\\
&\mathrel{\hphantom{=}}+ v_{3,1} t_1^{-1}(a_1v_{4,1}+a_{1,3}v_{3,2}+ a_{1,2}v_{2,3} + v_{1,4})t_2^{-1}(a_2 v_{5,1}+a_{2,3}v_{3,3}+v_{2,4})\nonumber\\
&\mathrel{\hphantom{=}}+v_{3,1}t_1^{-1}(a_1 v_{4,1}+a_{1,3}v_{3,2}+a_{1,2}v_{2,3}+v_{1,4})t_2^{-1}(a_{2,3}v_{3,1}+v_{2,2})t_1^{-1}(a_1 v_{4,2}+a_{1,3}v_{3,3}+a_{1,2}v_{2,4}+v_{1,5}).
\end{align}

Now recall the rulings $\ruling_v^i$, $i=1,\dots, 6$ at the vertex $v$ from Example~\ref{example:six valent} which induce fixed-point-free involutions on six-edges $h_{v,i}$, $i=1,\dots, 6$ near the vertex $v$.
Since there is a natural map 
\[
\pi_v:\aug(\Lambda,\mu;\field)\to \aug(v;\field)
\]
induced by the inclusion $i_v:A(v)\to A(\Lambda,\mu)$, let us investigate $\aug(\Lambda,\mu;\field)$ by the induced augmentation in $\aug(v;\field)$ with respect to the decomposition 
\[
\aug(v;\field)= \coprod_{i=1}^{6}\aug^{\ruling_v^i}(v;\field),
\]
and let us denote by $\aug(\Lambda,\mu;\ruling_v^i;\field)\coloneqq \pi_v^{-1}(\aug^{\ruling_v^i}(v;\field))$ and $\aug(\Lambda,\mu;\epsilon^i;\field)\coloneqq \pi_v^{-1}(\epsilon^i)$.

For $\ruling_v^i=\{\{1,a_i\},\{2,b_i\},\{3,c_i\}\}$ and $1\le i\le 6$, let us consider an (canonical) augmentation $\epsilon_v^i\in \aug^{\ruling_v^1}(v;\field)$ satisfying
\begin{align*}
\epsilon_v^i(v_{k,\ell}) \coloneqq 
\begin{cases}
1 &\text{ if }(k,\ell)=(1,a_i-1), (a_i,7-a_i), (2,b_i-2),(b_i,8-b_i), (3,c_i-3), (c_i,9-c_i);\\
0 &\text{ otherwise. }
\end{cases}
\end{align*}
For example, when $\ruling_v^1=\{\{1,6\},\{2,5\},\{3,4\}\}$, the (canonical) augmentation $\epsilon_v^1\in \aug^{\ruling_v^1}(v;\field)$ becomes
\begin{align}\label{eqn:example aug_v1}
\epsilon_v^1(g) \coloneqq 
\begin{cases}
1 &\text{ if }g=v_{1,5}, v_{2,3}, v_{3,1}, v_{4,5}, v_{5,3}, v_{6,1};\\
0 &\text{ otherwise. }
\end{cases}
\end{align}
Now consider an augmentation $\epsilon^1\in \aug(\Lambda,\mu;\field)$ extending $\epsilon_v^1$.
Then $\epsilon^1$ is determined by the value on $a_{1,2}$, $a_{2,3}$, $a_{1,3}$, and $t_i$, $i=1,2,3$ which should satisfy the followings:
\begin{align}\label{eqn:example aug_1}
\bar {t_1} +\bar{a_{1,3}}=0;\nonumber\\
\bar {t_2} +1 + \bar{a_{2,3} t_1^{-1} a_{1,2}}=0;\nonumber\\
\bar {t_3} +\bar{t_1^{-1}} + \bar{t_1^{-1}a_{1,2}t_2^{-1}a_{2,3}t_1^{-1}} =0,
\end{align}
where $\bar g$ means $\epsilon^1(g)$ and the above induced from (\ref{eqn:exmaple differential}) by applying $\epsilon^1$ and (\ref{eqn:example aug_v1}).
Let us count the $\FF_q$-points satisfying (\ref{eqn:example aug_1}).
There are $(q-1)(2q-1)$ and $(q-1)^2(q-2)$-many $\FF_q$-points when $\bar {t_2}=-1$ and $\bar {t_2}\neq -1$, respectively. In total, we have
\[
(q-1)(2q-1)+(q-1)^2(q-2)=(q-1)^3+q(q-1).
\]

Denote by $(\tilde{\Lambda},\mu)$ the Legendrian graph obtained from $(\Lambda,\mu)$ by adding an additional base point at each left half-edge of $v$. Then the map $\pi_v:\aug(\tilde{\Lambda},\mu;\ruling_v^i;\field)\to \aug^{\ruling_v^i}(v;\field)$ is a fiber bundle. For example, this follows from Corollary \ref{cor:B^v action for V}. By a similar argument, one can also show the same holds for $\Lambda$. In particular, for any $\epsilon^i\in\aug^{\ruling_v^i}(\Lambda,\mu;\field)$ with $\pi_v(\epsilon^i)\in \aug^{\ruling_v^i}(v;\field)$, we have
\[
\#\left(\aug(\Lambda,\mu;\ruling_v^i;\FF_q)\right)=\#\left(\aug(\Lambda,\mu;\epsilon^i;\FF_q)\right)\times \#\left( \aug^{\ruling_v^i}(v;\FF_q) \right).
\]

Especially for $\epsilon^1$,
\begin{align*}
\#\left(\aug(\Lambda,\mu;\ruling_v^1;\FF_q)\right)&=\#\left(\aug(\Lambda,\mu;\epsilon^1;\FF_q)\right)\times \#\left( \aug^{\ruling_v^1}(v;\FF_q) \right)\\
&=\left((q-1)^3+q(q-1)\right)\times (q-1)^{\frac{\val(v)}{2}}q^{A(\ruling_v^i)}\\
&=\left((q-1)^3+q(q-1)\right)\times (q-1)^3q^6.
\end{align*}
By a similar computation for $\epsilon^i$, $i=2,\dots,6$,
\begin{align*}
\#\left(\aug(\Lambda,\mu;\ruling_v^2;\FF_q)\right)&= (q-1)^5 q^6,&
\#\left(\aug(\Lambda,\mu;\ruling_v^3;\FF_q)\right)&= (q-1)^5 q^6;\\
\#\left(\aug(\Lambda,\mu;\ruling_v^4;\FF_q)\right)&= (q-1)^3 q^6,&
\#\left(\aug(\Lambda,\mu;\ruling_v^5;\FF_q)\right)&= (q-1)^4 q^6;\\
\#\left(\aug(\Lambda,\mu;\ruling_v^6;\FF_q)\right)&= (q-1)^4 q^6.&
\end{align*}
Then we have
\begin{align*}
\#\left( \aug(\Lambda;\FF_q) \right)&=\left((q-1)^3+q(q-1)\right)\times (q-1)^3q^6 + 2(q-1)^5 q^6 +2 (q-1)^4 q^6 +(q-1)^3 q^6 ;\\
\augnumber(\Lambda;\FF_q)&=\frac{\left((q-1)^3+q(q-1)\right)\times (q-1)^3q^6 + 2 (q-1)^5 q^6 +2 (q-1)^4 q^6 +(q-1)^3 q^6}{q^{12}} .
\end{align*}

On the other hand, recall from Example~\ref{ex:ruling polynomial} that
\[
R(\Lambda;q,z)=q^{3}(1+z^{-2})+2q^{5/2} z^{-1} + 2q^{2} z^{-2} + q^{3/2}z^{-3}.
\]
Note that $d(\Lambda)=\max\deg_z R(\Lambda;z^2,z)=6$, $\hat B(\Lambda)=\frac{\val(v)}{2}+\#\{\text{base points}\}=6$. This implies
\[
q^{-\frac{d+\hat B}{2}}z^{\hat B}R(\Lambda;q,z)=q^{-6}z^{6}\left(q^{3}(1+z^{-2})+2q^{5/2} z^{-1} + 2q^{2} z^{-2} + q^{3/2}z^{-3}\right).
\]
Evaluate $z=(q^{1/2}-q^{-1/2})$, then it recovers $\augnumber(\Lambda;\FF_q)$.
This verifies Theorem~\ref{theorem:augmentation number is a ruling polynomial}.

\end{example}

\appendix
\section{Generalized stabilizations and Ekholm-Ng's (de)stabilizations}\label{appendix:proof of EN's stabilization}
Let $A=(\alg=\ring\langle G\rangle,\differential)$ be a DGA and $\varphi:I_{(\ell,m)}(\mu)\to A$ be a DGA inclusion for some $m\ge 1$ and $\mu:[m]\to\grading$.
We denote the image $\varphi(\ruling_{a,i})$ by $v_{a,i}$.

Then we define a DGA $\tilde SA^\pm=(\tilde S\alg,\bar\differential^\pm)$ as follows:
the underlying graded algebra $\tilde S\alg^\pm$ is generated by
\[
\tilde S\alg\coloneqq\ring\left\langle G\amalg\{\hat e^i, e^i ~\middle|~i\ge 1\}\right\rangle.
\]
The gradings and differentials for newly added generators are defined as
\begin{align*}
|e^i|&\coloneqq\begin{cases}
\fd & i=1;\\
\fd+|v_{1,i-1}|+1& i>1,
\end{cases}&
|\hat e^i|&\coloneqq |e^i|+1,&
\bar\differential^\pm(\hat e^i)&\coloneqq e^i,&
\bar\differential^\pm(e^i)&\coloneqq 0.
\end{align*}
%
Then it is obvious that $\tilde SA^\pm$ is a stabilization of $A$ by cancelling pairs $(\hat e^i, e^i)$.

Let us focus on the positive stabilization case first.
Then we define a map $\Phi^+:\tilde S\alg\to \tilde S\alg$ between graded algebras defined as
\begin{align*}
\Phi^+(\hat e^b)&\coloneqq \hat e^b,&
\Phi^+(e^b)&\coloneqq e^b+\sum_{a<b}(-1)^{|e^a|}\hat e^a v_{a,b-a}.
\end{align*}
\begin{remark}
The map $\Phi^+$ is indeed a tame isomorphism, which is a composition
\[
\Phi^+\coloneqq\cdots \circ f^+_{[2]}\circ f^+_{[1]}:\tilde S\alg\to \tilde S\alg
\]
of infinite elementary isomorphisms $\{f^+_{[b]}:\tilde S\alg\to \tilde S\alg\}_{b\ge 1}$ such that $f^+_{[b]}$ fixes all generators but $e^b$ and
\[
f^+_{[b]}:e^b\mapsto e^b+\sum_{a<b}(-1)^{|e^a|}\hat e^a v_{a,b-a}.
\]

Note that the map $\Phi^+$ does not depend on the order of compositions.
\end{remark}

We denote the twisted differential of $\bar\differential^+$ by $\Phi^+$ by $\bar\differential'^+$ and the DGA $(\tilde S\alg,\bar\differential'^+)$ by $\tilde SA'^+$.
\begin{align*}
\bar\differential'^+&\coloneqq \Phi^+\circ\bar\differential^+\circ (\Phi^+)^{-1},&
\tilde SA'^+&\coloneqq(\tilde S\alg^+,\bar\differential'^+).
\end{align*}
Then by definition, we have a tame isomorphism between DGAs
\[
\Phi^+:\tilde SA^+\to \tilde SA'^+.
\]

\begin{lemma}
The following holds: for each $b\ge1$,
\begin{align*}
\bar\differential'^+(e^b)&=
\begin{cases}
\displaystyle\sum_{a<b} (-1)^{|e^a|-1} e^a v_{a,b-a} & b\le m;\\
\displaystyle\hat e^{b-m}+\sum_{a<b} (-1)^{|e^a|-1} e^a v_{a,b-a}& b>m,
\end{cases}\\
\bar\differential'^+(\hat e^b)&=\Phi^+(e^b).
\end{align*}
\end{lemma}
\begin{proof}
This follows from the direct computation: for each $b\ge 1$, it is obvious that 
\begin{align*}
\bar\differential'^+(\hat e^b) &= (\Phi^+\circ\bar\differential^+\circ(\Phi^+)^{-1})(\hat e^b)=(\Phi^+\circ\bar\differential^+)(\hat e^b)=\Phi^+(e^b)
\end{align*}
since $\Phi^+$ fixes $\hat e^b$.

On the other hand, 
\begin{align*}
\bar\differential'^+(e^b)&=(\Phi^+\circ \bar\differential^+\circ(\Phi^+)^{-1})(e^b)\\
&=(\Phi^+\circ\bar\differential^+)\left(e^b-\sum_{a<b}(-1)^{|e^a|}\hat e^a v_{a,b-a}\right)\\
&=\Phi^+\left(-\sum_{a<b}(-1)^{|e^a|} e^a v_{a,b-a}+\sum_{a<b} \hat e^a \bar\differential^+(v_{a,b-a})\right)\\
&=\sum_{a<b}(-1)^{|e^a|-1} \Phi^+(e^a) v_{a,b-a}+\sum_{a<b}\hat e^a\left(\delta_{b-a,m}+\sum_{a<c<b}(-1)^{|v_{a,c-a}|-1}v_{a,c-a,b-a}\right)\\
&=\sum_{a<b}(-1)^{|e^a|-1} e^a v_{a,b-a} +\sum_{c<a<b}(-1)^{|e^a|-1+|e^c|}\hat e^c v_{c,a-c,b-a}\\
&\mathrel{\hphantom{=}}+\sum_{a<c<b} (-1)^{|v_{a,c-a}|-1}\hat e^a v_{a,c-a,b-a}+\sum_{a<b}\delta_{b-a,m}\hat e^a.
\end{align*}

Since $|e^a|+|e^c|-1\equiv|v_{a,c-a}|\mod{2}$, the above expression is equivalent to
\begin{align*}
&\mathrel{\hphantom{=}}\sum_{a<b}(-1)^{|e^a|-1} e^a v_{a,b-a}+\cancelto{}{\sum_{c<a<b}(-1)^{|v_{a,c-a}|}\hat e^c v_{c,a-c,b-a}}\\
&\mathrel{\hphantom{=}}+\cancelto{}{\sum_{a<c<b} (-1)^{|v_{a,c-a}|-1}\hat e^a v_{a,c-a,b-a}}+\sum_{a<b}\delta_{b-a,m}\hat e^a.
\end{align*}

Since $\sum_{a<b} \delta_{b-a,m}\hat e^a$ is either $0$ if $b\le m$ and $\hat e^{b-m}$ if $b>m$, we are done.
\end{proof}

Now we twist again the differential to obtain $\bar\differential''^+$ via the tame isomorphism $\Psi^+:\tilde S\alg\to \tilde S\alg$ between graded algebras as follows:
\begin{align*}
\Psi^+(\hat e^b)&\coloneqq \hat e^b + \sum_{a<m+b} (-1)^{|e^a|} e^a v_{a,m+b-a},&
\Psi^+(e^b)&\coloneqq e^b,&
\bar\differential''^+&\coloneqq \Psi^+\circ \bar\differential'^+\circ(\Psi^+)^{-1}.
\end{align*}

\begin{remark}
As before, $\Psi^+$ is a tame isomorphism which is a composition
\[
\Psi^+=\cdots\circ g^+_{[2]}\circ g^+_{[1]}
\]
of infinite elementary isomorphisms $g^+_{[b]}:\tilde S\alg\to \tilde S\alg$ which fix all generators but $\hat e^b$ such that
\[
g^+_{[b]}(\hat e^b)=\hat e^b+\sum_{a<m+b}(-1)^{|e^a|}e^a v_{a,m+b-a}.
\]

As before, this does not depend on the order of compositions.
\end{remark}

Then obviously, we have a tame isomorphism between DGAs
\[
\Psi^+:\tilde SA'^+\to\tilde SA''^+.
\]

\begin{lemma}
The following holds: for each $b\ge 1$,
\begin{align*}
\bar\differential''^+(e^b)&=
\begin{cases}
\bar\differential'^+(e^b) & b\le m;\\
\hat e^{b-m}& b>m,
\end{cases}&
\bar\differential''^+(\hat e^b)&=0.
\end{align*}
\end{lemma}
\begin{proof}
This also follows from the direct computation: since $\Psi^+$ fixes $e^b$ for each $b\ge 1$, we have
\begin{align*}
\bar\differential''^+(e^b)&=(\Psi^+\circ \bar\differential'^+\circ(\Psi^+)^{-1})(e^b)=(\Psi^+\circ\bar\differential'^+)(e^b).
\end{align*}

If $b\le m$, then 
\begin{align*}
\bar\differential''^+(e^b)&=\Psi^+\left(\sum_{a<b}(-1)^{|e^a|-1}e^a v_{a,b-a}\right)=\sum_{a<b}(-1)^{|e^a|-1}e^a v_{a,b-a},
\end{align*}
and if $b>m$, then
\begin{align*}
\bar\differential''^+(e^b)&=\Psi^+\left(\hat e^{b-m}+\sum_{a<b}(-1)^{|e^a|-1}e^a v_{a,b-a}\right)=(\Psi^+\circ(\Psi^+)^{-1})(\hat e^{b-m})=\hat e^{b-m}
\end{align*}
as desired.

This also implies the existence of anti-derivative of $\hat e^b$ for each $b\ge1$, which is precisely, $e^{b+m}$.
Therefore, we have $\bar\differential''^+(\hat e^{b})=0$ for all $b\ge 1$. Indeed, 
\begin{align*}
\bar\differential''^+(\hat e^b)&=(\Psi^+\circ\bar\differential'^+\circ(\Psi^+)^{-1})(\hat e^b)\\
&=(\Psi^+\circ\bar\differential'^+)\left(\hat e^b-\sum_{a<m+b}(-1)^{|e^a|}e^a v_{a,m+b-a}\right)\\
&=\Psi^+\left(\Phi^+(e^b)+\sum_{a<m+b}(-1)^{|e^a|-1} \bar\differential'^+(e^a v_{a,m+b-a})\right).
\end{align*}

Let us consider the inside first.
Then we have
\begin{align*}
&\mathrel{\hphantom{=}}\Phi^+(e^b)+\sum_{a<m+b}(-1)^{|e^a|-1} \bar\differential'^+(e^a v_{a,m+b-a})\\
&=e^b+\sum_{a<b}(-1)^{|e^a|}\hat e^av_{a,b-a}+\sum_{a<m+b}(-1)^{|e^a|-1}\bar\differential'^+(e^a) v_{a,m+b-a}
-\sum_{a<m+b} e^a\bar\differential^+(v_{a,m+b-a})\\
&=e^b+\sum_{a<b}(-1)^{|e^a|}\hat e^av_{a,b-a}\\
&\mathrel{\hphantom{=}}+\sum_{c<a<m+b}(-1)^{|e^a|-1}\left(\delta_{a-c,m}\hat e^c+\cancelto{}{(-1)^{|e^c|-1}e^c v_{c,a-c}}\right) v_{a,m+b-a}\\
&\mathrel{\hphantom{=}}-\sum_{a<m+b}e^a\delta_{m+b-a,m}-\cancelto{}{\sum_{a<c<m+b} e^a(-1)^{|v_{a,c-a}|-1}v_{a,c-a,m+b-c}}.
\end{align*}

Since we have 
\[
\sum_{a<m+b}e^a\delta_{m+b-a,m}=e^b\quad\text{ and }\quad
\sum_{c<a<m+b}(-1)^{|e^a|}\delta_{a-c,m}\hat e^c=\sum_{m+1\le a<m+b}(-1)^{|e^a|}\hat e^{a-m},
\]
we have
\begin{align*}
&\mathrel{\hphantom{=}}\cancelto{}{e^b}+\sum_{a<b}(-1)^{|e^a|}\hat e^av_{a,b-a}+\sum_{m+1\le a<m+b}(-1)^{|e^a|-1}\hat e^{a-m} v_{a,m+b-a}-\cancelto{}{e^b}\\
&=\cancelto{}{\sum_{a<b}(-1)^{|e^a|}\hat e^av_{a,b-a}}+\cancelto{}{\sum_{a'<b}(-1)^{|e^{a'}|-1}\hat e^{a'} v_{a',b-a'}},
\end{align*}
where the last equality comes from that $|e^{a+m}|=|e^a|+2$ and $v_{a,m+b-a}=v_{a',b-a'}$ for $a=m+a'$. We are done.
\end{proof}

\begin{proof}[Proof of Proposition~\ref{proposition:generalized stabilizations}]
As mentioned earlier, the DGA $\tilde SA^+=(\tilde S\alg,\bar\differential^+)$ is a stabilization of $A=(\alg,\differential)$, where the cancelling pairs are $(\hat e^b, e^b)$ for $b\ge 1$.

On the other hand, the DGA $\tilde SA^+$ is tame isomorphic to the twisted DGA $\tilde SA''^+=(\tilde S\alg, \bar\differential''^+)$ defined as above via 
\[
\Psi^+\circ\Phi^+:\tilde SA\to\tilde SA'',
\]
which has infinite cancelling pairs $(e^{m+b},\hat e^b)$ for $b\ge 1$. By cencelling out these pairs, what we have is isomorphic to the generalized stabilization $S^{\fd+}_\varphi A$ by identifying $e^b$ with $e_b^+$ for $b\in[m]$.

For the generalized stabilization $S^{\fd-}_\varphi A$, we define tame isomorphisms $\Phi^-$ and $\Psi^-$ by using elementary isomorphisms $f^-_{[b]}$ and $g^-_{[b]}$ fixing all generators but $e^b$ and $\hat e^b$ so that
\begin{align*}
f^-_{[b]}(e^b)&\coloneqq e^b+\sum_{a<b} (-1)^{|v_{m+1-b,b-a}|-1}v_{m+1-b,b-a} \hat e^a,\\
g^-_{[b]}(\hat e^a)&\coloneqq \hat e^b+\sum_{a<b}v_{m+1-b,b-a} e^a.
\end{align*}
The rest of the proof is essentially the same as before and we omit the detail.
\end{proof}

\section{Complementary results on augmentations}\label{appendix:technical results for augmentation varieties}

\subsection{Orbits of augmentation varieties for internal DGAs}\label{sec:Orbits of augmentation varieties for internal DGAs}

Let $v$ be a vertex of type $(\ell,r)$ with $\valency=\ell+r$ as in Section \ref{sec:aug_variety_vertex}, and $\mu:I(v)=[\valency]\rightarrow\ZZ$ be the Maslov potential. The main purpose of this subsection is to show Lemma \ref{lem:stratification of aug var at a vertex}: 
(1) follows immediately from Proposition/Definition \ref{def/prop:structure of Morse complexes at a vertex}.(4) below;
(2) and (3) are contained Lemma \ref{lem:properties of orbits at a vertex}.

Let us denote $n(v_{i,j})\coloneqq n(v,i,j)$.
\begin{proposition/definition}\label{def/prop:structure of Morse complexes at a vertex}
Let $v$ be a $\valency$-valency vertex as above, with $\valency$ even. Let $d\in \MC(v;\field)$. 
\begin{enumerate}
\item
For each $i\in \Zmod{\valency}$ and $t\geq 0$, an element $x\in C_v$ is called \emph{$(i,t)$-admissible}, if 
\[
x=Z^t \left(c_0e_i+\sum_{k>0}c_kZ^{n(v_{i,k})}e_{i+k} \right),
\]
for some $c_k\in\field$ with $c_0\in\field^*$, and $x$ is homogeneous. In particular, $c_k\neq 0$ implies 
\[
|e_i|=|Z^{n(v_{i,k})}e_{i+k}| \quad \text{ and hence } \quad\mu(i)-\mu(i+k)+n(v_{i,k})=0.
\] 
When $t=0$, we simply say that $x$ is \emph{$i$-admissible}.\newline
Notice that $x$ is $(i,t)$-admissible if and only if $x=Z^ty$, with $y$ $i$-admissible.

\item
For any $i$-admissible $x\in C_v$, \emph{define}
$\ruling_d(x)\coloneqq k$ if $dx$ is $(i+k,n(v_{i,k}))$-admissible.

\item
For any $i\in\ZZ/\valency$, \emph{define}
\begin{equation*}
\ruling_{[d]}(i)\coloneqq\max\{\ruling_d(x) \mid \textrm{$x\in C_v$ is $i$-admissible.}\}
\end{equation*}
\end{enumerate}
Then, we have:
\begin{enumerate}
\item
For any $i$-admissible $x$, $\ruling_d(x)$ is well-defined and $0<\ruling_d(x)<\valency$. In particular, $0<\ruling_{[d]}(i)<\valency$.
\item 
The set of $(i,t)$-admissible elements is preserved by the action of $B(v;\field)$. Moreover, 
for any $g\in B(v;\field)$ and any $i$-admissible $x$, have 
\[
\ruling_d(x)=\ruling_{g\cdot d}(g(x)).
\] 
In particular, $\ruling_{[d]}=\ruling_{[g\cdot d]}$. That is,
$\ruling_{[d]}$ depends only on the isomorphism class of $d$, i.e. the orbit $B(v;\field)\cdot d$.
\item
If $\ruling_{[d]}(i)=k$, then $\ruling_{[d]}(i+k)=\valency-k$. In particular,
each isomorphism class $B(v;\field)\cdot d$ in $\MC(v;\field)$ induces an involution $\ruling=\ruling([d])\in \NR(v)$ such that 
\[\ruling(i)=i+\ruling_{[d]}(i)\mod{\valency}.\]
\item
Each isomorphism class $B(v;\field)\cdot d$ contains a unique canonical differential $d_{\ruling}$ with $\ruling=\ruling([d])$.
\end{enumerate}
\end{proposition/definition}

\begin{proof}
Let $x\in C_v$ be any $i$-admissible element. Say, 
\[
x=\sum_{k\geq 0}x_kZ^{n(v_{i,k})}e_{i+k}
\] 
for some $x_k\in\field$ with $x_0\neq 0$. Here, we use $n(v_{i,0})\coloneqq0$. For any $f\in\End(C_v,F^{\bullet})$, i.e. a homogeneous $\field[Z]$-superlinear endomorphism, by Remark \ref{rem:morphisms of Morse complexes at a vertex}, we have:
\begin{align*}
f(x)&=\sum_{k\geq 0}x_k\sum_{j\geq 0}(-1)^{|f|\cdot n(v_{i,k})}Z^{n(v_{i,k})}\langle f(e_{i+k}), Z^{n(v_{i+k,j})}e_{i+k+j}\rangle Z^{n(v_{i+k,j})}e_{i+k+j}\nonumber\\
&=\sum_{l\geq 0}\left(\sum_{0\leq k\leq l}(-1)^{|f|\cdot n(v_{i,k})}x_k\langle f(e_{i+k}), Z^{n(v_{i+k,l-k})}e_{i+l}\rangle \right) Z^{n(v_{i,l})}e_{i+l}\nonumber\\
&=:\sum_{l\geq 0}c_l(f) Z^{n(v_{i,l})}e_{i+l}
\end{align*}
where $\langle f(e_a),Z^{n(v_{a,b})}e_{a+b}\rangle$ denotes the $\field$-coefficient of $Z^{n(v_{a,b})}e_{a+b}$ in $f(e_a)$, and we have used the fact that $f$ super-commutes with $Z$, and $n(v_{i,k})+n(v_{i+k,j})=n(v_{i,k+j})$.

\begin{proof}[Proof of {\rm (1)}]
Take $f\coloneqq d$, then $\langle d e_a,e_a\rangle=0$ and $dx=\sum_{l>0}c_l(d) Z^{n(v_{i,l})}e_{i+l}$.
By a similar calculation, we obtain:
\begin{equation*}
d^2x=\sum_{j>0}\left(\sum_{0<l<j}(-1)^{n(v_{i,l})}c_l(d)\langle de_{i+l},Z^{n(v_{i+l,j-l})}e_{i+j}\rangle\right)Z^{n(v_{i,j})}e_{i+j}
\end{equation*} 
In particular, 
\[
\langle d^2x, Z^2e_i\rangle=\sum_{0<l<\valency}(-1)^{n(v_{i,l})}c_l(d)\langle de_{i+l},Z^{n(v_{i+l,j-l})}e_{i+j}\rangle=-x_0
\] 
is non-zero, as $d^2+Z^2=0$. 
Let $l_0\coloneqq \min\{l| c_l(d)\neq 0\}$, then it follows that $0<l_0<\valency$. 
Moreover, by definition of $l_0$, we can rewrite 
\[
dx=Z^{n(v_{i,l_0})}\sum_{k\geq 0}c_{l_0+k}(d)Z^{n(v_{i+l_0,k})}e_{i+l_0+k},
\] 
with $c_{l_0}(d)\neq 0$. In other words, $dx$ is $(i+l_0,n(v_{i,l_0}))$-admissible. 
Hence, $\ruling_d(x)=l_0$ is well-defined, and $0<\ruling_d(x)<\valency$. This shows (1).
\end{proof}

\begin{proof}[Proof of {\rm (2)}]
For any $g\in B(v;\field)$ and any $t\geq 0$, we have 
\[
g(Z^tx)=Z^tg(x)=Z^t\sum_{l\geq 0}c_l(g)Z^{n(v_{i,l})}e_{i+l},
\] 
with $c_0(g)=x_k\langle g(e_i),e_i\rangle\neq 0$. That is, $g(Z^tx)$ is $(i,t)$-admissible. Thus, $g$ preserves the set of $(i,t)$-admissible elements.
Moreover, $dx$ is $(i+k,n(v_{i,k}))$-admissible if and only if $(g\cdot d)(g(x))=g\circ d(x)$ is. It follows that $\ruling_d(x)=\ruling_{g\cdot d}(g(x))$ and hence $\ruling_{[d]}$ is well-defined. This shows (2).
\end{proof}

\begin{proof}[Proof of {\rm (3)}] 
If $\ruling_{[d]}(i)=k$, then $0<k<\valency$. By definition, $k=\ruling_{[d]}(i)=\ruling_d(x)$ for some $i$-admissible $x$. In addition, $dx=Z^{n(v_{i,k})}y$ for some $(i+k)$-admissible $y$. 
We then define $g\in B(v;\field)$ by
\begin{align*}
\begin{cases}
g(e_p)=e_p, \qquad \text{ if } p\neq i, i+k;\\
g(x)=e_i; \\
g(y)=e_{i+k}.
\end{cases}
\end{align*}
 Replacing $d$ by $g\cdot d$ and $x$ by $e_i$, we obtain 
 \[x=e_i,\qquad de_i=Z^{n(v_{i,k})}e_{i+k}.\] 
It then follows from $d^2+Z^2=0$ that $0\leq n(v_{i,k})\leq 2$, and
\begin{align*}
de_{i+k}&=-(-1)^{n(v_{i,k})}Z^{2-n(v_{i,k})}e_i\\
&=-(-1)^{n(v_{i+k,\valency-k})}Z^{n(v_{i+k,\valency-k})}e_{i+k+(\valency-k)},
\end{align*} 
which implies that $\ruling_d(e_{i+k})=\valency-k$. This shows that $\ruling_{[d]}(i+k)\geq \valency-k$.

Suppose $\ruling_{[d]}(i+k)=j>\valency-k$. Then there exists a $(i+k)$-admissible element $e_{i+k}'$, such that $de_{i+k}'=Z^{n(v_{i+k,j})}e_{i+k+j}'$ and $e_{i+k+j}'$ is $(i+k+j)$-admissible. Again, by $d^2+Z^2=0$, we must have 
\[de_{i+k+j}'=-(-1)^{n(v_{i+k+j,\valency-j})}Z^{n(v_{i+k+j,\valency-j})}e_{i+k}'.\] 
It follows that 
\begin{align*}
&\mathrel{\hphantom{=}}d \left(e_i+cZ^{n(v_{i,k+j-\valency})}e_{i+k+j-\valency}'\right)\\
&=Z^{n(v_{i,k})}e_{i+k}-c(-1)^{n(v_{i,k+j-\valency})+n(v_{i+k+j,\valency-j})}Z^{n(v_{i,k+j-\valency})+n(v_{i+k+j,\valency-j})}e_{i+k}'\\
&=Z^{n(v_{i,k})} \left(e_{i+k}-(-1)^{n(v_{i,k})}ce_{i+k}'\right)
\end{align*}
Since $e_{i+k}'$ is $(i+k)$-admissible, can choose $c\in\field^*$ so that 
\[
w\coloneqq e_{i+k}-(-1)^{n(v_{i,k})}ce_{i+k}'=\sum_{l>k}*_lZ^{n(v_{i+k,l-k})}e_{i+l},
\]
i.e., $w$ is {\em not} $(i+k)$-admissible.
By assumption, 
\[z\coloneqq e_i+cZ^{n(v_{i,k+j-\valency})}e_{i+k+j-\valency}'\] 
is $i$-admissible. Then by the choice of $c$, $dz=Z^{n(v_{i,k})}w$ is $(i+l_0,n(v_{i,l_0}))$-admissible for some $l_0>k$. Hence, $\ruling_d(z)=l_0$, which implies that $\ruling_{[d]}(i)\geq l_0>k$, a contradiction! Therefore, we must have $\ruling_{[d]}(i+k)=\valency-k$, as desired. The remaining part of (3) now follows by a direct check of Definition \ref{def:canonical augmentation at a vertex}.
\end{proof}

\begin{proof}[Proof of {\rm (4)}] 
Let $\ruling\coloneqq \ruling([d])\in \NR(v)$ be the involution induced by $\ruling_{[d]}$. Notice that $\ruling$ is equivalent to a partition $I(v)=\{1,2,\ldots,\valency\}=U\amalg L$ together with a bijection $\ruling:U\xrightarrow[]{\sim}L$ satisfying
\[
\mu(i)-\mu(\ruling(i))-1+n(v_{i,\ruling(i)-i})=0, \qquad \text{ for all } i\in U \text{ with } i<\ruling(i).
\]
For each $i\in U$, denote $j\coloneqq \ruling(i)$. By definition of $\ruling_{[d]}$, there exists an $i$-admissible element $e_i'\in C_v$ such that $(-1)^{\mu(i)}de_i'=Z^{n(v_{i,\ruling(i)-i})}e_{\ruling(i)}'$, where $e_{\ruling(i)}'\in C_v$ is $\ruling(i)$-admissible. Then $\{e_p':1\leq p\leq \valency\}$ defines an element 
$g\in B(v;\field)$ by 
\[
g(e_p)\coloneqq e_p', \qquad \text{ for all } 1\leq p\leq \valency.
\] 
It follows that, for all $i\in U$, have 
\[
(g^{-1}\cdot d)(e_i)=g^{-1}\circ d\circ g(g^{-1}(e_i'))=g^{-1}((-1)^{\mu(i)}Z^{n(v_{i,\ruling(i)-i})}e_{\ruling(i)}')=(-1)^{\mu(i)}Z^{n(v_{i,\ruling(i)-i})}e_{\ruling(i)}.
\] 
Then by the condition $(g^{-1}\cdot d)^2+Z^2=0$, we see that $g^{-1}\cdot d=d_{\ruling}$, the canonical differential associated to $\ruling$ as in Definition \ref{def:canonical augmentation at a vertex}. In other words, $d=g\cdot d_{\ruling}$ and $B(v;\field)\cdot d=B(v;\field)\cdot d_{\ruling}$. By construction, $d_{\ruling}$ is uniquely determined by $\ruling([d])$, hence by $B(v;\field)\cdot d$. This shows (4).
\end{proof}
This proves the proposition.
\end{proof}

Similar to \cite[Lem.5.7,Cor.5.8]{Su2017}, we have 

\begin{lemma}\label{lem:properties of orbits at a vertex}
Let $v$ be a vertex as before and $\ruling\in\NR(v)$.
\begin{enumerate}
\item
For any $d\in B(v;\field)\cdot d_{\ruling}$, and any $i\in U$, there exists a unique $i$-admissible element in $C_v$ of the form 
\[
e_i'=e_i+\sum_{j\in A_{\ruling}(i)}a_{j-i}Z^{n(v_{i,j-i})}e_j,\qquad a_{j-i}\in\field,
\] 
such that $de_i'$ is $(\ruling(i),n(v_{i,\ruling(i)-i}))$-admissible. 
Moreover, $e_i'=e_i'(d)$ depends algebraically on $d\in B(v;\field)\cdot d_{\ruling}$. 
\item
The $\Stab^\ruling(v;\field)$-principal bundle 
\begin{equation*}
\pi_{\ruling}:B(v;\field)\to B(v;\field)\cdot d_{\ruling}
\end{equation*} 
admits a natural algebraic section $\varphi_{\ruling}$, i.e. $\varphi_{\ruling}(d)\cdot d_{\ruling}=d$ for all $d\in B(v;\field)\cdot d_{\ruling}$. In other words, we have a trivialization of $\pi_{\ruling}$: 
\begin{equation*}
B(v;\field)\isomorphic  \Stab^\ruling(v;\field)\times B(v;\field)\cdot d_{\ruling}
\end{equation*}
\item
In addition, we have:
\begin{equation}
B(v;\field)\cdot d_{\ruling}\isomorphic (\field^*)^{\frac{\valency}{2}}\times\field^{A_v(\ruling)}
\end{equation}
with $A_v(\ruling)$ given in Definition \ref{def:indices for an involution at a vertex}.
\end{enumerate}

\end{lemma}

The proof is essentially the same as that in \cite[Lem.5.7,Cor.5.8]{Su2017}. For completeness, we give the details.

\begin{proof}[Proof of {\rm (1)}]
\emph{Existence}. By Proposition/Definition \ref{def/prop:structure of Morse complexes at a vertex}, for any $i\in U$, we have:
\[
\ruling(i)-i=\max\{\ruling_d(x) \mid x \text{ is $i$-admissible}\}.
\]
Hence, we can take an $i$-admissible element of the form 
\[x=e_i+\sum_{j\in I_v(i)}a_jZ^{n(v_{i,j})}e_{i+j},\] 
such that $\ruling_d(x)=\ruling(i)-i$, i.e. $dx$ is $(\ruling(i),n(v_{i,\ruling(i)-i}))$-admissible.

Let us apply induction argument on
\[
j=j(x_0) \coloneqq \min\{j \mid j\in I_v(i)\setminus (A_{\ruling}(i)-i) \text{ and $a_j=\langle x_0,Z^{n(v_{i,j})}e_{i+j} \rangle \neq 0$} \}.
\] 
\noindent{}\emph{Inductive step 0:} If $j=+\infty$, i.e., $a_j=0$ for all $j\in I(i)\setminus (A_{\ruling}(i)-i)$, then $e_i'=x$ is a desired element for $(1)$.\footnote{Here, we use the convention $\min~\emptyset:=\infty$.} 

\noindent{}\emph{Inductive step 1:} Otherwise, set $x_0=x$ and construct $x_1$ by considering the following cases:
\begin{enumerate}
\item[(i)] If $i+j>n$, in particular, $j>\ruling(i)-i$. Take \[x_1:=x_0-a_jZ^{n(v_{i,j})}e_{i+j}.\] Then, $x_1$ is $i$-admissible, and $dx$ is still $(\ruling(i), n(v_{i,\ruling(i)-i}))$-admissible.
 
\item[(ii)] If $i+j\leq n$, i.e. $i+j\in I(v)$, then by definition of $A_{\ruling(i)}$, either $i+j\in L$, or $i+j\in U$ and $\ruling(i+j)>\ruling(i)$. 

\begin{enumerate}
\item[(ii-1)] If $i+j\in L$, denote $k:=\ruling^{-1}(i+j)\in U$, then $k<i+j$ and $\ruling_d(i+j)=n-(i+j-k),$ by Proposition/Definition \ref{def/prop:structure of Morse complexes at a vertex}.
Hence, there exists an $(i+j)$-admissible element of the form 
\[y_{i+j}=e_{i+j}+\sum_{l>j}*_lZ^{n(v_{i+j,l-j})}e_{i+l}\] 
such that $dy_{i+j}=(k,n(v_{i+j,n-(i+j-k)}))$-admissible. 
It follows that $x_1=x_0-a_jZ^{n(v_{i,j})}y_{i+j}$ is $i$-admissible and $dx_1$ is still $(\ruling(i),n(v_{i,\ruling(i)-i})$-admissible, but $j(x_1)>j=j(x_0)$. 

\item[(ii-2)] If $i+j\in U$ and $\ruling(i+j)>\ruling(i)$, then there exists a $(i+j)$-admissible element of the form 
\[
y_{i+j}=e_{i+j}+\sum_{l>j}*_lZ^{n(v_{i+j,l-j})}e_{i+l}
\] 
such that $dy_{i+j}$ is $(\ruling(i+j),n(v_{i+j,\ruling(i+j)-(i+j)})$-admissible. It follows again that $x_1=x_0-a_jZ^{n(v_{i,j})}y_{i+j}$ is $i$-admissible and $dx_1$ is still $(\ruling(i),n(v_{i,\ruling(i)-i}))$-admissible, but $j(x_1)>j(x_0)$.
\end{enumerate}
\end{enumerate}

\noindent{}\emph{Inductive step 2}: If $j(x_1)=\infty$, then $e_i'=x_1$ is a desired element for $(1)$. Otherwise, replace $x_0$ by $x_1$ and repeat the procedure above. Inductively, for some sufficiently large $N$, we obtain in the end an $i$-admissible element of the form $x_N=e_i+\sum_{j\in I_v(i)}a_jZ^{n(v_{i,j})}e_{i+j}$ such that $dx_N$ is $(\ruling(i),n(v_{i,\ruling(i)-i}))$-admissible and $j(x_N)=\infty$. That is, $e_i':=x_N$ is a desired element for $(1)$. This shows the \emph{existence}.\\

\noindent{}\emph{Uniqueness}. We show the uniqueness by induction on $|A_{\ruling}(i)|$. If $A_{\ruling}(i)=\emptyset$, then $e_i'=e_i$, which is clearly unique. For the inductive procedure, assume the uniqueness holds when $|A_{\ruling}(i)|<k$, and consider the case when $|A_{\ruling}(i)|=k$. Let 
\[e_i'=e_i+\sum_{j\in A_{\ruling}(i)}a_{j-i}Z^{n(v_{i,j-i})}e_j\] 
be any element satisfying $(1)$. Since $A_{\ruling}(j)\subsetneq A_{\ruling}(i)$ for all $j\in A_{\ruling}(i)$, by induction we can rewrite 
\[e_i'=e_i+\sum_{j\in A_{\ruling}(i)}b_{j-i}Z^{n(v_{i,j-i})}e_j'.\] 
Here, by the inductive hypothesis, for all $j\in A_{\ruling}(i)$, $e_j'$ is the element in $(1)$ for $j$, uniquely determined by $d$. We want to show the uniqueness of $b_{j-i}$'s.

Assume $A_{\ruling}(i)=\{i_1<i_2<\dots<i_k\}$ and $\ruling(A_{\ruling}(i))=\{j_1<j_2<\dots<j_k\}\subset L$. By definition of $A_{\ruling}(i)$, we know $\ruling(i_l)<\ruling(i)$ for all $1\leq l\leq k$. By the conditions of $(1)$, $de_i'$ is $(\ruling(i),n(v_{i,\ruling(i)-i}))$-admissible, hence $\langle de_i',Z^{n(v_{i,\ruling(i_p)-i})}e_{\ruling(i_p)}\rangle=0$ for all $1\leq p\leq k$. That is, the following system of linear equations for $\{b_{i_j-i}\}_{1\leq j\leq k}$ holds:
\begin{equation}
\left(\epsilon_{p,q}\right)_{p,q}\left(b_{i_q-i}\right)_{q}=\left(-\langle Z^{n(v_{i,\ruling(i_p)-i})}e_{\ruling(i_p)},de_i\rangle\right)_p,
\end{equation}
where 
\begin{equation}
\epsilon_{p,q}=
\begin{cases}
(-1)^{n(v_{i,i_q-i})}\langle Z^{n(v_{i_q,\ruling(i_p)-i_q})}e_{\ruling(i_p)},de_{i_q}'\rangle & \textrm{if $i_q<\ruling(i_p)$};\\
0 & \textrm{otherwise}.
\end{cases}
\end{equation}
\emph{Denote} $i_q':=\ruling^{-1}(j_q)\in\{i_1,i_2,\ldots,i_k\}$. \emph{Define}
\[
\epsilon_{p,q}'=
\begin{cases}
(-1)^{n(v_{i,i_q'-i})}\langle Z^{n(v_{i_q',j_p-i_q'})}e_{j_p},de_{i_q'}'\rangle & \textrm{if $i_q'<j_p$};\\
0 & \textrm{otherwise}.
\end{cases}
\]
Then the following two coefficient matrices are similar:
\[
\left(\epsilon_{p,q}\right)_{p,q}\sim\left(\epsilon_{p,q}'\right)_{p,q}
\]
And by definition, $de_{i_q'}'$ is $(j_q,n(v_{i_q',j_q-i_q'}))$-admissible, hence $\langle Z^{n(i_q',j_p-i_q')}e_{j_p},de_{i_q'}'\rangle=0$ if $p<q$ and $\langle Z^{n(v_{i_q',j_q-i_q'})}e_{j_q},de_{i_q'}'\rangle\neq 0$. Therefore, the square matrix $\left(\epsilon_{p,q}'\right)_{p,q}$ is lower triangular and invertible, then so is $\left(\epsilon_{p,q}\right)_{p,q}$. It follows that
\begin{equation}
\left(b_{i_q-i}\right)_{q}=\left(\epsilon_{p,q}\right)_{p,q}^{-1}\left(-\langle Z^{n(v_{i,\ruling(i_p)-i})}e_{\ruling(i_p)},de_i\rangle\right)_p
\end{equation}
By induction, $e_{i_q}'$'s are uniquely determined by $d$, hence so is the right hand side. The uniqueness in $(1)$ then follows. The previous equation also shows by induction that $e_i'=e_i'(d)$ depends algebraically on $d$. This shows $(1)$.
\end{proof}

\begin{proof}[Proof of {\rm (2)}]
Take any $j\in L$, then there is $i\in U$ such that  $j=\ruling(i)$. By (1), $(-1)^{\mu(i)}de_i'=c_jZ^{n(v_{i,j-i})}e_j'$ for some $c_j\in\field^*$ and some $j$-admissible element of the form 
\[e_j'=e_j+\sum_{l>0}*_lZ^{n(v_{j,l})}e_{j+l}.\] 
Then $(1)$ shows that both $c_j$ and $e_j'$ are uniquely determined, and depend algebraically on $d$. 

Now, define an unipotent isomorphism $\varphi_0(d)\in U(v;\field)\subset B(v;\field)$ by 
\[\varphi_0(d)(e_i):=e_i',\qquad \textrm{ for all } 1\leq i\leq n.\] 
It follows that $\varphi_0:B(v;\field)\cdot d_{\ruling}\rightarrow U(v;\field)$ defines a canonical algebraic map. Moreover, for any $d\in B(v;\field)\cdot d_{\ruling}$ and any $i\in U$, have
\begin{align*}
(\varphi_0(d)^{-1}\cdot d)(e_i)&=\varphi_0(d)^{-1}\circ d\circ \varphi_0(d)(e_i)
=\varphi_0(d)^{-1}\circ d(e_i')\\
&=c_{\ruling(i)}\varphi_0(d)^{-1}((-1)^{\mu(i)}Z^{n(v_{i,\ruling(i)-i})}e_{\ruling(i)}')
=c_{\ruling(i)}(-1)^{\mu(i)}Z^{n(v_{i,\ruling(i)-i})}e_{\ruling(i)}.
\end{align*}
Now, for the canonical section of $\pi_{\ruling}:B(v;\field)\rightarrow B(v;\field)\cdot d_{\ruling}$, we simply take 
\begin{align*}
\varphi_{\ruling}(d):=D(d)\circ \varphi_0(d),
\end{align*}
where
\begin{align}\label{eqn:D(d)}
\begin{cases}
D(d)(e_i')\coloneqq e_i' &\text{ for }i\in U;\\
D(d)(e_j')\coloneqq c_je_j' &\text{ for }j\in L.
\end{cases}
\end{align}
It follows that $(\varphi_{\ruling}(d)^{-1}\cdot d)(e_i)=(-1)^{\mu(i)}Z^{n(v_{i,\ruling(i)-i})}e_{\ruling(i)}=d_{\ruling}(e_i)$ for all $i\in U$. By the condition \[((\varphi_{\ruling}(d)^{-1})\cdot d)^2+Z^2=0,\] we then see that $(\varphi_{\ruling}(d)^{-1})\cdot d=d_{\ruling}$, i.e. $\varphi_{\ruling}(d)\cdot d_{\ruling}=d$, as desired. This shows (2).
\end{proof}

\begin{proof}[Proof of {\rm (3)}]
By (2), there is an identification between $d\in B(v;\field)\cdot d_{\ruling}$ and $\varphi_{\ruling}(d)$, where $\varphi_{\ruling}$ is the canonical section of $\pi_{\ruling}:B(v;\field)\rightarrow B(v;\field)\cdot d_{\ruling}$. Use the notations in the proof of (2) above, the general form of $\varphi_{\ruling}(d)$ is $\varphi_{\ruling}(d)=D(d)\circ \varphi_0(d)$, where 
\begin{align*}
\begin{dcases}
\varphi_0(d)(e_i)=e_i'=e_i+\sum_{j\in A_{\ruling}(i)}*_{ij}Z^{n(v_{i,j-i})}e_j, &\text{ for } i\in U;\\
\varphi_0(d)(e_i)=e_i+\sum_{j\in I_v(i)}\star_{ij}Z^{n(v_{i,j})}e_{i+j}, &\text{ for }i\in L.
\end{dcases}
\end{align*}
Together with \eqref{eqn:D(d)}, it follows that
\begin{align*}
B(v;\field)\cdot d_{\ruling}
&\isomorphic \{(*_{ij})_{i,j} \mid i\in U, j\in A_{\ruling}(i), *_{ij}\in \field \}\\
&\mathrel{\hphantom{\isomorphic}}\times \{ (\star_{ij})_{i,j} \mid i\in L, j\in I_v(i), \star_{ij}\in \field \} \\
&\mathrel{\hphantom{\isomorphic}}\times \{ (c_i)_i \mid i\in L, c_i \in \field^*\} \\
&\isomorphic (\field^*)^{\frac{\valency}{2}}\times \field^{A_v(\ruling)}.
\end{align*}
This shows (3).
\end{proof}

\subsection{Augmentations for elementary bordered Legendrian graphs with a vertex}

In this and the next subsection, we study the structure of the augmentation variety associated to any elementary bordered Legendrian graph with a vertex of general type, i.e. whose only singularity is a vertex of any type.
In the end, generalizing \cite[Thm.5.10]{Su2017}, we'll see this leads naturally to a ruling decomposition for the augmentation variety associated to any bordered Legendrian graph, at least when we impose a base point at each right cusp, and at each left half-edge of any vertex. 

Let $(\cV,\bfmu)\in \BLT^\mu_\front$ be an elementary bordered Legendrian graph in Definition~\ref{definition:elementary bordered Legendrian graphs} of type $(n_\Left,n_\Right)$ containing a vertex $v$ of type $(\ell,r)$.
Let us impose a base point on each left-half edges of $v$.
The borders and the additional base points are labelled from top to bottom, see Figure \ref{fig:vertex} (left).
Consider Ng's resolution $\Res(\cV)\in\BLT^\mu_\lag$ of $\cV$. Note that there are $\binom{\ell}2$-many additional crossings induced from Ng's resolution. Let us label them as in Figure \ref{fig:vertex} (right).

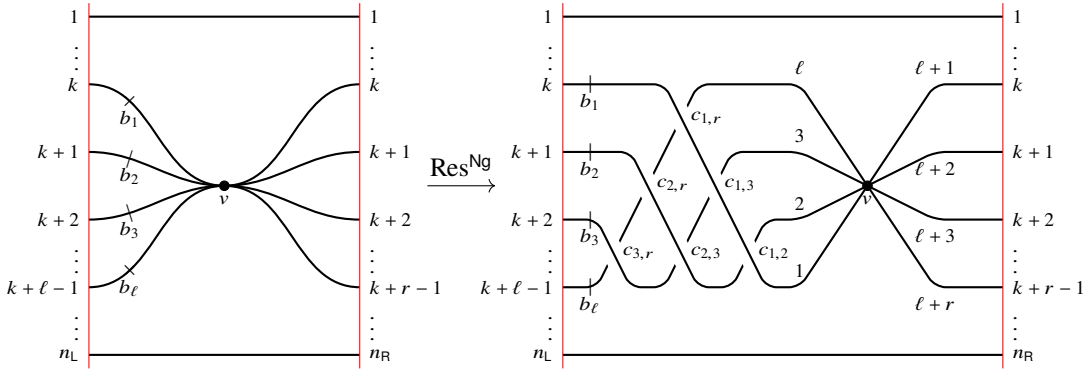
\begin{figure}[ht]
\[
\begin{tikzpicture}[scale=0.9,baseline=-.5ex]
\begin{scope}[xshift=-2cm]
\draw[thick](-1,1.5) node[left] {\scriptsize$k$} to[out=0,in=180] (1,0);
\draw(-0.4,1.25) node{\scriptsize$\times$} node[below] {\scriptsize$b_1$};
\draw[thick](1,0)  to[out=0,in=180] (3,1.5) node[right] {\scriptsize$k$};
\draw[thick](-1,.5) node[left] {\scriptsize$k+1$} to[out=0,in=180] (1,0);
\draw(-0.4,.4) node{\scriptsize$/$} node[below] {\scriptsize$b_2$};
\draw[thick](1,0)  to[out=0,in=180] (3,0.5) node[right] {\scriptsize$k+1$};
\draw[thick](-1,-.5) node[left] {\scriptsize$k+2$} to[out=0,in=180] (1,0);
\draw(-0.4,-.4) node{\scriptsize$\textbackslash$} node[below] {\scriptsize$b_3$};
\draw[thick](1,0)  to[out=0,in=180] (3,-0.5) node[right] {\scriptsize$k+2$};
\draw[thick](-1,-1.5) node[left] {\scriptsize$k+\ell-1$} to[out=0,in=180] (1,0);
\draw(-0.4,-1.25) node{\scriptsize$\times$} node[below] {\scriptsize$b_\ell$};
\draw[thick](1,0)  to[out=0,in=180] (3,-1.5) node[right] {\scriptsize$k+r-1$};
\draw[thick](-1, 2.5) node[left]{\scriptsize$1$} -- (3,2.5) node[right]{\scriptsize$1$};
\draw[thick](-1, -2.5) node[left]{\scriptsize$n_\Left$} -- (3,-2.5) node[right]{\scriptsize$n_\Right$};
\draw[fill] (1,0) circle (2pt) node[below] {\scriptsize$v$};
\draw[red] (-1,2.7)--(-1,-2.7);
\draw[red] (3,2.7)--(3,-2.7);
\draw (-1,2) node[left]{\scriptsize$\vdots$};
\draw (3,2) node[right]{\scriptsize$\vdots$};
\draw (-1,-1) node[left]{\scriptsize$\vdots$};
\draw (3,-1) node[right]{\scriptsize$\vdots$};
\draw (-1,-2) node[left]{\scriptsize$\vdots$};
\draw (3,-2) node[right]{\scriptsize$\vdots$};
\end{scope}
\draw[->](2,0)--(3,0);
\draw(2.5,0) node[above] {$\Res$};
\begin{scope}[xshift=5cm]
\draw[thick,rounded corners](-1,-1.5) node[left] {\scriptsize$k+\ell-1$} -- (-0.5,-1.5) -- (1,1.5) -- (2.5,1.5) node[above]{\scriptsize$\ell$} -- (3.5,0);
\draw(-.6,-1.5) node{\scriptsize$|$} node[below] {\scriptsize$b_\ell$};
\draw[line width=8pt,white,rounded corners](-1,-.5) -- (-0.5,-.5) -- (0,-1.5) -- (0.5,-1.5) -- (1.5,.5) -- (2.5,0.5);
\draw[thick,rounded corners](-1,-.5) node[left] {\scriptsize$k+2$}-- (-0.5,-.5) -- (0,-1.5) -- (0.5,-1.5) -- (1.5,.5) -- (2.5,0.5)node[above]{\scriptsize$3$} -- (3.5,0);
\draw(-.6,-.5)  node{\scriptsize$|$} node[below] {\scriptsize$b_3$};
\draw[line width=8pt,white,rounded corners](-1,.5) -- (0,.5) -- (1,-1.5) -- (1.5,-1.5) -- (2, -0.5) -- (2.5,-0.5);
\draw[thick,rounded corners](-1,.5)node[left] {\scriptsize$k+1$} -- (0,.5) -- (1,-1.5) -- (1.5,-1.5) -- (2, -0.5) -- (2.5,-0.5) node[above]{\scriptsize$2$} -- (3.5,0);
\draw(-.6,.5)  node{\scriptsize$|$} node[below] {\scriptsize$b_2$};
\draw[line width=8pt,white,rounded corners](-1,1.5) -- (0.5,1.5) -- (2,-1.5) -- (2.5,-1.5);
\draw[thick,rounded corners](-1,1.5) node[left] {\scriptsize$k$} -- (0.5,1.5) -- (2,-1.5) -- (2.5,-1.5) node[above]{\scriptsize$1$} -- (3.5,0);
\draw(-.6,1.5) node{\scriptsize$|$} node[below] {\scriptsize$b_1$};
\draw[thick,rounded corners] (3.5,0) -- (4.5,1.5) node[above]{\scriptsize$\ell+1$} -- (5.5,1.5) node[right]{\scriptsize$k$};
\draw[thick,rounded corners] (3.5,0) -- (4.5,0.5)node[below]{\scriptsize$\ell+2$} -- (5.5,0.5) node[right]{\scriptsize$k+1$};
\draw[thick,rounded corners] (3.5,0) -- (4.5,-0.5)node[below]{\scriptsize$\ell+3$} -- (5.5,-0.5)node[right]{\scriptsize$k+2$};
\draw[thick,rounded corners] (3.5,0) -- (4.5,-1.5) node[below]{\scriptsize$\ell+r$}-- (5.5,-1.5) node[right]{\scriptsize$k+r-1$};
\draw[thick](-1, 2.5) node[left]{\scriptsize$1$} -- (5.5,2.5) node[right]{\scriptsize$1$};
\draw[thick](-1, -2.5) node[left]{\scriptsize$n_\Left$} -- (5.5,-2.5) node[right]{\scriptsize$n_\Right$};
\draw[red] (-1,2.7)--(-1,-2.7);
\draw[red] (5.5,2.7)--(5.5,-2.7);
\draw (-1,2) node[left]{\scriptsize$\vdots$};
\draw (5.5,2) node[right]{\scriptsize$\vdots$};
\draw (-1,-1) node[left]{\scriptsize$\vdots$};
\draw (5.5,-1) node[right]{\scriptsize$\vdots$};
\draw (-1,-2) node[left]{\scriptsize$\vdots$};
\draw (5.5,-2) node[right]{\scriptsize$\vdots$};
\draw[fill] (3.5,0) circle (2pt) node[below] {\scriptsize$v$};
\draw (1.75,-1) node[right] {\scriptsize$c_{1,2}$};
\draw (1.25,0) node[right] {\scriptsize$c_{1,3}$};
\draw (0.75,1) node[right] {\scriptsize$c_{1,r}$};
\draw (0.75,-1) node[right] {\scriptsize$c_{2,3}$};
\draw (0.25,0) node[right] {\scriptsize$c_{2,r}$};
\draw (-0.25,-1) node[right] {\scriptsize$c_{3,r}$};
\end{scope}
\end{tikzpicture}
\]
\caption{Bordered Legendrian graph $V$ and $\Res(V)$.}
\label{fig:vertex}
\end{figure}

\subsubsection{Augmentations for $\cV$}
Recall the construction of DGA $A=(\alg,\differential)\coloneqq A^\CE(V,\mu)$.
The algebra $\alg$ is a unital associative algebra over $\ZZ[ t_i^{\pm 1}\mid 1\leq i\leq \ell]$ freely generated by the left border generators, crossings, and vertex generators:
\begin{align*}
\sfG&=\{ a_{i,j}\mid 1\leq i<j\leq n_\Left\} \amalg \{ c_{i,j} \mid 1\leq i<j\leq \ell \}\amalg \{v_{i,j} \mid i\in\Zmod{\valency}, j>0 \};\\
\alg&=\ZZ[ t_i^{\pm 1}\mid 1\leq i\leq \ell]\langle \sfG \rangle,
\end{align*} 
where $t_i^{\pm 1}$ corresponds to the base point $b_i$. The grading is given by: 
\begin{align*}
|t_i^{\pm 1}|&=0,& |a_{i,j}|&=\mu_\Left(i)-\mu_\Left(j)-1,& |c_{i,j}|&=\mu(i)-\mu(j),& |v_{i,j}|&=\mu(i)-\mu(j)-1+n(v_{i,j})
\end{align*} 
as in Section \ref{sec:aug_variety_vertex}. 
By definition, the differential is given by $\differential t_i^{\pm 1}=0$ and 
\begin{align*}
\differential a_{i,j}&=\sum_{i<k<j}(-1)^{|a_{i,k}|+1}a_{i,k}a_{k,j},\nonumber\\
\differential v_{i,j}&=\delta_{j,\valency}+\sum_{0<k<j}(-1)^{|v_{i,k}|+1}v_{i,k}v_{i+k,j-k},\nonumber\\
\differential c_{p,q}&=t_p^{-\sigma_p}a_{k+p-1,k+q-1}t_q^{\sigma_q}+\sum_{p<o<q}t_p^{-\sigma_p}a_{k+p-1,k+o-1}t_o^{\sigma_o}c_{o,q}\nonumber\\
&\mathrel{\hphantom{=}}+(-1)^{|c_{p,q}|+1}\left(v_{p,q-p}+\sum_{p<o<q}c_{p,o}v_{o,q-o}\right)
\end{align*} 
where $\sigma_i\coloneqq (-1)^{\mu(i)}$.

Then by Theorem/Definition~\ref{theorem:well-definedness of a bordered DGA}, we have a diagram of DGAs 
\begin{equation*}
A(\cV)\coloneqq(A_\Left\stackrel{\phi_\Left}{\longrightarrow} A\stackrel{\phi_\Right}{\longleftarrow} A_\Right),
\end{equation*}
where $\phi_\Left$ is the natural inclusion of the DG-subalgebra $A_\Left$ generated by $a_{i,j}$'s, and 
\[
A_\Right=\left(\ZZ\langle b_{i,j}:1\leq i<j\leq n_R\rangle, \partial_\Right \right)
\] 
is the DGA generated by the right border generators as in Example/Definition~\ref{example:border DGA}. 
See, (\ref{equation:sign corrected differential}) for the construction of $\phi_\Right$.

Let $A\coloneqq(a_{i,j})_{1\leq i,j\leq n_\Left}$ and $B\coloneqq(b_{i,j})_{1\leq i,j\leq n_\Right}$ be the strictly upper-triangular matrix with entries $a_{i,j}$'s and $b_{i,j}$'s, with $a_{i,j}=0$ and $b_{i,j}=0$ for $i\geq j$, respectively. We write them in the following block matrices
\begin{align*}
A&=\begin{pmatrix}
A_{1,1} & A_{1,2} & A_{1,3}\\
0 & A_{2,2} & A_{2,3}\\
0 & 0 & A_{3,3}
\end{pmatrix},&
B&=\begin{pmatrix}
B_{1,1} & B_{1,2} & B_{1,3}\\
0 & B_{2,2} & B_{2,3}\\
0 & 0 & B_{3,3}
\end{pmatrix},
\end{align*}
where $A_{2,2}=(a_{k+p-1,k+q-1})_{1\leq p,q\leq \ell}$ and $B_{2,2}=(b_{k+p-1,k+q-1})_{1\leq p,q\leq r}$. The other $A_{\bullet,\bullet}$'s and $B_{\bullet,\bullet}$'s are defined in the same way.
 
Let $\tilde{c}_{i,j}\coloneqq(-1)^{|c_{i,j}|+1}c_{i,j}$ for the crossing generators $\{ c_{i,j} \mid 1\leq i<j\leq \ell \}$, and define 
\begin{align*}
S&\coloneqq \diag(t_1^{\sigma_1},\ldots,t_{\ell}^{\sigma_{\ell}}),& 
C&\coloneqq(c_{p,q})_{1\leq p,q\leq \ell},&
\tilde{C}&\coloneqq(\tilde{c}_{p,q})_{1\leq p,q\leq \ell}.
\end{align*}
Similarly define matrices $V_{i,j}$, $i\in\ZZ/2$, $j\geq 0$ for the vertex generators $\{v_{i,j} \mid i\in\ZZ/\valency, j>0 \}$  by
\begin{equation*}
\begin{pmatrix}
V_{1,2j} & V_{1,2j+1}\\
V_{2,2j+1} & V_{2,2j}
\end{pmatrix}\coloneqq
\begin{pmatrix}
(v_{p,j\valency+q-p})_{1\leq p,q\leq \ell} & (v_{p,j\valency+\ell+q-p})_{1\leq p\leq \ell, 1\leq q\leq r}\\
(v_{\ell+p,j\valency+r+q-p})_{1\leq p\leq r, 1\leq q\leq \ell} & (v_{\ell+p,j\valency+q-p})_{1\leq p,q\leq r}
\end{pmatrix}.
\end{equation*}
Then we have:
\begin{equation}\label{eqn:iota_R(B)}
\phi_\Right(B)=\begin{pmatrix}
A_{1,1} & A_{1,2}S(I+C)V_{1,1}& A_{1,3}+A_{1,2}S(I+C)V_{1,2}\left(\displaystyle\sum_{i\geq 0}\tilde{C}^i\right)S^{-1}A_{2,3}\\
0 & V_{2,0} & V_{2,1}\left(\displaystyle\sum_{i\geq 0}\tilde{C}^i\right)S^{-1}A_{2,3}\\
0 & 0 & A_{3,3}
\end{pmatrix}
\end{equation}

Dualizing $A(\cV)$, we then obtain a diagram of augmentation varieties:
\begin{equation*}
\aug(\cV;\field)\coloneqq \left(\aug(V_\Left;\field)\xleftarrow[]{r_\Left}\aug(V;\field)\xrightarrow[]{r_\Right}\aug(V_\Right;\field)\right)
\end{equation*}
Also, the inclusion $i:I_v\hookrightarrow A$ of the DG-subalgebra $I_v$ as in Section \ref{sec:aug_variety_vertex} induces a map 
\begin{equation*}
i^*:\aug(V;\field)\rightarrow\aug(v;\field).
\end{equation*}

\begin{lemma}
As in the above setup, any augmentation $\epsilon\in\aug(V;\field)$ is equivalent to the following:
\begin{align*}
\epsilon_\Left\coloneqq r_\Left(\epsilon)&\in\aug(V_\Left;\field),&
\epsilon_v \coloneqq i^*(\epsilon)&\in\aug(v;\field);\\
\epsilon(c_{p,q})&\in\field , 1\leq p\leq \ell,&
\epsilon(t_i)&\in\field^\times, 1\leq i\leq \ell,
\end{align*}
satisfying $\epsilon(c_{p,q})=0$ if $|c_{p,q}|\neq 0$, and for $|c_{p,q}|=1$ we have:
\begin{align*}
0&=\epsilon(t_p^{-\sigma_p})\epsilon_\Left(a_{k+p-1,k+q-1})\epsilon(t_q^{\sigma_q})+\sum_{p<o<q}\epsilon(t_p^{-\sigma_p})\epsilon_\Left(a_{k+p-1,k+o-1})\epsilon(t_o^{\sigma_o})\epsilon(c_{o,q})\\
&\mathrel{\hphantom{=}} +\epsilon_v(v_{p,q-p})+\sum_{p<o<q}\epsilon(c_{p,o})\epsilon_v(v_{o,q-o}).
\end{align*}
Equivalently,
\begin{equation}\label{eqn:u_C}
\epsilon_\Left(A_{2,2})\epsilon(S)(I+\epsilon(C))=-\epsilon(S)(I+\epsilon(C))\epsilon_v(V_{1,0}).
\end{equation}
\end{lemma}

\subsubsection{An identification via Morse complexes}
We'll see immediately that the condition for $\epsilon(t_p^{\sigma_p})$ and $\epsilon(c_{p,q})$ has a simple interpretation in terms of Morse complexes.

\begin{definition}\label{def:matrix notations at V}
Let $(V,\mu)$ be a bordered Legendrian graph involving a single vertex $v$, and let $\epsilon$ be any augmentation of $A(V)$. We make the following definitions:
\begin{enumerate}
\item
For the trivial tangle $V_\Left$, define
\begin{align*}
C_1(V_\Left)&\coloneqq\bigoplus_{1\leq i<k}\field\cdot e_i^\Left,&
C_2(V_\Left)&\coloneqq\bigoplus_{1\leq p\leq \ell}\field\cdot e_{k+p-1}^\Left,&
C_3(V_\Left)&\coloneqq\bigoplus_{k+\ell\leq j\leq n_\Left}\field\cdot e_j^\Left,
\end{align*} 
where the grading is given by $|e_i^\Left|\coloneqq-\mu_\Left(i)$ as usual. Recall that, by Definition \ref{def:Morse complex}, we have 
\[
C(V_\Left)\isomorphic C_1(V_\Left)\oplus C_2(V_\Left)\oplus C_3(V_\Left).
\] 
Similarly define $C_a(V_\Right)$ similarly for $1\leq a\leq 3$, and we have $C(V_\Right)\isomorphic\bigoplus_{a=1}^3C_a(V_\Right)$.

Let us denote column vectors and diagonal matrices
\begin{align*}
\vec{e}^\Left&\coloneqq(e_1^\Left,\ldots,e_{n_\Left}^\Left)^t,&
\vec{e}^{\Left,1}&\coloneqq(e_1^\Left,\ldots,e_{k-1}^\Left)^t;\\
\vec{e}^{\Left,2}&\coloneqq(e_k^\Left,\ldots,e_{k+\ell-1}^\Left)^t,&
\vec{e}^{\Left,3}&\coloneqq(e_{k+\ell}^\Left,\ldots,e_{n_\Left}^\Left)^t;\\
(-1)^{\mu_\Left}&\coloneqq
\diag\left((-1)^{\mu_\Left(1)},\ldots,(-1)^{\mu_\Left(n_\Left)}\right),&
(-1)^{\mu_{\Left,1}}&\coloneqq
\diag\left((-1)^{\mu_\Left(1)},\ldots,(-1)^{\mu_\Left(k-1)}\right);\\ 
(-1)^{\mu_{\Left,2}}&\coloneqq
\diag\left((-1)^{\mu_\Left(k)},\ldots,(-1)^{\mu_\Left(k+\ell-1)}\right),&
(-1)^{\mu_{\Left,3}}&\coloneqq
\diag\left((-1)^{\mu_{\Left}(k+\ell)},\ldots,(-1)^{\mu_{\Left}(n_\Left)}\right).
\end{align*}
Define $\vec{e}^{\Right,i}, \vec{e}^\Right$ and  $(-1)^{\mu_\Left,i}, (-1)^{\mu_\Right}$ similarly, with $\ell, n_\Left$ replaced by $r,n_\Right$ respectively.

\item
Let us define 
\begin{align*}
d^\Left&\coloneqq d(\epsilon_\Left)=\sum_{1\leq i<j\leq n_\Left}(-1)^{\mu_\Left(i)}\epsilon_\Left(a_{i,j})e_j^\Left\otimes (e_i^\Left)^*;\\
d^\Right&\coloneqq d(\epsilon_\Right)=\sum_{1\leq i<j\leq n_\Right}(-1)^{\mu_\Right(i)}\epsilon_\Right(a_{i,j})e_j^\Right\otimes (e_i^\Right)^*.
\end{align*}
Equivalently, we have 
\begin{align*}
d^\Left(\vec{e}^\Left)&=(-1)^{\mu_\Left}\epsilon_\Left(A)\cdot\vec{e}^\Left,&
d^\Right(\vec{e}^\Right)&=(-1)^{\mu_\Right}\epsilon_\Right(A)\cdot\vec{e}^\Right,
\end{align*} 
where the right hand side is matrix multiplication.
We can then denote:
\begin{align*}
d^\Left&=\begin{pmatrix}
d_{1,1}^\Left & 0 & 0\\
d_{2,1}^\Left & d_{2,2}^\Left & 0\\
d_{3,1}^\Left & d_{3,2}^\Left & d_{3,1}^\Left
\end{pmatrix},&
d^\Right&=\begin{pmatrix}
d_{1,1}^\Right & 0 & 0\\
d_{2,1}^\Right & d_{2,2}^\Right & 0\\
d_{3,1}^\Right & d_{3,2}^\Right & d_{3,1}^\Right
\end{pmatrix}
\end{align*}
according to the decomposition $C(V_\Left)\isomorphic C_1(V_\Left)\oplus C_2(V_\Left)\oplus C_3(V_\Left)$, and that of $C(V_\Right)$. Equivalently, we have:
\begin{align*}
d_{j,i}^\Left(\vec{e}^{\Left,i})&=(-1)^{\mu_{\Left,i}}\epsilon_\Left(A_{i,j})\cdot\vec{e}^{\Left,j}\\
d_{j,i}^\Right(\vec{e}^{\Right,i})&=(-1)^{\mu_{\Right,i}}\epsilon_\Right(B_{i,j})\cdot\vec{e}^{\Right,j}
\end{align*}

\item
For the vertex $v$, define 
\begin{align*}
C(v_\Left)&\coloneqq\bigoplus_{1\leq i\leq \ell}\field\cdot e_i^v,& 
C(v_\Right)&\coloneqq\bigoplus_{\ell+1\leq i\leq \ell+r}\field\cdot e_i^v,
\end{align*}
where $|e_i^v|\coloneqq-\mu(i)$. Recall that, by Definition \ref{def:Morse complex at a vertex}, we have 
\[
\bar C_v\coloneqq C_v/(Z=0)\isomorphic C(v_\Left)\oplus C(v_\Right).
\]

\noindent{Denote} for $i\in\ZZ/2$
\[
\left\{\begin{array}{lll}
\vec{e}^{v,i}\coloneqq (e_1^v,\ldots,e_\ell^v)^t,\quad & (-1)^{\mu_i}\coloneqq \diag\left((-1)^{\mu(1)},\ldots,(-1)^{\mu(\ell)}\right) & \textrm{ if $i$ is odd};\\
\vec{e}^{v,i}\coloneqq (e_{\ell+1}^v,\ldots,e_{\ell+r}^v)^t,\quad & (-1)^{\mu_i}\coloneqq \diag\left((-1)^{\mu(\ell+1)},\ldots,(-1)^{\mu(\ell+r)}\right) & \textrm{ if $i$ is even}.
\end{array}\right.
\]

\item
Let us define $d^v\coloneqq d(\epsilon_v)$ as in Lemma \ref{lem:aug and Morse complex for a vertex}, 
\[
d^v=\sum_{i\in\Zmod{\valency}, j>0}(-1)^{\mu(i)}Z^{n(v_{i,j})}e_{i+j}^v\otimes (e_i^v)^*.
\]
We can rewrite $d^v$ as 
\[
d^v= \sum_{i\in\Zmod{2},j\geq 0}Z^jd_{i,j}^v
\] 
where $d_{i,j}^v\in\hom^{1-j}(C(v_i),C(v_{i+j}))$ is given by
\begin{align*}
d_{i,j}^v(\vec{e}^{v,i})=(-1)^{\mu_i}\epsilon_v(V_{i,j})\cdot\vec{e}^{v,i+j}.
\end{align*}

\item
We will always use the following \emph{identifications}:
\begin{align*}
C(v_\Right)&\isomorphic C_2(V_\Right) 
\quad\textrm{ via $\vec{e}^{v,2}=\vec{e}^{\Right,2}$};\\
C_i(V_\Left)&\isomorphic C_i(V_\Right)\quad \textrm{ via $\vec{e}^{\Left,i}=\vec{e}^{\Right,i}$, for $i=1,3$}.
\end{align*}
Notice that the gradings are also preserved.

\item
Define $g_C=g(\epsilon(S),\epsilon(C)): C_2(V_\Left)\xrightarrow[]{\sim}C(v_\Left)$ by 
\[
(-1)^{\mu_\Left(k+i-1)}g_C(e_{k+i-1}^\Left)\coloneqq \epsilon(t_i^{\sigma_i})e_i^v+\sum_{i<j}\epsilon(t_i^{\sigma_i}c_{i,j})e_j^v.
\]
Equivalently, we have
\begin{align*}
g_C(\vec{e}^{\Left,2})=(-1)^{\mu_{\Left,2}}\epsilon(S)(I+\epsilon(C))\vec{e}^{v,1}. 
\end{align*}
Clearly, $g_C\in B(C_2(V_\Left);\field)$, under the identification $C_2(V_\Left)\isomorphic C(v_\Left)$ via $e_{k+p-1}^\Left=e_p^v$.
\end{enumerate}

\end{definition}

\begin{remark}\label{rem:identification for aug var of V}

Now, observe that 
\begin{align*}
(-1)^{\mu_{\Left,2}}\epsilon(S)(I+\epsilon(C))&=\epsilon(S)(I+\epsilon(C))(-1)^{\mu_1};\\(-1)^{\mu_{\Left,2}}\epsilon_\Left(A_{2,2})&=-\epsilon_\Left(A_{2,2})(-1)^{\mu_{\Left,2}}, 
\end{align*}
then the equation (\ref{eqn:u_C}) can be re-written as:
\begin{equation}\label{eqn:twistor u_C}
g_C\circ d_{2,2}^\Left = d_{1,0}^v\circ g_C
\end{equation}

As a consequence, we obtain an identification
\begin{equation*}
\aug(V;\field)\isomorphic \MC(V;\field)
\end{equation*}
where $\MC(V;\field)$ is the variety of triples $(d^\Left,g_C,d^v)$ such that, $d^\Left$ defines a Morse complex for $V_\Left$, $d^v$ defines a Morse complex for $v$, and Equation (\ref{eqn:twistor u_C}) holds with $g_C\in B(C_2(V_\Left);\field)$. 
\end{remark}

\subsubsection{Induced morphisms between Morse complexes}
By definition, $(d^\Left)^2=0$ is equivalent to 
\begin{equation}\label{eqn:d^L}
\begin{cases}
(d_{i,i}^\Left)^2=0, &1\leq i\leq 3\\
d_{2,2}^\Left\circ d_{2,1}^\Left+d_{2,1}^\Left\circ d_{1,1}^\Left=0,\\
d_{3,3}^\Left\circ d_{3,2}^\Left+d_{3,2}^\Left\circ d_{2,2}^\Left=0,\\
d_{3,2}^\Left\circ d_{2,1}^\Left+d_{3,1}^\Left\circ d_{1,1}^\Left +d_{3,3}^\Left\circ d_{3,1}^\Left=0.
\end{cases}
\end{equation}
Similarly for $d^\Right$.
In addition, by a direct calculation, $(d^v)^2+Z^2=0$ is equivalent to 
\begin{equation}\label{eqn:d^2+T^2=0}
\sum_{j+k=m}(-1)^{k}d_{i+k,j}^v\circ d_{i,k}^v+\delta_{m,2}\identity=0
\end{equation}
for all $m\geq 0$. In particular, it implies that:
\begin{equation}\label{eqn:d_v^2+T^2=0}
\begin{cases}
(d_{i,0}^v)^2=0, & i=1,2\\
d_{1,1}^v\circ d_{1,0}^v-d_{2,0}^v\circ d_{1,1}^v=0,\\
d_{2,1}^v\circ d_{2,0}^v-d_{1,0}^v\circ d_{2,1}^v=0,\\
d_{1,2}^v\circ d_{1,0}^v+d_{1,0}^v\circ d_{1,2}^v-d_{2,1}^v\circ d_{1,1}^v+\identity=0,\\
d_{2,2}^v\circ d_{2,0}^v+d_{2,0}^v\circ d_{2,2}^v-d_{1,1}^v\circ d_{2,1}^v+\identity=0.
\end{cases}
\end{equation}

On the other hand, by the formula (\ref{eqn:iota_R(B)}) for $\iota_\Right(B)$, we obtain
\[
\epsilon_\Right(B)=
\begin{pmatrix}
\epsilon(A_{1,1}) & \epsilon(A_{1,2})\epsilon(S+SC)\epsilon(V_{1,1}) & \epsilon(A_{1,3})+\epsilon(A_{1,2})\epsilon(S+SC)\epsilon(V_{1,2})(\epsilon(S+SC))^{-1}\epsilon(A_{2,3})\\
0 & \epsilon(V_{2,0}) & \epsilon(V_{2,1})(\epsilon(S+SC))^{-1}\epsilon(A_{2,3})\\
0 & 0 & \epsilon(A_{3,3})
\end{pmatrix}
\]
By the identifications in Definition \ref{def:matrix notations at V}.(5), it follows that
\begin{equation}\label{eqn:d^R}
\begin{pmatrix}
d_{1,1}^\Right & 0 & 0\\
d_{2,1}^\Right & d_{2,2}^\Right & 0\\
d_{3,1}^\Right & d_{3,2}^\Right & d_{3,3}^\Right
\end{pmatrix}=
\begin{pmatrix}
d_{1,1}^\Left& 0 & 0\\
\tilde{d}_{1,1}^v\circ d_{2,1}^\Left& d_{2,0}^v& 0\\
d_{3,1}^\Left+d_{3,2}^\Left\circ\tilde{d}_{1,2}^v\circ d_{2,1}^\Left& d_{3,2}^\Left\circ\tilde{d}_{2,1}^v& d_{3,3}^\Left
\end{pmatrix}
\end{equation}
where for all $j\geq 0$, we \emph{define} 
\begin{equation}
\begin{pmatrix}
\tilde{d}_{1,2j}^v & \tilde{d}_{2,2j+1}^v\\
\tilde{d}_{1,2j+1} & \tilde{d}_{2,2j}
\end{pmatrix}\coloneqq
\begin{pmatrix}
g_C^{-1} & 0\\
0 & \identity
\end{pmatrix}\circ
\begin{pmatrix}
d_{1,2j}^v & d_{2,2j+1}^v\\
d_{1,2j+1} & d_{2,2j}
\end{pmatrix}\circ
\begin{pmatrix}
g_C & 0\\
0 & \identity
\end{pmatrix}
\end{equation}

\begin{proposition/definition}
\emph{Define} $C(V)\coloneqq C(V_\Left)\oplus C(v_\Right)$. Let $\epsilon$ be any augmentation of $\cA(V)$. 
\begin{enumerate}
\item
The equation \eqref{eqn:d_v^2+T^2=0} still holds with $d_{i,j}^v$'s replaced by $\tilde{d}_{i,j}^v$'s, and
\begin{equation}
d_{2,2}^\Left=\tilde{d}_{1,0}^v
\end{equation}
It follows that 
\[
\begin{cases}
\tilde{d}_{1,1}^v=d_{1,1}^v\circ g_C:(C_2(V_\Left),d_{2,2}^\Left)\rightarrow(C(v_\Right),d_{2,0}^v)\\
\tilde{d}_{2,1}^v=g_C^{-1}\circ d_{2,1}^v:(C(v_\Right),d_{2,0}^v)\rightarrow (C_2(V_\Left),d_{2,2}^\Left)
\end{cases}
\]
are \emph{co-chain maps}, and the last two equalities in \eqref{eqn:d_v^2+T^2=0} show that they are \emph{homotopy inverse} to each other. 

\item
\emph{Define} $d^V=d^V(\epsilon)\in\End(C(V))$ by 
\begin{equation}
\begin{cases}
d^V|_{C(V_\Left)}\coloneqq d^\Left;\\
d^V|_{C(v_\Right)}\coloneqq d_{2,0}^v+d_{3,2}^\Left\circ\tilde{d}_{2,1}^v.
\end{cases}
\end{equation}
Then $d^V$ defines a differential of degree $1$, and we obtain a short exact sequence of complexes:
\begin{equation*}
0\rightarrow (C(V_\Left),d^\Left)\xrightarrow[]{\varphi_\Left} (C(V),d^V)\xrightarrow[]{q_2} (C(v_\Right),d_{2,0}^v)\rightarrow 0,
\end{equation*}
where $\varphi_\Left$ is a canonical inclusion and $q_2$ a canonical quotient.

\item
\emph{Define} $i_2=i_2(\epsilon):C_2(V_\Left)\hookrightarrow C(V)$ by
\begin{equation}
i_2\coloneqq \identity-(\tilde{d}_{1,1}^v+d_{3,2}^\Left\circ \tilde{d}_{1,2}^v).
\end{equation}
Use the identifications in Definition \ref{def:matrix notations at V}.(5), \emph{define} a $\field$-linear map $\varphi_\Right=\varphi_\Right(\epsilon):C(V)\rightarrow C(V_\Right)$ by 
\begin{equation}
\begin{cases}
\varphi_\Right|_{C_1(V_\Left)\oplus C(v_\Right)\oplus C_3(V_\Left)}\coloneqq \identity;\\
\varphi_\Right|_{C_2(V_\Left)}\coloneqq \tilde{d}_{1,1}^v+ d_{3,2}^\Left\circ\tilde{d}_{1,2}^v.
\end{cases}
\end{equation}
Then $i_2:(C_2(V_\Left),d_{2,2}^\Left)\hookrightarrow(C(V),d^V)$ and $\varphi_\Right:(C(V),d^V)\twoheadrightarrow (C(V_\Right),d^\Right)$ are co-chain maps, and we obtain a short exact sequence of complexes:
\begin{equation*}
0\rightarrow (C_2(V_\Left),d_{2,2}^\Left)\xrightarrow[]{i_2}(C(V),d^V)\xrightarrow[]{\varphi_\Right} (C(V_\Right),d^\Right)\rightarrow 0
\end{equation*}
Moreover, the composition $q_2\circ i_2:(C_2(V_\Left),d_{2,2}^\Left)\rightarrow (C(v_\Right),d_{2,0}^v)$ is a co-chain homotopy equivalence with homotopy inverse $-\tilde{d}_{2,1}^v$.

\item
\emph{Define} a $\field$-linear map $\psi_\Left:C(V)\rightarrow C(V_\Left)$ by
\begin{equation}
\begin{cases}
\psi_\Left|_{C(V_\Left)}\coloneqq \identity;\\
\psi_\Left|_{C(v_\Right)}\coloneqq\tilde{d}_{2,1}^v.
\end{cases}
\end{equation}
Then $\psi_\Left$ is a co-chain map, and $\psi_\Left\circ\varphi_\Left=\identity$.

\noindent{}\emph{Define} a $\field$-linear map $\psi_\Right:C(V_\Right)\rightarrow C(V)$ by
\begin{equation}
\begin{cases}
\psi_\Right|_{C_3(V_\Left)}\coloneqq id;\\
\psi_\Right|_{C(v_\Right)}\coloneqq\tilde{d}_{2,1}^v;\\
\psi_\Right|_{C_1(V_\Left)}\coloneqq id+\tilde{d}_{1,2}^v\circ d_{2,1}^\Left.
\end{cases}
\end{equation}
Then $\psi_\Right$ is a co-chain map.

\item
\emph{Define} $\varphi\coloneqq\varphi_\Right\circ\varphi_\Left:(C(V_\Left),d^\Left)\rightarrow(C(V_\Right),d^\Right)$ and $\psi\coloneqq \psi_\Left\circ\psi_\Right:(C(V_\Right),d^\Right)\rightarrow(C(V_\Left),d^\Left)$. \emph{Define} a $\field$-linear map $h_\Left:C(V_\Left)\rightarrow C(V_\Left)$ of degree $-1$ by
\begin{equation}
\begin{cases}
h_\Left|_{C_1(V_\Left)\oplus C_3(V_\Left)}\coloneqq 0;\\
h_\Left|_{C_2(V_\Left)}\coloneqq\tilde{d}_{1,2}^v.
\end{cases}
\end{equation}
\emph{Define} a $\field$-linear map $h_\Right:C(V_\Right)\rightarrow C(V_\Right)$ of degree $-1$ by
\begin{equation}
\begin{cases}
h_\Right|_{C_3(V_\Left)}\coloneqq 0;\\
h_\Right|_{C(v_\Right)}\coloneqq d_{2,2}^v+d_{3,2}^\Left\circ\tilde{d}_{2,3}^v;\\
h_\Right|_{C_1(V_\Left)}\coloneqq (\tilde{d}_{1,3}^v+d_{3,2}^\Left\circ\tilde{d}_{1,4}^v)\circ d_{2,1}^\Left.
\end{cases}
\end{equation}
Then we have
\begin{equation}
\begin{cases}
\psi\circ\varphi-\identity=d^\Left\circ h_\Left+h_\Left\circ d^\Left;\\
\varphi\circ\psi-\identity=d^\Right\circ h_\Right+h_\Right\circ d^\Right.
\end{cases}
\end{equation}
In particular, $\varphi$ and $\psi$ are co-chain homotopy equivalences and are homotopy inverse to each other. 
\end{enumerate}
\end{proposition/definition}

\begin{proof}
\noindent{}$(1)$. By a direct check, this follows immediately from the definition.

\noindent{}$(2)$.
The only nontrivial part is to show $(d^V)^2=0$.
By definition of $d^\Left$, we have $(d^\Left)^2=0$. It suffices to show that, for all $x\in C(v_\Right)$, we have $d^2(x)=0$. By definition, we have:
\begin{align*}
(d^V)^2(x)&=d^V(d_{2,0}^v(x))+d^V(d_{3,2}^\Left\circ\tilde{d}_{2,1}^v(x))\\
&=(d_{2,0}^v)^2(x)+d_{3,2}^\Left\circ \tilde{d}_{2,1}^v\circ d_{2,0}^v(x)+d_{3,3}^\Left\circ d_{3,2}^\Left\circ \tilde{d}_{2,1}^v(x)\\
&=d_{3,2}^\Left\circ d_{2,2}^\Left(\tilde{d}_{2,1}^v(x))+d_{3,3}^\Left\circ d_{3,2}^\Left(\tilde{d}_{2,1}^v(x))\\
&=0
\end{align*}
Here, in the second equality, we used $(1)$, and in the third equality, we used Equation (\ref{eqn:d^L}).

\noindent{}In fact, the above computation also shows that $d_{3,2}^\Left\circ \tilde{d}_{2,1}^\Left: (C(v_\Right),d_{2,0}^v)[-1]\rightarrow (C(V_\Left),d^\Left)$ is a co-chain map, and $(C(V),d^V)=\cone(d_{3,2}^\Left\circ \tilde{d}_{2,1}^\Left)$.

\noindent{}$(3)$. Firstly, we show $i_2$ is a co-chain map. Clearly, $i_2$ is of degree $0$. In addition, for all $x\in C_2(V_\Left)$, we have
\begin{align*}
d^V\circ i_2(x)&=d^V(x-\tilde{d}_{1,1}^v(x)-d_{3,2}^\Left\circ\tilde{d}_{1,2}^v(x))\\
&=d_{2,2}^\Left(x)+d_{3,2}^\Left(x)-\tilde{d}_{2,0}^v\circ \tilde{d}_{1,1}^v(x)-d_{3,2}^\Left\circ \tilde{d}_{2,1}^v\circ \tilde{d}_{1,1}^v(x)-d_{3,3}^\Left\circ d_{3,2}^\Left\circ\tilde{d}_{1,2}^v(x)\\
&=d_{2,2}^\Left(x)-\tilde{d}_{1,1}^v\circ d_{2,2}^\Left(x)+d_{3,2}^\Left\circ(id-\tilde{d}_{2,1}^v\circ \tilde{d}_{1,1}^v+\tilde{d}_{1,0}^v\circ\tilde{d}_{1,2}^v)(x)\\
&=d_{2,2}^\Left(x)-\tilde{d}_{1,1}^v\circ d_{2,2}^\Left(x)-d_{3,2}^\Left\circ\tilde{d}_{1,2}^v\circ\tilde{d}_{1,0}^v(x)\\
&=(\identity-\tilde{d}_{1,1}^v-d_{3,2}^\Left\circ\tilde{d}_{1,2}^v)\circ d_{2,2}^\Left(x)\\
&=i_2\circ d_{2,2}^\Left(x)
\end{align*}
as desired. Here, in the third equality, we used the equation (\ref{eqn:d^L}) and the fact that $\tilde{d}_{1,1}^v$ is a co-chain map. In the fourth and fifth equalities, we used the equation (\ref{eqn:d_v^2+T^2=0}) and $d_{2,2}^\Left=\tilde{d}_{1,0}^v$. 

Next, we show $\varphi_\Right$ is a co-chain map. Clearly, $\varphi_\Right$ is of degree $0$. It suffices to show 
\begin{equation}
\varphi|_{C(V_\Left)}=\varphi:(C(V_\Left),d^\Left)\rightarrow(C(V_\Right),d^\Right)
\end{equation} 
is a co-chain map, and $\varphi_\Right\circ d^V(x)=d^\Right\circ\varphi_\Right(x)$ for all $x\in C(V_\Right)$. 

Use the decompositions $C(V_\Left)=C_1(V_\Left)\oplus C_2(V_\Left)\oplus C_3(V_\Left)$ and
$C(V_\Right)=C_1(V_\Left)\oplus C(v_\Right)\oplus C_3(V_\Left)$, we can write $\varphi$ in a matrix form:
\[
\varphi=\begin{pmatrix}
\identity& 0& 0\\
0& \tilde{d}_{1,1}^v& 0\\
0& d_{3,2}^\Left\circ\tilde{d}_{1,2}^v& id
\end{pmatrix}
\]
It follows from Equation (\ref{eqn:d^R}) that
\begin{align*}
d^\Right\circ \varphi&=
\begin{pmatrix}
d_{1,1}^\Left& 0 & 0\\
\tilde{d}_{1,1}^v\circ d_{2,1}^\Left& d_{2,0}^v& 0\\
d_{3,1}^\Left+d_{3,2}^\Left\circ\tilde{d}_{1,2}^v\circ d_{2,1}^\Left& d_{3,2}^\Left\circ\tilde{d}_{2,1}^v& d_{3,3}^\Left
\end{pmatrix}\circ
\begin{pmatrix}
\identity& 0& 0\\
0& \tilde{d}_{1,1}^v& 0\\
0& d_{3,2}^\Left\circ\tilde{d}_{1,2}^v& \identity
\end{pmatrix}\\
&=
\begin{pmatrix}
d_{1,1}^\Left& 0 & 0\\
\tilde{d}_{1,1}^v\circ d_{2,1}^\Left& d_{2,0}^v\circ\tilde{d}_{1,1}^v & 0\\
d_{3,1}^\Left+d_{3,2}^\Left\circ\tilde{d}_{1,2}^v\circ d_{2,1}^\Left& d_{3,2}^\Left\circ\tilde{d}_{2,1}^v\circ\tilde{d}_{1,1}^v+d_{3,3}^\Left\circ d_{3,2}^\Left\circ\tilde{d}_{1,2}^v & d_{3,3}^\Left
\end{pmatrix}\\
&=
\begin{pmatrix}
d_{1,1}^\Left& 0 & 0\\
\tilde{d}_{1,1}^v\circ d_{2,1}^\Left& \tilde{d}_{1,1}^v\circ d_{2,2}^\Left & 0\\
d_{3,1}^\Left+d_{3,2}^\Left\circ\tilde{d}_{1,2}^v\circ d_{2,1}^\Left& d_{3,2}^\Left\circ\tilde{d}_{2,1}^v\circ\tilde{d}_{1,1}^v-d_{3,2}^\Left\circ \tilde{d}_{1,0}^v\circ\tilde{d}_{1,2}^v & d_{3,3}^\Left
\end{pmatrix}\\
&=
\begin{pmatrix}
d_{1,1}^\Left& 0 & 0\\
\tilde{d}_{1,1}^v\circ d_{2,1}^\Left& \tilde{d}_{1,1}^v\circ d_{2,2}^\Left & 0\\
d_{3,1}^\Left+d_{3,2}^\Left\circ\tilde{d}_{1,2}^v\circ d_{2,1}^\Left& d_{3,2}^\Left+d_{3,2}^\Left\circ\tilde{d}_{1,2}^v\circ \tilde{d}_{1,0}^v & d_{3,3}^\Left
\end{pmatrix}\\
&=
\begin{pmatrix}
\identity& 0& 0\\
0& \tilde{d}_{1,1}^v& 0\\
0& d_{3,2}^\Left\circ\tilde{d}_{1,2}^v& \identity
\end{pmatrix}\circ
\begin{pmatrix}
d_{1,1}^\Left & 0 & 0\\
d_{2,1}^\Left & d_{2,2}^\Left & 0\\
d_{3,1}^\Left & d_{3,2}^\Left & d_{3,1}^\Left
\end{pmatrix}\\
&=\varphi\circ d^{\Left}.
\end{align*}
Here we've used Equation (\ref{eqn:d^L}), Equation (\ref{eqn:d_v^2+T^2=0}) for $\tilde{d}_{i,j}^v$, and $(1)$ above. Besides, for all $x\in C(V_\Right)$, by Equation (\ref{eqn:d^R}), we have:
\begin{align*}
\varphi_\Right\circ d^V(x)&=\varphi_\Right(d_{2,0}^v(x)+d_{3,2}^\Left\circ\tilde{d}_{2,1}^v(x))\\
&=d_{2,0}^v(x)+d_{3,2}^\Left\circ\tilde{d}_{1,1}^v(x)\\
&=d^\Right(x)\\
&=d^\Right\circ \varphi_\Right(x).
\end{align*}
This shows that $\varphi_\Right$ is indeed a co-chain map.

Then, we show $(\varphi_\Right,i_2)$ induces a short exact sequence of complexes. 
Clearly, $i_2$ is injective and $\varphi_\Right$ is surjective. In addition, by definition, for all $x\in C_2(V_\Left)$, we have:
\begin{align*}
\varphi_\Right\circ i_2(x)&=
\varphi_\Right(x-\tilde{d}_{1,1}^v(x)-d_{3,2}^\Left\circ\tilde{d}_{1,2}^v(x))\\
&=0,
\end{align*}
and clearly we have $\ker(\varphi_\Right)=\im(i_2)$.

Finally, by definition, we have $q_2\circ i_2(x)=-\tilde{d}_{1,1}^v(x)$ for all $x\in C_2(V_\Left)$. By $(1)$, it is a co-chain homotopy equivalence with homotopy inverse $-\tilde{d}_{2,1}^v$. 
This finishes the proof of $(3)$.

\noindent{}$(4)$. Clearly, $\psi_\Left\circ\varphi_\Left=\identity$. To show $\psi_\Left$ is a cochain map, it suffices to show that, for all $x\in C_2(v_\Right)$, we have $\psi_\Left\circ d^V(x)=d^\Left\circ\psi_\Left(x)$. Indeed, by definition, we have:
\begin{align*}
\psi_\Left\circ d^V(x)
&=\psi_\Left(d_{2,0}^v(x)+d_{3,2}^\Left\circ\tilde{d}_{2,1}^v(x))\\
&=\tilde{d}_{2,1}^v\circ d_{2,0}^v(x)+d_{3,2}^\Left\circ\tilde{d}_{2,1}^v(x)\\
&=d_{2,2}^\Left\circ\tilde{d}_{2,1}^v(x)+d_{3,2}^\Left\circ\tilde{d}_{2,1}^v(x)\\
&=d^\Left\circ\psi_\Left(x),
\end{align*}
as desired.

Next, we show $\psi_\Right$ is a co-chain map. The only nontrivial parts are to show that, for all $x\in C(v_\Right)$ or $x\in C_1(V_\Left)\isomorphic C_1(V_\Right)$, we have $\psi_\Right\circ d^\Right(x)=d^V\circ\psi_\Right(x)$. 

\noindent{}If $x\in C(v_\Right)$, we have 
\begin{align*}
\psi_\Right\circ d^\Right(x)
&=\psi_\Right(d_{2,0}^v(x)+d_{3,2}^\Left\circ\tilde{d}_{2,1}^v(x))\\
&=\tilde{d}_{2,1}^v\circ d_{2,0}^v(x)+d_{3,2}^\Left\circ\tilde{d}_{2,1}^v(x)\\
&=d_{2,2}^\Left\circ\tilde{d}_{2,1}^v(x)+d_{3,2}^\Left\circ\tilde{d}_{2,1}^v(x)\\
&=d^V\circ\psi_\Right(x).
\end{align*}
If $x\in C_1(V_\Left)$, we have
\begin{align*}
\psi_\Right\circ d^\Right(x)
&=\psi_\Right(d_{1,1}^\Left(x)+\tilde{d}_{1,1}^v\circ d_{2,1}^\Left(x)
+d_{3,1}^\Left(x)+d_{3,2}^\Left\circ\tilde{d}_{1,2}^v\circ d_{2,1}^\Left(x))\\
&=d_{1,1}^\Left(x)+\tilde{d}_{1,2}^v\circ d_{2,1}^\Left\circ d_{1,1}^\Left(x)
+\tilde{d}_{2,1}^v\circ\tilde{d}_{1,1}^v\circ d_{2,1}^\Left(x)
+d_{3,1}^\Left(x)+d_{3,2}^\Left\circ\tilde{d}_{1,2}^v\circ d_{2,1}^\Left(x)\\
&=d^Vx-d_{2,1}^\Left(x)-\tilde{d}_{1,2}^v\circ\tilde{d}_{1,0}^v\circ d_{2,1}^\Left(x)
+\tilde{d}_{2,1}^v\circ\tilde{d}_{1,1}^v\circ d_{2,1}^\Left(x)
+d_{3,2}^\Left\circ\tilde{d}_{1,2}^v\circ d_{2,1}^\Left(x)\\
&=d^Vx+\tilde{d}_{1,0}^v\circ\tilde{d}_{1,2}^v\circ d_{2,1}^\Left(x)
+d_{3,2}^\Left\circ\tilde{d}_{1,2}^v\circ d_{2,1}^\Left(x)\\
&=d^Vx+d^V\circ\tilde{d}_{1,2}^v\circ d_{2,1}^\Left(x)\\
&=d^V\circ\psi_\Right(x)
\end{align*}
as desired.

\noindent{}$(5)$. Firstly, we show $\psi\circ\varphi-\identity=d^\Left\circ h_\Left+h_\Left\circ d^\Left$.
For all $x\in C_3(V_\Left)$, clearly we have $\psi\circ\varphi(x)-x=0=d^\Left\circ h_\Left+h_\Left\circ d^\Left(x)$. For all $x\in C_2(V_\Left)$, we have
\begin{align*}
\psi\circ\varphi(x)
&=\psi(\tilde{d}_{1,1}^v(x)+d_{3,2}^\Left\circ\tilde{d}_{1,2}^v(x))\\
&=\tilde{d}_{2,1}^v\circ\tilde{d}_{1,1}^v(x)+d_{3,2}^\Left\circ\tilde{d}_{1,2}^v(x))\\
&=(id+\tilde{d}_{1,2}^v\circ\tilde{d}_{1,0}^v+\tilde{d}_{1,0}^v\circ\tilde{d}_{1,2}^v)(x)
+d_{3,2}^\Left\circ\tilde{d}_{1,2}^v(x))\\
&=x+h_\Left(d_{2,2}^\Left(x))+d^\Left(\tilde{d}_{1,2}^v(x))\\
&=x+h_\Left(d^\Left(x))+d^\Left\circ h_\Left(x)
\end{align*}
Finally, for all $x\in C_1(V_\Left)$, we have
\begin{align*}
\psi\circ\varphi(x)-x
&=\psi(x)-x=\tilde{d}_{1,2}^v\circ d_{2,1}^\Left(x)\\
&=h_\Left(d_{2,1}^\Left(x))+h_\Left(d_{1,1}^\Left(x)+d_{3,1}^\Left(x))+d^\Left\circ h_\Left(x)\\
&=h_\Left\circ d^\Left(x)+d^\Left\circ h_\Left(x)
\end{align*}
as desired.

Now, we show $\varphi\circ\psi-\identity=d^\Right\circ h_\Right+h_\Right\circ d^\Right$.
Clearly, for all $x\in C_3(V_\Right)=C_3(V_\Left)$, we have $\varphi\circ\psi(x)-x=0=d^\Right\circ h_\Right(x)+h_\Right\circ d^\Right(x)$. For all $x\in C(v_\Right)$, we have
\begin{align*}
\varphi\circ\psi(x)-x
&=\varphi(\tilde{d}_{2,1}^v(x))
=(\tilde{d}_{1,1}^v+d_{3,2}^\Left\circ\tilde{d}_{1,2}^v)\circ\tilde{d}_{2,1}^v(x)-x\\
&=(\tilde{d}_{2,0}^v\circ\tilde{d}_{2,2}^v+\tilde{d}_{2,2}^v\circ\tilde{d}_{2,0}^v)(x)
+d_{3,2}^\Left\circ\tilde{d}_{1,2}^v\circ\tilde{d}_{2,1}^v(x)\\
&=d^\Right\circ d_{2,2}^v(x)-d_{3,2}^\Left\circ\tilde{d}_{2,1}^v\circ d_{2,2}(x)
+d_{3,2}^\Left\circ\tilde{d}_{1,2}^v\circ\tilde{d}_{2,1}^v(x)+d_{2,2}^v\circ d_{2,0}^v(x)\\
&=d^\Right\circ d_{2,2}^v(x)+d_{3,2}^\Left\circ(\tilde{d}_{2,3}^v\circ d_{2,0}^v-\tilde{d}_{1,0}^v\circ\tilde{d}_{2,3}^v)(x)+d_{2,2}^v\circ d_{2,0}^v(x)\\
&=d^\Right \circ d_{2,2}^v(x)+d_{3,3}^\Left\circ d_{3,2}^\Left\circ \tilde{d}_{2,3}^v(x)
+(d_{2,2}^v+d_{3,2}^\Left\circ\tilde{d}_{2,3}^v)\circ d_{2,0}^v(x)\\
&=d^\Right\circ h_\Right(x)+h_\Right(d_{2,0}^v(x))\\
&=d^\Right\circ h_\Right(x)+h_\Right\circ d^\Right(x)
\end{align*}
Here, in the fourth equality we've used the identity 
\[
\tilde{d}_{2,3}^v\circ d_{2,0}^v-\tilde{d}_{1,2}^v\circ\tilde{d}_{2,1}^v+\tilde{d}_{2,1}^v\circ d_{2,2}^v-\tilde{d}_{1,0}^v\circ\tilde{d}_{2,3}^v=0
\]
from Equation (\ref{eqn:d^2+T^2=0}) for $\tilde{d}_{i,j}^v$'s with $m=3$.
Finally, for all $x\in C_1(V_\Right)=C_1(V_\Left)$, let us show $\varphi\circ\psi(x)-x=d^\Right\circ h_\Right(x)+h_\Right\circ d^\Right(x)$. By Equation (\ref{eqn:d^R}), observe that
\begin{align*}
h_\Right\circ d^\Right(x)
&=h_\Right(d_{1,1}^\Left+\tilde{d}_{1,1}^v\circ d_{2,1}^\Left
+d_{3,1}^\Left+d_{3,2}^\Left\circ\tilde{d}_{1,2}^v\circ d_{2,1}^\Left)(x)\\
&=h_\Right\circ d_{1,1}^\Left(x)+h_\Right\circ\tilde{d}_{1,1}^v\circ d_{2,1}^\Left(x)\\
&=h_\Right\circ d_{1,1}^\Left(x)+(d_{2,2}^v+d_{3,2}^\Left\circ\tilde{d}_{2,3}^v)\circ\tilde{d}_{1,1}^v\circ d_{2,1}^\Left(x)
\end{align*}
It suffices to show
\[
\varphi\circ\psi(x)-x-(d_{2,2}^v+d_{3,2}^\Left\circ\tilde{d}_{2,3}^v)\circ\tilde{d}_{1,1}^v\circ d_{2,1}^\Left(x)=d^\Right\circ h_\Right(x)+h_\Right\circ d_{1,1}^\Left(x).
\]
Indeed, we have
\begin{align*}
&\mathrel{\hphantom{=}}\varphi\circ\psi(x)-x-(d_{2,2}^v+d_{3,2}^\Left\circ\tilde{d}_{2,3}^v)\circ\tilde{d}_{1,1}^v\circ d_{2,1}^\Left(x)\\
&=\varphi(x+\tilde{d}_{1,2}^v\circ d_{2,1}^\Left(x))-x-(d_{2,2}^v+d_{3,2}^\Left\circ\tilde{d}_{2,3}^v)\circ\tilde{d}_{1,1}^v\circ d_{2,1}^\Left(x)\\
&=(\tilde{d}_{1,1}^v+d_{3,2}^\Left\circ \tilde{d}_{1,2}^v)\circ\tilde{d}_{1,2}^v\circ d_{2,1}^\Left(x)-(d_{2,2}^v+d_{3,2}^\Left\circ\tilde{d}_{2,3}^v)\circ\tilde{d}_{1,1}^v\circ d_{2,1}^\Left(x)\\
&=(\tilde{d}_{1,1}^v\circ\tilde{d}_{1,2}^v-d_{2,2}^v\circ\tilde{d}_{1,1}^v)\circ d_{2,1}^\Left(x)+d_{3,2}^\Left\circ(\tilde{d}_{1,2}^v\circ\tilde{d}_{1,2}^v-\tilde{d}_{2,3}^v\circ\tilde{d}_{1,1}^v)\circ d_{2,1}^\Left(x)\\
&=(d_{2,0}^v\circ\tilde{d}_{1,3}^v-\tilde{d}_{1,3}^v\circ\tilde{d}_{1,0}^v)\circ d_{2,1}^\Left(x)+d_{3,2}^\Left\circ(-\tilde{d}_{1,4}^v\circ\tilde{d}_{1,0}^v+\tilde{d}_{2,1}^v\circ\tilde{d}_{1,3}^v-\tilde{d}_{1,0}^v\circ\tilde{d}_{1,4}^v)\circ d_{2,1}^\Left(x)\\
&=(d_{2,0}^v+d_{3,2}^\Left\circ\tilde{d}_{2,1}^v)\circ\tilde{d}_{1,3}^v\circ d_{2,1}^\Left(x)+\tilde{d}_{1,3}^v\circ d_{2,1}^{\Left}\circ d_{1,1}^\Left(x)\\
&\mathrel{\hphantom{=}}+d_{3,2}^\Left\circ\tilde{d}_{1,4}^v\circ d_{2,1}^\Left\circ d_{1,1}^\Left(x)+d_{3,3}^\Left\circ d_{3,2}^\Left\circ\tilde{d}_{1,4}^v\circ d_{2,1}^\Left(x)\\
&=d^\Right\circ(\tilde{d}_{1,3}^v+d_{3,2}^\Left\circ\tilde{d}_{1,4}^v)\circ d_{2,1}^\Left(x)
+(\tilde{d}_{1,3}^v+d_{3,2}^\Left\circ\tilde{d}_{1,4}^v)\circ d_{2,1}^\Left(d_{1,1}^\Left(x))\\
&=d^\Right\circ h_\Right(x)+h_\Right\circ d_{1,1}^\Left(x)
\end{align*}f
as desired. Here, in the fourth equality, we've used the identities
\[
\begin{cases}
\tilde{d}_{1,3}^v\circ\tilde{d}_{1,0}^v-d_{2,2}^v\circ\tilde{d}_{1,1}^v+\tilde{d}_{1,1}^v\circ\tilde{d}_{1,2}^v-d_{2,0}^v\circ\tilde{d}_{1,3}^v=0;\\
\tilde{d}_{1,4}^v\circ\tilde{d}_{1,0}^v-\tilde{d}_{2,3}^v\circ\tilde{d}_{1,1}^v+\tilde{d}_{1,2}^v\circ\tilde{d}_{1,2}^v-\tilde{d}_{2,1}^v\circ\tilde{d}_{1,3}^v+\tilde{d}_{1,0}^v\circ\tilde{d}_{1,4}^v=0
\end{cases}
\]
from Equation \eqref{eqn:d^2+T^2=0} for $\tilde{d}_{i,j}$'s with $m=3$ and $4$ respectively.
This finishes the proof of $(5)$.

\end{proof}

\begin{corollary}\label{cor:acyclic}
Let $\cT=(T_\Left\rightarrow T\leftarrow T_\Right)$ be a bordered Legendrian graph, and $\epsilon\in\aug(T;\field)$ be an augmentation. In particular, $\epsilon_\Left\coloneqq \epsilon|_{T_\Left}$ is acyclic. Then $\epsilon_\Right\coloneqq \epsilon|_{T_\Right}$ is also acyclic.
\end{corollary}

\begin{proof}
By \cite[Cor.5.3]{Su2017}, it suffices to proof the case when $T=V$ is an bordered Legendrian graph involving a single vertex. But this is a direct corollary of $(5)$ in the Proposition/Definition above, in which we showed $\varphi$ is a co-chain homotopy equivalence, in particular a quasi-isomorphism. 
\end{proof}

\subsubsection{A group action}
We use the notations in Definition \ref{def:Morse complex} and Definition \ref{def:Morse complex at a vertex}.

\begin{definition}
\emph{Let} $B^\Left\coloneqq B(V_\Left;\field)$ and $B^v\coloneqq B(v;\field)$ be the two automorphism groups associated to $V_\Left$ and $v$ respectively. 
For $1\leq i\leq 3$, \emph{define} $B_{i,i}^\Left\coloneqq B(C_i(V_\Left);\field)$. 
For $i\in\ZZ/2$, \emph{define} $B_{i,0}^v\coloneqq B(C(v_i);\field)$, with $v_i$'s as in Definition \ref{def:matrix notations at V}.(3). 
Notice that there is a \emph{canonical identification} $B_{2,2}^\Left\isomorphic B_{1,0}^v$ if we identify $C_2(V_\Left)$ with $C(v_\Left)$ via $e_{k+p-1}^\Left=e_p^v$. 



\noindent{}For any $g^\Left\in B^\Left$, similar to $d^\Left$, we can write $g^\Left$ in a matrix form:
\[
g^\Left\coloneqq
\begin{pmatrix}
g_{1,1}^\Left & 0 & 0\\
g_{2,1}^\Left & g_{2,2}^\Left & 0\\
g_{3,1}^\Left & g_{3,2}^\Left & g_{3,3}^\Left
\end{pmatrix}
\]
Similarly, one can \emph{define} $B^\Right\coloneqq B(V_\Right;\field)$ and $B_{i,i}^\Right$ for $1\leq i\leq 3$. And, there are \emph{canonical identifications} $B_{i,i}^\Right\isomorphic B_{i,i}^\Left$ for $i=1,3$, and $B_{2,2}^\Right\isomorphic B_{2,0}^v$.

\noindent{}For any $g^v\in B^v$, similar to $d^v$, we can write $g^v$ as:
\[
g^v\coloneqq \sum_{i\in\ZZ/2,j\geq 0}Z^j g_{i,j}^v
\]
for some $g_{i,j}^v\in\hom^{-j}(C(v_i),C(v_{i+j}))$.

\noindent{}\emph{Denote} $B\coloneqq B^\Left\times B^v$.
By the identification in Remark \ref{rem:identification for aug var of V}, we can define a natural algebraic group action of $B$ on $\aug(V;\field)\isomorphic\MC(V;\field)$ as follows: for any $g=(g^\Left,g^v)\in B$ and any $d=(d^\Left, g_C, d^v)\in\MC(V;\field)$, we take
\[
g\cdot d=(g\cdot d^\Left, g\cdot g_C, g\cdot d^v)\coloneqq(g^\Left\cdot d^\Left, g_{1,0}^v\circ g_C\circ (g_{2,2}^\Left)^{-1} , g^v\cdot d^v)
\]
Here, $g^\Left\cdot d^\Left$ and $g^v\cdot d^v$ are the group actions in Definition \ref{def:Morse complex} and Definition \ref{def:Morse complex at a vertex} respectively. It is direct to check that $g\cdot d$ satisfies the Equation (\ref{eqn:twistor u_C}). Hence, the group action is well-defined.
\end{definition}

\noindent{}\emph{Question}: Given any $g\in B$, when does this group action preserve the isomorphism class of $d^\Right$, i.e. $(g\cdot d)^\Right=g^\Right\cdot d^\Right$ for some $g^\Right=g^\Right(g,d)\in B^\Right$?
The answer is the key to counting augmentations for $V$.

\begin{lemma}\label{lem:isomorphism lifting for a vertex}
Let $g\in B$, and $\epsilon$ (or $d=(d^\Left,g_C,d^v)$) be any augmentation for $\cA(V)$. We have:
\begin{enumerate}
\item
If $g=(g^\Left,g^v)\in B$, with both of $g^\Left$ and $g^v$ \emph{block-diagonal}. That is, $g_{j,i}^\Left=0$ if $i<j$, and $g_{i,j}^v=0$ if $j>0$. Then $(g\cdot d)^\Right=g^\Right\cdot d^\Right$ with $g^\Right=g^\Right(g,d)\coloneqq\diag(g_{1,1}^\Left, g_{2,0}^v, g_{3,3}^\Left)$.

\item
If $g=(g^\Left,\identity)\in B$, with $g_{i,i}^\Left=\identity$, and $g_{j,i}^\Left=0$ for all $1\leq i<j\leq 3$ unless $(j,i)=(3,2)$. In particular, $(g^{-1})_{i,i}^\Left=\identity$, $(g^{-1})_{3,2}^\Left=-g_{3,2}^\Left$, and $g_{j,i}^\Left=0$ otherwise. Then $(g\cdot d)^\Right=g^\Right\cdot d^\Right$ with $g^\Right=g^\Right(g,d)$ given by
\[
g^\Right\coloneqq
\begin{pmatrix}
\identity & 0 & 0\\
0 & \identity & 0\\
g_{3,2}^\Left\circ\tilde{d}_{1,2}^v\circ d_{2,1}^\Left & g_{3,2}^\Left\circ\tilde{d}_{2,1}^v & \identity
\end{pmatrix}
\]

\item
If $g=(g^\Left,\identity)\in B$, with $g_{i,i}^\Left=\identity$, and $g_{j,i}^\Left=0$ for all $1\leq i<j\leq 3$ unless $(j,i)=(2,1)$. In particular, $(g^{-1})_{i,i}^\Left=\identity$, $(g^{-1})_{2,1}^\Left=-g_{2,1}^\Left$, and $g_{j,i}^\Left=0$ otherwise. Then $(g\cdot d)^\Right=g^\Right\cdot d^\Right$ with $g^\Right=g^\Right(g,d)$ given by
\[
g^\Right\coloneqq
\begin{pmatrix}
\identity & 0 & 0\\
\tilde{d}_{1,1}^v\circ g_{2,1}^\Left & \identity & 0\\
d_{3,2}^\Left\circ\tilde{d}_{1,2}^v\circ g_{2,1}^\Left & 0 & \identity
\end{pmatrix}
\]

\item
If $g=(g^\Left,\identity)\in B$, with $g_{i,i}^\Left=\identity$, and $g_{j,i}^\Left=0$ for all $1\leq i<j\leq 3$ unless $(j,i)=(3,1)$. In particular, $(g^{-1})_{i,i}^\Left=\identity$, $(g^{-1})_{3,1}^\Left=-g_{3,1}^\Left$, and $g_{j,i}^\Left=0$ otherwise. Then $(g\cdot d)^\Right=g^\Right\cdot d^\Right$ with $g^\Right=g^\Right(g,d)$ given by
\[
g^\Right\coloneqq
\begin{pmatrix}
\identity & 0 & 0\\
0 & \identity & 0\\
g_{3,1}^\Left & 0 & \identity
\end{pmatrix}
\]
\end{enumerate}
\end{lemma}
\begin{proof}
\begin{enumerate}
\item
If $g=(g^\Left,g^v)\in B$, with both of $g^\Left$ and $g^v$ \emph{block-diagonal}. Then
\begin{align*}
&\mathrel{\hphantom{=}}\begin{pmatrix}
(g\cdot d)_{1,1}^\Right & 0 & 0\\
(g\cdot d)_{2,1}^\Right & (g\cdot d)_{2,2}^\Right & 0\\
(g\cdot d)_{3,1}^\Right & (g\cdot d)_{3,2}^\Right & (g\cdot d)_{3,3}^\Right
\end{pmatrix}\\
&=
\begin{pmatrix}
(g\cdot d)_{1,1}^\Left& 0 & 0\\
(\tilde{g\cdot d})_{1,1}^v\circ (g\cdot d)_{2,1}^\Left& (g\cdot d)_{2,0}^v& 0\\
(g\cdot d)_{3,1}^\Left+(g\cdot d)_{3,2}^\Left\circ(\tilde{g\cdot d})_{1,2}^v\circ (g\cdot d)_{2,1}^\Left& (g\cdot d)_{3,2}^\Left\circ(\tilde{g\cdot d})_{2,1}^v& (g\cdot d)_{3,3}^\Left
\end{pmatrix}\\
&=
\begin{pmatrix}
g_{1,1}^\Left\circ d_{1,1}^\Left\circ (g_{1,1}^\Left)^{-1} & 0 & 0\\
g_{2,0}^v\circ(\tilde{d}_{1,1}^v\circ d_{2,1}^\Left)\circ(g_{1,1}^\Left)^{-1} & g_{2,0}^v\circ d_{2,0}^v\circ (g_{2,0}^v)^{-1} & 0\\
g_{3,3}^\Left\circ (d_{3,1}^\Left+d_{3,2}^{\Left}\circ \tilde{d}_{1,2}^v\circ d_{2,1}^\Left)\circ (g_{1,1}^\Left)^{-1} & g_{3,3}^\Left\circ(d_{3,2}^\Left\circ \tilde{d}_{2,1}^v)\circ(g_{2,0}^v)^{-1} & g_{3,3}^\Left\circ d_{3,3}^\Left\circ(g_{3,3}^\Left)^{-1}
\end{pmatrix}\\
&=
\begin{pmatrix}
g_{1,1}^\Left & 0 & 0\\
0 & g_{2,0}^v & 0\\
0 & 0 & g_{3,3}^\Left
\end{pmatrix}\circ
\begin{pmatrix}
d_{1,1}^\Left& 0 & 0\\
\tilde{d}_{1,1}^v\circ d_{2,1}^\Left& d_{2,0}^v& 0\\
d_{3,1}^\Left+d_{3,2}^\Left\circ\tilde{d}_{1,2}^v\circ d_{2,1}^\Left& d_{3,2}^\Left\circ\tilde{d}_{2,1}^v& d_{3,3}^\Left
\end{pmatrix}\circ
\begin{pmatrix}
g_{1,1}^\Left & 0 & 0\\
0 & g_{2,0}^v & 0\\
0 & 0 & g_{3,3}^\Left
\end{pmatrix}^{-1}\\
&=
\begin{pmatrix}
g_{1,1}^\Left & 0 & 0\\
0 & g_{2,0}^v & 0\\
0 & 0 & g_{3,3}^\Left
\end{pmatrix}\circ
\begin{pmatrix}
d_{1,1}^\Right & 0 & 0\\
d_{2,1}^\Right & d_{2,2}^\Right & 0\\
d_{3,1}^\Right & d_{3,2}^\Right & d_{3,3}^\Right
\end{pmatrix}\circ
\begin{pmatrix}
g_{1,1}^\Left & 0 & 0\\
0 & g_{2,0}^v & 0\\
0 & 0 & g_{3,3}^\Left
\end{pmatrix}^{-1}\\
&=g^\Right\cdot d^\Right
\end{align*}
with $g^\Right=g^\Right(g)=\diag(g_{1,1}^\Left,g_{2,0}^v,g_{3,3}^\Left)\in B^\Right$, as desired.

\item
If $g=(g^\Left,\identity)\in B$, with $g_{i,i}^\Left=\identity$, and $g_{j,i}^\Left=0$ for all $1\leq i<j\leq 3$ unless $(j,i)=(3,2)$. The only two nontrivial calculations are
\begin{align*}
(g\cdot d)_{3,1}^\Right
&=(g\cdot d)_{3,1}^\Left + (g\cdot d)_{3,2}^\Left\circ(\tilde{g\cdot d})_{1,2}^v\circ (g\cdot d)_{2,1}^\Left\\
&=(d_{3,1}^\Left+g_{3,2}^\Left\circ d_{2,1}^\Left) + (d_{3,2}^\Left+g_{3,2}^\Left\circ d_{2,2}^\Left - d_{3,3}^\Left\circ g_{3,2}^\Left)\circ\tilde{d}_{1,2}^v\circ d_{2,1}^\Left\\
&=d_{3,1}^\Right+g_{3,2}^\Left\circ d_{2,1}^\Left + (g_{3,2}^\Left\circ\tilde{d}_{1,0}^v-d_{3,3}^\Left\circ g_{3,2}^\Left)\circ\tilde{d}_{1,2}^v\circ d_{2,1}^\Left\\
&=d_{3,1}^\Right + g_{3,2}^\Left\circ(\tilde{d}_{2,1}^v\circ\tilde{d}_{1,1}^v-\tilde{d}_{1,2}^v\circ\tilde{d}_{1,0}^v)\circ d_{2,1}^v - d_{3,3}^\Left\circ g_{3,2}^\Left\circ\tilde{d}_{1,2}^v\circ d_{2,1}^\Left\\
&=d_{3,1}^\Right + g_{3,2}^\Left\circ\tilde{d}_{2,1}^v\circ d_{2,1}^\Right + g_{3,2}^\Left\circ\tilde{d}_{1,2}^v\circ d_{2,1}^\Left\circ d_{1,1}^\Left - d_{3,3}^\Right\circ g_{3,2}^\Left\circ\tilde{d}_{1,2}^v\circ d_{2,1}^\Left\\
&=d_{3,1}^\Right + g_{3,2}^\Left\circ\tilde{d}_{2,1}^v\circ d_{2,1}^\Right + g_{3,2}^\Left\circ\tilde{d}_{1,2}^v\circ d_{2,1}^\Left\circ d_{1,1}^\Right - d_{3,3}^\Right\circ g_{3,2}^\Left\circ\tilde{d}_{1,2}^v\circ d_{2,1}^\Left\\
&=d_{3,1}^\Right + g_{3,2}^\Right\circ d_{2,1}^\Right + g_{3,1}^\Right\circ d_{1,1}^\Right -d_{3,3}^\Right\circ g_{3,1}^\Right,\\
(g\cdot d)_{3,2}^\Right
&=(g\cdot d)_{3,2}^\Left\circ(\tilde{g\cdot d})_{2,1}^v\\
&=(d_{3,2}^\Left+g_{3,2}^\Left\circ d_{2,2}^\Left - d_{3,3}^\Left\circ g_{3,2}^\Left)\circ\tilde{d}_{2,1}^v\\
&=d_{3,2}^\Right+ g_{3,2}^\Left\circ\tilde{d}_{2,1}^v\circ d_{2,0}^v - d_{3,3}^\Right\circ g_{3,2}^\Left\circ\tilde{d}_{2,1}^v\\
&=d_{3,2}^\Right + g_{3,2}^\Right\circ d_{2,2}^\Right- d_{3,3}^\Right\circ g_{3,2}^\Right
\end{align*}
Otherwise, we have $(g\cdot d)_{j,i}^\Right=d_{j,i}^\Right$. It follows that
\begin{align*}
&\mathrel{\hphantom{=}}\begin{pmatrix}
(g\cdot d)_{1,1}^\Right & 0 & 0\\
(g\cdot d)_{2,1}^\Right & (g\cdot d)_{2,2}^\Right & 0\\
(g\cdot d)_{3,1}^\Right & (g\cdot d)_{3,2}^\Right & (g\cdot d)_{3,3}^\Right
\end{pmatrix}\\
&=
\begin{pmatrix}
d_{1,1}^\Right & 0 & 0\\
d_{2,1}^\Right & d_{2,2}^\Right & 0\\
d_{3,1}^\Right + g_{3,2}^\Right\circ d_{2,1}^\Right + g_{3,1}^\Right\circ d_{1,1}^\Right -d_{3,3}^\Right\circ g_{3,1}^\Right & d_{3,2}^\Right + g_{3,2}^\Right\circ d_{2,2}^\Right- d_{3,3}^\Right\circ g_{3,2}^\Right & d_{3,3}^\Right
\end{pmatrix}\\
&=
\begin{pmatrix}
\identity & 0 & 0\\
0 & \identity & 0\\
g_{3,1}^\Right & g_{3,2}^\Right & \identity
\end{pmatrix}\circ
\begin{pmatrix}
d_{1,1}^\Right & 0 & 0\\
d_{2,1}^\Right & d_{2,2}^\Right & 0\\
d_{3,1}^\Right & d_{3,2}^\Right & d_{3,3}^\Right
\end{pmatrix}\circ
\begin{pmatrix}
\identity & 0 & 0\\
0 & \identity & 0\\
-g_{3,1}^\Right & -g_{3,2}^\Right & \identity
\end{pmatrix}\\
&=g^\Right\circ d^\Right\circ (g^\Right)^{-1} 
\end{align*}
as desired.

\item
If $g=(g^\Left,\identity)\in B$, with $g_{i,i}^\Left=\identity$, and $g_{j,i}^\Left=0$ for all $1\leq i<j\leq 3$ unless $(j,i)=(2,1)$. The only two nontrivial calculations are
\begin{align*}
(g\cdot d)_{3,1}^\Right
&=(g\cdot d)_{3,1}^\Left + (g\cdot d)_{3,2}^\Left\circ(\tilde{g\cdot d})_{1,2}^v\circ (g\cdot d)_{2,1}^\Left\\
&=(d_{3,1}^\Left - d_{3,2}^\Left\circ g_{2,1}^\Left) + d_{3,2}^\Left\circ\tilde{d}_{1,2}^v\circ(d_{2,1}^\Left+g_{2,1}^\Left\circ d_{1,1}^\Left - d_{2,2}^\Left\circ g_{2,1}^\Left)\\
&=d_{3,1}^\Right + (d_{3,2}^\Left\circ\tilde{d}_{1,2}^v\circ g_{2,1}^\Left)\circ d_{1,1}^\Left + d_{3,2}^\Left\circ(\tilde{d}_{1,0}^v\circ\tilde{d}_{1,2}^v-\tilde{d}_{2,1}^v\circ\tilde{d}_{1,1}^v)\circ g_{2,1}^\Left\\
&=d_{3,1}^\Right + (d_{3,2}^\Left\circ\tilde{d}_{1,2}^v\circ g_{2,1}^\Left)\circ d_{1,1}^\Left - d_{3,3}^\Left\circ (d_{3,2}^\Left\circ\tilde{d}_{1,2}^v\circ g_{2,1}^\Left) - (d_{3,2}^\Left\circ\tilde{d}_{2,1}^v)\circ(\tilde{d}_{1,1}^v\circ g_{2,1}^\Left)\\
&=d_{3,1}^\Right + g_{3,1}^\Right\circ d_{1,1}^\Right - d_{3,3}^\Right\circ g_{3,1}^\Right - d_{3,2}^\Right\circ g_{2,1}^\Right,\\
(g\cdot d)_{2,1}^\Right
&=(\tilde{g\cdot d})_{1,1}^v\circ (g\cdot d)_{2,1}^\Left\\
&=\tilde{d}_{1,1}^v\circ (d_{2,1}^\Left + g_{2,1}^\Left\circ d_{1,1}^\Left - d_{2,2}^\Left\circ g_{2,1}^\Left)\\
&=d_{2,1}^\Right + g_{2,1}^\Right\circ d_{1,1}^\Left - d_{2,0}^v\circ\tilde{d}_{1,1}^v\circ g_{2,1}^\Left\\
&=d_{2,1}^\Right + g_{2,1}^\Right\circ d_{1,1}^\Left - d_{2,0}^v\circ g_{2,1}^\Right.
\end{align*}
Otherwise, we have $(g\cdot d)_{j,i}^\Right=d_{j,i}^\Right$. It follows that
\begin{align*}
(g\cdot d)^\Right
&=
\begin{pmatrix}
d_{1,1}^\Right & 0 & 0\\
d_{2,1}^\Right + g_{2,1}^\Right\circ d_{1,1}^\Left - d_{2,0}^v\circ g_{2,1}^\Right & d_{2,2}^\Right & 0\\
d_{3,1}^\Right + g_{3,1}^\Right\circ d_{1,1}^\Right - d_{3,3}^\Right\circ g_{3,1}^\Right - d_{3,2}^\Right\circ g_{2,1}^\Right & 0 & d_{3,3}^\Right
\end{pmatrix}\\
&=
\begin{pmatrix}
\identity & 0 & 0\\
g_{2,1}^\Right & \identity & 0\\
0 & 0 & \identity
\end{pmatrix}\circ
\begin{pmatrix}
d_{1,1}^\Right & 0 & 0\\
d_{2,1}^\Right & d_{2,2}^\Right & 0\\
d_{3,1}^\Right & d_{3,2}^\Right & d_{3,3}^\Right
\end{pmatrix}\circ
\begin{pmatrix}
\identity & 0 & 0\\
-g_{2,1}^\Right & \identity & 0\\
0 & 0 & \identity
\end{pmatrix}\\
&=g^\Right\circ d^\Right\circ (g^\Right)^{-1}
\end{align*}
as desired. 

\item

If $g=(g^\Left,\identity)\in B$, with $g_{i,i}^\Left=\identity$, and $g_{j,i}^\Left=0$ for all $1\leq i<j\leq 3$ unless $(j,i)=(3,1)$. The only nontrivial calculation is
\begin{align*}
(g\cdot d)_{3,1}^\Right
&=(g\cdot d)_{3,1}^\Left + (g\cdot d)_{3,2}^\Left\circ(\tilde{g\cdot d})_{1,2}^v\circ(g\cdot d)_{2,1}^\Left\\
&=(d_{3,1}^\Left+g_{3,1}^\Left\circ d_{1,1}^\Left - d_{3,3}^\Left\circ g_{3,1}^\Left) + d_{3,2}^\Left\circ\tilde{d}_{1,2}^v\circ d_{2,1}^\Left\\
&=d_{3,1}^\Right + g_{3,1}^\Right d_{1,1}^\Right - d_{3,3}^\Right\circ g_{3,1}^\Right.
\end{align*}
Otherwise, we have $(g\cdot d)_{j,i}^\Right=d_{j,i}^\Right$. It follows that
\begin{align*}
(g\cdot d)^\Right
&=
\begin{pmatrix}
d_{1,1}^\Right & 0 & 0\\
d_{2,1}^\Right & d_{2,2}^\Right & 0\\
d_{3,1}^\Right + g_{3,1}^\Right\circ d_{1,1}^\Right - d_{3,3}^\Right\circ g_{3,1}^\Right & d_{3,2}^\Right & d_{3,3}^\Right
\end{pmatrix}\\
&=\begin{pmatrix}
\identity & 0 & 0\\
0 & \identity & 0\\
g_{3,1}^\Right & 0 & \identity
\end{pmatrix}\circ
\begin{pmatrix}
d_{1,1}^\Right & 0 & 0\\
d_{2,1}^\Right & d_{2,2}^\Right & 0\\
d_{3,1}^\Right & d_{3,2}^\Right & d_{3,3}^\Right
\end{pmatrix}\circ
\begin{pmatrix}
\identity & 0 & 0\\
0 & \identity & 0\\
-g_{3,1}^\Right & 0 & \identity
\end{pmatrix}\\
&=g^\Right\circ d^\Right\circ (g^\Right)^{-1}
\end{align*}
as desired.\qedhere
\end{enumerate}
\end{proof}

Regard $B^\Left$ as a subgroup of $B$ via the obvious inclusion. Recall Definition \ref{def:aug var with bdy conditions}, we then have
\begin{proposition}\label{prop:aug-orbit vs orbit-orbit}
Let $(\ruling_\Left,\ruling_\Right)$ be a pair of boundary conditions, with $\ruling_\Left\in\NR(V_\Left)$ and $\ruling_\Right\in\NR(V_\Right)$ respectively. Then the group action of $B^\Left$ on $\aug(V;\field)$ preserves the sub-variety $\aug(V,\ruling_\Left,\ruling_\Right;\field)$. 
Moreover, for any $\epsilon_\Left\in\aug^{\ruling_\Left}(V;\field)$, the augmentation variety with boundary conditions $\aug(V,\epsilon_\Left,\ruling_R;\field)$ is independent of the choice of $\epsilon_\Left$ in $\aug^{\ruling_\Left}(V;\field)$ up to canonical isomorphism.
And, we have a natural isomorphism
\begin{align*}
\aug(V,\ruling_\Left,\ruling_\Right;\field)
&\isomorphic\aug^{\ruling_\Left}(V_\Left;\field)\times\aug(V,\epsilon_\Left,\ruling_\Right;\field)\\
&\isomorphic\field^{n_\Left/2}\times(\field^\times)^{A_b(\ruling_\Left)}\times\aug(V,\epsilon_\Left,\ruling_\Right;\field)
\end{align*}
\end{proposition}
\begin{proof}
Any element $g^\Left\in B^\Left\subset B$ can be written uniquely as a product $g^\Left=g_4^\Left\circ g_3^\Left\circ g_2^\Left\circ g_1^\Left$ with $g_i^\Left$ of the form as in Lemma~\ref{lem:isomorphism lifting for a vertex}$~(1)$. The rest follows  immediately from Lemma \ref{lem:isomorphism lifting for a vertex} and Lemma \ref{lem:properties of orbits at a vertex}.
\end{proof}

\subsection{Ruling decomposition for augmentation varieties of bordered Legendrian graphs}

\begin{definition}
Use the notations in Definition \ref{def:canonical augmentation at a vertex} and Lemma \ref{lem:stratification of aug var at a vertex}, for any $\ruling_v\in\NR(v)$, \emph{define} $\aug(V,\epsilon_\Left,\ruling_v,\ruling_\Right;\field)$ to be the sub-variety of $\aug(V,\epsilon_\Left,\ruling_\Right;\field)$ whose augmentations $d=(d^\Left,g_C,d^v)$ further satisfy $d^v\in B^v\cdot d_{\ruling_v}$. \emph{Define}  $\aug(V,\ruling_\Left,\ruling_v,\ruling_\Right;\field)$ similarly.
\end{definition}
Then we have the following decomposition 
\begin{equation}\label{eqn:decomposition 1 for V}
\aug(V,\epsilon_\Left,\ruling_\Right;\field)\isomorphic \coprod_{\ruling\in\NR(v)}\aug(V,\epsilon_\Left,\ruling_v,\ruling_\Right;\field)
\end{equation}
of $\aug(V,\epsilon_\Left,\ruling_\Right;\field)$ into the set-theoretic disjoint union of locally closed sub-varieties. 
There is a similar decomposition of $\aug(V,\ruling_\Left,\ruling_\Right;\field)$, and the previous isomorphism is compatible with the decompositions. That is, we have an induced natural isomorphism 
\[
\aug(V,\ruling_\Left,\ruling_v,\ruling_\Right;\field)\isomorphic \field^{n_\Left/2}\times(\field^*)^{A_b(\ruling_\Left)}\times\aug(V,\epsilon_\Left,\ruling_v,\ruling_\Right;\field).
\] 

By the result above, to count augmentations for $\aug(V,\epsilon_\Left,\ruling_\Right;\field)$ or $\aug(V,\epsilon_\Left,\ruling_v,\ruling_\Right;\field)$, we can now \emph{assume} $\epsilon_\Left=\epsilon_{\ruling_\Left}$ is a fixed \emph{standard} augmentation. Equivalently, the differential is $d^\Left=h^\Left\cdot d_{\ruling_\Left}$ for some diagonal element $h^\Left\in B^\Left$, and $d_{\ruling_\Left}$ is the canonical differential associated to $\ruling_\Left$ as in Definition \ref{def:canonical augmentation at a vertex}.

Fix any augmentation $\epsilon$ (or $d=(d^\Left,g_C,d^v)$) of $\cA(V)$, similar to $\tilde{d}^v$, \emph{define} $\tilde{g}^v\in B^v$ by 
\begin{equation}
\begin{pmatrix}
\tilde{g}_{1,2j}^v & \tilde{g}_{2,2j+1}^v\\
\tilde{g}_{1,2j+1} & \tilde{g}_{2,2j}
\end{pmatrix}\coloneqq
\begin{pmatrix}
g_C^{-1} & 0\\
0 & \identity
\end{pmatrix}\circ
\begin{pmatrix}
g_{1,2j}^v & g_{2,2j+1}^v\\
g_{1,2j+1} & g_{2,2j}
\end{pmatrix}\circ
\begin{pmatrix}
g_C & 0\\
0 & \identity
\end{pmatrix}
\end{equation}

\begin{lemma}\label{lem:group action for V}
Let $\epsilon$ (or $d=(d^\Left,g_C,d^v)$) be any augmentation of $\cA(V)$ such that $d^\Left$ is standard. Then for any $g=(\identity,g^v)\in B$, we have:
\begin{enumerate}
\item
If $g_{i,0}^v=\identity$, and $g_{i,j}^v=0$ for all $i\in\ZZ/2$ and $j>0$ unless $(i,j)=(1,1)$. In particular, $(g^{-1})_{i,0}^v=\identity$, $(g^{-1})_{1,1}^v= -g_{1,1}^v$, and $(g^{-1})_{i,j}^v=0$ otherwise. Then $(g\cdot d)^\Right=g^\Right\cdot d^\Right$ with $g^\Right=g^\Right(g,d)$ given by
\[
g^\Right\coloneqq
\begin{pmatrix}
\identity & 0 & 0\\
\tilde{g}_{1,1}^v\circ d_{2,1}^\Left & \identity & 0\\
0 & 0 & \identity
\end{pmatrix}
\]

\item
If $g_{i,0}^v=\identity$, and $g_{i,j}^v=0$ for all $i\in\ZZ/2$ and $j>0$ unless $(i,j)=(2,1)$. In particular, $(g^{-1})_{i,0}^v=\identity$, $(g^{-1})_{2,1}^v= -g_{2,1}^v$, and $(g^{-1})_{i,j}^v=0$ otherwise. Then $(g\cdot d)^\Right=g^\Right\cdot d^\Right$ with $g^\Right=g^\Right(g,d)$ given by
\[
g^\Right\coloneqq
\begin{pmatrix}
\identity & 0 & 0\\
0 & \identity & 0\\
0 & d_{3,2}^\Left\circ\tilde{g}_{2,1}^v & \identity
\end{pmatrix}
\]

\item
If $g_{i,0}^v=\identity$, and $g_{i,j}^v=0$ for all $i\in\ZZ/2$ and $j>0$ unless $(i,j)=(1,2)$. In particular, $(g^{-1})_{i,0}^v=\identity$, $(g^{-1})_{1,2}^v= -g_{1,2}^v$, and $(g^{-1})_{i,j}^v=0$ otherwise. Then $(g\cdot d)^\Right=g^\Right\cdot d^\Right$ with $g^\Right=g^\Right(g,d)$ given by
\[
g^\Right\coloneqq
\begin{pmatrix}
\identity & 0 & 0\\
0 & \identity & 0\\
-d_{3,2}^\Left\circ\tilde{g}_{1,2}^v\circ d_{2,1}^\Left & 0 & \identity
\end{pmatrix}
\]

\item
If $g_{i,0}^v=\identity$, $g_{i,1}^v=0$ for all $i\in\ZZ/2$, and $g_{1,2}^v=0$. In particular, $(g^{-1})_{i,0}^v=\identity$, $(g^{-1})_{i,1}^v= 0$, and $(g^{-1})_{1,2}^v=0$. Then $(g\cdot d)^\Right=g^\Right\cdot d^\Right$ with $g^\Right=\identity$.
\end{enumerate}
\end{lemma}
\begin{proof}
Notice that if $g=(\identity,g^v)$ with $g_{1,0}^v=\identity$, then $g\cdot g_C=g_C$, and we have $\tilde{g\cdot d}^v=\tilde{g}^v\cdot\tilde{d}^v$.

\noindent{}$(1)$. If $g=(\identity,g^v)$ with $g_{i,0}^v=\identity$, and $g_{i,j}^v=0$ for all $i\in\ZZ/2$ and $j>0$ unless $(i,j)=(1,1)$, then $\tilde{g}^v$ is of the same form as $g^v$.
The only two nontrivial calculations are
\begin{align*}
(g\cdot d)_{2,1}^\Right
&=(\tilde{g\cdot d})_{1,1}^v\circ (g\cdot d)_{2,1}^\Left\\
&=(\tilde{d}_{1,1}^v+\tilde{g}_{1,1}^v\circ \tilde{d}_{1,0}^v - d_{2,0}^v\circ\tilde{g}_{1,1}^v)\circ d_{2,1}^\Left\\
&=d_{2,1}^\Right - \tilde{g}_{1,1}^v\circ d_{2,1}^\Left\circ d_{1,1}^\Left - d_{2,0}^v\circ\tilde{g}_{1,1}^v\circ d_{2,1}^\Left\\
&=d_{2,1}^\Right + \tilde{g}_{1,1}^v\circ d_{2,1}^\Left\circ d_{1,1}^\Left - d_{2,0}^v\circ\tilde{g}_{1,1}^v\circ d_{2,1}^\Left\\
&=d_{2,1}^\Right + g_{2,1}^\Right\circ d_{1,1}^\Right - d_{2,2}^\Right\circ g_{2,1}^\Right\\
(g\cdot d)_{3,1}^\Right
&=(g\cdot d)_{3,1}^\Left +(g\cdot d)_{3,2}^\Left\circ(\tilde{g\cdot d})_{1,2}^v\circ (g\cdot d)_{2,1}^\Left\\
&=d_{3,1}^\Left + d_{3,2}^\Left\circ(\tilde{d}_{1,2}^v-\tilde{d}_{2,1}^v\circ \tilde{g}_{1,1}^v)\circ d_{2,1}^\Left\\
&=d_{3,1}^\Right - d_{3,2}^\Right\circ(\tilde{g}_{1,1}^v\circ d_{2,1}^\Left)\\
&=d_{3,1}^\Right - d_{3,2}^\Right\circ g_{2,1}^\Right
\end{align*}
Here, in the first equation we have used the fact that $d^\Left$ is standard. In particular, it implies that $d_{2,1}^\Left\circ d_{1,1}^\Left=0$.
In all the other cases, we have $(g\cdot d)_{j,i}^\Right=d_{j,i}^\Right$. It follows that
\begin{align*}
(g\cdot d)^\Right&=
\begin{pmatrix}
d_{1,1}^\Right & 0 & 0\\
d_{2,1}^\Right + g_{2,1}^\Right\circ d_{1,1}^\Right - d_{2,2}^\Right\circ g_{2,1}^\Right & d_{2,2}^\Right & 0\\
d_{3,1}^\Right - d_{3,2}^\Right\circ g_{2,1}^\Right & d_{3,2}^\Right & 0
\end{pmatrix}\\
&=
\begin{pmatrix}
\identity & 0 & 0\\
g_{2,1}^\Right & \identity & 0\\
0 & 0 & \identity
\end{pmatrix}\circ
\begin{pmatrix}
d_{1,1}^\Right & 0 & 0\\
d_{2,1}^\Right & d_{2,2}^\Right & 0\\
d_{3,1}^\Right & d_{3,2}^\Right & 0
\end{pmatrix}\circ
\begin{pmatrix}
\identity & 0 & 0\\
-g_{2,1}^\Right & \identity & 0\\
0 & 0 & \identity
\end{pmatrix}\\
&=g^\Right\circ d^\Right\circ (g^\Right)^{-1}
\end{align*}
as desired.

\noindent{}$(2)$. If $g_{i,0}^v=\identity$, and $g_{i,j}^v=0$ for all $i\in\ZZ/2$ and $j>0$ unless $(i,j)=(2,1)$, then $\tilde{g}^v$ is of the same form as $g^v$. The only two nontrivial calculations are
\begin{align*}
(g\cdot d)_{3,2}^\Right
&=(g\cdot d)_{3,2}^\Left\circ(\tilde{g\cdot d})_{2,1}^v\\
&=d_{3,2}^\Left\circ(d_{2,1}^v + \tilde{g}_{2,1}^v\circ d_{2,0}^v - \tilde{d}_{1,0}^v\circ\tilde{g}_{2,1}^v)\\
&=d_{3,2}^\Right +(d_{3,2}^\Left\circ\tilde{g}_{2,1}^v)\circ d_{2,0}^v + d_{3,3}^\Left\circ d_{3,2}^\Left\circ\tilde{g}_{2,1}^v\\
&=d_{3,2}^\Right +(d_{3,2}^\Left\circ\tilde{g}_{2,1}^v)\circ d_{2,0}^v - d_{3,3}^\Left\circ (d_{3,2}^\Left\circ\tilde{g}_{2,1}^v)\\
&=d_{3,2}^\Right + g_{3,2}^\Right\circ d_{2,2}^\Right - d_{3,3}^\Right\circ g_{3,2}^\Right\\
(g\cdot d)_{3,1}^\Right
&=(g\cdot d)_{3,1}^\Left + (g\cdot d)_{3,2}^\Left\circ(\tilde{g\cdot d})_{1,2}^v\circ (g\cdot d)_{2,1}^\Left\\
&=d_{3,1}^\Left + d_{3,2}^\Left\circ(\tilde{d}_{1,2}^v+\tilde{g}_{2,1}^v\circ\tilde{d}_{1,1}^v)\circ d_{2,1}^\Left\\
&=d_{3,1}^\Right + g_{3,2}^\Right\circ d_{2,1}^\Right
\end{align*}
Here, in the first equation, we have used the fact that $d^\Left$ is standard. In particular, it implies that $d_{3,3}^\Left\circ d_{3,2}^\Left=0$.
In all the other cases, we have $(g\cdot d)_{j,i}^\Right=d_{j,i}^\Right$. It follows that
\begin{align*}
(g\cdot d)^\Right&=
\begin{pmatrix}
d_{1,1}^\Right & 0 & 0\\
d_{2,1}^\Right & d_{2,2}^\Right & 0\\
d_{3,1}^\Right + g_{3,2}^\Right\circ d_{2,1}^\Right & d_{3,2}^\Right + g_{3,2}^\Right\circ d_{2,2}^\Right - d_{3,3}^\Right\circ g_{3,2}^\Right & 0
\end{pmatrix}\\
&=
\begin{pmatrix}
\identity & 0 & 0\\
0 & \identity & 0\\
0 & g_{3,2}^\Right & \identity
\end{pmatrix}\circ
\begin{pmatrix}
d_{1,1}^\Right & 0 & 0\\
d_{2,1}^\Right & d_{2,2}^\Right & 0\\
d_{3,1}^\Right & d_{3,2}^\Right & 0
\end{pmatrix}\circ
\begin{pmatrix}
\identity & 0 & 0\\
0 & \identity & 0\\
0 & -g_{3,2}^\Right & \identity
\end{pmatrix}\\
&=g^\Right\circ d^\Right\circ (g^\Right)^{-1}
\end{align*}
as desired.

\noindent{}$(3)$. If $g_{i,0}^v=\identity$, and $g_{i,j}^v=0$ for all $i\in\ZZ/2$ and $j>0$ unless $(i,j)=(1,2)$, then $\tilde{g}^v$ is of the same form as $g^v$.
The only nontrivial calculation is
\begin{align*}
(g\cdot d)_{3,1}^\Right
&=(g\cdot d)_{3,1}^\Left + (g\cdot d)_{3,2}^\Left\circ(\tilde{g\cdot d})_{1,2}^v\circ(g\cdot d)_{2,1}^\Left\\
&=d_{3,1}^\Left + d_{3,2}^\Left\circ(\tilde{d}_{1,2}^v+\tilde{g}_{1,2}^v\circ\tilde{d}_{1,0}^v-\tilde{d}_{1,0}^v\circ\tilde{g}_{1,2}^v)\circ d_{2,1}^\Left\\
&=d_{3,1}^\Right -d_{3,2}^\Left\circ\tilde{g}_{1,2}^v\circ d_{2,1}^\Left\circ d_{1,1}^\Left + d_{3,3}^\Left\circ d_{3,2}^\Left\circ\tilde{g}_{1,2}^v\circ d_{2,1}^\Left\\
&=d_{3,1}^\Right + g_{3,1}^\Right\circ d_{1,1}^\Right - d_{3,3}^\Right\circ g_{3,1}^\Right. 
\end{align*}
In all the other cases, we have $(g\cdot d)_{j,i}^\Right=d_{j,i}^\Right$. It follows that
\begin{align*}
(g\cdot d)^\Right&=
\begin{pmatrix}
d_{1,1}^\Right & 0 & 0\\
d_{2,1}^\Right & d_{2,2}^\Right & 0\\
d_{3,1}^\Right + g_{3,1}^\Right\circ d_{1,1}^\Right - d_{3,3}^\Right\circ g_{3,1}^\Right & d_{3,2}^\Right & 0
\end{pmatrix}\\
&=
\begin{pmatrix}
\identity & 0 & 0\\
0 & \identity & 0\\
g_{3,1}^\Right & 0 & \identity
\end{pmatrix}\circ
\begin{pmatrix}
d_{1,1}^\Right & 0 & 0\\
d_{2,1}^\Right & d_{2,2}^\Right & 0\\
d_{3,1}^\Right & d_{3,2}^\Right & 0
\end{pmatrix}\circ
\begin{pmatrix}
\identity & 0 & 0\\
0 & \identity & 0\\
-g_{3,1}^\Right & 0 & \identity
\end{pmatrix}\\
&=g^\Right\circ d^\Right\circ (g^\Right)^{-1}
\end{align*}
as desired.

\noindent{}$(4)$. If $g_{i,0}^v=\identity$, $g_{i,1}^v=0$ for all $i\in\ZZ/2$, and $g_{1,2}^v=0$, then $\tilde{g}^v$ is of the same form as $g^v$. By definition, it is direct to see that
\begin{align*}
(g\cdot d)^\Right&=
\begin{pmatrix}
(g\cdot d)_{1,1}^\Right & 0 & 0\\
(g\cdot d)_{2,1}^\Right & (g\cdot d)_{2,2}^\Right & 0\\
(g\cdot d)_{3,1}^\Right & (g\cdot d)_{3,2}^\Right & (g\cdot d)_{3,3}^\Right
\end{pmatrix}\\
&=
\begin{pmatrix}
(g\cdot d)_{1,1}^\Left& 0 & 0\\
(\tilde{g\cdot d})_{1,1}^v\circ (g\cdot d)_{2,1}^\Left& (g\cdot d)_{2,0}^v& 0\\
(g\cdot d)_{3,1}^\Left+(g\cdot d)_{3,2}^\Left\circ(\tilde{g\cdot d})_{1,2}^v\circ (g\cdot d)_{2,1}^\Left& (g\cdot d)_{3,2}^\Left\circ(\tilde{g\cdot d})_{2,1}^v& (g\cdot d)_{3,3}^\Left
\end{pmatrix}\\
&=
\begin{pmatrix}
d_{1,1}^\Left& 0 & 0\\
\tilde{d}_{1,1}^v\circ d_{2,1}^\Left& d_{2,0}^v& 0\\
d_{3,1}^\Left+d_{3,2}^\Left\circ\tilde{d}_{1,2}^v\circ d_{2,1}^\Left& d_{3,2}^\Left\circ\tilde{d}_{2,1}^v& d_{3,3}^\Left
\end{pmatrix}\\
&=d^\Right
\end{align*}
as desired.

This finishes the proof of the lemma.
\end{proof}

\begin{corollary}\label{cor:B^v action for V}
For any $\epsilon_\Left\in\aug^{\ruling_\Left}(V_\Left;\field)$ which is standard, the group action of $B^v$ on $\aug(V,\epsilon_\Left,\ruling_\Right;\field)$ is well-defined, and preserves the sub-varieties $\aug(V,\epsilon_\Left,\ruling_v,\ruling_\Right;\field)$.
\end{corollary}

\begin{proof}
By linear algebra, any group element $g=g^v\in B^v\subset B$ can be written uniquely as a product $g^v=g_4^v\circ g_3^v\circ g_2^v\circ g_1^v\circ g_0^v$ with $g_i^v$ of the form as in Lemma \ref{lem:group action for V}.(i) for $1\leq i\leq 4$, and $g_0^v$ block-diagonal, i.e. of the form as in Lemma \ref{lem:isomorphism lifting for a vertex}.(1). Now, the result follows from Lemma \ref{lem:isomorphism lifting for a vertex}.(1) and Lemma \ref{lem:group action for V} above.
\end{proof}

Fix $\epsilon_\Left$ (or $d^\Left$) to be standard, \emph{define} $\aug^1(V,\epsilon_\Left,\ruling_v,\ruling_\Right;\field)$ to be the subvariety of $\aug(V,\epsilon_\Left,\ruling_v,\ruling_\Right;\field)$ consisting of $d=(d^\Left,g_C,d^v)$ with $g_C=\identity$, hence $d_{1,0}^v=g_C\circ d_{2,2}^\Left\circ g_C^{-1}=d_{2,2}^\Left$. 
Clearly, we have a natural isomorphism
\begin{align}\label{eqn:decomposition 2 for V}
\aug(V,\epsilon_\Left,\ruling_v,\ruling_\Right;\field)&\isomorphic\aug^1(V,\epsilon_\Left,\ruling_v,\ruling_\Right;\field)\times B_{1,0}^v\\
d=(d^\Left,g_C,d^v)&\mapsto ((d^\Left,id,\tilde{d}^v), g_C)\nonumber
\end{align}

For simplicity, we now \emph{assume} $d^\Left\isomorphic d_{\ruling_\Left}$ is the canonical augmentation. Notice that there are natural maps 
\begin{align*}
r_\Left&:\NR(v)\rightarrow \GNR(v_\Left),&
r_\Right&:\NR(v)\rightarrow \GNR(v_\Right)
\end{align*}
such that 
\begin{align*}
r_\Left(\ruling_v)(i)&\coloneqq
\begin{cases}
\ruling_v(i) &\text{ if }1\leq i<\ruling_v(i)\leq \ell;\\
i & \text{ otherwise},
\end{cases}&
r_\Right(\ruling_v)(i)\coloneqq
\begin{cases}
\ruling_v(i) &\text{ if }\ell+1\leq i<\ruling_v(i)\leq \ell+r;\\
i & \text{ otherwise}.
\end{cases}
\end{align*}
Then $d_{1,0}^v=d_{2,2}^\Left$ is the canonical differential associated to $\ruling_v|_{v_\Left}=r_\Left(\ruling_v)$.

\emph{Define} $B_v^1$ to be the subgroup of $B^v$ whose elements $g^v$ satisfy $g_{1,0}^v=id$. Again, by the previous lemma, there is an induced group action of $B_v^1$ on $\aug^1(V,\epsilon_\Left,\ruling_v,\ruling_\Right;\field)$.

For each fixed $\ruling_v\in\NR(v)$, recall that $\ruling_v$ determines a partition $I_v\coloneqq\{1,2,\ldots,\ell+r\}=U(\ruling_v)\amalg L(\ruling_v)$, together with a bijection $\ruling_v:U(\ruling_v)\xrightarrow[]{\sim}L(\ruling_v)$, such that $i<\ruling_v(i)$ and $|v_{i,\ruling_v(i)-i}|=\mu(i)-\mu(\ruling_v(i))-1+n(v_{i,\ruling_v(i)-j})=0$ for all $i\in U(\ruling_v)$. For each $i\in U(\ruling_v)$, \emph{define} 
\begin{align*}
A^1(\ruling_v,i)&\coloneqq\{j\in U(\ruling_v) \mid i<j<\ruling_v(j)<\ruling_v(i), \ruling_v(j)\leq \ell, \textrm{and $\mu(j)=\mu(i)$}\};\\
A^2(\ruling_v,i)&\coloneqq\{j\in U(\ruling_v) \mid i<j\leq \ell<\ruling_v(j)<\ruling_v(i), \textrm{and $\mu(j)=\mu(i)$}\};\\
A^3(\ruling_v,i)&\coloneqq\{j\in U(\ruling_v) \mid i<j<\ruling_v(j)<\ruling_v(i), j\geq \ell+1, \textrm{and $|v_{i,j-i}|+1=0$}\}.
\end{align*}
Then $A_{\ruling_v}(i)=A^1(\ruling_v,i)\amalg A^2(\ruling_v,i)\amalg A^3(\ruling_v,i)$ by Definition \ref{def:indices for an involution at a vertex}.
For each $i\in L(\ruling_v)$, \emph{define}
\[
I(\ruling_v,i)\coloneqq\{j>0~|~i+j>\ell, \textrm{and $|v_{i,j}|+1=0$, i.e. $|e_i|=|Z^{n(v_{i,j})}e_{i+j}|$}\}.
\]
\emph{Define} $\aug^2(V,\epsilon_\Left,\ruling_v,\ruling_\Right;\field)$ to be the sub-variety of $\aug^1(V,\epsilon_\Left,\ruling_v,\ruling_\Right;\field)$, consisting of (\emph{``partially canonical''}) augmentations $d=(d^\Left,id,d_c^v)$ such that: 
\begin{align}\label{eqn:d_c^v}
(-1)^{\mu(i)}d_c^v(e_i^v)=
\begin{dcases}
e_{\ruling_v(i)}^v &\text{ if } 1\leq i<\ruling_v(i)\leq \ell \text{ or }\ell+1\leq i<\ruling_v(i)\leq \ell+r;\\
Z\left(e_{\ruling_v(i)}^v+\sum_{j\in A^2(\ruling_v,i)}\epsilon_{ij}e_{\ruling_v(j)}^v\right) &\text{ if } 1\leq i\leq \ell< \ruling_v(i)\leq \ell+r,
\end{dcases}
\end{align}
for some $\epsilon_{ij}\in\field$.
Notice that the differential $d_c^v\in\MC(v;\field)$ is uniquely determined by the conditions above, as the other values $d_c^v(e_{\ruling_v(i)}^v)$ are uniquely determined via $(d_c^v)^2+Z^2=0$.

\emph{Define} $B(\ruling_v;\field)$ to be the subgroup of $B_v^1$ whose elements $g^v$ further satisfy:
\begin{enumerate} 
\item
for $i\in U(\ruling_v)$, we have
\begin{align*}
g^v(e_i^v)=
\begin{dcases}
e_i^v+Z\sum_{j\in A^3(\ruling_v,i)} *_{i,j-i}e_j^v &\text{ if } i\leq \ell;\\
e_i^v+\sum_{j\in A^3(\ruling_v,i)}*_{i,j-i}e_j^v &\text{ if } i\geq \ell+1,
\end{dcases}
\end{align*}
for some $*_{i,j-i}\in\field$;
\item 
for $i\in L(\ruling_v)$, we have 
\[
g^v(e_i^v)=*_{i,0}e_i^v+\sum_{j\in I(\ruling_v,i)}*_{i,j}Z^{n(v_{i,j})}e_{i+j}^v
\] 
for some $*_{i,0}\in\field^*$ and $*_{i,j}\in\field$. 
\end{enumerate}

One can check directly that this indeed defines a group. 

\begin{lemma}\label{lem:partially canonical differential at a vertex}
There is an natural isomorphism
\begin{align*}
B(\ruling_v;\field)\times\aug^2(V,\epsilon_\Left,\ruling_v,\ruling_\Right;\field)&\xrightarrow[]{\sim}\aug^1(V,\epsilon_\Left,\ruling_v,\ruling_\Right;\field)\\
(g^v,d_c=(d^\Left,\identity,d_c^v))&\mapsto d=(d^\Left,\identity,(g^v)^{-1}\cdot d_c^v).
\end{align*}
\end{lemma}

\begin{proof}
The argument is similar to the proof of Lemma \ref{lem:stratification of aug var at a vertex}, and Lemma \ref{lem:properties of orbits at a vertex}. By the discussion above, the map is well-defined. For any $d=(d^\Left,\identity,d^v)\in\aug^1(V,\epsilon_\Left,\ruling_v,\ruling_\Right;\field)$, it suffices to show that there exists a unique pair $(g^v,d_c)$ with $g^v\in B(\ruling_v;\field)$ and $d_c=(d^\Left,\identity,d_c^v)\in\aug^2(V,\epsilon_\Left,\ruling_v,\ruling_\Right;\field)$, such that $(g^v)^{-1}\cdot d_c^v=d^v$. We will determine $g^v$ by defining $(e_i^v)''=g^v(e_i^v)$ inductively.

By Lemma \ref{lem:properties of orbits at a vertex}.(1), for each $i\in U(\ruling_v)$, there exists a unique $i$-admissible element in $C_v$ of the form 
\[
(e_i^v)'=e_i^v+\sum_{j\in A_{\ruling_v}(i)}a_{i,j-i}Z^{n(v_{i,j-i})}e_j^v,\qquad a_{i,j-i}\in\field,
\] 
such that $d^v(e_i^v)'$ is $(\ruling_v(i),n(v_{i,\ruling_v(i)-i}))$-admissible. \emph{Denote} $d^v(e_i^v)'= Z^{n(v_{i,j-i})}(e_j^v)'$ with $j\coloneqq\ruling_v(i)$, then 
\[(e_j^v)'=a_{j,0}e_j^v+\sum_{p\in I_v(j)}a_{j,p}Z^{n(v_{j,k})}e_{j+p}^v
\] 
for some $a_{j,0}\in\field^\times$ and $a_{j,p}\in\field$, uniquely determined by $d^v$. 
\begin{enumerate}
\item If $i\geq \ell+1$, then $A_{\ruling_v}(i)=A^3(\ruling_v,i)$ and $I_v(\ruling_v(i))=I(\ruling_v,\ruling_v(i))$. \emph{Define} $(e_i^v)''\coloneqq (e_i^v)'$ and $(e_{\ruling_v(i)}^v)''\coloneqq (e_{\ruling_v(i)}^v)'$, we have $d^v(e_i^v)''=(e_{\ruling_v(i)}^v)''$. 

\item If $i<\ruling_v(i)\leq \ell$, by the condition that $d_{1,0}^v$ is standard, we have $(e_i^v)'=e_i^v$ and 
\[
d^v(e_i^v)'=(e_{\ruling_v(i)}^v)'=a_{\ruling_v(i),0}e_{\ruling_v(i)}^v+\sum_{p\in I(\ruling_v,i)}a_{\ruling_v(i),p}Z^{n(v_{\ruling_v(i),p})}e_{\ruling_v(i)+p}^v.
\]
\emph{Define} $(e_i^v)''\coloneqq e_i^v$, and $(e_{\ruling_v(i)}^v)''\coloneqq (e_{\ruling_v(i)}^v)'$, we have $d^v(e_i^v)''=(e_{\ruling_v(i)}^v)''$.

\item If $i\leq \ell<\ruling_v(i)$, then $I_v(i)=I(\ruling_v,i)$. Again by the condition $d_{1,0}^v$ is standard, we have
\[
(e_i^v)'=e_i^v+\sum_{j\in A^2(\ruling_v,i)\amalg A^3(\ruling_v,i)}a_{i,j-i}Z^{n(v_{i,j-i})}e_j^v.
\] 
If $A^2(\ruling_v,i)=\emptyset$, \emph{define} 
\[
(e_i^v)''\coloneqq (e_i^v)'=e_i^v+\sum_{j\in A^3(\ruling_v,i)}a_{i,j-i}Z^{n(v_{i,j-i})}e_j^v
\] 
and $(e_{\ruling_v(i)}^v)''\coloneqq (e_{\ruling_v(i)}^v)'$, then we have $d^v(e_i^v)''=Z(e_{\ruling_v(i)}^v)''$.
Otherwise, by induction, we can \emph{define} $(e_j^v)''$ and $(e_{\ruling_v(j)}^v)''$ for all $j\in A^2(\ruling_v,i)$, such that 
\begin{align*}
(e_j^v)''&=e_j^v+Z\sum_{k\in A^3(\ruling_v,j)} *_{j,k-j}e_k^v;\\
(e_{\ruling_v(j)}^v)''&=*_{j,0}e_{\ruling_v(j)}^v+\sum_{k\in I(\ruling_v,j)}*_{\ruling_v(j),k}Z^{n(v_{\ruling_v(j),k})}e_{\ruling_v(j)+k}^v;\\
d^v\left((e_j^v)''\right)&=Z\left((e_{\ruling_v(j)}^v)''+\sum_{k\in A^2(\ruling_v,j)}\epsilon_{j,k}(e_{\ruling_v(k)}^v)''\right),
\end{align*}
for some uniquely determined $*_{j,0}\in\field^*$ and $*_{j,p},\epsilon_{j,k}\in\field$.
In addition, we can re-write uniquely
\[
(e_i^v)'=e_i^v+\sum_{j\in A^2(\ruling_v,i)\amalg A^3(\ruling_v,i)}b_{i,j-i}Z^{n(v_{i,j-i})}(e_j^v)'',\qquad b_{i,j-i}\in\field
\]
\emph{Define} $(e_i^v)''\coloneqq e_i^v+\displaystyle\sum_{j\in A^3(\ruling_v,i)b_{i,j-i}}Z(e_j^v)''$ and $(e_{\ruling_v(i)}^v)''\coloneqq (e_{\ruling_v(i)}^v)'$, then 
\begin{align*}
d^v(e_i^v)''&=d^v(e_i^v)'-d^v\left(\sum_{j\in A^2(\ruling_v,i)}b_{i,j-i}(e_j^v)''\right)\\
&=Z(e_{\ruling_v(i)}^v)'-\sum_{j\in A^2(\ruling_v,i)}b_{i,j-i}d^v(e_j^v)''\\
&=Z((e_{\ruling_v(i)}^v)''+\sum_{p\in A^2(\ruling_v,i)}\epsilon_{i,p}(e_p^v)''),
\end{align*}
for some uniquely determined $\epsilon_{i,p}\in\field$. Here, the last equality follows from the inductive hypothesis.
\end{enumerate}
This finishes the construction of $(e_i^v)''$ for all $1\leq i\leq \ell+r$.

Now, \emph{define} $g^v$ via $g^v(e_i^v)=(e_i^v)''$. By the discussion above, we have $g^v\in B(\ruling_v;\field)$, and $d_c^v\coloneqq g^v\cdot d^v$ defines $d_c\coloneqq (d^\Left,\identity,d_c^v)\in\aug^2(V,\epsilon_\Left,\ruling_v,\ruling_\Right;\field)$ uniquely. This finishes the proof of the lemma.
\end{proof}

\subsubsection{Resolutions and normal rulings}\label{subsubsec:resolutions and normal rulings}

For our purpose later on in this section, we give an alternative definition of resolutions of a vertex and normal rulings for bordered Legendrian graphs.

For each fixed $\ruling_v\in\NR(v)$, recall that $r_\Left(\ruling_v)\in\GNR(v_\Left)$ (resp. $r_\Right(\ruling_v)\in\GNR(v_\Right)$) is an involution (or generalized normal ruling) for $v_\Left$ (resp. $v_\Right$) with possible fixed points. By Definition \ref{def:indices for an involution at a trivial tangle}, $r_\Left(\ruling_v)$ determines a partition $I(v_\Left)=\{1,2,\ldots,\ell\}=U(r_\Left(\ruling_v))\amalg L(r_\Left(\ruling_v))\amalg H(r_\Left(\ruling_v))$, together with a bijection
$r_\Left(\ruling_v):U(r_\Left(\ruling_v))\xrightarrow[]{\sim}L(r_\Left(\ruling_v))$. Similar statement applies to $r_\Right(\ruling_v)$ with $I(v_\Left)$ replaced by $I(v_\Right)=\{\ell+1,\ldots,\ell+r\}$. 
\emph{Let} $\ell(\ruling_v)=r(\ruling_v)=|H(r_\Left(\ruling_v))|$ be the number of fixed points of $r_\Left(\ruling_v)$ (equivalently, $r_\Right(\ruling_v)$). Then $\ell-\ell(\ruling_v)$ and $r-r(\ruling_v)$ are both even. Also, $n_\Left-\ell=n_\Right-r$.

\begin{definition}\label{def:resolution of an involution at a vertex}
For each fixed $\ruling_v\in\NR(v)$, we define a Legendrian tangle (or bordered Legendrian link) $V(\ruling_v)$, called a \emph{resolution} of $V$ with respect to $\ruling_v$, as follows: 

\noindent{}\emph{Define} $V_1(\ruling_v)$ as a Legendrian tangle of type $(n_\Left, n_\Left+\ell(\ruling_v)-\ell)$ obtained from $V_\Left$ such that, for all $i\in U(r_\Left(\ruling_v))$, so $1\leq i<j\coloneqq\ruling_v(i)\leq \ell$, connect the strands $k+i-1$ and $k+j-1$ of $V_\Left$ by a right cusp. Then $V_1(\ruling_v)$ consists of $\hat{c}_\Right(\ruling_v)=(\ell-\ell(\ruling_v))$ right cusps and some additional crossings.  We term these crossings as \emph{markings}.


\noindent{}Say, $H(r_\Left(\ruling_v))=\{h_1^\Left<\ldots<h_{\ell(\ruling_v)}^\Left\}$ and $H(r_\Right(\ruling_v))=\{h_1^\Right<\ldots<h_{r(\ruling_v)}^\Right\}$.
Recall that $\ruling_v$ determines a bijection $\ruling_v:H(r_\Left(\ruling_v))\xrightarrow[]{\sim} H(r_\Right(\ruling_v))$, which can be represented by a positive braid $\beta(\ruling_v)$ with the minimal number of crossings. We can regard $\beta(\ruling_v)$ as a braid with $n_\Left+\ell(\ruling_v)-\ell$ strands, by adding $k$ and $n_\Left-\ell-k$ parallel strands from top and bottom respectively.\\ 
\emph{Define} $V_2(\ruling_v)$ to be the Legendrian tangle of type $(n_\Left+\ell(\ruling_v)-\ell,n_\Left+\ell(\ruling_v)-\ell=n_\Right+r(\ruling_v)-r)$ obtained from $(V_1(\ruling_v))_\Right$ by adding the braid $\beta(\ruling_v)$ from the right hand side.

\noindent{}\emph{Define} $V_3(\ruling_v)$ as a Legendrian tangle of type $(n_\Right+r(\ruling_v)-r,n_\Right)$ obtained from $V_\Right$ as follows: for all $\ell+i\in U(r_\Right(\ruling_v))$, so $\ell+1\leq \ell+i<\ell+j\coloneqq\ruling_v(\ell+i)\leq \ell+r$, connect the strands $k+i-1$ and $k+j-1$ of $V_\Right$ by a left cusp. Then $V_3(\ruling_v)$ consists of $(r-r(\ruling_v))/2$ left cusps and some additional crossings. Again, we term these crossings as \emph{markings}.

\noindent{}Finally, \emph{define} $V(\ruling_v)\coloneqq V_1(\ruling_v)\circ V_2(\ruling_v)\circ V_3(\ruling_v)$ as the concatenation of the Legendrian tangles above. The Maslove potential of $V(\ruling_v)$ is induced from that of $V$. We impose a \emph{base point} at each right cusp of $V(\ruling_v)$.

The construction of $V(\ruling_v)$ is not unique, but the ambiguity is not essential.
\end{definition}

\begin{definition}\label{def:normal ruling for V}
\emph{Define} $\NR'(V(\ruling_v),\ruling_\Left,\ruling_\Right)\subset\NR(V(\ruling_v),\ruling_\Left,\ruling_\Right)$ to be the set of $\ZZ$-graded normal rulings $\ruling'$ for $V(\ruling_v)$ such that
\begin{enumerate}
\item
$\ruling'|_{(V(\ruling_v))_\Left}=\ruling_\Left, \ruling'|_{(V(\ruling_v))_\Right}=\ruling_\Right$;
\item
Any marking of $V(\ruling_v)$ is not a switch, i.e. all the switches of $\ruling'$ are contained in $V_2(\ruling_v)$;
\item
Any two strands of $(V_2(\ruling_v))_\Left$ contained in $H(r_\Left(\ruling_v))$, are not paired by
$\ruling'|_{(V_2)_\Left}\in\NR((V_2)_\Left)$.
\end{enumerate}
\emph{Define} $\NR(V,\ruling_\Left,\ruling_\Right)$ to be the set of $\ruling=(\ruling_v,\ruling')$ such that $\ruling_v\in\NR(v)$, and $\ruling'\in\NR'(V(\ruling_v),\ruling_\Left,\ruling_\Right)$.
Any $\ruling\in\NR(V,\ruling_\Left,\ruling_\Right)$ is called a \emph{normal ruling} for $V$. 
\end{definition}

\begin{remark}
The resolution $V_{\ruling_v}$ described in Definition \ref{def:resolution at a vertex} is different from the $V(\ruling_v)$ defined above. However, from the definition, there is a canonical identification $\NR'(V(\ruling_v),\ruling_\Left,\ruling_\Right)\isomorphic\bfR(V_{\ruling_v},M_{\ruling_v};\ruling_\Left,\ruling_\Right)$, with the latter described in Definition \ref{def:normal ruling LG}. Hence, there is a canonical identification $\NR(V,\ruling_\Left,\ruling_\Right)\isomorphic\bfR(V,\ruling_\Left,\ruling_\Right)$ between the set of normal rulings for $V$ defined above, and that in Definition \ref{def:normal ruling LG}. 
\end{remark}

\subsubsection{Ruling decomposition}\label{subsubsec:ruling decomposition for V}

Now, we have seen that $\aug(V,\epsilon_\Left,\ruling_v,\ruling_\Right;\field)$ is non-empty if and only if $\aug^2(V,\epsilon_\Left,\ruling_v,\ruling_\Right;\field)$ is non-empty. It suffices to determine the structure of $\aug^2(V,\epsilon_\Left,\ruling_v,\ruling_\Right;\field)$.

Recall that there is a ruling decomposition \cite[Thm.5.10]{Su2017} of the augmentation variety for $V(\ruling_v)$:
\begin{equation*}
\aug(V(\ruling_v),\epsilon_\Left,\ruling_\Right;\field)=\coprod_{\ruling'\in\NR(V(\ruling_v);\ruling_\Left,\ruling_\Right)}\aug^{\ruling'}(V(\ruling_v),\epsilon_\Left,\ruling_\Right;\field)
\end{equation*}
with
\begin{equation*}
\aug^{\ruling'}(V(\ruling_v),\epsilon_\Left,\ruling_\Right;\field)\isomorphic (\field^\times)^{-\chi(\ruling')+B}\times \field^{r(\ruling')}
\end{equation*}
Observe that the crossings of degree 0 (excluding the markings) are in one-to-one correspondence with the pairs of $1\leq i<j\leq \ell$ in $U(\ruling_v)$ such that $\mu(i)=\mu(j)$ and $\ell+1\leq\ruling_v(j)<\ruling_v(i)\leq \ell+r$. \emph{Denote} by $z_{i,j}$ the corresponding generators of $\cA(V(\ruling_v))=\cA(\Res(V(\ruling_v)))$. 
\emph{Define} a sub-variety $\aug^2(V(\ruling_v),\epsilon_\Left,\ruling_\Right;\field)\subset\aug(V(\ruling_v),\epsilon_\Left,\ruling_\Right;\field)$, whose augmentations $\epsilon$ satisfy
\begin{enumerate}
\item
$\epsilon(a)=0$ if $a$ is a marking in $V(\ruling_v)$;
\item
The normal ruling determined by $\epsilon|_{(V_2(\ruling_v))_\Left}$ satisfies Definition \ref{def:normal ruling for V}.(3). Equivalently, for any Reeb chord $b$ connecting two strands of $(V_2(\ruling_v))_\Left$ contained in $H(r_\Left(\ruling_v))$, we have $\epsilon|_{(V_2(\ruling_v))_\Left}(b)=0$.
\end{enumerate} 
For each $\ruling'\in\NR'(V(\ruling_v);\ruling_\Left,\ruling_\Right)$, \emph{define} $\aug^{\ruling',2}(V(\ruling_v),\epsilon_\Left,\ruling_\Right)\subset\aug^{\ruling'}(V(\ruling_v),\epsilon_\Left,\ruling_\Right;\field)$ similarly, with $\aug^{\ruling'}(V(\ruling_v),\epsilon_\Left,\ruling_\Right;\field)$ given in Definition \ref{def:ruling strata}.

\begin{lemma}\label{lem:identification of aug var for resolution}
There is a natural identification
\begin{align*}
\aug^2(V,\epsilon_\Left,\ruling_v,\ruling_\Right;\field)&\isomorphic\aug^2(V(\ruling_v),\epsilon_\Left,\ruling_\Right;\field)\\
d_c=(d^\Left,\identity,d_c^v)&\mapsto (d^\Left,\epsilon(z_{i,j})=-\epsilon_{i,j}) 
\end{align*}
where $\epsilon_{i,j}$ is defined by $d_c^v$ as in Equation \eqref{eqn:d_c^v}.
And, we have a ruling decomposition:
\begin{equation*}
\aug^2(V(\ruling_v),\epsilon_\Left,\ruling_\Right;\field)\isomorphic\amalg_{\ruling'\in\NR'(V(\ruling_v);\ruling_\Left,\ruling_\Right)}\aug^{\ruling',2}(V(\ruling_v),\epsilon_\Left,\ruling_\Right;\field)
\end{equation*}
with
\begin{equation*}
\aug^{\ruling',2}(V(\ruling_v),\epsilon_\Left,\ruling_\Right;\field)\isomorphic (\field^\times)^{s(\ruling')}\times \field^{r(\ruling')}.
\end{equation*}
where $s(\ruling')$ is the number of switches, and $r(\ruling')$ is the number of $\ZZ$-graded returns (excluding the markings), of $\ruling'$ in $V(\ruling_v)$ as in Definition \ref{def:returns and departures}.
\end{lemma}

\begin{proof}
The first result follows from the identification between augmentations and Morse complexes \cite[Lem.5.2]{Su2017}.
What remains follows from the ruling decomposition \cite[Thm5.10]{Su2017} for $\aug(V(\ruling_v),\epsilon_\Left,\ruling_\Right;\field)$.
\end{proof}

Now, generalizing Lemma \ref{lem:aug var elementary}~(4), we have

\begin{proposition}\label{prop:ruling decomposition}
For any $\epsilon_\Left\in\aug^{\ruling_\Left}(V_\Left;\field)$, there is a natural decomposition 

\begin{equation*}
\aug(V,\epsilon_\Left,\ruling_\Right;\field)=\coprod_{\ruling\in\NR(V,\ruling_\Left,\ruling_\Right)}\aug^{\ruling}(V,\epsilon_\Left,\ruling_\Right;\field)
\end{equation*}
over the finite set $\NR(V,\ruling_\Left,\ruling_\Right)$ ($=\bfR(V;\ruling_\Left,\ruling_\Right)$), and for all $\ruling=(\ruling_v,\ruling')\in\NR(V,\ruling_\Left,\ruling_\Right)$, we have
\begin{equation*}
\aug^{\ruling}(V,\epsilon_\Left,\ruling_\Right;\field)\isomorphic (\field^\times)^{-\chi(\ruling)+\hat{B}}\times\field^{\hat{r}(\ruling)+A(\ruling)}
\end{equation*}
where $\hat{r}(\ruling)$ is defined as follows:
\emph{Let} $u(v)$ be the number of degree 0 crossings in Ng's resolution $\Res(v)$ of $v$. \emph{Let} $c(\ruling')$ be the number of degree 0 crossings (excluding the markings) in $V(\ruling_v)$. \emph{Define} $A_b(r_\Left(\ruling_v))$ as in Definition \ref{def:indices for an involution at a trivial tangle}. \emph{Define} $r(\ruling')$ as in Lemma \ref{lem:identification of aug var for resolution}.  Then
\begin{equation*}
\hat{r}(\ruling)\coloneqq r(\ruling')+u(v)-A_b(r_\Left(\ruling_v))-c(\ruling')
\end{equation*}  
\end{proposition}

\begin{proof}
By Equations (\ref{eqn:decomposition 1 for V}), (\ref{eqn:decomposition 2 for V}), Lemma \ref{lem:partially canonical differential at a vertex}, and Lemma \ref{lem:identification of aug var for resolution}, the proof is a direct calculation.
\end{proof}

\begin{remark}
By combining with \cite[Thm.5.10]{Su2017}, the gluing property of augmentation varieties implies that, the above ruling decomposition for $V$ generalizes directly to all bordered Legendrian graphs when we impose a base point at each right cusp and each left half-edge of a vertex.
\end{remark}

\bibliographystyle{abbrv}
\bibliography{references}

\end{document}